\documentclass[english,reqno,11pt,empty]{amsart}

\usepackage[margin=3cm,top=2.5cm,bottom=2.5cm]{geometry}

\usepackage{amsmath, enumerate, amssymb, amsthm, mathrsfs,bm,enumerate} 
\usepackage{enumitem,microtype,esint} 
\usepackage[dvipsnames]{xcolor}
\usepackage{hyperref}
\hypersetup{
	colorlinks=true,
	citecolor=blue!60!black,
	linkcolor=red!60!black,
	urlcolor=green!40!black,
	filecolor=yellow!50!black,
	breaklinks=true,
	pdfpagemode=UseNone,
	bookmarksopen=false,
}
\usepackage{natbib}
\usepackage[many]{tcolorbox}

\usepackage{physics}
\usepackage{amsmath}
\usepackage{tikz}
\usepackage{mathdots}
\usepackage{yhmath}
\usepackage{cancel}
\usepackage{color}
\usepackage{siunitx}
\usepackage{array}
\usepackage{multirow}
\usepackage{amssymb}
\usepackage{gensymb}
\usepackage{tabularx}
\usepackage{extarrows}
\usepackage{booktabs}
\usetikzlibrary{fadings}
\usetikzlibrary{patterns}
\usetikzlibrary{shadows.blur}
\usetikzlibrary{shapes}

\newcommand{\re}{\mathrm{Re}\,}

\newcommand{\R}{\mathbb R}
\newcommand{\C}{\mathbb C}
\newcommand{\T}{\mathbb T}
\renewcommand{\S}{\mathbb S}
\newcommand{\N}{\mathbb N}
\newcommand{\Z}{\mathbb Z}

\renewcommand{\AA}{\mathcal A}
\newcommand{\BB}{\mathcal B}
\newcommand{\CC}{\mathcal C}
\newcommand{\OO}{\mathcal O}
\newcommand{\DD}{\mathcal D}

\newcommand{\FF}{\mathcal F}
\newcommand{\GG}{\mathcal X}
\newcommand{\HH}{\mathcal H}

\newcommand{\LL}{\mathcal L}
\newcommand{\MM}{\mathcal M}
\newcommand{\PP}{\mathsf P}
\newcommand{\QQ}{\mathcal Q}
\newcommand{\RR}{\mathcal R}
\renewcommand{\SS}{\mathcal S}
\newcommand{\TT}{\mathcal T}
\newcommand{\UU}{\mathcal U}
\newcommand{\VV}{\mathcal Y}

\newcommand{\BBB}{\mathscr B}

\newcommand{\FFF}{\mathscr F}
\newcommand{\HHH}{\mathscr H}
\newcommand{\GGG}{\mathscr X}

\newcommand{\MMM}{\mathscr M}

\newcommand{\Ss}{H}       % Principal (le Space small)
\newcommand{\Ssp}{H^{\bullet}}    % Dissipation (le Space small plus)
\newcommand{\Ssm}{H^{\circ}}    % Dual de dissipation (le Space small minus)
\newcommand{\SSs}{\HH}       					% Principal
\newcommand{\rSSs}[1]{\HH^{#1}}	 		% Principal en prÃ©cisant la rÃ©gularitÃ©
\newcommand{\SSsp}{\HH^{\bullet}}    		 	 	 % Dissipation
\newcommand{\rSSsp}[1]{\HH^{{\bullet}, #1}}     % Dissipation en prÃ©cisant la rÃ©gularitÃ©
\newcommand{\SSsm}{\HH^{\circ}}   			 	% Dual de dissipation
\newcommand{\rSSsm}[1]{\HH^{{\circ}, #1}}    % Dual de dissipation en prÃ©cisant la rÃ©gularitÃ©
\newcommand{\SSSs}{\HHH}       			% Principal
\newcommand{\rSSSs}[1]{\HHH^{#1}}     % Principal en prÃ©cisant la rÃ©gularitÃ©

\newcommand{\Sl}{X}       % Principal (le Space large)
\newcommand{\Slp}{X^{\bullet}}    % Dissipation (le Space large plus)
\newcommand{\Slm}{X^{\circ}}    % Dual de dissipation (le Space large minus)
\newcommand{\SSl}{\GG}       					% Principal
\newcommand{\rSSl}[1]{\GG^{#1}}	 		% Principal en prÃ©cisant la rÃ©gularitÃ©
\newcommand{\SSlp}{\GG^{\bullet}}    		 	 	 % Dissipation
\newcommand{\rSSlp}[1]{\GG^{\bullet, #1}}     % Dissipation en prÃ©cisant la rÃ©gularitÃ©
\newcommand{\SSlm}{\GG^{\circ}}   			 	% Dual de dissipation
\newcommand{\rSSlm}[1]{\GG^{{\circ}, #1}}    % Dual de dissipation en prÃ©cisant la rÃ©gularitÃ©
\newcommand{\SSSl}{\GGG}       			% Principal
\newcommand{\rSSSl}[1]{\GGG^{#1}}     % Principal en prÃ©cisant la rÃ©gularitÃ©

\newcommand{\Sh}{\bm{X}}       % Principal (le Space large)
\newcommand{\Shp}{\bm{X}^{\bullet}}    % Dissipation (le Space large plus)
\newcommand{\Shm}{\bm{X}^{\circ}}    % Dual de dissipation (le Space large minus)
\newcommand{\SSh}{\bm{\GG}}       					% Principal
\newcommand{\rSSh}[1]{\bm{\GG}^{#1}}	 		% Principal en prÃ©cisant la rÃ©gularitÃ©
\newcommand{\SShp}{\bm{\GG}^{\bullet}}    		 	 	 % Dissipation
\newcommand{\rSShp}[1]{\bm{\GG}^{\bullet, #1}}     % Dissipation en prÃ©cisant la rÃ©gularitÃ©
\newcommand{\SShm}{\bm{\GG}^{\circ}}   			 	% Dual de dissipation
\newcommand{\SSSh}{\pmb{\GGG}}       			% Principal
\newcommand{\SSSm}{\FFF}       % Principal
\newcommand{\rSSSm}[1]{\FFF^{#1}}       % Principal (en prÃ©cisant la rÃ©gularitÃ©)

\newcommand{\Sg}{Y}       % Principal (le Space general)
\newcommand{\Sgp}{Y^{\bullet}}    % Dissipation (le Space general plus)
\newcommand{\Sgm}{Y^{\circ}}    % Dual de dissipation (le Space general minus)
\newcommand{\SSg}{\VV}       % Principal
\newcommand{\rSSg}[1]{\VV^{#1}}       % Principal (en prÃ©cisant la rÃ©gularitÃ©)
\newcommand{\SSgp}{\VV^{\bullet}}    % Dissipation
\newcommand{\rSSgp}[1]{\VV^{\bullet, #1}}    % Dissipation (en prÃ©cisant la rÃ©gularitÃ©)
\newcommand{\SSgm}{\VV^{\circ}}    % Dual de dissipation
\newcommand{\rSSgm}[1]{\VV^{\circ, #1}}    % Dual de dissipation (en prÃ©cisant la rÃ©gularitÃ©)

\newcommand{\Sr}{W}       % Principal (le Space regular)
\theoremstyle{plain}
\newtheorem{theo}{Theorem}
\newtheorem{prop}[theo]{Proposition}
\newtheorem{lem}[theo]{Lemma}

\newtheorem{exe}[theo]{Example}

\newtheorem{hypL}{\textsc{Structural Linear Assumptions}}
\newtheorem*{hypLE}{\textsc{Assumptions on enlarged spaces}}
\newtheorem{hypQ}{\textsc{Structural Quadratic Assumptions}}
\newtheorem*{hypQE}{\textsc{Structural Assumptions -- non symmetric case}}
\newtheorem{defi}[theo]{Definition}
\newtheorem{cor}[theo]{Corollary}
\theoremstyle{remark}
\newtheorem{rem}[theo]{Remark}

\numberwithin{equation}{section}
\numberwithin{theo}{section}

\def\le{\leqslant}
\def\ge{\geqslant}
\def\leq{\leqslant}
\def\geq{\geqslant}

\DeclareMathOperator{\supp}{supp}

\DeclareMathOperator{\Span}{Span}
\DeclareMathOperator{\id}{Id}

\DeclareMathOperator{\range}{\mathrm{Range}}
\DeclareMathOperator{\nul}{\mathrm{Ker}}

\newcommand{\ini}{\textnormal{in}}
\newcommand{\err}{\textnormal{err}}

\newcommand{\ns}{\textnormal{NS}}

\newcommand{\hyd}{\textnormal{hydro}}
\newcommand{\mix}{\textnormal{mix}}
\newcommand{\kin}{\textnormal{kin}}
\newcommand{\sym}{\mathrm{sym}}

\def\eps{{\varepsilon}}

\newcommand{\Nt}{|\hskip-0.04cm|\hskip-0.04cm|}

\newcommand{\la}{\langle}
\newcommand{\ra}{\rangle}
\newcommand{\lla}{\left\langle}
\newcommand{\rra}{\right\rangle}

\newcommand{\rectL}{\mathbb{L}}
\newcommand{\rectsL}{\mathbb{M}}

\newcommand{\BurA}{{\mathbf{A}}}
\newcommand{\BurB}{{\mathbf{B}}}

\makeatletter
\newsavebox{\@brx}
\newcommand{\dlla}[1][]{\savebox{\@brx}{\(\m@th{#1\langle}\)}%
	\mathopen{\copy\@brx\kern-0.5\wd\@brx\usebox{\@brx}}}
\newcommand{\drra}[1][]{\savebox{\@brx}{\(\m@th{#1\rangle}\)}%
	\mathclose{\copy\@brx\kern-0.5\wd\@brx\usebox{\@brx}}}
\makeatother

\renewcommand{\d}{\mathrm{d}}

\newcommand{\Id}{\mathrm{Id}}

\newcommand{\hypst}[1]{\textnormal{\textbf{(#1)}}}

\newcommand{\dom}{\mathscr{D}}

\allowdisplaybreaks

\newcommand{\thttl}[1]{\textit{\textbf{#1}}}

\newcommand{\step}[2]{\medskip\noindent\textit{Step #1: #2.}}

\newcommand{\Wave}{\textnormal{wave}}
\newcommand{\disp}{\textnormal{disp}}
\newcommand{\Inc}{\textnormal{inc}}
\newcommand{\Bou}{\textnormal{Bou}}

\usepackage{natbib}
\usepackage[many]{tcolorbox}

\title[Hydrodynamic limits for kinetic equations]{Hydrodynamic limits for kinetic equations preserving mass, momentum and energy: a spectral and unified approach in the presence of a spectral gap}
\def\theauthor{P. Gervais, B. Lods}

\hypersetup{pdfauthor={\theauthor}}
\author{P. Gervais}
 
\address{Universit\`{a} degli Studi di Torino, Department of Economics, Social Sciences, Applied Mathematics and Statistics ``ESOMAS'', Corso Unione Sovietica, 218/bis, 10134 Torino, Italy.} \email{pierre.gervais@inria.fr}

\author{B. Lods}
\address{Universit\`{a} degli Studi di Torino \& Collegio Carlo Alberto, Department of Economics, Social Sciences, Applied Mathematics and Statistics ``ESOMAS'', Corso Unione Sovietica, 218/bis, 10134 Torino, Italy.}
\email{bertrand.lods@unito.it}

\date{\today}

\begin{document}

	\maketitle 
	
	\begin{abstract}
		Triggered by the fact that, in the hydrodynamic limit, several different kinetic equations of physical interest all lead to the same Navier-Stokes-Fourier system, we develop in the paper an abstract framework which allows to explain this phenomenon. The method we develop can be seen as a significant improvement of known approaches for which we fully exploit some structural assumptions on the linear and nonlinear collision operators as well as a good knowledge of the Cauchy theory for the limiting equation. {In particular, we fully exploit the fact that the collision operator is preserving both momentum and kinetic energy.}  We adopt a perturbative framework in a Hilbert space setting and first develop a general and fine spectral analysis of the linearized operator and its associated semigroup. Then, we introduce a splitting adapted to the various regimes (kinetic, acoustic, hydrodynamic) present in the kinetic equation which allows, by a fixed point argument, to construct a solution to the kinetic equation and prove the convergence towards suitable solutions to the Navier-Stokes-Fourier system. Our approach is robust enough to treat, in the same formalism, the case of the Boltzmann equation with hard and moderately soft potentials, with and without cut-off assumptions, as well as the Landau equation for hard and moderately soft potentials in presence of a spectral gap. New well-posedness and strong convergence results are obtained within this framework. In particular, for initial data with algebraic decay with respect to the velocity variable, our approach provides the first result concerning the strong Navier-Stokes limit  from Boltzmann equation without Grad cut-off assumption or Landau equation. The method developed in the paper is also robust enough to apply, at least at the linear level, to quantum kinetic equations for Fermi-Dirac or Bose-Einstein particles.\end{abstract}
	
\tableofcontents
	
	\section{Introduction}

	\subsection{From nonlinear collisional model to Navier-Stokes-Fourier system}
	
	The connection between the  Navier-Stokes  and Boltzmann equations originates seemingly from the work  \cite{H1912} regarding the mathematical treatment of the axioms of physics. Since this original idea, the derivation of  suitable hydrodynamic equations from nonlinear kinetic equations  has attracted a lot of attention in the recent years. We will review later in this introduction several of the main contributions in the field, illustrating in particular the large variety of models considered in the literature, but we wish to focus here on some striking universal features shared by several binary collisional models in the diffusive scaling. 
	
	Namely, for kinetic equations in adimensional form given by the evolution of a {particles number} density $f^{\eps}(x,v,t)$ {(with $x \in \R^d$ denoting position, $v \in \R^d$ the velocity, $t \ge 0$ the time and $\eps > 0$ the mean free path between particles collisions)} 
	\begin{equation}\label{eq:Kin-Intro}
		\partial_{t}f^{\eps}+\frac{1}{\eps}v\cdot \nabla_{x}f^{\eps} =\frac{1}{\eps^{2}}\LL f^{\eps} + \frac{1}{\eps}\QQ(f^{\eps},f^{\eps}), \qquad f^{\eps}(x,v,0)=f_{\ini}(x,v)\end{equation}
	for some suitable \emph{linear operator} $\LL$ and \emph{quadratic operator} $\QQ$, it has been shown in various contexts that, in the limit $\eps \to 0$, the solution $f^{\eps}$ converges (in some sense to determine) towards a ``macroscopic'' distribution $f_{\ns}(x,v,t)$ of the form
	\begin{equation}\label{eq:11}
		f_{\ns}(x,v,t)= \left(\varrho(t, x) + u(t, x) \cdot v   + C_{0}\theta(t, x)\left( |v|^2 - E \right)\right) \mu(v)\end{equation}
	where $C_{0} >0,E >0$ are depending {only on the universal distribution $\mu$ (independent of $f_\ini$)}. More surprisingly, it is also known that the triple of functions 
	$$(\varrho(t,x),u(t,x),\theta(t,x)) \in \R \times \R^{d}\times \R$$
	{associated to} the macroscopic mass, mean velocity and temperature of the gas
	are suitable solutions to the Navier-Stokes-Fourier system
	\begin{equation}\label{eq:NSFint}
		\begin{cases}
			\partial_{t}u-\kappa_{\Inc}\,\Delta_{x}u + \vartheta_{\Inc}\,u\cdot \nabla_{x}\,u = \nabla_{x}p \,,\\[6pt]
			\partial_{t}\,\theta-\kappa_{\Bou}\,\Delta_{x}\theta +\vartheta_{\Bou}\,u\cdot \nabla_{x}\theta=0,\\[8pt]
			\nabla_{x}\cdot u=0\,, \qquad \nabla_x\left(\varrho + \theta\right)= 0\,,
		\end{cases}
	\end{equation}
	where the third line describe respectively the \emph{incompressibility} condition of the fluid and the \emph{Boussinesq relation} between mass and temperature. The pressure of the fluid $p$ is here above obtained implicitly as a Lagrange multiplier associated to the incompressibility constraint $\nabla_{x}\cdot u=0.$ 
	
	The striking phenomena we wish to discuss in this paper is the fact that a large variety of kinetic models described by \eqref{eq:Kin-Intro} provide in the hydrodynamic limit the \emph{same Navier-Stokes-Fourier system} \eqref{eq:NSFint}, making that system a \emph{universal hydrodynamic limit} for \eqref{eq:Kin-Intro}. The only memory of the original equation \eqref{eq:Kin-Intro} kept in  the system \eqref{eq:NSFint} is encapsulated in the various coefficients:
	$$\kappa_{\Inc} >0, \quad \kappa_{\Bou} >0,$$
	which represent the viscosity and thermal conductivity, as well as and $\vartheta_\Inc,\vartheta_\Bou$ , all of being defined explicitly in terms of the operators $\LL$ and $\QQ$ that encode the collision process. We refer to Section \ref{sec:detail} for more details on those coefficients. 
	
	Recall that, in the kinetic equation \eqref{eq:Kin-Intro}, the unknown $f^\eps(x,v,t)$ denotes typically the density of particles having position $x \in \R^{d}$ and velocity $v \in \R^{d}$ at time $t\geq0$ while the parameter $\eps$ represents the \emph{Knudsen number} which is proportional to the mean free path between collisions. Typically, small values of $\eps$ correspond to a case in which particles suffer a very large number of collisions. The hydrodynamic limit $\eps \to 0$ consists in assuming that the mean free path is negligible when compared to the typical physical scale length. We refer to \cite{Cercignani,Sone} for details on the kinetic description of gases.
	
	That kinetic equation \eqref{eq:Kin-Intro} leads to \eqref{eq:NSFint} in the limit $\eps \to 0^{+}$ is a well-understood fact that have been proven, for several type of solutions and various mode of convergence, in the case of the classical Boltzmann equation for which 
	\begin{equation}\label{eq:QQBoltz}
		\QQ(f,f)(v)=\QQ_{\mathrm{Boltz}}(f,f)=\int_{\R^{d}\times\S^{d-1}}B(|v-v_{*}|,\sigma)\left[f(v')f(v_{*}')-f(v)f(v_{*})\right]\d v_{*}\d\sigma\end{equation}
	where 
	$$v'=\frac{v+v_*}{2}+\frac{|v-v_*|}{2}\sigma, \qquad v_{*}'=\frac{v+v_*}{2}-\frac{|v-v_*|}{2}\sigma, \qquad \sigma \in \S^{d-1}$$
	and the collision kernel $B(|v-v_{*}|,\sigma)$ is given by
	$$B(|v-v_{*}|,\sigma)=|v-v_{*}|^{\gamma}\,b(\cos\theta), \qquad \cos \theta=\sigma \cdot \frac{v-v_{*}}{|v-v_{*}|}.$$
	The method developed in the paper allows to consider \emph{all kinds of collision kernel} of physical interest, covering the cases of hard {and Maxwell} potentials $(\gamma \ge 0)$ with and without cut-off assumptions as well as that of moderately soft potentials {(without cut-off assumption) for which $b(\cos \theta) \approx \theta^{-(d-1) - 2 s}$ and $\gamma + 2 s \ge 0$}. We refer to Appendix \ref{sec:Landau-Boltz} for details. 
	Besides this Boltzmann model, our approach is also robust enough to treat in the same formalism the case of the Landau equation
	\begin{equation*}
		\begin{split}
			\QQ(f,f)&=\QQ_{\mathrm{Landau}}(f,f)\\
			&=\nabla_{v}\cdot \int_{\R^{d}} |v-v_{*}|^{\gamma+2} \, \Pi_{v-v_{*}}
			\Big\{f(t,v_{*}) \nabla_{v} f(t,v) - f(t,v) {\nabla_{v_{*}} f}(t,v_{*}) \Big\}
			\, \d v_{*}\end{split}\end{equation*}
	where {$\gamma \geq -d$} and 
	$$\Pi_{z}=\mathrm{Id}-\frac{z \otimes z}{|z|^{2}}, \qquad z \in \R^{d} \setminus \{0\}$$
	denotes the projection in the direction orthogonal to $z \in \R^{d},z\neq 0$. As before, our results cover the two cases of hard {or Maxwell $(\gamma \ge 0)$ and moderately soft potentials $(\gamma + 2 \ge 0)$}.   
	For both these models, the solutions to \eqref{eq:Kin-Intro} converges to a solution $f$ given by \eqref{eq:11} where
	$$\mu(v)=(2\pi)^{-\frac{d}{2}}\exp\left(-\frac{|v|^{2}}{2}\right)$$
	is a Maxwellian distribution with  {unit mass, unit energy and mean zero velocity}, which is an equilibrium state of the collision operator $\QQ$, i.e.
	$$\QQ(\mu,\mu)=0$$
	whereas $\LL$ is the linearized operator around that equilibrium, i.e. 
	\begin{equation}\label{eq:linearized}
		\LL f=\QQ(\mu,f)+\QQ(f,\mu)\end{equation}
	for any suitable $f$ for which this makes sense. \medskip

	In this paper, we introduce an abstract framework allowing to recover the above universal behaviour, as well as the well-posedness of \eqref{eq:Kin-Intro} in a perturbative framework. Even tough the Boltzmann and Landau equations are the two main models we have in mind as field of applications of our method, we wish again to point out that we are able to prove the convergence towards \eqref{eq:NSFint} for much more general models than those ones. In particular, we can handle general linear operator $\LL$ and do not ask for the rest of the analysis that $\LL$ and $\QQ$ are related through \eqref{eq:linearized}. 
	
	The abstract framework developed in the paper is very general and robust and rely only on core assumptions about the linear part $\LL$ and the quadratic part $\QQ$. In particular, our approach can also  be adapted to handle the case of the Boltzmann equation with \emph{relativistic velocities} and it is flexible enough to also encompass, at the price of some modifications, the case of  quantum kinetic model (for which the collision operator is actually trilinear). Work is in progress in that direction in order to prove the strong convergence of solutions to the Boltzmann-Fermi-Dirac equation towards the above NSF system \eqref{eq:NSFint}, see \cite{GL2023}.

	\subsection{Literature review}
	
	As said, the derivation of hydrodynamic limits from linear and nonlinear equation is an important problem which received a lot of attention since the pioneering work of \cite{H1912} and \cite{E1917}. We do not review here the vast literature on the problem of diffusion approximation for transport processes, just referring to the classical references \cite{BLP1979,BSS} and the more recent contributions \cite{GW,BM} and the references therein.
	
For nonlinear collisional models, we refer the reader to {\cite{SR,golse}} for a more exhaustive description of the mathematically relevant results in the field regarding the Boltzmann equation. Depending on the limiting equation and the type of convergence one is interested with, there are mainly three different approaches for the derivation of hydrodynamical limit from the Boltzmann equation: a first approach consists in justifying rigorously suitable (truncated) asymptotic expansions of the solution to the kinetic equation around some hydrodynamic solution
	\begin{equation*}\label{eq:C-E}
		f_{\eps}(t,x,v)=f_{0}(t,x,v)\left(1+\sum_{n}\eps^{n}F_{n}(t,x,v)\right)\end{equation*}
	where, typically $f_{0}(t,x,v)$ is a local Maxwelllian whose macroscopic fields are required to satisfy the limiting fluid model. With such an approach, the works \cite{caflisch} and \cite{demasi} obtained respectively the first rigorous justification of the compressible Euler limit up to the first singular time for the solution of the Euler system and a justification of the incompressible Navier-Stokes limit from Boltzmann equation. The work \cite{G2006} is another important reference on this line of research and we point out that, with such an approach, one is mainly interested with strong solutions  for both the kinetic and fluid equations.
	
	Regarding now \emph{weak solutions} at both the kinetic and fluid models, a very important program  has been introduced in {\cite{BaGoLe1,BaGoLe2}} whose goal was to prove the convergence of the renormalized solutions to the Boltzmann equation towards weak solutions to the compressible Euler system or to the incompressible Navier-Stokes equations. This program has been continued exhaustively and the convergence have been obtained in several important results (see {\cite{golseSR,golseSR1,jiang-masm,lever,lions-masm1,lions-masm2}} to mention just a few).  
	
	The present contribution belongs to the third line of research which investigates \emph{strong solutions close to equilibrium} and exploits a careful spectral analysis of the linearized kinetic equation. Strong solutions to the Boltzmann equation close to equilibrium have been obtained in a weighted $L^{2}$-framework in the work \cite{U1974} and the \emph{local-in-time} convergence of these solutions towards solution to the compressible Euler equations have been derived in {\cite{nishida}}. For the limiting incompressible Navier-Stokes solution, a similar result have been carried out in {\cite{BU1991}} for smooth  global solutions in $\R^{3}$ with a small {initial datum}. The recent work \cite{GT2020} recently removed this smallness assumption, allowing to treat also non global in time solutions to the Navier-Stokes equation. A recent extension to less restrictive integrability conditions has been obtained in \cite{G2023}.  Our work is falling into this framework and is closer in spirit to the work \cite{GT2020} than to \cite{BU1991} since it fully exploits the Cauchy theory of the limiting NSF system. This line of research, complemented for instance with  {\cite{B2015,BMM2019, CC2023}}, exploits a very careful description of the spectrum of the linearized Boltzmann equation derived in {\cite{EP1975}}.  We notice that they are framed in the space $L^{2}(\mu^{-1})$ where the linearized Boltzmann operator is self-adjoint and coercive. The fact that the analysis of \cite{EP1975} has been extended recently in \cite{G2021} to larger functional spaces of the type $L^{2}_{v}(\langle \cdot\rangle^{q})$ opens the gate to some refinements of several of the aforementioned results. We also mention here the work {\cite{zhao}} which deals with an energy method in $L^{2}(\mu^{-1})$ spaces (see also {\cite{guo,guo2}} and~\cite{R2021}) in order to prove the \emph{strong convergence} of the solutions to the Boltzmann or Landau equation towards the incompressible Navier-Stokes equation without resorting to the work of {\cite{EP1975}}. 
	
	Besides the above lines of research and contributions which are dealing mainly with Boltzmann or Landau equation, we wish to point out that other kinetic and fluid models have been considered in the literature. Exhaustive list of contributions to the field is out of reach and we just mention some recent works  spanning from high friction regimes for kinetic models of swarwing (see e.g. \cite{karper,figalli} for the Cucker-Smale model) to the reaction-diffusion limit for Fitzhugh-Nagumo kinetic equations \cite{crevat}. For  fluid-kinetic systems, the literature is even more important, we mention simply here the works \cite{goudona,goudonb} dealing with light or fine particles regimes for the Vlasov-Navier-Stokes system and refer to \cite{daniel} for the more recent advances on the subject. We also mention the challenging study of gas of charged particles submitted to electro-magnetic forces (Vlasov-Maxwell-Boltzmann system) for which several incompressible fluid limits have been derived recently in the monograph \cite{arsenio}.

	\subsection{Objectives of the paper}
	
	The main scope of the paper is threefold:
	\begin{enumerate}
		\item[\textbf{(I)}] First, we  provide a unified framework which allows to capture a large variety of quadratic models and explain the emergence of the universal NSF system \eqref{eq:NSFint} in the hydrodynamic limit. To do so, we provide a general though seemingly minimal set of Assumptions under which the NSF would emerge. Those are structural assumptions on the collision operator $\QQ$ as well as the linear operator $\LL$. They are related to physical properties of the kinetic equation: we assume in particular the rotational symmetry of $\LL$ and $\QQ$ due to the isotropy of the collision process as well as usual \emph{local conservation laws} related to mass, bulk velocity and energy. This work can be considered as a quantitative version of the founding paper by \cite{BaGoLe1} in which general collision operators are considered. We refer to Section \ref{scn:general_collisions} for more details.
		%In particular, our structural assumptions cover the case of elastic interactions but are not suited to inelastic interactions  as studied in \cite{ALT}.
		\item[\textbf{(II)}] Second, within the abstract framework considered here, we aim to provide a very fine spectral analysis of the linearized operator $\LL-v\cdot \nabla_{x}$ as well as a thorough description of the decay {and regularization} properties of the associated semigroup. As in previous contributions to the field, such an analysis is performed in a Fourier-based formalism under which the linearized operator of peculiar interest becomes 
		$$\LL_{\xi}:=\LL-i(\xi\cdot v)$$
		where the transport term has been transformed in the more tractable multiplication operator by $i(v\cdot \xi)$ in Fourier variable (see Section \ref{sec:detail} for details). The advantage of working in this Fourier-based formalism is that it encompass the various scales of frequencies according to
		$$|\xi| \simeq \eps, \quad |\xi| \ll \eps \quad \text{ or } \eps \ll |\xi|$$
		which let emerge the various (kinetic, hydrodynamic, dispersive) regimes of description at the  linearized level. Under the structural assumptions on the linear part $\LL$, we give a full description of the spectrum of $\LL_{\xi}$, including the asymptotic expansion of both its leading eigenvalues and associated spectral projectors in the regime of  small frequencies, $|\xi| \simeq 0$. Such a spectral description yields to result similar to those obtained in the seminal work \cite{EP1975} but we provide here a \textbf{\textit{completely new and more direct approach to this question}} in the unified and abstract framework. Our new approach is based upon a combination of Kato's perturbation theory \cite{K1995} and enlargment and factorization techniques from \cite{GMM2017}.

		\item[\textbf{(III)}] Finally,  we provide a  \emph{strong convergence} result from solutions $f^\eps$ to \eqref{eq:Kin-Intro} towards the solution $f$ given in \eqref{eq:11} associated to \eqref{eq:NSFint}. Moreover, the strong convergence result is in essence \emph{quantitative} since we carefully estimate the difference between the solution $f^{\eps}$ and the solution $f=f_{\ns}$ by introducing a suitable splitting of $f^{\eps}$ which, roughly speaking, can be given as
		$$f^{\eps}=f_{\ns} + h^{\eps}_{\err}$$
		where $h_{\err}^{\eps}$ is an error term that that we aim to estimate as
		$$\sup_{t\ge t_{*}}\|h_{\err}^{\eps}(t)\| \leq \beta(\eps), \qquad \lim_{\eps\to0}\beta(\eps)=0$$
		for any $t_{*} >0$ and  some quantified error estimate $\beta(\eps)$. Here above, the norm $\|\cdot\|$ is quite involved and takes into account several phenomena that produce different convergence rates (e.g, acoustic waves, dissipation of entropy). The restriction $t \geq t_{*}$ stems from the difficult task of estimating the initial layer and can be removed in the case of \emph{well-prepared} initial datum (see Theorem \ref{thm:hydrodynamic_limit}
		for a precise statement and a complete description of the difference $f^{\eps}-f_{\ns}$). \end{enumerate}
	
	As a by-product of our third objective \textbf{(III)} here above, we show, for this variety of model, a close-to-equilibrium Cauchy theory for the kinetic equation \eqref{eq:Kin-Intro} for suitably small value of $\eps$. 
	One of the main feature of our approach is that, inspired by the work \cite{GT2020}, our methodology is ‘‘top-down" from the limit equation to the kinetic equation rather than ‘‘bottom-up" as usually done. This means that, as far as possible, we adapt our approach to the existing Cauchy theory for the limiting system \eqref{eq:NSFint} and deduce the Cauchy theory for the kinetic equation \eqref{eq:Kin-Intro} by comparing it to the limiting equation \eqref{eq:NSFint} for small values of $\eps.$ This is achieved through a suitable fixed-point argument involving fluctuations around the solution $f_{\ns}$. The fixed-point argument is based upon a simple use of Banach fixed point theorem or, for the more general case considered in the paper, by the convergence of a suitable scheme mimicking Picard iteration. Such an approach allows in particular to obtain well-posedness results \emph{without} any smallness assumption on the initial datum $f_{\ini}$ but only under some smallness assumption on the scaling parameter $\eps$ yielding several improvements of known results in the field.

	Among the novelty of the paper, as just said, we adapt our approach to the existing Cauchy theory for the limiting system \eqref{eq:NSFint}. A lot of efforts in the present paper are given to adapt several tools   used in the estimates of the Navier-Stokes system and, in particular, we resort to several Fourier analysis tools as developed in \cite{BCD2011} to treat nonlinear terms. We in particular adapt the paraproduct estimates described in \cite{BCD2011} to handle $x$-estimates of products of the form $\QQ(f,g)$ (see Appendix \ref{scn:littlewood-paley} for more details). The case $d=2$ needs in particular a peculiar treatment for which we face several technical difficulties to handle nonlinear estimates.

	%Regarding the method used to achieve the above  objectives, as in previous contributions to the field, we study \eqref{eq:Kin-Intro} in its equivalent formulation
	%\begin{multline}\label{eq:FoKin}
	%\partial_{t}\widehat{f}^{\eps}(\xi,v,t)+\frac{i}{\eps}(\xi \cdot v)\widehat{f}(\xi,v,t)=\frac{1}{\eps^{2}}\LL \widehat{f}^{\eps}(\xi,v,t)\\
	%+\frac{1}{\eps}\int_{\R^d}\QQ(\widehat{f}^{\eps}(\xi,\cdot,t),\widehat{f}^{\eps}(\xi-\zeta,\cdot,t))(v)\d\zeta\end{multline}
	%where 
	%$$\widehat{f}^{\eps}(\xi,v,t)=\int_{\R^{d}}e^{-i \xi \cdot x}f^{\eps}(x,v,t)\d x$$
	%is the Fourier transform with respect to the position variable $x\in \R^{d}$ and we exploited the fact that $\LL$ and $\QQ$ are local in $x.$ In this framework, the linearized operator of peculiar interest becomes 
	%$$\LL_{\xi}:=\LL-i(\xi\cdot v)$$
	%where the transport term has been transformed in the more tractable multiplication operator by $i(v\cdot \xi)$ in Fourier variable. The idea of studying \eqref{eq:Kin-Intro} in the equivalent form \eqref{eq:FoKin} originates from the seminal work \cite{EP1975} where a careful spectral analysis of the linearized Boltzmann operator was performed. It enforces somehow the study of both \eqref{eq:Kin-Intro} and \eqref{eq:NSFint} in $L^{2}_{x}$-functional spaces. Here, we push forward this idea and try to extract from it minimal assumptions and optimal estimates for $\QQ$.

	Regarding the method used to achieve the above  objectives, as in previous contributions to the field, we start by studying \eqref{eq:Kin-Intro} without its non-linear part and in Fourier variables:
	\begin{equation}\label{eq:FoKin}
		\partial_{t}\widehat{f}^{\eps}(\xi,v,t)+\frac{i}{\eps}(\xi \cdot v)\widehat{f}(\xi,v,t)=\frac{1}{\eps^{2}}\LL \widehat{f}^{\eps}(\xi,v,t)
		\end{equation}
	where 
	$$\widehat{f}^{\eps}(\xi,v,t)=\int_{\R^{d}}e^{-i \xi \cdot x}f^{\eps}(x,v,t)\d x$$
	is the Fourier transform with respect to the position variable $x\in \R^{d}$ and we exploited the fact that $\LL$ is local in $x.$ In this framework, the linear  operator of peculiar interest becomes 
	$$\LL_{\xi}:=\LL-i(\xi\cdot v)$$
	where the transport term has been transformed in the more tractable multiplication operator by $i(v\cdot \xi)$ in Fourier variable. The idea of studying \eqref{eq:FoKin} originates from the seminal work \cite{EP1975} where a careful spectral analysis of the linearized Boltzmann operator was performed. It enforces somehow the study of both \eqref{eq:Kin-Intro} and \eqref{eq:NSFint} in $L^{2}_{x}$-functional spaces. Here, we push forward this idea and try to extract from it minimal assumptions and optimal estimates for $\LL$ and $\QQ$.
	\color{black}
	
	\subsection{Notations} In all the sequel, given a closed densely defined linear operator on a Banach space $Y$ of functions $f\::v \in \R^{d} \mapsto f(v)\in \C$, 
	$$L\::\:\dom(L) \subset Y \to Y$$
	we denote, for any $\xi \in \R^{d}$, the operator 
	$L_{\xi}\::\:\dom(L_{\xi}) \subset Y \to Y$
	by
	$$\dom(L_{\xi})=\{f \in \dom(L)\;;\; v f \in Y\}\, \qquad L_{\xi}f=f-i(v\cdot \xi)f, \qquad f \in \dom(L_{\xi}).$$
	The spectrum of $L$ is denoted $\mathfrak{S}(L)$ (or $\mathfrak{S}_X(L)$ if it appears necessary to explicit the underlying Banach space) and, for $z \in \C \setminus \mathfrak{S}(L)$, the resolvent of $L$ at $z$ is denoted by 
	$$\RR(z,L)=(z-L)^{-1} \in \BBB(Y)$$
where $\BBB(Y)$ is the space of all bounded linear operators on $Y$ (with its usual norm $\|\cdot\|_{\BBB(Y)}$).
	%
	%
	%When $(\cdot, \cdot)$ denotes the product of scalar valued functions, we will naturally denote for $\C^d$-valued functions $f, g$ and $\MMM_{d \times d}(\C)$-valued functions $A, B$
	%$$(f, g) = \sum_{n=1}^{d} (f_i, g_i), \quad (A, B) = \sum_{i, j=1}^{d} (A_{i, j}, B_{i,j}).$$
	We introduce, for any $a \in \R$ the right-half plane of the complex field $\C$ as
	$$\Delta_{a}:=\left\{z \in \C\,;\;\mathrm{Re} z >a\right\}.$$
	To handle now functions depending on the position variable $x\in \R^d$, we define the {inhomogeneous} Sobolev spaces of order $s \in \R$,
	$$\mathbb{H}^{s}_{x}(\R^{d})=\left\{f \in L^{2}(\R^{d})\;;\;\|f\|_{\mathbb{H}_{x}^{s}}^{2}=\int_{\R^{d}}\langle \xi\rangle^{2s}|\widehat{f}(\xi)|^{2}\d\xi <\infty\right\}$$
	and the homogeneous Sobolev space
	$$\dot{\mathbb{H}}^{s}_{x}(\R^{d})=\left\{f \in \mathscr{S}'_x(\R^{d})\;;\;\|f\|_{\dot{\mathbb{H}}_{x}^{s}}^{2}=\int_{\R^{d}}|\xi|^{2s}|\widehat{f}(\xi)|^{2}\d\xi <\infty\right\}$$
	where $\mathscr{S}'_x(\R^d)$ denotes the space of tempered distributions over $\R^d$. One can identify $\mathbb{H}^{s}_{x}(\R^{d})$ as the space of tempered distribution $f \in \mathscr{S}'_{x}(\R^{d})$ such that $\left(\Id-\Delta_{x}\right)^{s/2}f \in L^{2}_{x}(\R^{d})$ whereas $\dot{\mathbb{H}}_{x}^{s}(\R^{d})$ is the space of mappings $f \in \mathscr{S}'_{x}(\R^{d})$ such that $(-\Delta_{x})^{s/2}f = |\nabla_x |^s  f \in L^{2}_{x}(\R^{d})$.
	We also introduce the homogeneous Besov spaces for $p, q \in [1, \infty]$ and $s \in \R$
	$$\dot{\mathbb{B}}^s_{p, q}\left( \R^d \right) = \left\{ f \in \mathscr{S}'_x \left(\R^d\right) \, ; \, \| f \|_{ \dot{\mathbb{B}}^s_{p, q} }^q = \sum_{n \in \Z} \left( 2^{ n s } \| \dot{\Delta}_{n} f \|_{ L^p_x } \right)^q < \infty\right \}$$
	where the homogeneous dyadic projector $\dot{\Delta}_n$ from Littlewood-Paley theory is recalled in Appendix \ref{scn:littlewood-paley}.

	For a Banach space $(Y_{v},\|\cdot\|{Y_{v}})$ of mappings depending on the variable $v$, the space 
	$\mathbb{H}^{s}_{x}\left(Y_{v}\right)$
	denotes the space of functions $f\::\;(x,v) \mapsto f(x,v)$ such that $$\|f\|_{\mathbb{H}^{s}_{x}\left(Y_{v}\right)}=\left\|\,\|f(x,\cdot)\|_{Y_{v}}\,\right\|_{\mathbb{H}^{s}_{x}} < \infty.$$
	Equivalently, one has
	\begin{equation}\label{eq:normHsx}
		\|f\|_{\mathbb{H}^s_x\left(Y_v\right)}^2=\int_{\R^d}\,\langle\xi\rangle^{2s}\,\left\|\widehat{f}(\xi)\right\|_{Y}^2\,\d\xi.
	\end{equation}
A similar definition applies to Besov spaces.

	\subsection{Assumptions}
	We work in a general setting of a perturbed kinetic equation of the form \eqref{eq:Kin-Intro} which, for $\eps=1$, reads
	\begin{gather*}
		(\partial_t + v \cdot \nabla_x)f = \LL f + \QQ(f, f),
	\end{gather*}
	where $\LL$ and $\QQ$ are local in $x$, that is to say, they act on functions depending only on $v$. Their actions on functions $f=f(x, v)$ depending on both $x$ and $v$ are naturally defined as
	$$[\LL f](x, v) = \big[ \LL f(x, \cdot) \big](v), \qquad \QQ(f, f)(x, v) = \Big( \QQ\big( f(x, \cdot), f(x, \cdot) \big) \Big)(v).$$
	
	%\subsubsection{Assumptions on the linear part}
	
	At the linear level, we make the following assumptions on the linearized operator $\LL$ in the space 
	$$\Ss=L^{2}\left(\mu^{-1}(v)\d v\right),$$
	of functions depending only on the velocity variable where $\mu \, : \, \R^d \to [0,\infty)$ is some measurable weight function.
	\begin{hypL}\label{AsL1} The linear operator $\LL\::\dom(\LL) \subset \Ss \to \Ss$ satisfies the following.
		\begin{enumerate}[label=\hypst{L\arabic*}]
			\item \label{L1} The operator $\LL$ is self-adjoint in $\Ss$ and commutes with orthogonal matrices:
			\begin{gather*}
				\langle \LL f, g \rangle_{\Ss} = \langle f, \LL g \rangle_{\Ss} = \langle \LL ({\Theta}f), \Theta g \rangle_\Ss,
			\end{gather*}
			for any $f, g \in \dom(\LL)$ and orthogonal matrix $\Theta \in \MMM_{d \times d}(\R)$, where $[\Theta f](v):=f(\Theta v)$.
			\item \label{L2} The weight function $\mu$ is nonnegative, normalized, radial, and such that:
			\begin{gather*}
				\mu=\mu(|v|) \ge 0, \qquad \int_{\R^d} \mu(v) \d v = 1,\\
				E = \int_{\R^d} |v|^2 \mu(v) \d v < \infty, \qquad K = \frac{1}{E^2} \int_{\R^d} | v |^4 \mu(v) \d v < \infty.
			\end{gather*}
			\item \label{L3} The null-space of $\LL$ is given by
			\begin{equation*} 
				\nul\left( \LL \right) =\mathrm{Span}\left\{ \mu, v_1 \mu,  v_2 \mu, \dots, v_d \mu, |v|^2 \mu \right\}\end{equation*}
			and there exists a Hilbert space $\Ssp$ such that
			\begin{equation*}
				\dom(\LL) \subset \Ssp \subset \Ss, \qquad \| \cdot \|_{\Ss} \le \| \cdot \|_{\Ssp},
			\end{equation*}
			and such that there holds
			\begin{equation*}
				\la \LL f, f \ra_{\Ss} \le - \lambda_\LL \| f \|^2_{\Ssp} \qquad \text{ for any } f \in \dom(\LL) \cap \nul\left( \LL \right)^\perp\,.
			\end{equation*}
			\item \label{L4} The operator $\LL$ can be decomposed as 
			$$\LL = \BB+ \AA, \qquad \dom(\BB)=\dom(\LL), \qquad \AA \in \BBB(\Ss),$$
			where the splitting is compatible with a hierarchy of Hilbert spaces $(\Ss_j)_{j=0}^{2}$ such that
			\begin{enumerate}
				\item \label{assumption_hierarchy} the spaces $\Ss_j$ continuously and densely embed into one another:
				$$\Ss_{2} \hookrightarrow \Ss_1 \hookrightarrow \Ss_0 = \Ss,$$
				\item \label{assumption_multi-v} the multiplication by $v$ is bounded from $\Ss_{j+1}$ to $\Ss_{j}$, i.e.
				\begin{equation*}
					\|v f\|_{\Ss_{j}} \lesssim \|f\|_{\Ss_{j+1}} \qquad f \in \Ss_{j+1}, \quad j=0,1,
				\end{equation*}
				\item \label{assumption_bounded_A} the operator $\AA : \Ss_j \rightarrow \Ss_{j+1}$ is bounded:
				$$\AA \in \BBB(\Ss_j,\Ss_{j+1}), \qquad j=0,1,$$
				\item \label{assumption_dissipative} the part $\BB_{\xi}$ is hypo-dissipative on each space $\Ss_j$ uniformly in $\xi  \in \R^d$, that is to say there exists $\lambda_\BB \geq \lambda_{\LL}$ such that, for $j=0,1,2$
				$$\mathfrak{S}_{\Ss_{j}}(\BB_{\xi}) \cap \Delta_{-\lambda_\BB} = \varnothing,$$
				and
				$$\sup_{\xi\in \R^{d}}
				\left\| \RR( z,\BB_{\xi})\right\|_{\BBB(\Ss_j)} \lesssim | \re z + \lambda_{\BB}|^{-1}, \qquad \forall z \in \Delta_{-\lambda_{\BB}}.$$
			\end{enumerate}
		\end{enumerate}
	\end{hypL}
	\begin{rem}
		Note that $K > 1$ by a simple use of Jensen's inequality applied to the probability measure $\mu(v)\d v$. Moreover, according to \ref{L1}, Assumption \ref{L3} can be formulated as follows: 	
		\begin{equation*}
			\forall f \in \dom(\LL), \quad \langle \LL f, \mu \rangle_{\Ss} = \langle \LL f, v \mu \rangle_{\Ss} = \left\langle \LL f, |v|^2 \mu \right\rangle_{\Ss} = 0,
		\end{equation*}
		that is to say $\nul(\LL)^{\perp}=\range(\LL)$, and  the operator $\LL$ has a spectral gap in $\Ss$:
		\begin{equation*}
			\mathfrak{S}_{\Ss}(\LL) \cap \Delta_{-\lambda_\LL} = \{0\}\,.
		\end{equation*}
	\end{rem}
	\begin{rem}\label{rem:KernelH2} Notice also that $\nul(\LL) \subset \Ss_2$. Indeed, given $f \in \nul(\LL)$, with the spitting given in \ref{L4}, $\LL f=0$ implies 
		$$f=\RR(0,\BB)\AA f$$
		and thanks to \ref{assumption_bounded_A}, $\AA f \in \Ss_1$ which, with now \ref{assumption_dissipative}, yields $\RR(0,\BB)\AA f\in \Ss_1.$ Thus $f \in \Ss_1$ and one can repeat the argument to deduce that $f \in \Ss_2$.
	\end{rem}
	
	\begin{exe}\label{rem:L1B1} We show in Appendix \ref{sec:Landau-Boltz} that the various Assumptions  \ref{L1}--\ref{L4} hold for  several models of physical interest, expliciting for each of those models the precise definition of the various spaces $\Ssp$ and $\Ss_{j}$ as well as the splitting $\LL=\AA+\BB$. {Typically, Assumptions \ref{L1}--\ref{L4} apply to the Boltzmann equation with hard potentials with or without Grad's cut-off assumptions or to Landau equation \emph{in spaces with Gaussian weights}.} To clarify right away the role of this set of Assumptions in our analysis, we  illustrate here the form of the spaces $\Ssp,\Ss_{j}$ in the case of Boltzmann equation with hard-spheres interactions. This corresponds to \eqref{eq:QQBoltz} with the choice
		$$B(|v-v_{*}|,\sigma)=|v-v_{*}|, \qquad v,v_{*} \in \R^{d}\times\R^{d}, \quad \sigma \in \S^{d-1}.$$
		In such a case, as said, $\mu$ is a Maxwellian distribution:
		\begin{equation}\label{eq:muMax}
		\mu(v) := (2 \pi)^{-d/2} \exp\left( - \frac{|v|^2}{2} \right), \qquad E = d, \quad K = 1 + \frac{2}{d} \, ,\end{equation}	and the usual linearized operator given by \eqref{eq:linearized} 
		is known to satisfy \ref{L1}--\ref{L2} with $\dom(\LL)=L^2( \langle v\rangle^2\mu^{-1}(v)\d v)$. Moreover, assumption \ref{L3}  is met with the choice 
		$$\Ssp=L^2\left( \la v\ra \mu^{-1}(v)\d v\right)$$
		Regarding assumption \ref{L4},  one can chose the hierarchy of spaces $\Ss_{j}$ as 
		$$\Ss_j := L^2\left( \la v \ra^{2 j} \mu^{-1}(v) \d v \right),$$
		for $j=0,1,2$. The splitting is taken to be Grad's splitting:
		$$(\BB f)(v) = - f(v) \int_{\R^d} |v-v_*| \mu(v_*) \d v_*, \qquad \AA f = \LL - \BB.$$
		Details are given in Appendix \ref{sec:Landau-Boltz}. We point out that, in full generality, $\Ssp$ maybe much more complicated than the above one and this is what motivated the introduction of the abstract framework \ref{L1}--\ref{L4}.\end{exe}
	
	\begin{defi}
		\label{defi:Hhstar}
		Under Assumption \ref{L2}, we define the ``dual'' space $\Ssm$ of the dissipation Hilbert space $\Ssp$ as the completion of~$\Ss$ for the norm
		$$\| f \|_{\Ssm} := \sup_{ \| \varphi \|_{\Ssp} \le 1 } \la f, \varphi \ra_{\Ss}.$$
	\end{defi}
	\begin{rem}
		\label{rem:rangHh}
		Since $\|\cdot\|_{\Ss} \leq \|\cdot\|_{\Ssp}$, for any $f \in\Ss$ one has from Cauchy-Schwarz inequality
		$$\|f\|_{\Ssm} \le \sup_{\|\varphi\|_{\Ssp}\leq 1} \|f\|_{\Ss}\|\varphi\|_{\Ss} \le \|f\|_{\Ss}$$
		we thus have the following comparison:
		$$\Ssp \hookrightarrow \Ss \hookrightarrow \Ssm, \qquad \|\cdot\|_{\Ssm} \leq \|\cdot\|_{\Ss} \leq \|\cdot\|_{\Ssp}.$$
	\end{rem}

	%	\subsubsection{Assumptions on the nonlinear term}
	At the nonlinear level, we make the following assumptions on $\QQ$.
	\begin{hypQ}\label{AsB1} The nonlinear operator $\QQ$ is satisfying the following assumptions:
		\begin{enumerate}[label=\hypst{B\arabic*}]
			\item \label{Bortho} The bilinear operator is $\Ss$-orthogonal to the null-space of $\LL$:
			$$\la \QQ(f, g), \mu \ra_{\Ss} = \la \QQ(f, g), v \mu \ra_{\Ss} = \la \QQ(f, g), |v|^2 \mu \ra_{\Ss} = 0,$$
			or, equivalently, in terms of integrals:
			$$\int_{\R^d} \QQ(f, g)(v) \d v = \int_{\R^d} v \QQ(f, g)(v) \d v = \int_{\R^d} |v|^2 \QQ(f, g)(v) \d v = 0.$$
			\item \label{Bisotrop} The bilinear operator commutes with orthogonal matrices:
			$$\la \QQ(f, g), h \ra_{\Ss} = \la \QQ\left( \Theta f, \Theta g \right), \Theta h \ra_{\Ss},$$
			for any orthogonal matrix $\Theta \in \MMM_{d \times d}(\R).$
			\item \label{Bbound} The bilinear operator satisfies the following dual estimate
			$$\| \QQ(f, g) \|_{\Ssm} \lesssim \| f \|_{\Ss} \| g \|_{\Ssp} + \| f \|_{\Ssp} \| g \|_{\Ss}\,,$$
			or, in other words, there holds
			$$\lla \QQ(f, g) , h \rra_{\Ss} \lesssim \| h \|_{\Ssp} \left( \| f \|_{\Ss} \| g \|_{\Ssp} + \| f \|_{\Ssp} \| g \|_{\Ss} \right)\,.$$
		\end{enumerate}
	\end{hypQ}

	\subsection{Main results -- first version}  Under the above structural assumptions \ref{L1}--\ref{L4}, the full description of the spectrum of $\LL_{\xi}$ and the decay {and regularization properties} of the associated semigroup are made explicit in the following.
	\begin{theo}[\textit{\textbf{Main spectral theorem}}]
		\label{thm:spectral_study}
		Assume \ref{L1}--\ref{L4}, there exist {explicitly computable constants} $C, \alpha_{0}, \lambda, \gamma, \sigma_0 > 0$ such that the following spectral and dynamical properties hold.\\
		
		\noindent \textit{\textbf{(1)}} \textit{\textbf{Localization of the spectrum.}}
		The spectrum of $\LL_\xi$ is localized as follows.
		\begin{itemize}
			\item If $|\xi| \ge \alpha_{0}$, the spectrum is at a positive distance from $\{ \re z \ge 0 \}$:
			$$ \mathfrak{S}_{\Ss}\left(\LL_{\xi}\right)  \cap \Delta_{-\gamma} = \varnothing.$$
			\item If $|\xi| \le \alpha_{0}$, the spectrum is at a positive distance from $\{ \re z \ge 0 \}$, except for a finite number of small eigenvalues:
			$$\mathfrak{S}_{\Ss}(\LL_{\xi}) \cap \Delta_{-\lambda} = \big\{ \lambda_\Inc(\xi), \, \lambda_\Bou(\xi), \, \lambda_{- \Wave}(\xi), \,\lambda_{+\Wave}(\xi) \big\},$$
			and these eigenvalues $\lambda_\star(\xi)$ expand for $\xi \to 0$ as
			\begin{subequations}
				\label{eq:lambda_star}
				\begin{gather}
					\lambda_{\pm \Wave}(\xi) = \pm i c | \xi | - \kappa_\Wave | \xi |^2 + \OO\left( | \xi |^3 \right),\\
					\lambda_{\star}(\xi) = - \kappa_\star | \xi |^2 + \OO\left( | \xi |^3 \right), \quad \star = \Bou, \Inc,
				\end{gather}
			\end{subequations}
			where the speed of sound is defined as
			\begin{equation}
				\label{eq:speed_sound}
				c:=\sqrt{\frac{KE}{d}},
			\end{equation}
			and the diffusion coefficients  {$\kappa_\star \in (0, \infty)$} are given by
			\begin{equation}\label{eq:kappa}
				\begin{split}
					\kappa_\Inc := - \dfrac{1}{(d-1)(d+1) } &\left \langle \LL^{-1}  \BurA  ,\BurA  \right \rangle_{\Ss}, 
					\qquad  \kappa_\Bou := - \dfrac{1}{d} \left \langle \LL^{-1} \BurB  , \BurB  \right \rangle_{\Ss},\\
					\kappa_\Wave &:= \frac{d-1}{2 d} \kappa_\Inc + \frac{E^2(K-1)}{2} \kappa_\Bou,
				\end{split}
			\end{equation}
			where the Burnett functions $\BurA $ and $\BurB $ are defined as
			\begin{equation}\label{eq:burnett}\begin{split}\begin{cases}
						\BurA (v) & := \displaystyle  \sqrt{ \frac{d}{E} } \left( v \otimes v - \frac{|v|^2}{d} \Id \right) \mu(v),\\
						\\
						\BurB (v) & := \displaystyle  \frac{1}{\sqrt{K (K-1)}} v \left( K - \frac{|v|^2}{E} \right) \mu(v).
			\end{cases}\end{split}\end{equation}
		\end{itemize}
		\noindent \textit{\textbf{(2)}} \textit{\textbf{Asymptotic behavior of the spectral projectors.}}
		For any non-zero $| \xi | \le \alpha_0$, the spectral projectors associated with these eigenvalues expand in $\BBB\left( \Ssm ; \Ssp \right)$ as
		\begin{equation}
			\label{eq:PPstar}
			\PP_\star(\xi) = \PP^{(0)}_\star\left( \frac{\xi}{|\xi|} \right) + i \xi \cdot \PP^{(1)}_\star\left( \frac{\xi}{|\xi|} \right) +  S_{\star}(\xi), \qquad \star=\Inc,\Bou,\pm\Wave,
		\end{equation}
		where  $S_{\star}(\xi) \in \BBB(\Ssm;\Ssp)$ with $\|S_{\star}(\xi)\|_{\BBB(\Ssm;\Ssp)} \lesssim | \xi |^2$.
		The zeroth order coefficients are  defined for any $\omega \in \S^{d-1}$ as
		\begin{gather*} 
			\PP_\Inc^{(0)}\left( \omega \right) f(v)=\dfrac{d}{E} \big( \Pi_\omega \langle f, v \mu \rangle_{\Ss}\big) \cdot v \mu(v),\\
			\PP_\star^{(0)}(\omega)f = \langle f, \psi_\star(\omega) \rangle_{\Ss} \, \psi_\star(\omega), \qquad \star = \Bou, \pm \Wave,
		\end{gather*}
		where we denoted $\Pi_\omega=\Id - \omega \otimes \omega$ the orthogonal projection onto $( \R \omega )^\perp$, and the first order terms write explicitly for any $f \in \ker(\LL)^\perp$ as
		\begin{gather*}
			\PP^{(1)}_\Inc(\omega) f(v) = \sqrt{\frac{d}{E}} \left\la f , \LL^{-1}\BurA \right\ra_{\Ss} \Pi_\omega v \mu,\\
			\PP^{(1)}_{\pm \Wave}(\omega) f = \left(\pm \frac{1}{\sqrt{2}  } \left\langle f , \LL^{-1}\BurA \omega \right\rangle_{\Ss} + E \sqrt{\frac{K-1}{2}} \left\langle f, \LL^{-1} \BurB \right\rangle_{\Ss} \right) \psi_{\pm \Wave}(\omega),
		\end{gather*}
		and
		\begin{equation}
			\label{eq:P_1_Bou}
			\PP^{(1)}_{\Bou}(\omega) f = \langle f, \LL^{-1} \BurB \rangle_{\Ss} \psi_\Bou,
		\end{equation}
		 {where the zeroth order eigenfunctions $\psi_{\pm \Wave}$ and $\psi_\Bou$ are defined as}
		\begin{gather}
			\label{eq:def_psi_wave}
			\psi_{\pm \Wave}(\omega, v) := \frac{1}{\sqrt{2 K}} \left( 1 \pm \sqrt{\frac{d K}{E}} \omega \cdot v  + \frac{1}{E}\left(|v|^2 - E\right) \right) \mu(v),\\
			\label{eq:def_psi_Bou}
			\psi_{\Bou}(v) := \frac{1}{\sqrt{K (K - 1) } } \left(K - \frac{|v|^2}{E} \right) \mu(v).
		\end{gather}
		\color{black}
		Notice, in particular, that
		\begin{equation}
			\label{eq:alternative_P_1_inc}
			\omega \cdot \PP^{(1)}_\Inc(\omega) f = \sqrt{\frac{d}{E}} \Big( \Pi_\omega \left\la f , \LL^{-1}\BurA \right\ra_{\Ss} \omega \Big) \cdot v \mu.
		\end{equation}
		\noindent \textit{\textbf{(3)}} \textit{\textbf{Resolvent bounds and decay estimates.}} Setting
		\begin{equation}\label{eq:splitPP}
			\PP(\xi)= \mathbf{1}_{|\xi| \le \alpha_0} \bigg(\PP_{\Bou}(\xi)+\PP_{\Inc}(\xi)+\PP_{+\Wave}(\xi)+\PP_{-\Wave}(\xi)\bigg)\end{equation}
		the spectral projector associated to the part of the spectrum from point \textbf{(1)}, the following resolvent bound  holds 
		\begin{equation}
			\label{eq:ResolI-PP}
			\sup_{z \in \Delta_{-\sigma_{0}}}\bigg\|\RR(z,\LL_{\xi})\left(\Id-\PP(\xi)\right)\bigg\|_{\BBB(\Ss)} \le C, \qquad \forall \xi\in \R^d  ,
		\end{equation} 
		where $\sigma_0 := \min\{\lambda, \gamma\}$. Finally, the $C^{0}$-semigroup $\left(U_{\xi}(t)\right)_{t\geq0}$ generated by $(\LL_{\xi},\dom(\LL_{\xi}))$ satisfies for any $\sigma \in (0, \sigma_0)$, any $\xi \in \R^{d}$ and any $f \in \Ss$
		\begin{subequations}\label{decay-semigroup}
			\begin{equation}
				\label{eq:decay}
				\begin{split} \sup_{t \ge 0} \, & e^{2 \sigma_0 t} \left\| U_{\xi}(t)\ {\left( \Id - \PP(\xi) \right)} f \right\|^2_{\Ss} \\
					&+ \int_0^\infty e^{2 \sigma t} \left\| U_{\xi}(t) \left( \Id - \PP(\xi) \right)f \right\|_{\Ssp}^2 \, \d t \le C_\sigma \| {\left( \Id - \PP(\xi) \right)}f \|_{\Ss}^2,
			\end{split}\end{equation}
			whereas, for any $f \in \Ssm$, 
			\begin{equation}
				\label{eq:decay_Hh'}
				\int_0^\infty e^{2 \sigma t} \left\| U_{\xi}(t) {\left( \Id - \PP(\xi) \right)} f \right\|_{\Ss}^2 \, \d t \le C_\sigma \|  {\left( \Id - \PP(\xi) \right)}f \|_{\Ssm}^2\,.
			\end{equation}
		\end{subequations}
	\end{theo}

	\begin{figure}[h]
		\centering
		
		\begin{tikzpicture}
			% Spectre cinétique
			\fill [pattern=north west lines, pattern color=red] (-5, -2.5) rectangle (-2.7, 2.5);
			\draw [color=red] (-2.7, -2.5) -- (-2.7, 2.5);
			\node[red] at (-4, -3) {$\Re z \le - \lambda$};

			% Valeurs propres ondes
			\draw[black, variable=\y, dash pattern=on 5pt] plot[smooth,domain=-2.25:2.25] ( {-0.5 * (\y)^2 }, \y );
			\node[blue] at ( {-0.5 * (1.7)^2} , {1.7 + 0.7} ) {$\lambda_{+ \Wave}(\xi)$};
			\fill [color=blue] ( {-0.5 * (2.2)^2} , 2.2) circle (0.05);
			\node[blue] at ( {-0.5 * (1.7)^2} , {-1.7 - 0.7} ) {$\lambda_{- \Wave}(\xi)$};
			\fill [color=blue] ( {-0.5 * (2.2)^2} , -2.2) circle (0.05);
			
			% Valeur propre Boussinesq
			\fill [color=blue] (-1.9, 0) circle (0.05);
			\node[blue] at (-1.9, -0.35) {$\lambda_\Bou(\xi)$};
			
			% Valeur propre incompressible
			\fill [color=blue] (-1.1, 0) circle (0.05);
			\node[blue] at (-1, 0.4) {$\lambda_\Inc(\xi)$};
			
			% Repère
			\draw[-stealth] (-5, 0) -- (1, 0);
			\draw[-stealth] (0, -2.5) -- (0, 2.5);
			
		\end{tikzpicture}
		\qquad\qquad\begin{tikzpicture}
			% Spectre cinétique
			\fill [pattern=north west lines, pattern color=red] (-5, -2.5) rectangle (-1.2, 2.5);
			\draw [color=red] (-1.2, -2.5) -- (-1.2, 2.5);
			\node[red] at (-3, -3) {$\Re z \le - \gamma$};
			
			% Repère
			\draw[-stealth] (-5, 0) -- (1, 0);
			\draw[-stealth] (0, -2.5) -- (0, 2.5);
			
		\end{tikzpicture}
		
		\caption{Localization of the spectrum of $\LL - i (v \cdot \xi) $ for $| \xi | \le \alpha_0 $ and for~$| \xi | \ge \alpha_0$ }
		
	\end{figure}
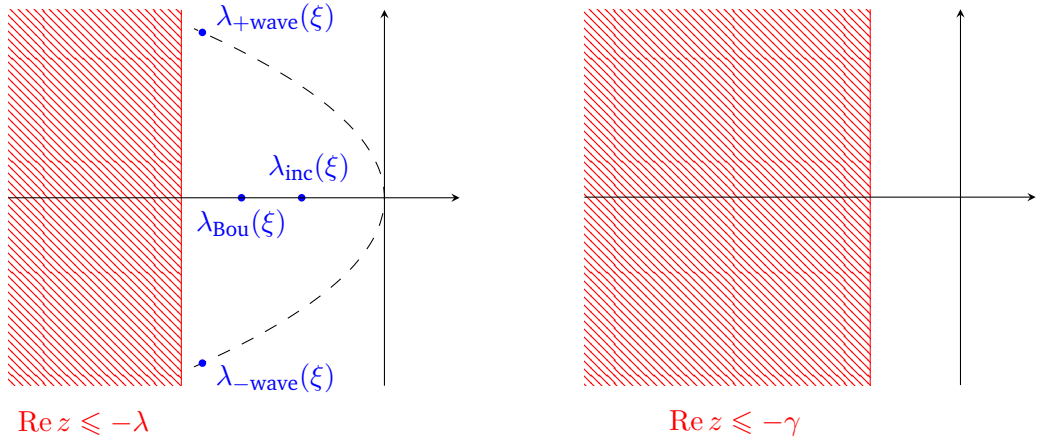

	\begin{rem}
		Recall from Remark \ref{rem:KernelH2} that \ref{L4} implies that $|\cdot|^2\mu \in \Ss_2$. Using then \ref{assumption_multi-v} twice, we deduce that the mapping $| \cdot |^4\mu$ belongs to $\Ss$. Thus,
		$$\int_{\R^d}|v|^8\mu(v)\d v < \infty.$$
		Consequently, $\BurA , \BurB  \in \Ss$ and $\LL^{-1}\BurA ,\LL^{-1}\BurB  \in \Ss$, and thus
		$$\kappa_\star < \infty, \qquad \star=\Bou,\Inc,\Wave.$$
		We point out that, in a sense, we only assume (almost) enough integrability for the diffusion coefficients $\kappa_\star$ to be finite. This is to be contrasted with the work \cite{Mellet} in which they prove that if $\kappa_\star = \infty$, then, under some appropriate scaling, one observes fractional diffusion in the limit $\eps \to 0$. In this framework, a corresponding version of Theorem \ref{thm:spectral_study} was proved in \cite{Puel}. We also refer to \cite{BM} for a unified spectral approach to the (fractional) diffusion limit for a large variety of linear collisional kinetic equations with a single conservation law. Finally, we point out that contrary to previous proofs of Theorem \ref{thm:spectral_study} for specific models, we do not assume that the weight $\mu$ decays like a gaussian.
	\end{rem} 
	\begin{rem}
		Notice that $\RR(z,\LL_{\xi})\left(\Id-\PP(\xi)\right)=\RR\Big(z,\LL_{\xi}(\Id-\PP(\xi))\Big)$ and, by virtue of \eqref{eq:PPstar} the above resolvent bound \eqref{eq:ResolI-PP} can be rewritten as
		$$\sup_{z \in \Delta_{-\lambda}}\bigg\|\RR(z,\LL_{\xi})-\sum_{\star=\Inc,\pm\Wave,\Bou}\left(z-\lambda_{\star}(\xi)\right)^{-1}\PP_{\star}(\xi)\bigg\|_{\BBB(\Ss)} \leq C.$$
		Note also that as $\LL$ is self adjoint in {$\Ss$}, the dual semigroup 
		$$\Big(U_{\xi}(t) \left(\Id - \PP(\xi)\right) \Big)^\star =  U_{-\xi}(t) \left(\Id - \PP(-\xi)\right) \,,\, \qquad t \geq 0,$$ automatically satisfies the same estimates \eqref{decay-semigroup}.
	\end{rem}
	
	\begin{rem}
		\label{rem:macro_representation_spectral}
		The zeroth order terms in the expansions of the projectors are macroscopic in the sense that
		$$\PP \PP^{(0)}_\star = \PP^{(0)}_\star \PP = \PP^{(0)}_\star, \quad \star = \Inc, \Bou, \pm \Wave.$$
		As a consequence, they can be characterized in terms of the macroscopic components $\varrho[\cdot], u[\cdot]$ and $\theta[\cdot]$ where
		\begin{equation}
			\label{eq:fluctuat}
			\begin{cases}
				\varrho_f &= \varrho[f] := \la f, \mu \ra_\Ss, \\
				\\
				u_f & = u[f] := \dfrac{d}{E} \la f, v \mu \ra_\Ss ,\\
				\\
				\theta_f &= \theta[f] := \dfrac{1}{E} \lla f , (|v|^2 - E) \mu \rra_\Ss  		
			\end{cases}
		\end{equation}
		for any $f\in \Ss.$ We refer to Proposition \ref{prop:macro_representation_spectral} for a precise statement.  
	\end{rem}

	For the Boltzmann equation with hard potential interaction, the above {theorem} has been proven first in the seminal work \cite{EP1975} whose  method has been {adapted} subsequently to encompass much more general models in the recent work \cite{YY} (see also \cite{LY2016, LY2017,YY23}). The method in these contributions is based upon some compactness argument and a study of the eigenvalue problem through the use of the Implicit Function Theorem. 
	
	The approach we perform in the present paper appears much more direct and simpler. Any explicit computation relies solely on the isotropy of the operator $\LL$. To be more precise, we adapt here the perturbation theory of eigenvalues introduced in \cite{K1995} and exploit the structural assumption \ref{L4} to prove the regularity and expansion of the eigen-projectors. Notice that, except for some peculiar cases (including the Boltzmann equation for hard-spheres interactions), our perturbative approach does not directly fall into the realm of the classical perturbation theory of  unbounded operators developed in \cite{K1995} since the multiplication operator $i(v\cdot\xi)$ is \emph{not} $\LL$-bounded in general. This induces some technical complications and requires to adapt the method of \cite{K1995} to the general situation we are dealing with here.  This is done, borrowing and pushing further some ideas of \cite{T2016}, by fully exploiting the splitting of $\LL$ as
	$$\LL=\AA+\BB$$
	where $\BB$ enjoys dissipative properties whereas $\AA$ is a regularizing operator which compensates the unboundedness of the multiplication by $v$ (see \ref{L4}).  Moreover, in contrast with existing results based upon \cite{EP1975}, our method takes into account the role of the dissipation space $\Ssp$ and its dual $\Ssm$. This allows us to emphasise and exploit regularizing effects of $\LL$ in the scale of spaces $\Ssp \hookrightarrow \Ss \hookrightarrow \Ssm$. A more detailed description of the approach we follow will be given in Section \ref{sec:newspec}. We point out already that we strongly use here the fact that all functional spaces considered here are Hilbert spaces: this allows to use a suitable ``diagonalization'' of the transport operator thanks to Fourier transform and also permits to deduce spectral properties of the semigroup $\left(U_{\xi}(t)\right)_{t\ge0}$ through some of its generator $\LL_{\xi}$ thanks to Gearhart-Pruss theorem.

	We strongly believe that the new  method we propose here to the fine spectral analysis of kinetic models is robust enough to be adapted to various contexts and can become a valid alternative to the technical approach of \cite{EP1975}. In our opinion, it replaces in an efficient way the compactness arguments introduced in \cite{EP1975} for the localization of the spectrum by a much modern and quantitative approach, combining enlargement techniques from \cite{GMM2017} to describe small frequencies $\xi\simeq 0$ with hypocoercivity methods from \cite{D2011} for frequencies $|\xi| \gtrsim 1.$ Moreover, since it is based on the isotropic nature of $\LL$ and $\QQ$, it can be directly adapted to more general equations of the type
	$$\partial_t f+ \mathfrak{a}(|v|) v \cdot \nabla_x f=\LL f+ \QQ(f,f)$$
	where $\mathfrak{a}(\cdot)$ is a suitable smooth radial mapping and 
	$$\nul \LL=\mathrm{Span}\left\{\mu,v_1\mu,\ldots,v_d\mu,\mathfrak{b}(|v|)\mu\right\}$$
	for a suitable radial mapping $\mathfrak{b}(\cdot)$ such that 
	$$\int_{\R^d}\QQ(f,f)\mathfrak{b}(|v|)\d v=0.$$ 
	The relativistic Boltzmann and Landau equations both fall within the above framework with
	$$\mathfrak{a}(|v|)=\frac{1}{\mathfrak{b}(|v|)}, \qquad \mathfrak{b}(|v|)=\sqrt{1+c_0^{-2}|v|^2}, \qquad \mu(v)=Z^{-1}e^{-\mathfrak{b}(|v|)}, \qquad c_0 > 0,$$
	where $Z >0$ is a normalization constant so that the Juttner distribution $\mu$ satisfies \ref{L1}. Our structural assumptions \ref{L1}--\ref{L4} can then easily be modified to cover such a case.  For instance, Assumption \ref{assumption_multi-v} should read now 
	$$\left\|\mathfrak{a}(|v|)vf\right\|_{\Ss_j} \lesssim \|f\|_{\Ss_{j+1}}.$$
	We point out that several moments of the Juttner distribution $\mu$ involving powers of $\mathfrak{a}(|v|)v$ and $\mathfrak{b}(|v|)$ would have to be considered in Assumption \ref{L2}. In particular the expressions of $\psi_\star$ and $\kappa_\star$ would be much more intricate. Thus, we do not pursue further this line of research since the present contribution is already quite technical and lengthy.  \bigskip
	
	Besides the thorough description of the spectrum of $\LL_\xi$ and the relevant eigen-projectors, Theorem \ref{thm:spectral_study} also describe the long-time behaviour of the associated linearized semigroup $\left(U_\xi(t)\right)_{t\ge0}$. Our approach uses, as said enlargement techniques from \cite{GMM2017} as well as an abstract hypocoercivity result from \cite{D2011}. The decay of the linearized semigroup $\left(U_{\xi}(t)\right)_{t\geq0}$ in \eqref{decay-semigroup} is one of the fundamental brick on which it is possible to build the Cauchy theory for \eqref{eq:Kin-Intro} whereas a comparison of $\left(U_\xi(t)\right)_{t \geq0}$ with the linearized semigroup associated to \eqref{eq:NSFint} is the main tool for the study of the hydrodynamic limit. This yields, in the Hilbert space setting $\mathbb{H}^s_x(\Ss_v)$ 
	to our main result as far as the above objective \textbf{(III)} is concerned:  
	\begin{theo}[\textit{\textbf{Hydrodynamic limit theorem}}]
		\label{thm:hydrodynamic_limit}
		Let $s > \frac{d}{2}$ be given  as well as some initial data 
		$$f_\ini \in\mathbb{H}^s_x \left( \Ss_v \right),$$ 
		satisfying additionally, if $d=2$, $f_\ini \in \dot{\mathbb{H}}^{-\alpha}_x \left( \Ss_v \right)$ for some $0 < \alpha < \frac{1}{2}.$
		
		Consider the solution of the Navier-Stokes-Fourier system   (see Theorem \ref{thm:spectral_study} and Proposition \ref{prop:equivalence_kinetic_hydrodynamic_INSF} for the definitions of the coefficients,  {and Theorem \ref{thm:cauchy_NSF} for the existence of this solution})
		\begin{equation}\label{eq:main-NSF} 
			\begin{cases}
				\partial_{t}u-\kappa_{\Inc}\,\Delta_{x}u + \vartheta_{\Inc}\,u\cdot \nabla_{x}\,u= \nabla_{x}p\,,\\[6pt]
				\partial_{t}\,\theta-\kappa_{\Bou}\,\Delta_{x}\theta +\vartheta_{\Bou}\,u\cdot \nabla_{x}\theta=0,\\[8pt]
				\nabla_{x}\cdot u=0\,, \qquad \nabla_{x}\left(\varrho + \theta\right)= 0\,,
			\end{cases}
		\end{equation}
		spanned by the initial conditions
		$$u(0, x) = \mathbb{P} u[f_\ini](x), \qquad \theta(0, x) = \frac{1}{K(K-1)}\left( ( K-1 ) \varrho[f_\ini](x) - \theta[f_\ini](x) \right),$$
		and which satisfies for some $T \in (0, \infty]$
		\begin{gather*}
			(\varrho, u, \theta) \in \CC_b\left( [0, T) ;\mathbb{H}^s_x \right), \qquad (\nabla_x \varrho, \nabla_x u, \nabla_x \theta) \in L^2\left( [0, T) ;\mathbb{H}^s_x \right).
		\end{gather*}
		Introducing, for any $(x,v) \in \R^{d}\times\R^{d}$ and $t \in [0,T)$,
		$$ f_\ns(t, x, v) = \left(\varrho(t, x) + u(t, x) \cdot v   + \frac{\theta(t, x)}{E(K-1)} \left( |v|^2 - E \right)\right) \mu(v),$$ 
		the following holds.
		\begin{enumerate}
			\item[\textit{\textbf{(1)}}]\textit{\textbf{Existence of a unique solution.}}
			There exists some small $c_0 > 0$ and $\eps_0 > 0$ such that the equation
			$$\partial_t f^\eps = \frac{1}{\eps^2} \left( \LL - \eps v \cdot \nabla_x \right) f^\eps + \frac{1}{\eps} \QQ\left( f^\eps, f^\eps \right), \quad f^\eps(0, x, v) = f_\ini(x, v)$$
			admits for any $\eps \in (0, \eps_0]$ a unique solution 
			$$f^\eps \in L^2_{\rm{loc}} \left( [0, T) ; \mathbb{H}^s_x \left( \Ssp_v\right) \right) \cap \CC\left([0,T);\mathbb{H}^s_x(\Ss_v)\right)$$
			such that
			$$\sup_{0 \le t < T} \| f^\eps(t) \|_{ \mathbb{H}^s_x \left( \Ss_v \right) } \le \frac{c_0}{\eps},$$
			which satisfies furthermore the following uniform estimate:
			$$\sup_{0 \le t < T} \| f^\eps(t) \|_{ \mathbb{H}^s_x \left( \Ss_v \right) }^2 + \int_0^T \| | \nabla_x |^{ 1-\alpha} f^\eps(t) \|_{ \mathbb{H}^{s - (1-\alpha) }_x \left( \Ssp_v \right) }^2 \d t \, \lesssim 1,$$
			where we recall that  $0 < \alpha < \frac{1}{2}$ if $d=2$ and $\alpha=0$ if $d\geq3.$
			\item[\textit{\textbf{(2)}}]\textit{\textbf{Decomposition and convergence of the solution.}}
			The solution $f^\eps$ splits as the sum of some limiting part $f_\ns$, some initial layers $(f^\eps_\disp, f^\eps_\kin)$, and a vanishing part~$f^\eps_\err$:
			\begin{equation}\label{eq:decomp}
				f^\eps = f_\ns + f^\eps_\disp + f^\eps_\kin + f^\eps_\err\end{equation} 
			where {each part belongs to $L^\infty\left( [0, T) ; \mathbb{H}_x^s \left( \Ss_v \right) \right)$ uniformly in $\eps$ and}
			\begin{itemize}
				\item       The dispersive part $f^\eps_\disp$ vanishes in an averaged sense:
				\begin{equation*}
					\lim_{\eps \to 0}\int_0^{t_*} \|f^\eps_\disp(\tau)\|^p_{L^\infty_x(\Ss_v)}\d\tau=0, \qquad \forall 0 < t_\star < T, \quad p >\frac{2}{d-1}
				\end{equation*}
				and uniformly away from $t=0$:
				\begin{equation*} 
					\lim_{\eps \to 0}\sup_{t_\star \le t < T} \| f^\eps_\disp(t) \|_{ L^\infty_x \left( \Ss_v \right) }=0, \qquad \forall 0 < t_\star < T.
				\end{equation*}
				\item     The kinetic part $f^\eps_\kin$ satisfies for some universal $\sigma > 0$
				\begin{equation}\label{eq:fKIN}
					\sup_{0 \le t < T} e^{2 \sigma t / \eps^2} \| f^\eps_\kin(t) \|_{ \mathbb{H}^s_x \left( \Ss_v \right) }^2 + \frac{1}{\eps^2} \int_0^T e^{2 \sigma t / \eps^2} \| f^\eps_\kin(t) \|_{ \mathbb{H}^s_x \left( \Ssp_v \right) }^2 \d t \, {\lesssim 1}.
				\end{equation} 
				\item        The error term $f^\eps_\err$ vanishes uniformly:
				\begin{equation*}
					\lim_{\eps \to 0} \sup_{0 \le t < T} \| f^\eps_\err(t) \|_{ \mathbb{H}^s_x \left( \Ss_v \right) }=0.
				\end{equation*}
			\end{itemize}
		\end{enumerate}
	\end{theo}

	\begin{rem}
		The rate of convergence of $f^\eps_\disp$ and $f^\eps_\err$ can be made explicit. Namely, the dispersive part $f^\eps_\disp$ satisfies:
		$$ \| f^\eps_\disp \|_{L^\infty_x \left( \Ss_v \right) } \lesssim 1 \land \left( \frac{\eps}{t} \right)^{ \frac{d-1}{2} } \left( {\| \PP_\disp f_\ini} \|_{ \dot{\mathbb{B}}^{\frac{d+1}{2}}_{1,1} \left( \Ss_v \right) } + \| \PP_\disp f_\ini \|_{ \mathbb{H}^s_x (\Ss_v) } \right),$$
		whereas the error term $f^\eps_\err$ is such that
		$$\| f^\eps_\err \|_{ L^\infty_t \mathbb{H}^s_x \Ss_v } \lesssim {\beta_{\disp}}(f_\ini, \eps) + \beta_\ns(f_\ns, f_\ini, \eps),$$
		where ${\PP_\disp}$ is defined in Definition \ref{def:hydrodynamic_projectors}, and ${\beta_{\disp}}$ and $\beta_\ns$ are rates of convergence to zero described in Proposition \ref{prop:source_term}.
	\end{rem}
\begin{rem}
Notice that the initial datum here above $f_\ini$ is not depending on $\eps$. On this respect, the initial datum is \emph{well-prepared} in a more restrictive sense than \cite{BU1991} (see in particular Remark 1.5 (ii) in \cite{BU1991}). 
We however point out that it would be possible to choose a family of initial data $f_\ini^{\eps}$ depending on $\eps$ provided with assume $\lim_{\eps\to0}f_{\ini}^{\eps}=f_\ini$ in some explicit and quantitative sense which would allow to to quantify the convergence of the error term as in the previous Remark. We also wish to emphasise that, aware of the general issue of sensitivity with respect to initial data for both kinetic and hydrodynamic equations, the kind of solutions we are considering in the present contribution is much more regular than weak (renormalised) solutions considered for instance in \cite{golseSR,golseSR1} or, at the level of fluid-dynamical equation, Leray solutions to the Navier-Stokes-Fourier system. In particular, pathological issues as those exhibited in e.g. \cite{albri,delellis} are naturally excluded by our analysis.
\end{rem}	

\begin{exe}\label{exa:L1B1} Elaborating on Example \ref{rem:L1B1}, we can formulate Theorem \ref{thm:hydrodynamic_limit} in the special case of the Boltzmann equation with hard spheres interactions. Recall that, in such a case, the collision operator $\QQ$ is given by \eqref{eq:QQBoltz} with $B(|v-v_{*}|,\sigma)=|v-v_{*}|.$ In such a case, considering for simplicity the physical dimension $d=3$, we can consider functional spaces associated with Gaussian weight
$$H_{v}=L^{2}(\mu^{-1}(v)\d v), \qquad  \Ssp=L^2\left( \la v\ra \mu^{-1}(v)\d v\right)$$
where $\mu$ is given by \eqref{eq:muMax} and, with an initial datum $f_\ini \in\mathbb{H}^s_x \left( \Ss_v \right),$ $s >\frac{3}{2},$ Theorem\ref{thm:hydrodynamic_limit}  provides, for any $\eps >0$ a unique solution $f^\eps \in L^2_{\rm{loc}} \left( [0, T) ; \mathbb{H}^s_x \left( \Ssp_v\right) \right) \cap \CC\left([0,T);\mathbb{H}^s_x(\Ss_v)\right)$ to the Boltzmann equation with moreover the convergence of $f^{\eps}$ to $f_\ns$ as $\eps \to 0$ (in some suitable sense, we refer to Example \ref{exe:rate} for a more explicit statement).\end{exe}
	
	To study both the kinetic equation \eqref{eq:Kin-Intro} and the Navier-Stokes-Fourier system \eqref{eq:NSFint}, we adopt a mild formulation which consists in writing the equations in Duhamel form
	\begin{equation*}\begin{split}
			f^\eps(t)&=U^\eps(t)f_\ini+\frac{1}{\eps}\int_0^t U^\eps(t-\tau) \QQ\left(f^\eps(\tau),f^\eps(\tau)\right)\d\tau\\
			&=:U^\eps(t)f_\ini+\Psi^\eps[f^\eps,f^\eps](t)\end{split}\end{equation*}
	where  $\left(U^\eps(t)\right)_{t\ge0}$ is the semigroup generated by $-{\eps^{-2}}\LL - {\eps}^{-1}v\cdot \nabla_x$. 
	
	{Of course, the most obvious difficulty in establishing the hydrodynamic limit $\eps \to 0$ lies in the control of the stiff term 
		$\frac{1}{\eps}\QQ\left(f^\eps,f^\eps\right)$, however one first needs to construct a solution for any $\eps$.} In \cite{BU1991, GT2020} in which the cutoff Boltzmann equation is considered, the authors only need to consider uniform  in time estimates for the semigroup $(U^\eps(t))_{t\ge0}$ to prove that $\Psi^\eps[f,f]$ is well-defined.  This means that, in such a case, $\Psi^\eps$ is bounded in a space of the type $L^\infty_t\mathbb{H}^s_x \Ss_v$. The study of the Boltzmann equation under cut-off assumption makes this approach possible since then, in the splitting
	$$\LL=\AA+\BB$$
	the dissipative  part $\BB$ is simply the multiplication by the collision frequency. 
	In \cite{CRT2022} in which the Landau equation is considered, the authors prove uniform  in time regularization estimates for the semigroup and the kinetic solution which subsequently yield the boundedness of $\Psi^\eps[f,f].$ The short-time regularization effects are of course due to the elliptic-like nature of the Landau collision operator.
	
	The abstract framework we are considering in this paper covers both cases, but the assumptions \ref{L1}--\ref{L4} or \ref{LE} are not strong enough to directly deduce regularization effects and boundededness of $\Psi^\eps$. To overcome this, we draw inspiration from suitable energy methods used to construct close-to-equilibrium solutions for Boltzmann and Landau that leads us to a study of the solution in spaces of the form    $$L^\infty_t \Ss \cap L^2_t \Ssp$$
	as expressed in \eqref{eq:fKIN}. See Section \ref{sec:detail} for more details of the functional setting and the mathematical difficulties encoutered in the proof of Theorem \ref{thm:hydrodynamic_limit}. 
	
	We wish to point out also that the nonlinear effects induced by the collision and the lack of control for the $L^2_{x}\Ss_v$ norm requires a refined inequality of quantity like $\| f g \|_{\mathbb{H}^s_x}$  of the type
	$$\| f g \|_{\mathbb{H}^s_x} \lesssim \| | \nabla_x |^{r} f \|_{\mathbb{H}^{s - r}} \| g \|_{\mathbb{H}^s_x}, \quad {0 \le r < \frac{d}{2} < s}.$$
	In the case of $d \ge 3$, one can take $r = 1$ so that the dissipation of the $L^2_x$-norm and the energy control the nonlinearity. In the case $d = 2$, we cannot take $r= 1$ so we consider $r = 1-\alpha$ where $\alpha > 0$. This explains 
	why we require, in this specific case, 
	$$f_\ini \in \dot{\mathbb{H}}^{-\alpha}$$
	which is related to the control of the heat semigroup of the type
	$$\| e^{t \Delta_x} | \nabla_x|^{1-\alpha}_x f \|_{L^2_{t, x}} \lesssim \| f \|_{ \dot{\mathbb{H}}^{-\alpha} }.$$
	Note that in \cite{GT2020, G2023}, the stronger assumption $f_\ini \in L^1_x \left( \Ss_v \right)$ was made.
	
	More details about  the proof of Theorem \ref{thm:hydrodynamic_limit} will be given in the next Section \ref{sec:detail}. However, let us already anticipate that the main relevant facts of our approach  lie in the following 
	\begin{enumerate}
		\item[(i)] We insist here on the fact that, in the broad generality we are dealing with here, Theorem \ref{thm:hydrodynamic_limit} is new even if results of similar flavors do exist in the literature for the Boltzmann or Landau equations. In particular, as already said, we do not assume here any special link between $\LL$ and $\QQ$ apart from the structure of $\nul(\LL)$ and compatible nonlinear estimates. Precisely, \emph{we do not require here} that $\LL$ is the linearized version of $\QQ$ around the equilibrium $\mu$.
		
		\item[(ii)] In dimension $d=2$, one knows that solutions to the NSF system exist globally in time.  When working in dimension $d\ge 3$ one can prove that the solutions to the NSF systems are global assuming 
		$$\left\|(\varrho_\ini,u_\ini,\theta_\ini)\right\|_{\mathbb{H}^{\frac{d}{2}-1}_x} \ll 1$$
		(see Appendix \ref{sec:N-S} for details), or equivalently, provided therefore that {the corresponding parts of the initial datum $f_\ini$ are small in $\mathbb{H}^{\frac{d}{2}-1}_{x}(\Ss_v)$ norms.} In both cases, solutions to \eqref{eq:Kin-Intro} constructed in Theorem \ref{thm:hydrodynamic_limit} are also global. This was already the case in the work \cite{GT2020} and this is an important contrast with respect to the result in \cite{BU1991}  which assumed small $f_{\ini}$ to generate global solutions. 
		
		\item[(iii)] In the same spirit, contrary to \cite{BU1991, CRT2022, CC2023}, we do not work with small $f_\ini$, but as in \cite{GT2020, G2023}, we consider the \emph{a priori} limit $f_{\ns}$ (which exists at least locally in time)  and construct the kinetic solution in its neighborhood with the same lifespan.  The smallness assumption we impose is transferred to the physical parameter $\eps$, i.e. assuming that a large number of collisions are experienced by the gas. This for instance extends the results of \cite{CRT2022,CC2023} to a larger class of initial data. Notice also that, since solutions to the NSF system \eqref{eq:NSFint} can be global depending on the properties of the initial data (such as {symmetry, etc.}), the kinetic solution to \eqref{eq:Kin-Intro} we construct are also global. 
		
		\item[(iv)] We also point out that our analysis is performed in the whole space $\R^{d}_{x}$. The strategy we adopt in the paper can be easily adapted to treat the case of a spatial torus $\T^{d}_{x}$. Furthermore, in such a case, assuming the initial datum $\PP f_\ini$ to be mean-free in space, i.e. 
		$$\int_{\T^d} \PP f_\ini(x) \d x = 0,$$ one can show the exponential trend to equilibrium for solutions to the kinetic equation \eqref{eq:Kin-Intro}. This is an easy consequence of Theorem \ref{thm:spectral_study} and this can be seen in the case of the Boltzmann equation (see \cite{GT2020, G2023, GMM2017, BMM2019}). The situation is much more delicate in the case of the whole space $\R^{d}_{x}$ and the trend to equilibrium for solutions to \eqref{eq:Kin-Intro} is not addressed in this paper. 
	\end{enumerate}
	Moreover, we wish to emphasize that, for \emph{well-prepared} initial datum, i.e. in the case in which $f_{\ini}$ is such that
	$$\nabla_x\cdot u[f_\ini]=0, \quad \nabla_x\left(\varrho[f_\ini]+\theta[f_\ini]\right)=0$$ 
	then  {no acoustic waves are produced:}
	$$f^\eps_\disp(t)=0 \qquad t \in [0,T]$$
	and, with the notations of Proposition \ref{prop:source_term}, there holds $\beta_\Wave(f_\ini,\eps)=0$. In particular, for a smooth initial datum $f_\ini \in \mathbb{H}^{s+1}_x(\Ss_v)$, this yields to the convergence rate  {$\beta_{\ns} = \mathcal{O}(\eps)$} which is optimal (see \cite{G2006}). Let us state this clearly in the following corollary.

	\begin{cor}[\thttl{Optimal convergence rate}]
		If the initial datum is smooth and well-prepared, in the sense that
		\begin{equation}\label{eq:addi}
		f_\ini \in \mathbb{H}^{s+1}_x \left( \Ss_v \right), \quad \nabla_x\cdot u[f_\ini]=0, \quad \nabla_x\left(\varrho[f_\ini]+\theta[f_\ini]\right)=0,\end{equation}	then the conclusion of Theorem \ref{thm:hydrodynamic_limit} holds with the decomposition
		$$f^\eps = f_\ns + f^\eps_\kin + f^\eps_\err,$$
		where, in this case, the error term is such that
		\begin{gather*}
			\sup_{0 \le t < T} \| f^\eps_\err(t) \|_{ \mathbb{H}^s_x \left( \Ss_v \right) } \lesssim \eps,
		\end{gather*}
		and, in particular, away from $t=0$
		$$\sup_{t_* \le t < T} \| f^\eps(t) - f_\ns(t) \|_{ \mathbb{H}^s_x \left( \Ss_v \right) } \lesssim \eps, \qquad \forall 0 < t_* < T.$$
	\end{cor}
	\color{black}
	\begin{exe}\label{exe:rate} With reference to Example \ref{exa:L1B1}, under the additional assumption \eqref{eq:addi}, the solution $f^{\eps}$ to the Boltzmann equation \eqref{eq:QQBoltz} with hard-sheres interactions is converging to $f_\ns$ with an explicit rate in, say $L^{\infty}_{x}\left(\Ss_{v}\right)$ with
$$\sup_{t_* \le t < T} \sup_{x} \| f^\eps(t) - f_\ns(t) \|_{ L^2(\mu^{-1}\d v) }  \lesssim \eps $$
 for any $< t_* < T.$
	\end{exe}

	\subsection{Main results -- second version in larger functional spaces} We improve also the two main results here above by showing that the same conclusion still holds in a larger functional space  $\Sl$ such that
	$$\Ss \hookrightarrow \Sl.$$
	To do so, our analysis requir{es a new set of Assumption which complement \ref{L1}--\ref{L4}:
		\begin{hypLE} Besides Assumptions \ref{L1}--\ref{L4}, one assumes that $\LL$ satisfies
			\begin{enumerate}[label=\hypst{LE}]
				\item \label{LE} Besides the splitting provided in \ref{L3}, the operator $\LL$ can be decomposed as 
				$$\LL = \BB^{(0)}+ \AA^{(0)}, \qquad \dom\left(\BB^{(0)}\right)=\dom(\LL), \qquad \AA^{(0)} \in \BBB(\Sl,\Ss)$$
				where the splitting is 	
				compatible with a hierarchy of Hilbert spaces $\left(\Sl_j\right)_{j=0}^{2}$ such that
				\begin{enumerate}
					\item the spaces $\Sl_j$ continuously and densely embed into one another:
					$$\Sl_{2} \hookrightarrow \Sl_{1} \hookrightarrow \Sl_{0} =\Sl, \quad \Ss \hookrightarrow \Sl,$$
					\item \label{assumption_large_multi-v} the multiplication by $v$ and its adjoint are bounded from $\Sl_{j+1}$ to $\Sl_{j}$:
					\begin{equation*}
						\|v f\|_{\Sl_{j}} \lesssim \|f\|_{\Sl_{j+1}}, \qquad 
						\|v^{\star} f\|_{\Sl_{j}} \lesssim \|f\|_{\Sl_{j+1}}, \quad j=0,1,
					\end{equation*}
					%		$$X_2 \xrightarrow{v} X_1 \xrightarrow{v} X_0 = X,$$
					%$$Y_{2} \xrightarrow{\times v, ~ (\times v)^\star} Y_{1} \xrightarrow{ \times v, ~ (\times v)^\star } Y_{0} = Y,$$
					\item the part $\BB^{(0)}_{\xi} = \BB^{(0)} - i v \cdot \xi$ is dissipative on each space $\Sl_j$ and $\Ss$ uniformly in $\xi  \in \R^d$, that is to say, for $\Sg = \Sl_0, \Sl_1, \Sl_2, \Ss$
					$$
					\mathfrak{S}_\Sg\left(\BB_{\xi}^{(0)} \right) \cap \Delta_{-\lambda_\BB} = \varnothing$$
					and $$\sup_{\xi \in \R^d}\left\|\RR\left(z,\BB^{(0)}_{\xi} \right)\right\|_{\BBB(\Sg)} \lesssim | \re z + \lambda_\BB|^{-1}, \qquad \forall z \in \Delta_{-\lambda_\BB}.$$
					Specifically, in the space $\Sl$, there holds
					$$\Re \lla  \BB^{(0)}_{\xi} f , f \rra_{\Sl} \le - \lambda_\BB \| f \|_{\Slp}^2, \qquad \forall f \in \Slp$$
					for some dissipation Hilbert space $\Slp$ satisfying
					$$\Ssp \hookrightarrow \Slp \hookrightarrow  \Sl, \qquad \| \cdot \|_{\Sl} \le \| \cdot \|_{\Slp},$$
					\item \label{assumption_large_bounded_A} the operator $\AA^{(0)}$ and its adjoint $(\AA^{(0)})^{\star}$ are bounded in the following spaces
					$$\AA^{(0)} \in \BBB(\Sl;\Ss) \cap \BBB(\Sl_{j};\Sl_{j+1}), \qquad   (\AA^{(0)}) ^{\star} \in \BBB(\Sl_{j};\Sl_{j+1}), \quad j=0,1\,.$$
					%		while its adjoint $\left(\AA^{(0)}\right)^{\star}$ satisfies
					%		$$\left(\AA^{(0)}\right)^{\star} \in \BBB(\Sl_{j};\Sl_{j+1}), \qquad j=0,1.$$
					%		 from $Y$ to $\Ss$, and both the operator $\AA^{(0)}$ and its adjoint $(\AA^{(0)})^{\star}$ are bounded from $Y_{j}$ to $Y_{j+1}$:
					%\begin{gather*}
					%	Y=Y_{0} \xrightarrow{\AA^{(0)}, ~ (\AA^{(0)})^{\star}} Y_{1} \xrightarrow{\AA^{(0)}, ~ (\AA^{(0)})^{\star}} Y_{2},\quad Y \xrightarrow{\AA^{(0)}} X.
					%\end{gather*}
				\end{enumerate}
			\end{enumerate}
	
			\begin{enumerate}[label=\hypst{BE}]
				\item \label{BE} The corresponding nonlinear assumption is then the following:
				$$\lla \QQ(f, g) , h \rra_{\Sl} \lesssim \| h \|_{\Slp} \left( \| f \|_{\Sl} \| g \|_{\Slp} + \| f \|_{\Slp} \| g \|_{\Sl} \right), \qquad f,g,h \in \Slp\,.$$
			\end{enumerate}
		\end{hypLE}
	 
		\begin{rem}
			As shown in \cite{GMM2017, BMM2019, G2023}, the operator $\LL$ and $\QQ$ satisfy \ref{LE} and \ref{BE} in the spaces
			$$\Sl_j = L^2\left( \la v \ra^{k + 2j} \d v \right), \qquad \Slp = L^2\left( \la v \ra^{k+1} \d v \right), \qquad \dom\left( \LL \right) = L^2\left( \la v \ra^{k+2} \d v \right)$$
			for some $k > 0$.
		\end{rem}
		\color{black}
		
		As in Definition \ref{defi:Hhstar}, one can define the dual space $\Slm$ of $\Slp$ as the completion of $\Sl$ for the norm
		$$\| f \|_{\Slm} := \sup_{ \| \varphi \|_{\Slp} \le 1 } \la f, \varphi \ra_{\Sl}.$$
		In that space $\Sl$, combining suitable enlargement techniques introduced in \cite{GMM2017} with a bootstrap argument, we derive the following improvement of the spectral Theorem \ref{thm:spectral_study}.
		\begin{theo}[\textit{\textbf{Enlarged spectral result}}]
			\label{thm:enlarged_thm}
			Assume \ref{L1}--\ref{L4} as well as \ref{LE}. Then the results of Theorem \ref{thm:spectral_study} hold with $(\Ss, \Ssp, \Ssm)$ replaced by $(\Sl, \Slp, \Slm)$. Furthermore the spectral projectors are \emph{regularizing} in the sense that, in the decomposition
			\begin{equation*}
				\PP_\star(\xi) = \PP^{(0)}_\star\left( \widetilde{\xi} \right) + i \xi \cdot \PP_\star^{(1)}\left( \widetilde{\xi} \right) + S_{\star}(\xi),
			\end{equation*}
			each term belongs to $\BBB\left(\Slm ; \Ssp \right)$ uniformly in $|\xi| \le \alpha_{0}$, and $\| S_\star(\xi) \|_{\BBB( \Slm ; \Ssp )} \lesssim | \xi |^2$. Finally, the decay estimate \eqref{decay-semigroup}
			extends to $\Sl$ as follows: for any $\xi \in \R^d$, any $\sigma \in (0,\sigma_0)$ and any $f \in \Sl$
			\begin{subequations}
				\label{decay-semigroup-EE} 
				\begin{align}
					\notag
					\sup_{t \ge 0} \, e^{2 \sigma_0 t} \big\| & U_{\xi}(t) \left( \Id - \PP(\xi) \right) f \big\|^2_{\Sl} \\
					\label{eq:decay-EE} 
					& + \int_0^\infty e^{2 \sigma t} \big\| U_{\xi}(t)\left( \Id - \PP(\xi) \right)f \big\|_{\Slp}^2 \, \d t \le C_\sigma \|\left( \Id - \PP(\xi) \right)f \|_{\Sl}^2, 
				\end{align}
				whereas for any $f \in \Slm$
				\begin{equation}
					\label{eq:decay_Ee'}
					\int_0^\infty e^{2\sigma t} \left\| U_{\xi}(t)\left( \Id - \PP(\xi) \right) f \right\|_{\Sl}^2 \, \d t \le C_\sigma \| \left( \Id - \PP(\xi) \right)f \|_{\Slm}^2\,.
				\end{equation}
			\end{subequations}
		\end{theo}
		
		 {The extension of Theorem \ref{thm:spectral_study} to a larger Hilbert space $\Sl$ is done using the enlargement procedure  developed in \cite{GMM2017}, which inspired the subtle bootstrap argument leading to the regularity properties of $\PP_\star$}. Again, we refer to Section \ref{sec:newspec} for a description of the proof.  {The aforementioned regularity of $\PP_\star$ can be improved in the presence of yet another suitable splitting of $\LL$, this time in Banach spaces instead of just Hilbert spaces}. This is done in Theorem \ref{thm:regularized_thm}, but, since such a result is not necessary for  the derivation of the NSF system \eqref{eq:NSFint}, we do not give the statement in this introduction.
		
		We assume for this next theorem \ref{L1}--\ref{L4} and \ref{Bortho}--\ref{Bisotrop}, as well as \ref{LE} and \ref{BE}.
		
		\begin{theo}[\textit{\textbf{Enlarged hydrodynamic limit theorem}}]
			\label{thm:hydrodynamic_limit-gen_symmetric}
			Under the assumptions of Theorem \ref{thm:hydrodynamic_limit}  {on $f_\ini \in \mathbb{H}^s_x \left( \Sl_v \right)$} and on the solution to the Navier-Stokes-Fourier system,  {the conclusion of Theorem \ref{thm:hydrodynamic_limit} still holds with the following differences:}
			\begin{enumerate}
				\item[\textit{\textbf{(1)}}]\textit{\textbf{Existence of a unique solution.}}
				There exists some small $c_0 > 0$ and $\eps_0 > 0$ such that the equation
				$$\partial_t f^\eps = \frac{1}{\eps^2} \left( \LL - \eps v \cdot \nabla_x \right) f^\eps + \frac{1}{\eps} \QQ\left( f^\eps, f^\eps \right), \quad f^\eps(0, x, v) = f_\ini(x, v)$$
				admits for any $\eps \in (0, \eps_0]$ a unique solution among those satisfying
				$$\sup_{0 \le t < T} \| f^\eps(t) \|_{ \mathbb{H}^s_x \left( \Sl_v \right) } \le \frac{c_0}{\eps}, \qquad f^\eps \in L^2_{\rm{loc}} \left( [0, T) ; \mathbb{H}^s_x \left( \Slp_v\right) \right),$$
				and it satisfies furthermore the following uniform estimate:
				$$\sup_{0 \le t < T} \| f^\eps(t) \|_{ \mathbb{H}^s_x \left( \Ss_v \right) }^2 + \int_0^T \| | \nabla_x |^{1-\alpha } f^\eps(t) \|_{ \mathbb{H}^{s-1+\alpha}_x \left( \Ssp_v \right) }^2 \d t \, \lesssim 1.$$\color{black}
				Moreover, %under Assumption \ref{BE}, 
				$$f^\eps \in\CC\left( [0, T) ; \mathbb{H}^s_x \left( \Sl_v \right) \right)\,.$$
%				whereas, under Assumption \ref{BED},
%				$$f^\eps \in \CC\left( [0, T) ; \mathbb{H}^s_x \left( \Sh_{-1} \right) \right).$$
				\item[\textit{\textbf{(2)}}]\textit{\textbf{Decomposition and convergence of the solution.}}
				The solution $f^\eps$ splits as the sum of some limiting part $f_\ns$, some initial layers $(f^\eps_\disp, f^\eps_\kin)$, and a vanishing part~$f^\eps_\err$:
				$$f^\eps = f_\ns + f^\eps_\disp + f^\eps_\kin + f^\eps_\err$$
				where  $f^\eps_\disp$ and $f^\eps_\err$ satisfy the same estimates, and this time the kinetic part $f^\eps_\kin$ satisfies
				$$\sup_{0 \le t < T} e^{2 \sigma t / \eps^2} \| f^\eps_\kin(t) \|_{ \mathbb{H}^s_x \left( \Sl_v \right) }^2 + \frac{1}{\eps^2} \int_0^T e^{2 \sigma t / \eps^2} \| f^\eps_\kin(t) \|_{ \mathbb{H}^s_x \left( \Slp_v \right) }^2 \d t \,  {\lesssim 1}.$$
				\color{black}
			\end{enumerate}
		\end{theo}
	\begin{exe} The above result provides a generalization of Theorem \ref{thm:hydrodynamic_limit} to \emph{larger functional spaces}. Keeping on elaborating on the Boltzmann equation with hard spheres interactions as in Exampe \ref{rem:L1B1} , we obtain now the existence, uniqueness of solutions $f^{\eps}$ to \eqref{eq:QQBoltz} in the space 
$$\CC\left( [0, T) ; \mathbb{H}^s_x \left( \Sl \right) \right) \cap L^{2}([0,T],\mathbb{H}^{s}_{x}(\Slp))$$
where now $\Sl$ and $\Slp$ are $L^{2}$--spaces with \emph{polynomial weights}:	
	$$\Sl := L^2\left( \la v \ra^k \d v \right), \quad \Slp := L^2\left( \la v \ra^{k+\frac{1}{2}}\d v\right), \quad k > 2.$$
	Theorem \ref{thm:hydrodynamic_limit-gen_symmetric} also provides the convergence in similar spaces of $f^{\eps}$ to $f_\ns$ as $\eps \to 0.$	\end{exe}
		
		Finally, drawing inspiration from works such as \cite{CM2017, CTW2016, HTT2020, CG2022} which dealt with the Boltzmann equation without cutoff or the Landau equation, we present an alternative to the nonlinear assumption \ref{BE} in which the two arguments of $\QQ$ do not play symmetric roles.
		\begin{hypQE}
			Besides  Assumptions \ref{L1}--\ref{L4}, we assume the following:
			\begin{enumerate}[label=\hypst{BED}]
				\item \label{BED} We consider a hierarchy of spaces $(\Sh_j)_{j=-2-s}^1$ for some $s \ge 0$ such that
				$$\Ss \hookrightarrow \Sh_1 \hookrightarrow \Sl = \Sh_{0} \hookrightarrow \dots \hookrightarrow \Sh_{-2-s},$$
			\end{enumerate}
			whose dissipation spaces are embedded in the same way:
			$$\Ssp \hookrightarrow \Shp_1 \hookrightarrow  \Slp = \Shp_{0} \hookrightarrow \dots \hookrightarrow \Shp_{-2-s},$$
			and such that the following conditions hold:
			\begin{enumerate}
				\item the assumption \ref{LE} is satisfied in the larger spaces $\Sh_{-2-s}, \dots, \Sh_{-2}$ (but not necessarily in $\Sh_1$) and with a splitting of $\LL$ that may be different in each space $\Sh_j$,
				\item the nonlinear operator satisfies the following \emph{non-closed} dual estimate:
				\begin{equation}\label{eq:BEDNon}
					\| \QQ(f, g) \|_{\Shm_{j} } 
					\lesssim 
					\| f \|_{\Sh_{j}} \| g \|_{\Shp_{j+1}}
					+ \| f \|_{\Shp_{j}} \| g \|_{\Sh_{j+1}}
					, \quad j=-1-s, \dots,  0,
				\end{equation}
				\item the nonlinear operator satisfies the following \emph{closed} dual estimates:
				\begin{equation}\label{eq:BEDclosed}
					\begin{split}
											\la \QQ(f, g), g \ra_{ \Sh_{j} }
						\lesssim \sum_{ \{ a, b, c \} = \{j, j, -1-s \} }  & \| f \|_{\Sh_{ a }} \| g \|_{\Shp_{b}} \| g \|_{\Shp_{c}} \\
						& + \| f \|_{\Shp_{a}} \| g \|_{\Sh_{b}} \| g \|_{\Shp_{c}}, \quad j = -1-s, \dots, 0,
					\end{split}
				\end{equation}
				\item $\AA$ sends $\Sh_{j-1}$ to $\Sh_{j}$ at the dual level, in the sense that
				$$\la \AA f, f \ra_{\Sh_{j}} \lesssim \| f \|^2_{ \Sh_{j-1} }, \quad j=-1-s, \dots, 0.$$
			\end{enumerate}
		\end{hypQE}

		\begin{theo}[\textit{\textbf{Enlarged hydrodynamic limit theorem under \ref{BED}}}]
			\label{thm:hydrodynamic_limit-gen_degenerate}
			Consider $s \in \N$ such that $s \ge \max\{3, d/2 + 1\}$ and denote for $j = -1, 0$ the spaces $\rSSh{s}$ and $\rSShp{s}$ defined by the norms
			$$\|f\|_{\rSSh{s}_j} = \| f \|_{ L^2_x \left( \Sh_j \right) } + \| \nabla_x^s f \|_{  L^2_x(\Sh_{j-s}) },
			\qquad \|f\|_{\rSShp{s}_j} = \| f \|_{ L^2_x \left( \Shp_j \right) } + \| \nabla_x^s f \|_{ L^2_x (\Shp_{j-s}) }.$$
			For any $f_\ini \in \rSSh{s}_0$, and under the assumptions of Theorem \ref{thm:hydrodynamic_limit-gen_symmetric} on the solution to the Navier-Stokes-Fourier system,  the conclusion of Theorem \ref{thm:hydrodynamic_limit-gen_symmetric} still holds with the following difference.
			There exists some small $c_0 > 0$ and $\eps_0 > 0$ such that the equation
				$$\partial_t f^\eps = \frac{1}{\eps^2} \left( \LL - \eps v \cdot \nabla_x \right) f^\eps + \frac{1}{\eps} \QQ\left( f^\eps, f^\eps \right), \quad f^\eps(0, x, v) = f_\ini(x, v)$$
				admits for any $\eps \in (0, \eps_0]$ a unique solution among those satisfying
				$$\sup_{0 \le t < T} \| f^\eps(t) \|_{ \rSSh{s}_0 } \le \frac{c_0}{\eps}, \qquad f^\eps \in L^2_{\rm{loc}} \left( [0, T) ; \rSShp{s}_0 \right).$$
				Moreover, it satisfies   the following uniform estimate:
				$$\sup_{0 \le t < T} \| f^\eps(t) \|_{ \rSSh{s}_0  }^2 + \int_0^T \| | \nabla_x |^{ \bm \alpha } f^\eps(t) \|_{ \left( \rSShp{s-\boldsymbol{\alpha}}_0 \right) }^2 \d t \, \lesssim 1$$
				and is continuous in the larger space $\rSSh{s}_{-1}$:
				$$f^\eps \in \CC\left( [0, T) ; \mathbb{H}^s_x \left( \rSSh{s}_{-1} \right) \right).$$
				Finally, the kinetic initial layer is such that
				$$\sup_{0 \le t < T} e^{2 \sigma t / \eps^2} \| f^\eps_\kin(t) \|_{ \rSSh{s}_0 }^2 + \frac{1}{\eps^2} \int_0^T e^{2 \sigma t / \eps^2} \| f^\eps_\kin(t) \|_{ \rSShp{s}_0 }^2 \, \d t \,  {\lesssim 1}.$$
		\end{theo}
	\begin{exe} The above result is particularly well-suited to the study of the Landau equation. In such a case, the collision operator is given by
	\begin{equation*}
			\QQ(f, f) = \nabla_v \cdot \int_{\R^3_{v_*} \times \S^{d-1}_\sigma } | v - v_* |^{\gamma+2} \Pi(v-v_*) \Big( f(v_*) \nabla_v f(v) - \nabla_{v_*} f(v_*) f(v) \Big) \d v_*,
		\end{equation*}
		where $\Pi(z) = \Id - |z|^{-2} z \otimes z$ is the orthogonal projection onto $z^\perp$. We refer to Appendix \ref{sec:Landau-Boltz} for more details.\end{exe}
		
		A few comments are in order regarding Theorem \ref{thm:hydrodynamic_limit-gen_symmetric} and \ref{thm:hydrodynamic_limit-gen_degenerate}:
		\begin{itemize}
			\item Only the kinetic part is now living in the smaller space $\mathbb{H}^s_x(\Sl_v)$, whereas the dispersive parts $f_\disp^\eps$ and $f^\eps_\err$ satisfy the same properties as in Theorem \ref{thm:hydrodynamic_limit}.  
			\item The method of proof of Theorem \ref{thm:hydrodynamic_limit-gen_symmetric} is the same as for \ref{thm:hydrodynamic_limit}; the existence of solution to the kinetic equation can be deduced from the application of Banach fixed point theorem {(and the proof is led independently of whether we assume \ref{BE} or not)}. Concerning Theorem \ref{thm:hydrodynamic_limit-gen_degenerate}, the approach is more involved and a careful study of an approximating scheme (a variation of Picard iterations) has to be done in order to overcome the difficulty induced by the lack of symmetry of $\QQ.$ 
			\item Notice that, as for Theorem \ref{thm:hydrodynamic_limit-gen_symmetric}, we \emph{do not require} here any smallness assumptions on $f_{\rm in}$ and the smallness is totally transferred in the parameter $\eps.$ Recalling that assumptions \ref{BED} are suited for the study of the Boltzmann equation without cut-off assumption and for the Landau equation, this is an important improvement with respect to the previous results in the field which all require some additional restriction on the size of the initial datum to derive the hydrodynamical limit (see \cite{CRT2022} for the Landau equation and \cite{CC2023} for the Boltzmann equation).
			\item {Finally, we emphasize that this provides, up to our knowledge, the first result concerning the strong Navier-Stokes limit for initial data with algebraic decay with respect to the velocity variable in the case of Boltzmann equation without cut-off assumption or Landau equation (see Appendix \ref{sec:Landau-Boltz}).}
			\end{itemize}
		We will comment with more details the conclusion of Theorems  \ref{thm:hydrodynamic_limit}, \ref{thm:hydrodynamic_limit-gen_symmetric} and \ref{thm:hydrodynamic_limit-gen_degenerate} in Section \ref{sec:detail} where a detailed description of the proof and the role of the various assumptions will be illustrated. 
		
		\subsection{Outline of the paper} In the next Section \ref{sec:detail}, we introduce the main ideas underlying the proofs of {the hydrodynamic limit Theorems} \ref{thm:hydrodynamic_limit}, \ref{thm:hydrodynamic_limit-gen_symmetric} and \ref{thm:hydrodynamic_limit-gen_degenerate}. Notations and mathematical objects that are used in the rest of the analysis are also introduced in Section \ref{sec:detail}.
		
		Section \ref{scn:spectral_study} gives the full proof of both the spectral theorems \ref{thm:spectral_study} and \ref{thm:enlarged_thm}. A detailed description of approach is given in Section \ref{sec:newspec} and the proof of Theorem \ref{thm:spectral_study} is then derived in various steps, together with its ‘‘regularized version‘‘ Theorem \ref{thm:regularized_thm}.
		
		Section \ref{scn:study_physical_space} establishes the main consequences of Theorem \ref{thm:spectral_study} on the semigroup  $(U^\eps(t))_{t\geq0}$ generated by the linear part $\eps^{-2} \left( \LL - \eps v \cdot \nabla_x \right)$ of \eqref{eq:Kin-Intro} in the various regimes/time scales relevant for the hydrodynamic limit. In particular, the comparison between the linearized semigroups associated to \eqref{eq:Kin-Intro} and \eqref{eq:NSFint} is given in Section \ref{scn:study_physical_space}.
		
		The main bilinear estimates are then established in Section \ref{sec:Bilin} as well as the main tools used for the hydrodynamic limit (and in particular the mild formulation of \eqref{eq:NSFint}). The proof of Theorem \ref{thm:hydrodynamic_limit} and \ref{thm:hydrodynamic_limit-gen_symmetric} under assumptions \ref{BE}, is then given in Section \ref{scn:proof_hydrodynamic_limit_symmetrizable}, whereas the proof of Theorem \ref{thm:hydrodynamic_limit-gen_degenerate} under Assumption \ref{BED} is given in Section \ref{scn:hydrodynamic_limit_BED}.
		
To make the paper self-contained, we end it with three different Appendices.	In Appendix \ref{sec:Landau-Boltz}, we discuss the general assumptions \ref{L1}--\ref{L4}, \ref{Bortho}--\ref{Bisotrop}, \ref{LE}--\ref{BE} and \ref{BED} for an extensive list of physical models including, as said, the classical Boltzmann and Landau equations as well as their quantum counterpart covering, at the linearized level, both the Fermi-Dirac and Bose-Einstein descriptions. Appendix \ref{sec:toolbox} gives the functional toolbox with particular emphasis to Littlewood-Paley theory (Section \ref{scn:littlewood-paley} and other results relevant for our analysis). We also present in Section \ref{scn:boostrap_projectors} the bootstrap argument for projection operators which is a cornerstone of Theorem \ref{thm:enlarged_thm}. The final Appendix \ref{sec:N-S} recalls the main properties of the Navier-Stokes-Fourier system  \eqref{eq:NSFint} that are needed for our analysis {and proves some results necessary for our framework}.
		
\subsection*{Acknowledgments} Both the author gratefully acknowledge the financial support from the Italian Ministry of Education, University and Research (MIUR), ``Dipartimenti di Eccellenza'' grant 2022-2027. They also thank Isabelle Gallagher and Isabelle Tristani for insightful discussions about hydrodynamic limits. 		
		
		\section{Detailed description of our proofs}\label{sec:detail}
		
		We give here a precise description of the main steps of our approach to prove the above two Theorem \ref{thm:hydrodynamic_limit} and \ref{thm:hydrodynamic_limit-gen_symmetric}. We use repeatedly the spectral properties of $\LL_\xi$ and the properties of the associated semigroup $\left(U_\xi(t)\right)_{t\ge0}$ as established in Theorems \ref{thm:enlarged_thm} and \ref{thm:enlarged_thm}. Notations are those introduced in those two results.

		\subsection{The functional setting} The conclusion of the Theorem \ref{thm:hydrodynamic_limit} and the splitting of $f^\eps$ in \eqref{eq:decomp} suggest to introduce the following definitions of position-velocity spaces and time-position-velocity spaces suited to the different regimes (kinetic, diffusive and mixed) we will consider in this work.
		\begin{defi}\label{defi:NORMSPACES} Let $s \in \R$ be given. 
			\begin{enumerate}
				\item For $\Sg=\Ss$ or $\Sg=\Sl$,  we define the position-velocity spaces
				$$\SSg = \SSg^s = \mathbb{H}^{s}_{x}(\Sg_v),$$
				which we endow with their natural norms $\|\cdot\|_{\SSg}$ defined in \eqref{eq:normHsx}.    
				We define in the same way the spaces 
				$$\rSSgp{s}=\mathbb{H}^s_x(\Sgp_v), \qquad \rSSgm{s}=\mathbb{H}^s_x(\Sgm_v)$$
				with $\Sgp_v=\Ssp$ or $\Slp$ and $\Sgm_v=\Ssm$ or $\Slm.$
				\item Given $T \in (0,\infty]$ and $\sigma \in (0,\sigma_0)$ (with $\sigma_0$ defined in Theorem \ref{thm:spectral_study} or \ref{thm:enlarged_thm}), we introduce here the kinetic-type time-position-velocity space 
				$$\SSSl=\rSSSl{s}(T, \sigma, \eps):= \big\{ f \in \CC_b\left( [0, T) ; \rSSl{s} \right) \, ; \, \Nt f \Nt_{\rSSSl{s}} < \infty \big\},$$
				where the norm $\Nt \cdot \Nt_{\rSSSl{s}}$ is given by
				$$	\Nt f \Nt_{ \rSSSl{s}}^2 := \sup_{0 \le t < T} \, e^{2 \sigma t / \eps^2 } \| f(t) \|_{\rSSl{s}}^2 + \frac{1}{\eps^2} \int_0^T e^{2 \sigma t / \eps^2} \| f(t) \|_{{\rSSlp{s}}}^2 \d t.$$
				\item Given $T \in (0,\infty]$, $0 < \eta \ll 1$ and some mapping $\phi$ such that
				\begin{equation}\label{eq:nablaPhi}
					\left|\nabla_x\right|^{1-\alpha}\phi \in L^2\left( [0, T) ; \rSSsp{s} \right)
				\end{equation}
				where $\alpha \in \left(0, \frac{1}{2}\right)$ if $d = 2$ and $\alpha = 0$ if $d \ge 3$, we introduce the parabolic-type time-position-velocity space 
				$$\SSSs=\rSSSs{s}(T, \phi, \eta):= \left\{ f \in \CC_b\left( [0, T) ; \rSSs{s} \right) \, ; \, \Nt f \Nt_{\rSSSs{s}} < \infty \right\}$$
				endowed with the norm
				$$
				\Nt f \Nt_{\rSSSs{s}}^2 := \sup_{0 \le t < T} \left\{ w_{\phi, \eta}(t)^2 \| f(t) \|_{\rSSsp{s}}^2 + w_{\phi, \eta}(t)^2 \int_0^t \left\| |\nabla_x|^{1-\alpha} f( \tau ) \right\|_{\rSSsp{s}}^2 \d \tau \right\}, 
				$$
				where  we set
				\begin{equation}\label{eq:wphieta}
					w_{\phi, \eta}(t) = \exp\left( \frac{1}{2 \eta^2} \int_0^t\| |\nabla_x|^{1-\alpha} \phi(\tau) \|_{\rSSsp{s}}^2 \d \tau \right)\,.\end{equation} 
				\item Finally, with the notation of the previous points, we introduce the mixed-type time-position-velocity space 
				$$\SSSm = \rSSSm{s}(T, \phi, \eta, \eps):= \left\{ \CC_b\left( [0, T) ; \rSSs{s} \right) \, ; \, \Nt f \Nt_{\rSSSm{s}} < \infty \right\}$$
				endowed with the norm
				$$\Nt f \Nt_{\rSSSm{s}}^2 := \sup_{0 \le t < T } \left\{ w_{\phi, \eta}(t)^2 \| f(t) \|_{\rSSs{s}}^2 + \frac{1}{\eps^2} w_{\phi, \eta}(t)^2 \int_0^t \| f(\tau) \|_{\rSSsp{s}}^2 \d \tau \right\}.$$
		\end{enumerate}\end{defi}
		\begin{rem}
			 {Drawing inspiration from \cite{GT2020} or more precisely \cite{G2023},} the above norm $\Nt\cdot\Nt_{\SSSs}$ and $\Nt \cdot \Nt_\SSSm$, {and more specifically the time weight $w_{\phi, \eta}$,} are designed so that, no matter how big $\phi$ is, there holds
			\begin{equation}
				\label{eq:NS_exponential_weight}
				\forall t \in [0, T), \quad w_{\phi, \eta}(t) \left\| w_{\phi, \eta}^{-1} | \nabla_{x} |^{1-\alpha} \phi \right\|_{L^2 \left( [0, t] ; {\rSSsp{s}} \right) } \le \eta
			\end{equation}
			as can be seen by a simple computation (see Proposition \ref{prop:special_bilinear_hydrodynamic}).  Notice also that $w_{\phi, \eta}$ is decreasing and satisfies the bounds $0 < w_{\phi, \eta}(T) \le w_{\phi, \eta} \le 1$, consequently 
			\begin{equation}
				\label{eq:roughNt}
				\Nt f \Nt_{\rSSSs{s}}^2 \le \sup_{0 \leq t \leq T}\|f(t)\|_{\rSSsp{s}}^2+\int_0^T\left\|\,|\nabla_x|^{1-\alpha}f(\tau)\right\|_{\rSSsp{s}}^2\d \tau \le w_{\phi, \eta}(T)^{-2} \Nt f \Nt_{\rSSSs{s}}^2,
			\end{equation}
			and
			\begin{equation*}
				\Nt f \Nt_{\rSSSm{s}}^2 \le \sup_{0 \leq t \leq T}\|f(t)\|_{\rSSs{s}}^2 + \frac{1}{\eps^2} \int_0^T\left\|\,f(\tau)\right\|_{\rSSsp{s}}^2\d \tau \le w_{\phi, \eta}(T)^{-2} \Nt f \Nt_{\rSSSm{s}}^2.
			\end{equation*}
			We point out already that we will work in the proof of Theorems \ref{thm:hydrodynamic_limit}, \ref{thm:hydrodynamic_limit-gen_symmetric} and \ref{thm:hydrodynamic_limit-gen_degenerate} with the choice $\phi=f_\ns$ in \eqref{eq:nablaPhi}. 
		\end{rem}
		\begin{rem}
			Clearly, Theorems \ref{thm:hydrodynamic_limit-gen_symmetric} and \ref{thm:hydrodynamic_limit-gen_degenerate} suggest that, in the decomposition of the solution $f^\eps$ in \eqref{eq:decomp},  we will look for the kinetic part $f^\eps_\kin$ in the space $\SSSl$ in the sense that
			$$\Nt f^\eps_\kin\Nt_{\rSSSl{s}}  < \infty, \qquad s >\frac{d}{2},$$
			and, although its is not as obvious, the error term $f^\eps_\err$ will be constructed as the sum of a part in $\SSSs$ and one in $\SSSm$, which explains the global energy estimate satisfied by the solution.
			
			The first two regimes $\SSSl$ and $\SSSs$ are preserved by the two corresponding parts $U^\eps_\kin(\cdot)$ and $U^\eps_\hyd(\cdot)$ of the linearized flow of the equation \eqref{eq:Kin-Intro} (see Sections \ref{scn:reduction_problem}--\ref{scn:pseudo_hydro_kinetic} for their definitions), and the mixed regime $\SSSm$ is introduced to describe the interactions between the two different regimes (see Lemmas \ref{lem:decay_semigroups_hydro} and \ref{lem:decay_regularization_convolution_kinetic_semigroup}).
			\color{black}
		\end{rem}

		\subsection{Reduction of the problem}
		\label{scn:reduction_problem}
		We present in this section how we frame the problem of hydrodynamic limits. We start from the integral formulation
		\begin{equation}\label{eq:Kin-Mild}
			f^\eps(t) = U^\eps(t) f_\ini + \Psi^\eps\left[ f^\eps, f^\eps \right](t),
		\end{equation}
		where, denoting $2 \QQ^\sym(f, g) := \QQ(f, g) + \QQ(g, f)$
		\begin{gather*}
			\Psi^\eps[f,g](t) := \frac{1}{\eps} \int_0^t U^\eps(t - \tau) \QQ^\sym (f(\tau), g(\tau)) \d \tau
		\end{gather*}
		and $(U^\eps(t))_{t\geq0}$ is the $C^0$-semigroup in $\SSg^s$ generated by the full linearized operator (in original variables):
		$$\mathcal{G}_\eps f:=-\frac{1}{\eps} v\cdot \nabla_x f +\frac{1}{\eps^2}\LL f, \qquad  \dom(\mathcal{G}_\eps)=\left\{f \in \SSg\,;\,\LL f \in \SSg\right\}$$
		where $\Sg=\Ss$ or $\Sg=\Sl.$ Notice that the semigroup $\left(U^\eps(t)\right)_{t\geq0}$ is related to the semigroup $\left(U_{\eps\,\xi}(t)\right)_{t \ge0}$; for $g \in \SSg$, setting
		$$\widehat{g}(\xi)=\widehat{g}(\xi,\cdot)=\FF_{x}\left[g(x,\cdot)\right](\xi)=\int_{\R^{d}}e^{i\xi\cdot x}g(x,\cdot)\d x \in \Sg$$
		one has
		$$\FF_x\left[\mathcal{G}_\eps g\right](\xi) = \eps^{-2}\LL_{\eps\,\xi}\widehat{g}(\xi,\cdot)$$
		so that
		$$\FF_x U^\eps(t)\FF_x^{-1}=U_{\eps\,\xi}\left(\frac{t}{\eps^2}\right), \qquad t\ge0.$$
		
		We define now the linearized semigroup $(U_\ns(t))_{t\ge0}$, adopting again a Fourier-based description which involves the projectors $\PP^{(0)}_\Inc$ and $\PP_\Bou^{(0)}$ as defined in Theorem \ref{thm:spectral_study}.
		\begin{defi}
			\label{def:U_NS_V_NS}
			We define the \emph{diffusive} Navier-Stokes semigroup $\left(U_\ns(t)\right)_{t\geq0}$ through its Fourier transform for any $g \in \SSgm=\SSsm$ or $\SSgm=\SSlm$ and $t \ge 0$ as
			\begin{align*}
				U_\ns(t) g  = \FF_x^{-1} \Big[ & \exp( - t \kappa_\Inc | \xi |^2 ) \PP^{(0)}_\Inc \left( \widetilde{\xi} \right) \widehat{g}(\xi) \\
				& +  \exp( - t \kappa_\Bou | \xi |^2 ) \PP^{(0)}_\Bou \left( \widetilde{\xi} \right) \widehat{g}(\xi) \Big], \quad \widetilde{\xi} := \frac{\xi}{|\xi|}.
			\end{align*}
			We also define the one parameter family   $\left(V_\ns(t)\right)_{t\ge 0}$ as
			\begin{align}\label{eq:defi_V_ns}
				V_\ns(t) g := \FF^{-1}_x\Big[  \exp( - t \kappa_\Inc | \xi |^2 ) \PP^{(1)}_\Inc\left( \widetilde{\xi} \right) \widehat{g}(\xi)  + \exp( - t \kappa_\Bou | \xi |^2 ) \PP^{(1)}_\Bou\left( \widetilde{\xi} \right) \widehat{g}(\xi) \Big].
			\end{align}
		\end{defi}
		The link between the above objects and  solutions to the NSF system \eqref{eq:NSFint} is given by the following Proposition whose proof is postponed to Appendix \ref{sec:N-S}.
		
		\begin{prop}
			\label{prop:equivalence_kinetic_hydrodynamic_INSF}
			Consider $T_0 > 0$ and $(\varrho, u, \theta) \in L^\infty\left( [0, T_0] ; \mathbb{H}^s_x \right) \cap L^2\left( [0, T_0] ; H^{s+1}_x \right)$,
			and define the corresponding macroscopic distribution $f$ as
			\begin{equation}\label{eq:f-macro}
				f(t, x, v) = \varrho(t, x) \mu(v) + u(t, x) \cdot v \mu(v) + \frac{1}{E(K-1)} \theta(t, x) \left( |v|^2 - E \right) \mu(v).\end{equation}
			The macroscopic distribution $f$ satisfies the integral equation 
			\begin{equation}\begin{split}
					\label{eq:NS-IntVNS}
					f(t) & = U_\ns(t) f_\ini + \Psi_\ns\left[ f, f \right](t) \\
					& = U_\ns(t) f_\ini + \int_0^t \nabla_x \cdot  V_\ns(t-\tau) \QQ^\sym (f(\tau), f(\tau) )(t).
			\end{split}\end{equation}
			if and only if the coefficients $(\varrho, u, \theta)$ satisfy the incompressible Navier-Stokes-Fourier equations:
			$$
			\begin{cases}
				\partial_t u + \vartheta_\Inc u \cdot \nabla_x u = \kappa_\Inc \Delta_x u - \nabla_x p, & \nabla_x \cdot u = 0, \\
				\\
				\partial_t \theta + \vartheta_\Bou u \cdot \nabla_x \theta = \kappa_\Bou \Delta_x \theta, & \nabla_x (\varrho + \theta) = 0,
			\end{cases}
			$$
			where we denoted
			\begin{gather*}
				\vartheta_\Inc = - \frac{\vartheta_1}{2} \left( \frac{d}{E} \right)^{\frac{3}{2}}, \qquad
				\vartheta_\Bou = -\frac{1}{ K \sqrt{K(K-1)} } \left( 2 \vartheta_2 + \frac{2 \vartheta_3}{E (K-1)} \right)
			\end{gather*}
			with $\vartheta_1,\vartheta_2$ and $\vartheta_3$ defined in Lemma \ref{lem:Q_A}.
		\end{prop}
		With this at hands, one sees that \eqref{eq:NS-IntVNS} provides a \emph{kinetic formulation} of the NSF system \eqref{eq:NSFint}. As said in the Introduction, our approach is ‘‘top-down" so we start with solutions $(\varrho,u,\theta)$ to the Navier-Stokes-Fourier system \eqref{eq:NSFint} to recover information about the kinetic equation \eqref{eq:Kin-Intro} in its mild formulation \eqref{eq:Kin-Mild}. This allows in particular to define the ‘‘kinetic formulation'' to the NSF system
		\begin{equation}
			\label{eq:reduction_NS_kin}
			f_\ns(t)=U_\ns(t)f_\ini + \Psi_\ns[f_\ns,f_\ns](t).
		\end{equation}
		 Then, on the basis of the above Proposition, the hydrodynamic limit problem consists in proving
		$$\lim_{\eps \to0}\Big(U^\eps(t) f_\ini + \Psi^\eps\left[ f^\eps, f^\eps \right](t) \Big)= U_\ns(t) f_\ini + \Psi[f_\ns, f_\ns](t) = f_\ns(t)$$
		in some precise sense. We point out already that using the representation \eqref{eq:f-macro}, the solution $f_\ns$ to \eqref{eq:reduction_NS_kin} actually belongs to $\SSSs$ (see Lemma \ref{lem:NS_parabolic_space}).\color{black} 
		The key point will be therefore to split suitably $U^\eps(\cdot)$ (and $\Psi^\eps$ accordingly) in order to prove the convergence. The splitting will be based upon the different parts of the spectrum identified in Theorem \ref{thm:spectral_study}:
		$$U^\eps(t)=U^\eps_\ns(t)+U^\eps_\Wave(t)+U^\eps_\kin(t).$$
		Here $U^\eps_\ns(t)$ is the leading order term which is expected to converge, as $\eps\to0$ towards the linearized Navier-Stokes semigroup $U_\ns(t)$ whereas $U^\eps_\Wave(t)$ contains the acoustic waves {responsible for dispersive effects (which are absent if the initial data is well-prepared)}, and the combination of these two semigroups can be seen as a \emph{pseudo-hydrodynamic semigroup} encapsulating the macroscopic behavior of the solution $f^\eps$. The part $(U^\eps_\kin(t))_{t\ge0}$ keeps track of the  {pseudo-}kinetic (microscopic) behavior of the solution which is exponentially small in $\frac{t}{\eps^2}$ due to the dissipation of entropy, enhanced by the numerous collisions in this hydrodynamic scaling. Since 
		$$\Psi^\eps[f,f](t)=\frac{1}{\eps}\int_0^t U^\eps(t-\tau)\QQ^{\sym}(f(\tau),f(\tau))\d\tau$$
		the above splitting of $U^\eps(\cdot)$ induces a similar splitting of the nonlinear term as 
		$$\Psi^\eps[f,f]=\Psi^\eps_\ns[f,f]+\Psi^\eps_\Wave[f,f]+\Psi^\eps_\kin[f,f].$$
		Precise definitions of these objects are given in the next sections. Before this, we 
		briefly describe the main difficulties faced in the proof of Theorem \ref{thm:hydrodynamic_limit}:
		\begin{enumerate}
			\item  As said,  the major difficulty in establishing the hydrodynamic limit $\eps \to 0$ lies in the control of the stiff term $\frac{1}{\eps}\QQ\left(f^\eps,f^\eps\right)$.  This requires a precise understanding of the asymptotic behavior when $\eps \to 0$ of both $U^\eps(t)f$ and convolutions of the type
			$$\frac{1}{\eps} \int_0^t U^\eps(t-\tau)\varphi(\tau)\d\tau$$
			in various norms, having in mind that $\varphi=\QQ(f,f)$.   {Furthermore,} the nonlinear operator $\QQ$ induces a loss of regularity in the sense that $\QQ(f, f) \in \Ssm$ when $f \in \Ssp$, where we recall
			$$ {\Ssp \subset \Ss \subset \Ssm = ( \Ssp )'}.$$
			One of the difficulty is therefore to show that  convolution by $\left( U^\eps(t) \right)_{t \ge 0}$ is able to compensate this loss of regularity.
			\item  As explained in the introduction, in the abstract framework considered here, our minimal assumptions on $\LL$ and $\QQ$ are not sufficient to deduce in a direct way regularization estimates or direct boundedness of $\Psi^\eps$ as it is the case for the Boltzmann equation under cut-off assumptions in \cite{BU1991,GT2020} or for the Landau equation \cite{CRT2022}. In a more explicit way, our splitting
			$$\LL=\AA+\BB$$ 
			does not induce, in full generality, regularization estimates of the form
			$$t \mapsto \| \exp(t \BB) \|_{ \BBB\left( \Ssm ; \Ss \right) } +  \| \exp(t \BB) \|_{ \BBB\left( \Ss ; \Ssp \right) } \in L^1_{\text{loc}}((0,T))$$
			which would allow to compensate the unboundedness of $\QQ$ in the Duhamel nonlinear term $\Psi^\eps$. Inspired by known energy methods introduced for instance in \cite{G2004} and which rely on a suitable  dissipation in $L^2$-norm, the abstract functional setting which is adapted to our framework is the one involving spaces of the type
			$$L^\infty_t \Ss \cap L^2_t \Ssp.$$
			{Such spaces correspond to the above defined space $\SSSl, \SSSs$ and $\SSSm$.}
		\end{enumerate}
		\subsection{The pseudo-hydrodynamic and pseudo-kinetic projectors}
		\label{scn:pseudo_hydro_kinetic}
		In this section, we denote by $(\Sg, \SSg, \SSgp, \SSgm)$ the spaces $(\Ss, \SSs, \SSsp, \SSsm)$ under assumption \ref{L1}--\ref{L4}, as well as $(\Sl, \SSl, \SSlp, \SSlm)$ under the extra assumption \ref{LE}.
		We introduce the \textit{pseudo-hydrodynamic projector}, denoted $\PP_\hyd^{\eps}$, corresponding to the small eigenvalues identified in Theorem \ref{thm:spectral_study}, and defined as a Fourier multiplier: 
		\begin{gather*}
			\PP_\hyd^\eps g := \FF^{-1}_\xi \Big[ \PP(\eps \xi) \widehat{g}(\xi)\Big].
		\end{gather*}
		Using the splitting of $\PP$ in \eqref{eq:splitPP}, one sees that it is made up of two parts; one corresponding to the acoustic modes, denoted $\PP^\eps_\Wave$, and another one corresponding to the Navier-Stokes-Fourier modes, denoted $\PP^\eps_\ns$:
		$$\PP_{\hyd}^{\eps}=\PP^{\eps}_{\Wave} + \PP^{\eps}_{\ns}$$
		defined in the following.
		\begin{defi}
			\label{def:hydrodynamic_projectors}
			The projectors $\PP_{\Wave}^{\eps}$ and $\PP_{\ns}^{\eps}$ are defined through their Fourier transform, namely for any $g \in \SSgm$
			\begin{gather*}
				\PP_\Wave^\eps g := \FF^{-1}_{\xi} \Big[\PP_{+\Wave} (\eps \xi) \widehat{g}(\xi) + \PP_{-\Wave} (\eps \xi) \widehat{g}(\xi)\Big],\\
				\PP_\ns^\eps g := \FF^{-1}_{\xi} \Big[\PP_{\Inc} (\eps \xi) \widehat{g}(\xi) + \PP_{\Bou} (\eps \xi) \widehat{g}(\xi)\Big].
			\end{gather*}
			We also define the limit of the first one as $\eps \to 0$ provided by the expansions of the projectors in Theorem \ref{thm:spectral_study}:
			\begin{equation*}
				\PP_\disp g := \FF^{-1}_{\xi} \Big[\PP_{+\Wave}^{(0)} \left( \widetilde{\xi} \right) \widehat{g}(\xi) + \PP_{-\Wave}^{(0)} \left( \widetilde{\xi} \right) \widehat{g}(\xi)\Big], \quad \widetilde{\xi} := \frac{\xi}{|\xi|}.
			\end{equation*}
		\end{defi}
		Using these projectors, we define the corresponding partial semigroups:  
		\begin{gather*}
			U^\eps_\star(\cdot):= \PP^\eps_\star U^\eps(\cdot) = U^\eps (\cdot)\PP^\eps_\star, \quad \star = \hyd, \ns, \Wave,
		\end{gather*}
		which gives
		\begin{equation*}
			U^\eps_\hyd(\cdot) := U^\eps_\ns(\cdot) + U^\eps_\Wave(\cdot).
		\end{equation*}
		More precisely, the above semigroups are defined as follows.
		\begin{defi}
			\label{def:hydro_semigroups}
			The pseudo-Navier-Stokes (diffusive) part $\left(U^\eps_\ns(t)\right)_{t\ge 0}$ is defined through its Fourier transform for any $g \in \SSgm$ and $t \ge 0$ as
			\begin{align}
				\label{eq:Unst}
				U^\eps_\ns(t) g = \FF_\xi^{-1} \Big[&  \exp\left( - \eps^{-2} t \lambda_\Bou(\eps \xi) \right) \PP_\Bou(\eps \xi)\widehat{g}(\xi)   \\
				\notag
				& + \exp\left( - \eps^{-2} t \lambda_\Inc(\eps \xi) \right) \PP_\Inc(\eps \xi)\widehat{g}(\xi)\Big] ,
			\end{align}
			whereas the pseudo-acoustic (dispersive) part $\left(U^\eps_\Wave(t)\right)_{t\ge 0}$ is defined as
			\begin{align*}
				\label{eq:UWave}
				U^\eps_\Wave(t) g = \FF_\xi^{-1} \Big[ & \exp\left( \eps^{-2} t \lambda_{+ \Wave}(\eps \xi) \right) \PP_{+ \Wave}(\eps \xi)\widehat{g}(\xi)\\
				\notag
				& + \exp\left( \eps^{-2} t \lambda_{- \Wave}(\eps \xi) \right) \PP_{- \Wave}(\eps \xi)\widehat{g}(\xi)\Big].
			\end{align*}
		\end{defi}
		
		Because $\PP_\Wave^\eps \to \PP_\disp$ as $\eps \to 0$, the leading order terms of $U^\eps_\Wave(t)$ denoted respectively $U^\eps_\disp(t)$ will play also a crucial roles in the study of hydrodynamic limits:
		\begin{defi}
			The dispersive semigroup $\left(U^\eps_\disp(t)\right)_{t\ge 0}$ is defined as
			\begin{align*}
				U^\eps_\disp(t)g = \FF_\xi^{-1} \Big[  & \exp\left( i c \eps^{-1} t|\xi|  - t \kappa_\Wave | \xi |^2 \right) \PP^{(0)}_{+ \Wave}\left( \widetilde{\xi} \right) \widehat{g}(\xi) \\
				& + \exp\left( - i c \eps^{-1} t|\xi| - t \kappa_\Wave | \xi |^2 \right) \PP^{(0)}_{- \Wave}\left( \widetilde{\xi} \right) \widehat{g}(\xi) \Big].
			\end{align*}
		\end{defi}
		\color{black}
		\begin{rem}
			Recall that $U_\ns(t)$ and $V_\ns(t)$ were introduced in Definition \ref{def:U_NS_V_NS}. Observe that $U_\ns(t)$ is the leading order term in the expansion of $U^\eps_\ns(t)$ while  $\nabla_x \cdot V_{\ns}(t)$ is the leading order term of $U^\eps_\ns(t)$ on $\nul(\LL)^\perp$, that is to say
			$${U^\eps_\ns(t) \approx U_\ns(t), \qquad \left(U^\eps_\ns(t)\right)_{| \nul(\LL)^\perp } \approx \eps \nabla_{x} \cdot V_\ns(t)}$$
			as will be  {exploited} in Lemma \ref{lem:asymptotic_equiv_NS_semigroup}.
		\end{rem}
	 
		Note that the projectors $\PP_{\Inc}^{(0)}$, $\PP_{\Bou}^{(0)}$ and $\PP_{\pm \Wave}^{(0)}$ are macroscopic in the sense that they take values in
		$$\nul(\LL) = \left\{ \left(\varrho + u \cdot v + \theta \left( |v|^2 - E \right) \right) \mu \, ; \, \varrho, u, \theta \in L^2\left( \R^d \right) \right \}$$
		and vanish on its orthogonal, thus they can be characterized using the macroscopic components $\varrho, u$ and $\theta$ (see Remark \ref{rem:macro_representation_spectral}). Similarly, the first order projectors $\PP^{(1)}_\Inc$ and $\PP^{(1)}_\Bou$ restricted to $\nul(\LL)^\perp$ can be characterized in such a way, which will be useful for describing $V_\ns$.
		\begin{prop}
			\label{prop:macro_representation_spectral}
			The zeroth order projector related to the Navier-Stokes (incompressible) mode is characterized for $f=f(x,v) \in L^2_x(\Ss_v)$ by
			$$\varrho\left[ \PP^{(0)}_\Inc f \right] = \theta\left[ \PP^{(0)}_\Inc f \right] = 0, \qquad u \left[ \PP^{(0)}_\Inc f \right] = \frac{E}{d} \mathbb{P} u[f],$$
			the one related to the Fourier (Boussineq) mode for $f=f(x, v)$ by
			$$u\left[ \PP^{(0)}_\Bou f \right] =\varrho\left[ \PP^{(0)}_\Bou f \right] + \theta\left[ \PP^{(0)}_\Bou f \right] = 0,$$
			$$\sqrt{K(K-1)} \big( (K-1) \varrho \left[ \PP^{(0)}_\Bou f \right]  - \theta \left[ \PP^{(0)}_\Bou f \right]  \big) =  (K-1) \varrho[f] - \theta[f] \,,$$
			and the ones related to the acoustic modes (recall that $( \Id - \mathbb{P}) \nabla_x = \nabla_x $) for $f=f(x, v)$  by
			$$(K-1) \varrho\left[ \PP^{(0)}_{\pm \Wave} f \right] - \theta\left[ \PP^{(0)}_{\pm \Wave} f \right] = 0,$$
			$$\sqrt{2K} u\left[ \PP^{(0)}_{\pm \Wave} f \right] = (- \Delta_x)^{-\frac{1}{2}} \nabla_x \left( \varrho[f] + \theta[f]\right)  \pm c \left( \Id - \mathbb{P} \right) u[f],$$
			$$\sqrt{2K} \left( 1 - \frac{1}{K} \right)\left( \varrho \left[ \PP^{(0)}_{\pm \Wave} f \right] + \theta \left[ \PP^{(0)}_{\pm \Wave} f \right] \right) = ( \varrho[f] + \theta[f]) \pm c (- \Delta_x)^{-\frac{1}{2}} \nabla_x \cdot u[f].$$
			The first order projectors related to the Navier-Stokes (incompressible) mode satisfy the identities for $f(x, \cdot) \perp \nul(\LL)$
			$$\varrho\left[ \PP^{(1)}_\Inc f \right] = \theta\left[ \PP^{(1)}_\Inc f \right] = 0,$$
			$$\left( \frac{E}{d} \right)^{\frac{3}{2}} u \left[ \nabla_x \cdot \PP^{(1)}_\Inc f \right] = \mathbb{P} \left( \nabla_x \cdot \la f, \LL^{-1}\BurA \ra_{\Ss} \right),$$
			and the first order coefficient related to the Fourier (Boussinesq) mode for $f(x, \cdot) \perp \nul(\LL)$
			$$u\left[ \PP^{(1)}_\Bou f \right] = \varrho\left[ \PP^{(1)}_\Bou f \right]  + \theta\left[ \PP^{(1)}_\Bou f \right] = 0,$$
			$$  (K-1) \varrho\left[ \PP^{(1)}_\Bou f \right]  - \theta\left[ \PP^{(1)}_\Bou f \right]    = \frac{1}{\sqrt{K(K-1)}}\la f, \LL^{-1}\BurB \ra_{\Ss}.$$
			
		\end{prop}
		
		We end this section by defining, in a similar way, the pseudo-kinetic part of the whole linearized semigroup
		\begin{defi}\label{def:kinetic_semigroups}
			We define the \textit{pseudo-kinetic projector} $\PP_\kin^\eps$ through its Fourier transform for any $g \in \SSg$
			\begin{equation*}
				\PP_\kin^\eps g := \FF_\xi^{-1} \Big[ \left(\Id - \PP(\eps \xi)\right) \widehat{g}(\xi) \Big] = \left(\Id - \PP^\eps_\hyd\right) g,
			\end{equation*}
			as well as the corresponding  semigroup $\left(U^\eps_\kin(t)\right)_{t\ge0}$
			\begin{align*}
				U^\eps_\kin(t)g := \FF^{-1}_\xi \Big[ U_{\xi}(\eps^{-2}t)\left( \Id - \PP(\eps \xi) \right) \widehat{g}(\xi) \Big]=\FF^{-1}_\xi \Big[ \left( \Id - \PP(\eps \xi) \right)U_{\xi}(\eps^{-2}t)\widehat{g}(\xi) \Big].
			\end{align*}
		\end{defi}
		\color{black}
		
		\subsection{Decomposition of the solution}
		With the above definitions, we obtain the following decomposition of the semigroup $U^\eps(t)$ as
		\begin{equation}\label{eq:decompUeps}
			U^\eps(t) = U^\eps_\hyd(t) + U^\eps_\kin(t) 
			=U^\eps_\ns(t)+U^\eps_\Wave(t) + U^\eps_\kin(t)\qquad \quad t\ge0\end{equation}
		and we split the nonlinear integral operator $\Psi^\eps$ accordingly, that is to say as a hydrodynamic part and a kinetic part:
		$$		\Psi^\eps[f,g](t) = \Psi^\eps_\hyd [f, g](t) + \Psi^\eps_\kin [f, g](t),$$
		with
		$$		\Psi^\eps_\star [f, g](t) := \PP^\eps_\star \Psi^\eps[f,g](t) = \frac{1}{\eps} \int_0^t U^\eps_\star(t - \tau) \QQ^\sym (f(\tau), g(\tau)) \d \tau.$$
		
		The main idea behind the proof of Theorems \ref{thm:hydrodynamic_limit} or \ref{thm:hydrodynamic_limit-gen_symmetric} or \ref{thm:hydrodynamic_limit-gen_degenerate} is to consider an \textit{a priori} decomposition of the unknown $f^\eps$ in $\SSSl + \SSSm + \SSSs$:
		\begin{equation}\label{eq:decompfeps}
			f^\eps = f^\eps_\kin + f^\eps_\mix + f^\eps_\hyd, \qquad  f^\eps_\kin \in \SSSl, ~ f^\eps_\mix \in \SSSm, ~ f^\eps_\hyd \in \SSSs\end{equation}
		which will enable us to reduce the construction of a solution of \eqref{eq:Kin-Mild} to that of a solution to some appropriate system of equations for all the new unknowns
		$$(f^\eps_\kin,f^\eps_\mix,f^\eps_\hyd) \in \SSSl\times \SSSm\times \SSSs.$$
		The term $f^\eps_\mix$ is a coupling term between the purely kinetic $f^\eps_\kin$ and macroscopic $f^\eps_\hyd$ parts and which need to be studied separately. 
		
		Let us dive more deeply in such a strategy, aiming to determine the system solved by $(f^\eps_\kin,f^\eps_\mix,f^\eps_\hyd)$. The splitting \eqref{eq:decompfeps}  induces the \textit{a priori} decomposition of the kinetic part of the non-linear term:
		$$\Psi^\eps_\kin[f^\eps, f^\eps] = \Psi^\eps_\kin \left[ f^\eps_\kin, f^\eps_\kin \right] + 2 \Psi^\eps_\kin \left[f^\eps_\kin, f^\eps_\hyd + f^\eps_\mix\right]  + A_\mix $$
		where we expect the first two terms to belong to $\SSSl$ and the third one
		$$A_\mix:=\Psi^\eps_\kin \left[f^\eps_\mix + f^\eps_\hyd , f^\eps_\mix + f^\eps_\hyd \right]  \in \SSSm.$$

		In the same way, we introduce the following \textit{a priori} decomposition of the hydrodynamic part of the non-linearity, which will only be used to make the following presentation more compact:
		$$\Psi^\eps_\hyd\left[f^\eps, f^\eps\right]=\Psi^\eps_\hyd \left[f^\eps_\hyd, f^\eps_\hyd\right]  + A_\hyd\left[f^\eps_\hyd\right]  + B_\hyd$$	
		where we denoted
		$$A_\hyd\left[f^\eps_\hyd\right]=2 \Psi^\eps_\hyd\left[f^\eps_\hyd , f^\eps_\mix + f^\eps_\kin\right], \qquad B_\hyd:=\Psi^\eps_\hyd \left[f^\eps_\mix + f^\eps_\kin , f^\eps_\mix + f^\eps_\kin \right]\,.$$
		We consider an \emph{arbitrary} system of equations for each part, where $A_\mix$ is assigned to the equation for $f^\eps_\mix$:
		\begin{equation*}\label{eq:systemE}
			\begin{cases}
				f^\eps_\kin(t) = U^\eps_\kin(t) f_\ini + \Psi^\eps_\kin \left[ f^\eps_\kin, f^\eps_\kin \right](t) + 2 \Psi^\eps_\kin \left[f^\eps_\kin, f^\eps_\hyd + f^\eps_\mix\right],\\
				\\
				f^\eps_\mix(t) = A_\mix(t),\\
				\\
				f^\eps_\hyd(t) = U^\eps_\hyd(t) f_\ini + \Psi^\eps_\hyd\left[f^\eps_\hyd, f^\eps_\hyd\right](t) + A_\hyd\left[f^\eps_\hyd\right](t) + B_\hyd(t).
			\end{cases}
		\end{equation*}
		
		We see that, under the \emph{ansatz} \eqref{eq:decompfeps}, solving \eqref{eq:Kin-Mild} amounts to solve the above system for
		$$(f^\eps_\kin,f^\eps_\mix,f^\eps_\hyd) \in \SSSl\times \SSSm\times \SSSs,$$
	 {as well as proving the uniqueness of solutions to the original equation \eqref{eq:Kin-Intro} since our system is arbitrary.}
		
		In the hydrodynamic limit, we moreover expect $f^\eps_\hyd$ to be the leading term of $f^\eps$ converging to $f_\ns$, the other two terms being expected to converge to zero. Notice that we look for the solution $f^\eps_\hyd \in \SSSs$ and, as observed already, this is  the functional space to which $f_\ns$  actually belongs. In other words, we  expect
		$$\lim_{\eps\to0}\Nt f_\hyd^\eps-f_\ns \Nt_{\SSSs}=0.$$
		Moreover, we need to prove that, in the above system, all terms are well-defined and belong to the desired spaces. This is one the main technical difficulties of the work and, as explained in the introduction, will follow from a careful study of the behaviour of the various semigroups $U^\eps_\star(\cdot)$ as well as their action on convolutions. See Section \ref{scn:study_physical_space} for full proofs.

		\subsection{Removing the acoustic initial layer and the hydrodynamic limit}

		To justify the convergence of $f^\eps_\hyd$ towards $f_\ns$, we actually will need to  rewrite the equation for $f^\eps_\hyd(t)$ by removing its leading order terms, namely a Navier-Stokes part and an acoustic part . The construction of those leading order terms rely on already existing theory for the Navier-Stokes equations and the wave equation.
		
		More precisely, we split the hydrodynamic part $f^\eps_\hyd$ into an oscillating one $f^\eps_\disp(t)$ (which is explicit) and another one $f^\eps_\ns(t)$ that will be shown to be an approximation of the hydrodynamic limit $f_\ns(t)$:
		\begin{equation}\label{eq:fepshyd}
			f^\eps_\hyd(t) = f^\eps_\disp(t) + f^\eps_\ns(t), \qquad
			f^\eps_\disp(t) := U^\eps_\disp(t) f_{\ini}\,.
		\end{equation}
		Inserting this into the equation solved by $f^\eps_\hyd$, we see that $f^\eps_\ns(t)$ satisfies
		\begin{multline*}
			f^\eps_\ns(t) =  \left(U^\eps_\hyd(t)f_\ini - U^\eps_\disp(t)f_\ini\right)    + \Psi^\eps_\hyd\left[ f^\eps_\ns, f^\eps_\ns \right](t)\\+ 2 \Psi^\eps_\hyd\left[f^\eps_\disp, f^\eps_\ns\right](t) + A_\hyd\left[f^\eps_\ns\right](t)\\
			+ A_\hyd\left[f^\eps_\disp\right](t) + B_\hyd(t) + \Psi^\eps_\hyd\left[ f_\disp^\eps, f_\disp^\eps \right](t).
		\end{multline*}
		We further split $f^\eps_\ns$ into its \emph{a priori} limit $f_\ns(t)$ and an error term $g^\eps(t)$:
		$$f^\eps_\ns(t) = f_\ns(t) + g^\eps(t),$$
		so that, using \eqref{eq:reduction_NS_kin}, the part $g^\eps(t)$ satisfies the equation
		\begin{align*}
			g^\eps(t)  &= \left(U^\eps_\hyd(t)f_{\ini} - U_\ns(t)f_{\ini} - U^\eps_\disp(t)f_\ini\right)   + \left(\Psi^\eps_\hyd\left[ f_\ns , f_\ns \right](t) - \Psi_\ns\left[ f_\ns , f_\ns \right](t)\right ) \\
			&\phantom{+++} + 2 \Psi^\eps_\hyd\left[ f_\ns, g^\eps \right](t)   + 2 \Psi^\eps_\hyd\left[f^\eps_\disp, g^\eps \right](t) + A_\hyd\left[g^\eps\right](t) + \Psi^\eps_\hyd\left[g^\eps, g^\eps\right](t) \\
			&\phantom{+++} + 2 \Psi^\eps_\hyd\left[ f^\eps_\disp, f_\ns \right](t) + \Psi^\eps_\hyd\left[ f^\eps_\disp, f^\eps_\disp\right](t)  \\
			&\phantom{+++} + A_\hyd\left[f_\ns\right](t) + A_\hyd\left[ f^\eps_\disp\right](t) + B_\hyd(t) \\
			&=2 \Psi^\eps_\hyd\left[ g^\eps, f_\ns + f^\eps_\disp + f^\eps_\mix + f^\eps_\kin\right](t)  + \Psi^\eps_\hyd\left[g^\eps, g^\eps\right](t) + \SS^\eps(t)
		\end{align*}
		Here, we have denoted  the vanishing non-linear source term (which depends on $f^\eps_\kin$ and $f^\eps_\mix$ but not on $g^\eps$), as
		\begin{subequations}\label{eq:source}
			\begin{equation}\label{eq:sourceSS}
				\SS^\eps(t) =\SS^\eps_1(t) + \SS^\eps_2(t) + \SS^\eps_3[f^\eps_\kin, f^\eps_\mix](t)
			\end{equation} where  {the first two parts depend only on $f^\eps_\disp$ through $f_\ini$ and $f_\ns$, which are considered given}
			%\begin{align*}
			%		\SS^\eps(t) = & \left(U^\eps_\hyd(t)f_{\ini} - U_\ns(t)f_{\ini} - U^\eps_\disp(t)f_{\ini} \right)   + \left(\Psi^\eps_\hyd\left[ f_\ns , f_\ns \right](t)	 - \Psi_\ns\left[ f_\ns , f_\ns \right](t)\right)  \\
			%		& + \Psi^\eps_\hyd\left[ f^\eps_\disp, 2 f_\ns + f^\eps_\disp \right](t) \\
			%		& + A_\hyd\left[f_\ns + f^\eps_\disp\right](t) + B_\hyd(t) \\
			%		= & \SS^\eps_1(t) + \SS^\eps_2(t) + \SS^\eps_3[f^\eps_\kin, f^\eps_\mix](t),
			%	\end{align*}
			%	where each part writes
			\begin{equation}\label{eq:sourceS1S2} \begin{split}
					\SS^\eps_1(t)&=\left(U^\eps_\hyd(t)f_{\ini} - U_\ns(t)f_{\ini} - U^\eps_\disp(t)f_{\ini} \right)   + \left(\Psi^\eps_\hyd\left[ f_\ns , f_\ns \right](t)	 - \Psi_\ns\left[ f_\ns , f_\ns \right](t)\right)  \\
					\SS^\eps_2(t) &= \Psi^\eps_\hyd\left[ f^\eps_\disp, 2 f_\ns + f^\eps_\disp \right](t)\,,\end{split}\end{equation}
			and {the third one also depends on the partial solutions $f^\eps_\kin$ and $f^\eps_\mix$}
			\begin{equation}\label{eq:sourceS3} \begin{split}	
					\SS^\eps_3[f^\eps_\kin, f^\eps_\mix](t) &= A_\hyd\left[f_\ns + f^\eps_\disp\right](t) + B_\hyd(t)\\
					&= \Psi^\eps_\hyd\left[ f^\eps_\kin + f^\eps_\mix , f^\eps_\kin + f^\eps_\mix + f_\ns + f^\eps_\disp \right](t).
			\end{split}\end{equation}
		\end{subequations}
		
		\subsection{Summary of the proof}\label{sec:detail-sum} The above technical splitting allows us to consider the solution to \eqref{eq:Kin-Mild} we aim to construct in the form
		\begin{equation*}\begin{split}
				f^\eps(t)&=f^\eps_\kin(t)+f^\eps_\mix(t)+f^\eps_\hyd(t)\\
				&=f^\eps_\kin(t)+f^\eps_\mix(t)+f^\eps_\disp(t)+ f_\ns(t)+g^\eps(t)
			\end{split}
		\end{equation*}  
		where $f_\ns(\cdot)$ as well as $f_\ini$ (and thus $f^\eps_\disp(t) = U^\eps_\disp(t) f_\ini$) are functions to be considered as fixed parameters  since they depend only on the initial datum $f_\ini$ (and $\eps$). 
		
		According to the analysis performed in Section \ref{scn:study_physical_space} (see Lemma \ref{lem:decay_regularization_kinetic_semigroup}, Lemma \ref{lem:decay_semigroups_hydro} and Lemma \ref{lem:NS_parabolic_space} respectively) that
		$$\Nt U^\eps_\kin(\cdot) f_\ini \Nt_{\SSSl} \lesssim 1, \qquad \Nt f^\eps_\disp \Nt_{\SSSs} \lesssim 1, \qquad \Nt f_\ns \Nt_{\SSSs} \lesssim 1.$$
		Since $f_\ns$ is entirely determined by the solutions $(\varrho,u,\theta)$ to the NSF system \eqref{eq:main-NSF}, we point out that, in the definition of the space $\SSSs$, we choose the function $\phi$ to be \emph{exactly} the solution $f_\ns$. This corresponds, in Eq. \eqref{eq:wphieta}, to the choice of the weight function 
		$$w_{f_\ns,\eta}(t)=\exp\left( \frac{1}{2 \eta^2} \int_0^t\| |\nabla_x|^{1-\alpha} f_\ns(\tau) \|_{\rSSsp{s}}^2 \d \tau \right) \qquad t \ge0,$$
		with $\eta >0$ is a parameter which is still to free to be chosen {suitably small for the upcoming fixed point argument to work}. The above system considered in Section \ref{scn:reduction_problem} writes now:
		\begin{equation}\label{eq:systemKinMixG}
			\begin{cases}
				f^\eps_\kin(t) &= U^\eps_\kin(t) f_\ini + \Psi^\eps_\kin \left[ f^\eps_\kin, f^\eps_\kin \right](t) + 2 \Psi^\eps_\kin\left[ f^\eps_\kin, f^\eps_\hyd + f^\eps_\mix\right](t),\\
				\\
				f^\eps_\mix(t)    &=\Psi^\eps_\kin\left[\left(f_\ns + f^\eps_\disp\right)+ f_\mix^\eps+g^\eps\, , \,\left(f_\ns + f^\eps_\disp\right)+ f_\mix^\eps+g^\eps \right](t),\\
				\\	
				g^\eps(t) &= \Phi^\eps[ f^\eps_\kin, f^\eps_\mix ](t) g^\eps(t)+
				\Psi^\eps_\hyd\left[ g^\eps, g^\eps\right] + \SS^\eps(t).
			\end{cases}
		\end{equation}
		We will construct a solution $\left( f^\eps_\kin, f^\eps_\mix, g^\eps \right)$ of this system in the space
		$ \SSSl \times \SSSm \times \SSSs $	and more specifically, in a product of suitable balls in such spaces. This is achieves through a suitable use of Banach fixed point Theorem in the case of Assumptions \ref{BE} whereas, under Assumptions \ref{BED}, the situation is much more involved and we adapt a Picard-like scheme to construct our solution $(f^\eps_\kin,f^\eps_\mix,g^\eps)$. \\
		\medskip 
		
		As said already, in order to study the system \eqref{eq:systemKinMixG}, we first need to prove that all the various terms make sense under the \emph{ansatz} \eqref{eq:decompfeps}, that is we need to show that the various bilinear terms $\Psi^\eps_\kin \left[ f^\eps_\kin, f^\eps_\kin \right]$, $\Psi^\eps_\kin\left[ f^\eps_\kin, f^\eps_\hyd + f^\eps_\mix\right]$ are defined and belong to $\SSSl$ if $\left( f^\eps_\kin, f^\eps_\mix, g^\eps \right) \in \SSSl \times \SSSm \times \SSSs$, that the bilinear term appearing in the equation for $f_\mix^\eps$ is well defined and belong to $\SSSm$ while the bilinear terms involved in the equation for $g^\eps$ are well-defined and belong to $\SSSs.$ This is done in Section \ref{sec:Bilin} which is based on the thorough analysis led in Section \ref{scn:study_physical_space} of the various semigroups $U^\eps_\kin(\cdot)$, $U^\eps_\hyd(\cdot)$ and convolutions of the type 
		$$\frac{1}{\eps} U^\eps_\kin(\cdot) \ast \varphi \quad \text{ and } \quad \frac{1}{\eps}U^\eps_\hyd(\cdot) \ast \varphi.$$
		
		\section{Spectral analysis of the linearized operator}
		\label{scn:spectral_study}
		
		This section is mainly devoted to the proof of the main spectral result Theorem \ref{thm:spectral_study} in the Introduction about the linearized operator $\LL_{\xi} = \LL - i (v \cdot \xi).$
		
		\subsection{Description of the novel spectral approach} \label{sec:newspec}
		We give here a precise description of the main steps of our approach to prove the above two main results.
		
		{
		In order to prove our main spectral result, we adopt a ``direct method'' which appears much simpler than the original method of \cite{EP1975}. More precisely, their approach relies on the Lyapunov-Schmidt reduction process, which consists roughly in projecting the eigenvalue problem on the unperturbed eigenspace, whereas ours relies on Kato's reduction process, which consists roughly in rectifying the perturbed operator as another one defined on the unperturbed space. To some extent, we believe our approach to be somehow more natural and direct, fully exploiting the symmetry properties of the collision operators $\QQ$ and $\LL$. It is for sure of a more ``functional analytic flavour'' than the one of \cite{EP1975}.}
		
		First, to study the spectrum of $\LL_{\xi}$, we use the fact that, for $\xi=0,$ the spectrum of $\LL_{0}=\LL$ is explicit thanks to Assumption \ref{L3} and show that, for $|\xi|$ small enough, the structure of $\mathfrak{S}(\LL_{\xi})$ is similar to that of $\mathfrak{S}(\LL),$ i.e. there exists some explicit value $\alpha_{0}$ and $\gamma >0$ such that
		$$\mathfrak{S}(\LL_{\xi}) \cap \left\{z \in \mathbb{C}_{+}\;;\;\re z >-\gamma\right\}, \qquad \forall |\xi | \le \alpha_0$$
		consists in a finite number of eigenvalues. Such a localization of the spectrum is obtained here \emph{without resorting to any compactness argument}. This is the main contrast with respect to the original work \cite{EP1975} whose approach disseminates in the literature.
		
		The localization of the spectrum is not enough for the purpose of the paper and we also need to compute explicitly {the asymptotic spectrum $\mathfrak{S}(\LL_{\xi}) \cap \left\{z \in \mathbb{C}_{+}\;;\;\re z >-\gamma\right\}$ and the associated spectral projector}. This is done in a quantitative way, using Kato's perturbation theory as developed in \cite{K1995}.
		Typically, one observes that, since for any $\xi \in \R^{d}$, $\LL_{\xi}$ is a perturbation of $\LL$ by the multiplication operator with $-i(v\cdot \xi)$, for any $z \notin \mathfrak{S}(\LL_{\xi}) \cup \mathfrak{S}(\LL)$, the following expansion formulae are valid for $N \ge 0$:
		\begin{equation}
			\label{eq:factorization_L_ivxi_left}
			\RR(z,\LL_{\xi}) = \sum_{n = 0}^{N-1} \RR(z,\LL) \Big((- i v \cdot \xi) \RR(z,\LL)\Big)^n + \RR(z,\LL_{\xi}) \Big((- i v \cdot \xi) \RR(z,\LL)\Big)^N,\end{equation}
		as well as
		\begin{equation}		\label{eq:factorization_L_ivxi_right}
			\RR(z,\LL_{\xi}) = \sum_{n = 0}^{N-1} \RR(z,\LL) \Big((- i v \cdot \xi) \RR(z,\LL)\Big)^n + \RR(z,\LL_{\xi}) \Big((- i v \cdot \xi) \RR(z,\LL)\Big)^N.
		\end{equation}
		Various choice of the parameter $N\geq1$ would allow us to recover estimates on $\RR(z,\LL_{\xi})$ from known result on $\RR(z,\LL)$ and provide the asymptotic expansion of the eigenvalue and eigen-projectors. In a more specific way, the proof of Theorem \ref{thm:spectral_study} is done according to the following roadmap:
		\begin{itemize}
		\item In Lemma \ref{lem:localization_spectrum}, we show that the spectrum of $\LL_\xi := \LL - i v \cdot \xi$ contained in some right half plane is confined in a ball of radius of order $\xi$ centered around the origin, and establish some bounds on the resolvent. 
		\item In Lemma \ref{lem:expansion_projection}, we prove that the spectral projector associated with this part of the spectrum has a first order expansion as $\xi \to 0$.  {To study the aforementionned part of the spectrum, we then introduce $\rectL_{ \xi}$ which is a matrix conjugated to the restriction of $\LL_\xi$ to the corresponding stable subspace (sum of eigenspaces)}, thus allowing to rely on perturbation theory in finite dimension. 
		\item We establish in Lemma \ref{lem:rectified_operator} some invariance \textit{(isotropy)} properties satisfied by $\rectL_{ \xi}$ and give its first order expansion. 		
		\item A block matrix representation of $\rectL_{ \xi}$ is presented in Lemma \ref{lem:rectified_block_matrix},  thus identifying its only multiple eigenvalue, and isolating it from the three simple remaining ones as is shown in Lemma \ref{lem:expansion_rectified_matrix}.
		\item From that point on, we use finite dimensional perturbation theory and show that $\rectL_{ \xi}$ is diagonalizable and establish a second order expansion of its eigenvalues as well as a first order expansion of its spectral projectors in Lemma  \ref{lem:diagonalization_rectified}, from which we deduce the spectral decomposition of the original operator $\LL_\xi$, as well as expansions of the projectors in Lemma \ref{lem:expansion_projectors}.

		\end{itemize}

		Finally, we combine the resolvent bounds for $| \xi | \ll 1$ from the previous lemmas, and use a hypocoercivity theorem from \cite{D2011} for $ | \xi | \gtrsim$ to obtain an uniform exponential decay in $\Ss$ of the semigroup generated by $\LL_\xi$ on the stable subspace associated with the rest of the spectrum. We then improve this uniform decay estimate in $\Ss$ as an {integral regularization and decay in $\Ss-\Ssp$ and $\Ssm-\Ss$} by combining it with an energy method.
		
		\subsection{The spatially homogeneous setting}
		
		Before  undertaking the program described here above, it is important to recall the spectral picture in the spatially homogeneous setting corresponding to $\xi=0$. Assumptions \ref{L1}--\ref{L4} directly give the localization of the spectrum and the fact that $0$ is a semi-simple eigenvalue of $\LL$ with $d+2$-dimensional (geometric) multiplicity. Associated to such an eigenvalue, the spectral projection
		$$\PP:=\frac{1}{2i\pi}\oint_{|z|=r}\RR(z,\LL)\d z$$
		has the following properties: 
		%One has then the following 
		Let us present the properties of the orthogonal projection on the null space of $\LL$.
		
		\begin{prop}[\textit{\textbf{Representation formulae involving $\PP$}}]
			\label{prop:representation_P}
			We  recall   the macroscopic (fluctuations of) mass~$\varrho \in \R$, velocity $u \in \R^d$ and temperature~$\theta \in \R$ as defined in \eqref{eq:fluctuat}. Under Assumptions \ref{L1}-\ref{L2}, the spectral projector $\PP$ on the null-space $\nul(\LL)$ is $\Ss$-orthogonal and has the following explicit representation: \begin{equation}\label{eq:representation_P}
				\PP f(v) = \left(\varrho_f + u_f \cdot v   + \frac{\theta_f}{E(K-1)} \left( |v|^2 - E \right)\right) \mu(v),
			\end{equation}
			as well as, denoting $\Pi_\omega=\Id-\omega\otimes\omega$ the orthogonal projection on $\omega^\perp= \left\{ u \in \R^d ~| ~ u \perp \omega \right\}$ for any $\omega \in \S^{d-1}$:
			\begin{align*}
				\PP f(v) = \left( \Pi_\omega u_f \right) \cdot v \mu(v) & + \frac{1}{K(K-1)} \Big( (K-1) \varrho_f - \theta_f \Big) \left( K - \frac{|v|^2}{E} \right) \mu(v) \\
				& + \frac{1}{d c^2} (\varrho_f + \theta_f) |v|^2 \mu(v) + \big( \left( \Id - \Pi_\omega \right) u_f \big) \cdot v \mu(v)
			\end{align*}
			where we introduced the speed of sound $c$ in \eqref{eq:speed_sound}. More compactly,  {in terms of the eigenfunctions $\psi_{\pm \Wave}$ and $\psi_\Bou$ defined in \eqref{eq:def_psi_wave}--\eqref{eq:def_psi_Bou}}
			\begin{equation}
				\label{eq:representation_hydrodynamic_modes}
				\begin{aligned}
					\PP f(v)= \frac{d}{E} \Pi_\omega  \langle f, v\mu \rangle_\Ss v \mu &+ \langle f, \psi_{- \Wave}(\omega) \rangle_\Ss \psi_{- \Wave}(\omega) \\
					& + \langle f, \psi_{+ \Wave}(\omega) \rangle_\Ss \psi_{+ \Wave}(\omega) + \langle f, \psi_{\Bou} \rangle_\Ss \psi_{\Bou}.
				\end{aligned}
			\end{equation}
			We finally notice that the Burnett functions introduced in \eqref{eq:burnett} are  related to $v \mu \in \nul(\LL)$ and $\psi_\Bou \in \nul(\LL)$ through
			$$\BurA (v)=(\Id - \PP) [ v \otimes v \mu ] \quad \text{ and } \quad \BurB (v)=( \Id - \PP ) [ v \psi_\Bou ].$$
		\end{prop}
		\begin{rem}
			Note that, if a given function $\mu$ satisfies \ref{L2} , and if one defines $\PP$ as the $\Ss$-orthogonal projection on $\Span \{\mu, v_1, \dots, v_d \mu, |v|^2 \mu\}$, i.e. as \eqref{eq:representation_P}, then Proposition \ref{prop:representation_P} still holds.
		\end{rem}
		
		\subsection{Proof of Theorem \ref{thm:spectral_study}}
		
		We are now ready to attack the full proof of Theorem \ref{thm:spectral_study}. We begin with the localization of the spectrum.

		\begin{lem}[\textit{\textbf{Localization of the spectrum}}]
			\label{lem:localization_spectrum}
			For any gap size 
			$$0 < \lambda < \lambda_\LL,$$
			there exists some $C_0=C_0(\lambda) > 0$ and $\alpha_{0} = \alpha_{0}(\lambda) > 0$ that can be assumed small, such that the spectrum is localized as follows:
			$$\mathfrak{S}(\LL_{\xi}) \cap \Delta_{-\lambda} \subset \big\{ | z | \le C_0 | \xi | \big\}, \qquad \forall | \xi | \le \alpha_{0}\,.$$
			Moreover, there exist $C_1=C_1(\alpha_{0}, \lambda) > 0$ such that, for any $|\xi| \leq \alpha_{0}$
			\begin{equation}
				\label{eq:bound_L_xi}
				\sup_{ | z | = r } \| \RR(z,\LL_{\xi}) \|_{\BBB(\Ss;\Ssp) }  + \sup_{|z| = r} \| \RR(z, \LL_\xi ) \|_{ \BBB( \Ssm ; \Ss ) }  + \sup_{z\in\Omega} \| \RR(z,\LL_{\xi}) \|_{\BBB(\Ss)} \le C_1
			\end{equation}
			where $r=C_{0}\alpha_{0} >0$ and $\Omega := \Delta_{-\lambda} \cap \{ | z | \ge r \}$. 
		\end{lem}
		
		\begin{proof}
			In all the proof, we assume $0 < \lambda < \lambda_\LL$ to be \emph{fixed}.

			\step{1}{Resolvent properties in the spaces $\Ss_j$} We first observe that, $\BB_{\xi} - \lambda$ being dissipative in the spaces $\Ss_j$ by hypothesis \ref{assumption_dissipative}, it holds for any $j = 0, 1, 2$
			\begin{equation}
				\label{eq:ResBB_xi}
				\|\RR(z,\BB_{\xi})\|_{\BBB(\Ss_{j})} \lesssim 1, \qquad \forall z \in \Delta_{-\lambda},
			\end{equation}
			uniformly in $\xi \in \R^{d}$.
			In particular, the above is true for $\xi=0$, that is to say for $\RR(z,\BB)$. Using that $\mathfrak{S}_{\Ss}\left(\LL\right) \cap \Delta_{-\lambda} = \{0\}$ from \ref{L3}, we have the factorization formula for any $N \ge 0$:
			\begin{gather}
				\label{eq:shrinkage_A_B}
				\RR(z,\LL) = \sum_{n = 0}^{N-1} \RR(z,\BB) \Big( \AA \RR(z,\BB) \Big)^n + \Big(\RR(z,\BB) \AA \Big)^N \RR(z,\LL).
			\end{gather}
			holds for any $z \in \Delta_{-\lambda} \setminus\{0\}.$
			Furthermore, since $\LL$ is self-adjoint in $\Ss$ by hypothesis \ref{L1}, it is well-known that the zero eigenvalue is semi-simple so that
			\begin{equation}
				\label{eq:RRLLH}\|\RR(z,\LL)\|_{\BBB(\Ss)} \lesssim \frac{1}{|z|} \qquad \qquad
				\forall z \in \Delta_{-\lambda}\setminus\{0\}.
			\end{equation}	
			Thus, using the factorization \eqref{eq:shrinkage_A_B} with $N=1$, we deduce from a repeated use of \eqref{eq:ResBB_xi} 
			that, for any $f \in \Ss_{1}$ and any $z \in \Delta_{-\lambda}\setminus\{0\}$, 
			\begin{align*}
				\|\RR(z,\LL)f\|_{\Ss_{1}} \leq \|\RR(z,\BB)f\|_{\Ss_{1}} + \|\RR(z,\BB)\AA\RR(z,\LL)f\|_{\Ss_{1}} 
				\lesssim \|f\|_{\Ss_{1}} +\|\AA\RR(z,\LL)f\|_{\Ss_{1}},
			\end{align*}
			and using the boundedness of $\AA : \Ss_j \to \Ss_{j+1}$ from \ref{assumption_bounded_A} as well as \eqref{eq:RRLLH}
			\begin{align*}
				\|\RR(z,\LL)f\|_{\Ss_{1}}  & \lesssim \|f\|_{\Ss_{1}} +  \|\AA\|_{\BBB(\Ss_{1},\Ss)}\|\RR(z,\LL)f\|_{\Ss}\\
				&\lesssim \|f\|_{\Ss_{1}}+ |z|^{-1}\|f\|_{\Ss} \lesssim \left(1+\frac{1}{|z|}\right)\|f\|_{\Ss_{1}}
			\end{align*}
			where we used $\Ss_{1}  \hookrightarrow \Ss$ in the last inequality. Using this estimate and proceeding in the same way for $j=2$, we deduce that
			\begin{equation}
				\label{eq:bound_resolvent_L_X_j}
				\| \RR(z,\LL) \|_{\BBB(\Ss_{j})} \lesssim 1 + \frac{1}{|z|}, \qquad \forall z \in \Delta_{-\lambda}  \setminus \{0\}, \quad  j=0,1,2.
			\end{equation}
			We conclude that, in each space $\Ss_j$, the eigenvalue $0$ is semi-simple (i.e. a simple pole of $\RR(\cdot,\LL)$) and the resolvent writes in $\BBB(\Ss_j)$ as the sum of a singular part and a holomorphic part (see \cite[Chapter 3, Section 6.5]{K1995}):
			\begin{equation}
				\label{eq:pseudo_Laurent_expansion}  \RR(z,\LL) = z^{-1}\PP + \RR^\perp(z), 
				\qquad	\forall z \in \Delta_{-\lambda} \setminus \{0\}\,.
			\end{equation}
			with the regular part defined as
			\begin{equation}
				\label{eq:pseudo_Laurent_expansion1}
				\RR^{\perp}(z)=\sum_{n=1}^{\infty}z^{n}\mathsf{R_{0}}^{n+1}, \qquad \mathsf{R}_{0}:=\RR(0,\LL)\left(\mathrm{Id}-\PP\right)\,,
			\end{equation}
			where we notice that $\mathsf{R}_{0} =-\LL^{-1} \left(\Id-\PP\right) \in \BBB(\Ss_j)$.
			
			\step{2}{Localization of the spectrum and resolvent bound in $\BBB(\Ss)$}
			We draw inspiration from the proof of \cite[Lemma 2.16]{T2016}. Let us start with the following factorization formula permitted by the dissipativity hypothesis from \ref{assumption_dissipative} for some large enough $a > 0$:
			$$\RR(z,\LL_{\xi}) = \RR(z,\BB_{\xi}) + \RR(z,\LL_{\xi}) \AA \RR(z,\BB_{\xi}) \qquad \forall z \in \Delta_{a}$$
			and, expanding the term $\RR(z,\LL_{\xi})$ using \eqref{eq:factorization_L_ivxi_left} with $N = 1$, we deduce now
			\begin{align}
				\label{eq:factorization_R_xi_localization}
				\RR(z,\LL_{\xi}) = & \RR(z,\BB_{\xi}) + \RR(z,\LL) \AA \RR(z,\BB_{\xi}) + \RR(z,\LL_{\xi}) (- i v \cdot \xi) \RR(z,\LL) \AA \RR(z,\BB_{\xi}).
			\end{align}
			This allows to localize the spectrum using the bounds \eqref{eq:bound_resolvent_L_X_j} and the regularization hypothesis for $\AA$ coming from \ref{assumption_bounded_A}. Indeed, according to \ref{assumption_multi-v}
			$$\left\|v \RR(z,\LL) \AA \RR(z,\BB_{\xi}) \right\|_{\BBB(\Ss)} \lesssim \left\|\RR(z,\LL)\AA\RR(z,\BB_{\xi})\right\|_{\BBB(\Ss,\Ss_{1})}$$
			and, using now \eqref{eq:bound_resolvent_L_X_j} and the fact that $\AA \in \BBB(\Ss,\Ss_{1})$ from \ref{assumption_bounded_A}, we deduce that
			$$\left\|v \RR(z,\LL) \AA \RR(z,\BB_{\xi}) \right\|_{\BBB(\Ss)} \lesssim \left(1+\frac{1}{|z|}\right)\|\RR(z,\BB_{\xi})\|_{\BBB(\Ss)} \lesssim 1+\frac{1}{|z|},$$
			thanks to \eqref{eq:ResBB_xi}. In particular, for any $c_0 >0,$
			$$ \left\|\left(\xi \cdot v\right)\RR(z,\LL) \AA \RR(z,\BB_{\xi}) \right\|_{\BBB(\Ss)} \lesssim  |\xi|+c_0, \qquad \forall | z | > \frac{| \xi |}{c_0}\,,$$
			thus, considering $c_0, \alpha_{0} > 0$ small enough, we deduce  that
			$$\| \xi \cdot v \RR(z,\LL) \AA \RR(z,\BB_{\xi}) \|_{\BBB(\Ss)} \le \frac{1}{2}, \qquad \forall |z| > \frac{|\xi|}{c_0}, \quad | \xi | \le \alpha_0\, ,$$
			and in particular, for such a choice of $(\xi, z)$, the operator $\Id + (i v \cdot \xi) \RR(z,\LL) \AA\RR(z,\BB_{\xi})$ is invertible in $\BBB(\Ss)$ with
			\begin{equation}
				\label{eq:inverseId} \left\|\Big( \Id + (i v \cdot \xi) \RR(z,\LL) \AA\RR(z,\BB_{\xi}) \Big)^{-1}\right\|_{\BBB(\Ss)} \leq 2,\qquad \text{ for any } |z| > \frac{|\xi|}{c_0}, ~ |\xi | \le \alpha_0\, , 
			\end{equation}
			thus it follows from \eqref{eq:factorization_R_xi_localization} that
			\begin{align*}
				\RR(z,\LL_{\xi}) =\Big( \RR(z,\BB_{\xi}) + \RR(z,\LL) \AA \RR(z,\BB_{\xi}) \Big)\Big( \Id + (i v \cdot \xi) \RR(z,\LL) \AA\RR(z,\BB_{\xi}) \Big)^{-1} .
			\end{align*}
			Each term on the right hand side belongs to $\BBB(\Ss)$ for $z \in \Delta_{-\lambda} \cap \Big\{ |z| > c_0^{-1} | \xi | \Big\}$, thus the following localization of the spectrum holds:
			\begin{equation*}  \mathfrak{S}_{\Ss}(\LL_{\xi}) \cap \Delta_{-\lambda} \subset \Big\{z \in \C \; ; \; | z | \le c_0^{-1} | \xi | \Big\}, \qquad \forall |\xi| \le \alpha_{0}.
			\end{equation*}
			More precisely, using the bounds on $\RR(z,\BB_{\xi})$ from \ref{L4} and \eqref{eq:bound_resolvent_L_X_j} together with \eqref{eq:inverseId} we have
			\begin{equation}
				\label{eq:general_bound_L_xi} \| \RR(z,\LL_{\xi}) \|_{\BBB(\Ss)} \lesssim 1 + \frac{1}{|z|}
			\end{equation}
			for any $z \in \Delta_{-\lambda}, ~ | z | > \frac{|\xi|}{c_0},$ with $|\xi|\le \alpha_{0}$.  
			This proves the $\BBB(\Ss)$-bound in \eqref{eq:bound_L_xi}. This concludes this step.
			
			\step{3}{Resolvent bound in $\BBB(\Ss; \Ssp) \cap \BBB(\Ssm ; \Ss)$}
			First of all, note that the following identity for bounded operator $T\::\:\Ss \to \Ssp$ is proved in Appendix \ref{scn:duality} (where the adjoint $T^{\star}$ is considered for the inner product of $\Ss$):
			\begin{equation*}
				\| T^\star \|_{\BBB(\Ssm;\Ss)} = \| T \|_{ \BBB(\Ss;\Ssp) },
			\end{equation*}
			furthermore, using that $\LL_\xi^\star = \LL_{-\xi}$ since $\LL$ is self adjoint, one has
			$$ \RR(z,\LL_{\xi})^{\star}=\RR(\overline{z},\LL_{-\xi})\qquad \forall \xi \in \R^d, ~  z \in \mathfrak{S}(\LL_\xi),$$
			thus it is enough to prove the bound in $\BBB(\Ss; \Ssp)$.
			
			\medskip
			Using the dissipativity estimate for $\LL$ from \ref{L3} and the fact that the multiplication by~$i v \cdot \xi$ is skew-adjoint, we have for any $z > 0$
			\begin{align*}
				\lla (\LL_\xi - z) f, f \rra_{\Ss} & \le - \lambda_\LL \| (\Id - \PP) f \|_{\Ssp}^2 - z \| f \|^2_{\Ss}.
			\end{align*}
			Furthermore, using that $\PP$ is $\Ss$-orthogonal as well as the fact that $\PP^2 = \PP$ and $\| \PP \|_{\BBB(\Ss;\Ssp)} \le M$ for some $M > 0$
			\begin{align*}
				\lla (\LL_\xi - z) f, f \rra_{\Ss}  & \le - \lambda_\LL \| (\Id - \PP) f \|_{\Ssp}^2 - z \| \PP f \|^2_{\Ss} \\
				&  \le - \lambda_\LL \| (\Id - \PP) f \|_{\Ssp}^2 - z  {M^{-2}} \| \PP f \|^2_{\Ssp}.
			\end{align*}
			The term $\| (\Id - \PP) f \|_{\Ssp}^2$ can be estimated using the polar identity and Young's inequality (note that $\PP$ may not be $\Ssp$-orthogonal):
			\begin{align}
				\notag
				\| (\Id - \PP) f \|_{\Ssp}^2 & = \frac{1}{2} \| (\Id - \PP) f \|_{\Ssp}^2 + \frac{1}{2} \Big(\| f \|^2_{\Ssp} - \| \PP f \|_{\Ssp}^2 - 2 \lla (\Id - \PP) f, \PP f \rra_{\Ssp}\Big) \\
				\label{eq:degenerate_pythagora_X_star}
				& \ge \frac{1}{2} \| f \|^2_{\Ssp} - \| \PP f \|_{\Ssp}^2,
			\end{align}
			therefore we have
			\begin{align*}
				\lla (\LL_\xi - z) f, f \rra_{\Ss} \le - \frac{\lambda_\LL}{2} \| f \|_{\Ssp}^2 - \left(z  {M^{-2}} - \lambda_\LL \right) \| \PP f \|^2_{\Ssp}.
			\end{align*}
			We conclude that for $z_0 = {\lambda_\LL \,M^{2}}$, we have
			$$\forall f \in \dom(\LL_{\xi}), \quad \lla (\LL_\xi - z_0) f, f \rra_{\Ss} \le - \frac{\lambda_\LL}{2} \| f \|_{\Ssp}^2.$$ This, together with the comparison $\|\cdot\|_{\Ss} \leq \|\cdot\|_{\Ssp}$, implies  the resolvent bound
			\begin{equation}
				\label{eq:bound_L_xi_X_X_star_z_0}
				\| \RR(z_{0},\LL_{\xi})\|_{\BBB(\Ss;\Ssp)} \le \frac{2}{\lambda_\LL}.
			\end{equation}
			Using the resolvent identity
			\begin{equation*}
				\RR(z,\LL_{\xi}) = \RR(z_{0},\LL_{\xi})+ (z_0 - z) \RR(z_{0},\LL_{\xi})\RR(z,\LL_{\xi}),
			\end{equation*}
			we can combine \eqref{eq:general_bound_L_xi} and \eqref{eq:bound_L_xi_X_X_star_z_0} so as to obtain, for $|\xi| \le \alpha_{0}$,
			\begin{equation}
				\| \RR(z,\LL_{\xi}) \|_{\BBB(\Ss;\Ssp)}
				\lesssim 1 + |z| + \frac{1}{|z|}, \qquad \forall z \in \Delta_{-\lambda}, ~ |z| > c_0^{-1}|\xi|,
			\end{equation}
			from which we deduce the $\BBB(\Ss;\Ssp)$-bound of \eqref{eq:bound_L_xi}. This concludes the proof.
		\end{proof}
		
		\begin{figure}[h]
			
			\centering
			
				\begin{tikzpicture}[scale=1.8]
					% Spectre cinétique
					\fill [pattern=north west lines, pattern color=red] (-3, -1.5) rectangle (-1.5, 1.5);
					\draw [color=red] (-1.5, -1.5) -- (-1.5, 1.5);
					\node[red] at (-2.3, -1.8) {$\Re z \le - \lambda $};

					% Seuils cinétique
					\draw[ultra thick, dash pattern=on 5pt, color=red] (-2.4, -1.5) -- (-2.4, 1.5);
					\node[red] at (-2.7, 1.7) {$\Re z = - \lambda_\BB$};
					\draw[ultra thick, dash pattern=on 5pt, color=red] (-1.7, -1.5) -- (-1.7, 1.5);
					\node[red] at (-1.4, 1.7) {$\Re z = - \lambda_\LL$};
					
					% Spectre hydrodynamique
					\fill [pattern=north west lines, pattern color=blue] (0, 0) circle (1);
					\draw [color=blue] (0, 0) circle (1);
					\node[blue] at (1.5, 1.3) {$|z| \le C_0 | \xi |$};
					
					% Encerclement
					\draw[dash pattern=on 5pt] (0, 0) circle (1.2);
					\node at (1.3, -1.3) {$|z| = C_0 \alpha_0$};
					
					% Repère
					\draw[-stealth] (-3, 0) -- (2, 0);
					\draw[-stealth] (0, -1.5) -- (0, 1.5);
				\end{tikzpicture}
			
			\caption{Localization of the spectrum from Lemma \ref{lem:localization_spectrum}. The hatched blue part contains the ‘‘pseudo hydrodynamic part'' of the spectrum of $\LL_\xi$ (i.e. the perturbation of the macroscopic eigenvalue of $\LL$, that is to say $0$). The hatched red part contains the ‘‘pseudo kinetic part'' of the spectrum of $\LL_\xi$ (i.e. the perturbation of the microscopic part of the spectrum of $\LL$,   that is to say $\mathfrak{S}(\LL) \setminus \{0\}$).}
		\end{figure}
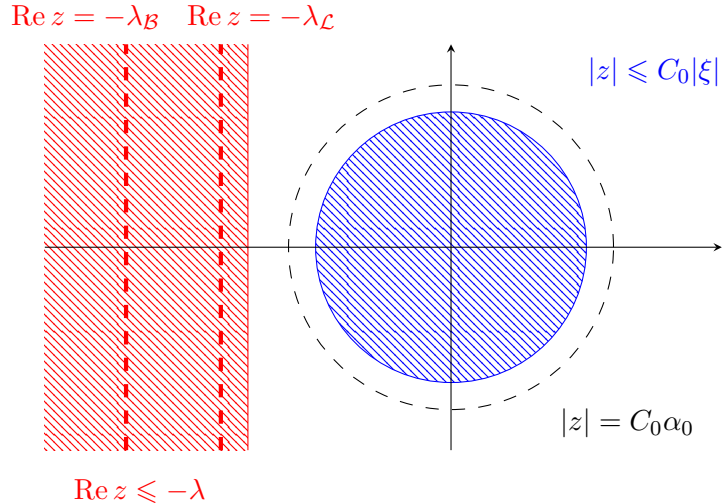

		\begin{lem}[\textit{\textbf{Expansion of the total projector}}]
			\label{lem:expansion_projection}
			With the notations of Lemma \ref{lem:localization_spectrum}, considering some $0 < \lambda < \lambda_\LL$, for any $|\xi| \le \alpha_{0}$, the spectral projector $\PP(\xi)$ associated with the $0$-group of Lemma \ref{lem:localization_spectrum} and defined as
			$$\PP(\xi) = \frac{1}{2 i \pi} \oint_{|z| = r } \RR(z,\LL_{\xi}) \d z, \qquad \forall | \xi | \le \alpha_0,$$
			where the integration along the circle $\{ |z| = r \}$ is counterclockwise, has the following  first order expansion in $\BBB(\Ssm ; \Ssp )$:
			\begin{equation*}
				\PP(\xi) = \PP + i \xi \cdot \Big( \PP v\mathsf{R}_{0} +\mathsf{R}_{0} v \PP \Big) + \mathsf{S}(\xi), \qquad \mathsf{R}_{0}:=\RR(0,\LL)\left(\Id-\PP\right)
			\end{equation*}
			where $\mathsf{S}(\xi) \in \BBB(\Ssm;\Ssp)$ with $\|\mathsf{S}(\xi)\|_{\BBB(\Ssm;\Ssp)} \lesssim  | \xi |^2$.
		\end{lem}
		
		\begin{proof} For a fixed $0 < \lambda <  \lambda_{\LL}$, let $\alpha_{0}=\alpha_{0}(\lambda)$ be provided by Lemma \ref{lem:localization_spectrum}. Since
			$$\PP(\xi)=\frac{1}{2i\pi}\oint_{|z|=r}\RR(z,\LL_{\xi}) \d z,$$
			we deduce directly from the bound \eqref{eq:bound_L_xi} that
			$$\| \PP(\xi) \|_{\BBB(\Ss;\Ssp)} + \| \PP(\xi) \|_{\BBB(\Ssm;\Ss)} \lesssim 1,$$
			and using that $\PP(\xi)^2 = \PP(\xi)$, we actually deduce
			\begin{equation}
				\label{eq:reg_P_xi_order_0}
				\| \PP(\xi) \|_{\BBB(\Ssm;\Ssp)} \lesssim 1, \quad \PP(0) = \PP \in \BBB(\Ssm ; \Ssp).
			\end{equation}
			We will refine this information and prove a second order Taylor expansion using the bootstrap formulae from Appendix \ref{scn:boostrap_projectors}.  Before doing so, we observe that the following factorization holds true
			\begin{equation*}
				\RR(z,\LL_{\xi}) =\RR(z,\BB_{\xi})+  \big( \RR(z,\BB_{\xi}) \AA \big)^2 \RR(z,\BB_{\xi}) + \big(\RR(z,\BB_{\xi}) \AA \big)^2 \RR(z,\LL_{\xi}),
			\end{equation*}
			where the first two terms are actually  $\BBB(\Ss)$-valued holomorphic in $z \in \Delta_{-\lambda}$. Therefore, $\PP(\xi)$ can be written equivalently as
			$$\PP(\xi)=\frac{1}{2i\pi}\oint_{|z|=r} \big(\RR(z,\BB_{\xi})\AA\big)^{2}\RR(z,\LL_{\xi})\d z.$$
			Since, according to the regularization properties \ref{assumption_bounded_A} of $\AA$ together with the resolvent bounds \eqref{eq:bound_L_xi} and \eqref{eq:ResBB_xi} for $\LL_\xi$ and $\BB_\xi$ respectively, it holds, uniformly in $|z|=r$ and $| \xi | \le \alpha_0$
			\begin{align*}
				\| \RR(z, \BB_\xi) \AA \|_{ \BBB(\Ss ; \Ss_1 ) } + \| \RR(z, \BB_\xi) \AA \|_{ \BBB(\Ss_1 ; \Ss_2 ) } + \| \RR(z, \LL_\xi) \|_{ \BBB(\Ssm ; \Ss ) } \lesssim 1,
			\end{align*}
			and thus, using $\Ss_{2} \hookrightarrow \Ss_1$
			\begin{equation*}
				\left\| \big(\RR(z,\BB_{\xi})\AA\big)^{2}\RR(z,\LL_{\xi}) \right\|_{ \BBB(\Ssm ; \Ss_2) } \lesssim 1,
			\end{equation*}
			which once integrated along $| z | = r$, gives
			\begin{equation}
				\label{eq:regularization_P_xi}  \| \PP(\xi) \|_{\BBB(\Ssm;\Ss_{j})} \lesssim 1, \qquad j=0,1,2, \qquad   | \xi | \le \alpha_0\,.
			\end{equation}
			\step{1}{First order expansion}
			A representation of the first order Taylor expansion is given by integrating \eqref{eq:factorization_L_ivxi_left} or \eqref{eq:factorization_L_ivxi_right} with $N = 1$ yielding
			$$		\PP(\xi) = \PP + \xi \cdot \PP^{(1)}(\xi), \qquad | \xi | \le \alpha_0,$$
			where $\PP^{(1)}(\xi)$ has the integral representation
			$$		\PP^{(1)}(\xi) := \frac{1}{2 i \pi} \oint_{|z| = r} \RR(z,\LL_{\xi}) (- i v) \RR(z,\LL) \d z = \frac{1}{2 i \pi} \oint_{|z| = r} \RR(z,\LL) (- i v) \RR(z,\LL_{\xi}) \d z.$$	
			We know that $\PP^{(1)}(\xi) \in \BBB(\Ssm ; \Ssp)$ (since $\xi\cdot \PP^{(1)}(\xi)=\PP(\xi)-\PP$) let us prove that this holds uniformly in $\xi$. We begin with the following estimate in $\BBB(\Ss_{1};\Ssp)$ which is made possible since the multiplication by $v$ is bounded from $\Ss_{1}$ to $\Ss$. Namely, thanks to \ref{assumption_multi-v} and the resolvent bounds \eqref{eq:bound_L_xi} and \eqref{eq:bound_resolvent_L_X_j}, we have uniformly in $|\xi| \le \alpha_0$
			\begin{equation}
				\label{eq:bound_PP1_H1X*}
				\begin{split}
					\| \PP^{(1)}(\xi) \|_{\BBB(\Ss_{1};\Ssp)} & \le \frac{1}{2 \pi} \oint_{|z|=r} \| \RR(z,\LL_{\xi}) v \RR(z,\LL) \|_{\BBB(\Ss_{1};\Ssp) } \d |z| \\
					& \lesssim \oint_{|z|=r} \| \RR(z,\LL_{\xi}) \|_{\BBB(\Ss;\Ssp) } \| \RR(z,\LL) \|_{\BBB(\Ss_{1})} \d |z|\lesssim 1.
				\end{split}
			\end{equation}
			Let us now explain how to extend such an estimate to $\BBB(\Ssm;\Ss)$. Using the formula \eqref{eq:bootstrap_P_1_A}, we have, for $|\xi| \le \alpha_0$
			$$\PP^{(1)}(\xi) = \PP(\xi) \PP^{(1)}(\xi) + \PP^{(1)}(\xi) \PP,$$
			and therefore
			\begin{align*}
				\| \PP^{(1)}(\xi) \|_{\BBB(\Ssm;\Ssp)} & \le \| \PP(\xi) \PP^{(1)}(\xi) \|_{\BBB(\Ssm;\Ssp)} + \| \PP^{(1)}(\xi) \PP \|_{\BBB(\Ssm;\Ssp)}\\
				& \le  \| \PP^{(1)}(\xi)^{\star} \PP(\xi)^{\star} \|_{\BBB(\Ssm;\Ssp)} + \| \PP^{(1)}(\xi) \PP \|_{\BBB(\Ssm;\Ssp)}
			\end{align*}
			where we used the adjoint identity \eqref{eq:twisted_adjoint_identity} (where the adjoint is considered for $\la \cdot , \cdot \ra_{\Ss}$). Since we have
			$	\PP(\xi)^\star = \PP(-\xi)$ and $\PP^\star = \PP,$
			one has $\PP^{(1)}(\xi)^{\star}=-\PP^{(1)}(-\xi)$. Therefore
			\begin{align*}
				\| \PP^{(1)}(\xi) \|_{\BBB(\Ssm;\Ssp)} & \le \| \PP^{(1)}(-\xi) \PP(-\xi) \|_{\BBB(\Ssm;\Ssp)} + \| \PP^{(1)}(\xi) \PP \|_{\BBB(\Ssm;\Ssp)} \\
				& \le \| \PP^{(1)}(-\xi) \|_{\BBB(\Ss_{1};\Ssp)} \| \PP(-\xi) \|_{\BBB(\Ssm;\Ss_{1}) } + \| \PP^{(1)}(\xi) \|_{\BBB(\Ss_{1};\Ssp)} \| \PP \|_{\BBB(\Ssm;\Ss_{1}) } \\
				& \lesssim \| \PP^{(1)}(-\xi) \|_{\BBB(\Ss_{1};\Ssp)} + \| \PP^{(1)}(\xi) \|_{\BBB(\Ss_{1};\Ssp)},
			\end{align*}
			where we used the regularizing property \eqref{eq:regularization_P_xi} of $\PP(\xi)$ in the last line (recall that $\PP=\PP(0)$). Using the above estimates \eqref{eq:bound_PP1_H1X*}, we deduce that
			$$\| \PP^{(1)}(\xi) \|_{\BBB(\Ssm;\Ssp)} \lesssim 1, \qquad |\xi| \le \alpha_0,$$
			which concludes this step. Notice that a clear implication is that, for any $|\xi| \le \alpha_0$
			$$\|\PP(\xi)-\PP\|_{\BBB(\Ssm;\Ssp)} \lesssim |\xi|.$$
			
			\step{2}{Second order expansion} In a similar fashion, we obtain a second order expansion integrating once again \eqref{eq:factorization_L_ivxi_left} or \eqref{eq:factorization_L_ivxi_right} with~$N = 2$:
			\begin{gather*}
				\PP(\xi) = \PP + i \xi \cdot \PP_{1} + \xi \otimes \xi : \PP^{(2)}(\xi), \qquad | \xi | \le \alpha_0,
			\end{gather*}
			where the first and second order terms are defined as
			\begin{gather*}
				\PP_{1} = - \frac{1}{2 i \pi} \oint_{|z| = r} \RR(z,\LL) v \RR(z,\LL) \d z,\\
				\PP^{(2)}(\xi) := - \frac{1}{2 i \pi} \oint_{ | z | = r } \Big( \RR(z,\LL) v \Big)^{\otimes 2} \RR(z,\LL_{\xi}) \d z = - \frac{1}{2 i \pi} \oint_{ | z | = r } \RR(z,\LL_{\xi}) \Big(v \RR(z,\LL)\Big)^{\otimes 2} \d z.
			\end{gather*}
			The first order term  is explicitly computable using the Laurent series expansion \eqref{eq:pseudo_Laurent_expansion}:
			\begin{equation*}
				%		\label{eq:def_P_1}
				\PP_{1} = - \frac{1}{2 i \pi} \oint_{|z| = r} \RR(z,\LL) v \RR(z,\LL) \d z = \mathsf{R}_{0} v \PP + \PP v \mathsf{R}_{0},
			\end{equation*}
			so in particular, one checks directly that
			$$		\PP_{1} = \PP \PP_{1} + \PP_{1} \PP.$$
			We use now the formula \eqref{eq:bootstrap_P_2_A}, which gives
			\begin{align*}
				\| \PP^{(2)}(\xi) \|_{\BBB(\Ssm;\Ssp)} \le \| \PP(\xi) \PP^{(2)}(\xi) \|_{\BBB(\Ssm;\Ssp)}  + \| \PP^{(1)}(\xi) \otimes \PP_{1} \|_{\BBB(\Ssm;\Ssp)}  + \| \PP^{(2)}(\xi) \PP \|_{\BBB(\Ssm;\Ssp)}.
			\end{align*}
			We use the bound from \textit{Step 1} for the second term, and, after using duality as in \textit{Step 1} (for the first term), we use the regularization property \eqref{eq:regularization_P_xi} of $\PP(\xi)$:
			\begin{align*}
				\| \PP^{(2)}(\xi) \|_{\BBB(\Ssm;\Ssp)} & \lesssim 1 + \| \PP^{(2)}(-\xi) \PP(-\xi) \|_{\BBB(\Ssm;\Ssp)}
				+ \| \PP^{(2)}(\xi) \PP \|_{\BBB(\Ssm;\Ssp)} \\
				& \lesssim  1 + \| \PP^{(2)}(-\xi) \|_{\BBB(\Ss_{2},\Ssp)}
				+ \| \PP^{(2)}(\xi) \|_{\BBB(\Ss_{2},\Ssp)}.
			\end{align*}
			As previously, using the fact that the multiplication by $v$ is bounded from $\Ss_{j}$ to $\Ss_{j+1}$ for $j=0,1$, we show using the resolvent bounds \eqref{eq:bound_L_xi} for $\LL_\xi$ and \eqref{eq:bound_resolvent_L_X_j} for $\LL$ that
			\begin{align*}
				\| \PP^{(1)}(\xi) \|_{\BBB(\Ss_{2};\Ssp)} & \le \frac{1}{2 \pi}  \oint_{|z|=r} \left\| \RR(z,\LL_{\xi}) \Big(v \RR(z,\LL)\Big)^{\otimes 2} \right\|_{\BBB(\Ss_{2};\Ssp) } \d|z |\\
				& \lesssim  \oint_{|z|=r} \| \RR(z,\LL_{\xi}) \|_{\BBB(\Ss;\Ssp) } \| \RR(z,\LL) \|_{\BBB(\Ss_{1})} \| \RR(z,\LL) \|_{\BBB(\Ss_{2})} \d| z|\lesssim 1.
			\end{align*}
			This concludes the proof, setting $\mathsf{S}(\xi) := \xi \otimes \xi : \PP^{(2)}(\xi)$.
		\end{proof}
		
		Following Kato's reduction process from \cite[Section I-4.6, pp. 32--34]{K1995}, we introduce for any $| \xi | \le \alpha_0$ a ``rectified'' version of $\LL_\xi$ in which we cut off any spectral points that is not a small eigenvalue. Precisely, following \cite[Section I-4.6]{K1995} and since 
		$$\|\PP(\xi)-\PP\|_{\BBB(\Ssm,\Ssp)} \lesssim |\xi| \qquad \forall | \xi | \le \alpha_0, \quad $$
		according to the previous Lemma and the injection $\Ssp \hookrightarrow \Ss \hookrightarrow \Ssm$, we deduce that for $\alpha_0$ small enough 
		$$ \|\left(\PP(\xi)-\PP\right)^{2}\|_{\BBB(\Ss)} < 1, \qquad \forall | \xi | \le \alpha_0,$$
		where we recall that $\Ssp  \hookrightarrow \Ss\hookrightarrow \Ssm.$
		In particular (see \cite[Eqs. (4.36)--(4.39), p. 33]{K1995}), we  can define
		\begin{align*}
			\UU_\xi &= \Big( \PP(\xi) \PP + (\Id - \PP(\xi))(\Id - \PP) \Big) \Big(\Id -(\PP(\xi) - \PP)^2 \Big)^{-\frac{1}{2}} \\
			&= \Big(\Id -(\PP(\xi) - \PP)^2 \Big)^{-\frac{1}{2}} \Big( \PP(\xi) \PP + (\Id - \PP(\xi))(\Id - \PP) \Big), \qquad |\xi| \le \alpha_0
		\end{align*}
		where we used the definition $(\Id - T)^{-\frac{1}{2}} = \sum_{n= 0}^\infty \binom{-\frac{1}{2}}{n} (-T)^n$ for any $T \in \BBB(\Ss)$. With such a definition, $\UU_{\xi}$ mapping isomorphically the null-space of $\LL$ onto the eigenspaces corresponding to the small eigenvalues of $\LL_\xi$:
		\begin{equation*}
			\nul(\LL) = \range(\PP) \xrightarrow{\UU_{\xi}} \range\big(\PP(\xi) \big),
		\end{equation*}
		we define the \emph{finite dimensional} rectified operator on $\nul(\LL)$:
		\begin{equation*} \quad \rectL_{\xi} := \Big( \UU_{\xi}^{-1} \, \LL_{\xi} \, \UU_{\xi} \Big)_{| \nul(\LL) } \in \BBB\left( \nul(\LL) \right), \qquad \forall | \xi | \le \alpha_0, \quad 
		\end{equation*}
		whose spectrum is related to that of $\LL_{\xi}$ by
		%\begin{equation*}
		$	\mathfrak{S}(\LL_{\xi}) \cap \Delta_{-\lambda} = \mathfrak{S}\left(\rectL_{\xi}\right).$
		%	\end{equation*}
		In the rest of this section, we will study the structure of $\rectL_{ \xi}$ so as to reduce its diagonalization to the perturbation of a diagonalizable matrix with simple eigenvalues. Using (Kato's) classical perturbation theory for matrices, this will provide expansion of the eigenvalues and eigenfunctions of $\rectL_\xi$, and in turn those of $\LL_{\xi}$.
		
		\begin{lem}[\textit{\textbf{Properties of the rectified operator}}]
			\label{lem:rectified_operator}
			The rectified operator $\rectL_{ \xi}$ has the following properties:
			\begin{enumerate}
				\item $\rectL_{ \xi}$ is compatible with any orthogonal $d \times d$ matrix:
				\begin{equation*}
					\Theta \rectL_{\Theta \xi} = \rectL_{\xi} \Theta,
				\end{equation*}
				in particular, $\rectL_{\xi}$ commutes with orthogonal matrices that fix $\xi$, and it preserves eveness and oddity in directions orthogonal to $\xi$:
				\begin{gather*}
					\Theta \xi = \xi \Longrightarrow \Theta \rectL_{\xi} = \rectL_\xi \Theta,\\
					\Big( u \perp \xi \text{ and } \varphi(-u) = \pm \varphi(u) \Big) \Longrightarrow \Big( \rectL_{\xi} \varphi \Big)(-u) = \pm \Big( \rectL_{\xi} \varphi \Big)(u).
				\end{gather*}
				\item $\rectL_{\xi}$ has the following second order expansion in $\BBB\left( \nul(\LL) \right)$
				\begin{equation*}
					\rectL_{\xi} =  -i \PP ( v \cdot \xi) + \PP (v \cdot \xi) \mathsf{R}_{0} (v \cdot \xi) + \mathbf{r}_0
					(\xi)			\end{equation*}
				where $\mathbf{r}_0(\xi) \in \BBB(\nul(\LL))$ is such that $\|\mathbf{r}_0(\xi)\|_{\BBB( \nul(\LL) ) } \lesssim |\xi|^3.$
			\end{enumerate}
		\end{lem}
		
		\begin{proof}
			\medskip
			As $\Theta \LL_{\Theta \xi} = \LL_\xi \Theta$ holds for any orthogonal matrix $\Theta$, it is clear that the resolvent satisfies the same relation $\Theta \RR(z, \LL_{\Theta \xi})= \RR(z,\LL_\xi)\Theta$, and thus $\Theta \PP(\Theta \xi) = \PP(\xi) \Theta$. Using the above definition of $\UU_{\xi}$, this implies $\Theta \UU_{\Theta \xi} = \UU_{\xi} \Theta$ and thus the first point of this lemma. We only need to check the expansion for $\rectL_{ \xi}$. According to Lemma \ref{lem:expansion_projection}, there holds
			\begin{gather*}
				\PP(\xi) \PP = \PP +i \mathsf{R}_{0} (v \cdot \xi) \PP + r_{1}(\xi), \\
				(\Id - \PP(\xi))(\Id - \PP) =  \Id - \PP - i \PP (v \cdot \xi) \mathsf{R}_{0}  + r_{2}(\xi),
			\end{gather*}
			and
			$$				(\PP(\xi) - \PP)^2=r_{3}(\xi),
			$$
			where the \emph{remainder operators} $r_{k}$ are such that
			$$\left\|r_{k}(\xi)\right\|_{\BBB(\Ssm;\Ssp)} \lesssim |\xi|^{2}, \qquad k=1,2,3,$$
			for $|\xi| \le \alpha_0$ with $\alpha_0$ small enough. Inserting this in the definition of $\UU_{\xi}$, there exist additional remainder operators $r_{k}(\xi) \in \BBB(\Ssm,\Ssp)$ \footnote{with for instance $r_{4}(\xi)=r_{1}(\xi)+r_{2}(\xi)$ and $r_{5}(\xi)=\sum_{n=1}^{\infty}{-\frac{1}{2} \choose n}(-r_{3}(\xi))^{n}.$} such that
			\begin{equation*}
				\label{eq:expansion_Kato_isomorphism}\begin{split}
					\UU_\xi & = \Big( \Id - i \PP (v \cdot \xi)\mathsf{R}_{0} + i\mathsf{R}_{0} (v \cdot \xi) \PP + r_{4}(\xi) \Big) \Big( \Id + r_{5}(\xi) \Big) \\
					& = \Id - i \PP (v \cdot \xi)\mathsf{R}_{0} + i\mathsf{R}_{0} (v \cdot \xi) \PP + r_{6}(\xi),
				\end{split}
			\end{equation*}
			where $\|r_{k}(\xi)\|_{\BBB(\Ssm,\Ssp)} \lesssim |\xi|^{2}$ for $k=1,\ldots,6.$ Note that there holds as well
			\begin{equation*}
				\label{eq:expansion_Kato_isomorphism_inverse}
				\UU^{-1}_\xi = \Id - i\mathsf{R}_{0} (v \cdot \xi) \PP + i \PP (v \cdot \xi)\mathsf{R}_{0} + r_{7}(\xi),
			\end{equation*}
			with $\|r_{7}(\xi)\|_{\BBB(\Ssm;\Ssp)} \lesssim |\xi|^{2}.$ Notice that in the above decompositions of $\UU_{\xi}$ and $\UU_{\xi}^{-1}$, all terms (except $\Id$) are in $\BBB\left( \Ssm ; \Ssp \right)$. We thus compute the second order expansion
			$$\rectL_\xi =\PP \rectL_\xi \PP=\PP \Big(\UU^{-1}_\xi \LL_\xi \UU_\xi\Big) \PP $$
			as follows. Observe that $\PP \mathsf{R}_{0}=\mathsf{R}_{0}\PP=0$ so that
			\begin{equation}
				\label{eq:expansion_Kato_isomorphismPP}
				\PP\UU^{-1}_\xi=\PP + i \PP (v \cdot \xi)\mathsf{R}_{0} + \PP\, r_{7}(\xi), \qquad \UU_{\xi}\PP=\PP +
				i\mathsf{R}_{0}(v\cdot\xi)\PP+r_{6}(\xi)\PP\,,
			\end{equation}
			and that $\range(r_{6}(\xi)\PP) \subset \dom(\LL)$ so that $r_{8}(\xi):=\LL\,r_{6}(\xi)\PP$ is well-defined and in the end we have $\|r_{8}(\xi) \|_{\BBB(\Ssm,\Ssp)} \lesssim |\xi|^{2}$. Therefore, using also that $\LL \PP=0$ while $\LL\mathsf{R}_{0}=\Id-\PP$, we deduce that
			\begin{equation*}\begin{split}
					\rectL_\xi & = \Big( \PP + i \PP (v \cdot \xi)\mathsf{R}_{0} + \PP r_{7}(\xi) \Big) (\LL - i v \cdot \xi) \Big( \PP + i\mathsf{R}_{0} (v \cdot \xi) \PP + r_{6}(\xi)\PP\Big) \\
					&= \Big( \PP + i \PP (v \cdot \xi)\mathsf{R}_{0} + \PP r_{7}(\xi) \Big) \Big(- i \PP (v \cdot \xi) \PP  + (v \cdot\xi) \mathsf{R}_{0}  (v \cdot \xi) \PP + r_{8}(\xi) + \tilde{r}_{1}(\xi)\Big) \\
					&= - i \PP (v \cdot \xi) \PP + \PP ( v \cdot \xi)\mathsf{R}_{0} (v \cdot \xi) \PP + \tilde{r}_{2}(\xi),
			\end{split}\end{equation*}
			where $\|\tilde{r}_{k}(\xi)\|_{\BBB(\Ssm,\Ssp)} \lesssim |\xi|^{3}$ for $k=1,2$ because the second order remainder term $\PP r_{8}(\xi)$ vanishes since $\PP\LL=0$. The lemma is proved.
		\end{proof}
		
		\begin{rem}
			Note that in order to compute a second order expansion of the rectified operator, only a first order expansion of the spectral projector is needed.
		\end{rem}
		
		\begin{lem}[\textit{\textbf{Block matrix representation of the rectified operator}}]
			\label{lem:rectified_block_matrix}
			For any non-zero $|\xi| \le \alpha_0$, the rectified operator $\rectL_\xi$ writes, along the $\Ss$-orthogonal decomposition
			\begin{equation*}
				\nul(\LL) = \Big\{ u \cdot v \mu ~ ; ~ u \perp \xi \Big \} \oplus \Span\left( \psi_{\Bou}, \psi_{+\Wave}\left( \widetilde{\xi} \right), \psi_{-\Wave}\left( \widetilde{\xi} \right) \right), \quad \widetilde{\xi}=\frac{\xi}{|\xi|},  
			\end{equation*}
			as a block matrix:
			\begin{equation}\label{eq:LLXIMXI}\rectL_{\xi} = \left(
				\begin{matrix}
					\lambda_\Inc(\xi) \, \Id_{(d-1) \times (d-1)} & 0_{(d-1) \times 3} \\
					0_{3 \times (d-1)} & \rectsL(\xi) \\
				\end{matrix}
				\right), \qquad \rectsL(\xi) \in \mathscr{M}_{3\times3}(\R)\end{equation}
			where the ``incompressible'' eigenvalue $\lambda_\Inc(\xi)$ can be expressed for any normalized $u \perp \xi$ as
			\begin{equation*}
				\lambda_\Inc(\xi) = \frac{d}{E} \left \langle \rectL_\xi (u \cdot v \mu), u \cdot v \mu \right \rangle_{\Ss}.
			\end{equation*}
		\end{lem}
		
		\begin{proof}
			Fix some non-zero $| \xi | \le \alpha_0$ and consider some $u \perp \xi$, using the first point of Lemma \ref{lem:rectified_operator}, the functions $\rectL_{\xi} ( u \cdot v \mu)$ and~$\left(\mu, \xi \cdot v \mu, |v|^2 \mu\right)$ are respectively odd and even in the direction $u$, and thus $\Ss$-orthogonal. Similarly, $\rectL_\xi \left(\mu, \xi \cdot v \mu, |v|^2 \mu\right)$ and $u \cdot v \mu$ are also~$\Ss$-orthogonal. Thus, one may write
			$$\rectL_{\xi} = \left(
			\begin{matrix}
				\rectsL'(\xi) & 0_{(d-1) \times 3} \\
				0_{3 \times (d-1)} & \rectsL(\xi) \\
			\end{matrix}
			\right),$$
			for some $(d-1) \times (d-1)$ matrix $\rectsL'(\xi)$, because $\left\{\left(a + b \xi + c | v |^2\right)\mu ~ | ~ a, b, c \in \R \right\}$ is spanned by $\psi_{\Bou}, \psi_{- \Wave}\left( \widetilde{\xi} \right)$ and $\psi_{+\Wave}\left( \widetilde{\xi} \right)$.
			
			\medskip
			In the case $d \ge 3$, we still have to show that $\rectsL'(\xi)$ is a multiplication matrix. To do so, consider a pair of $\Ss$-orthogonal functions $\varphi = u \cdot v \mu$ and $\varphi' = u' \cdot v \mu$ for some vectors~$u, u' \in \R^d$ such that~$(u, u', \xi)$ is an orthogonal triple. From the first point of Lemma~\ref{lem:rectified_operator}, the function $\rectL_{\xi} \varphi$ is odd in the direction $u$ and even in the direction $u'$, thus
			\begin{gather*}
				\left\langle \rectL_\xi (u \cdot v \mu), u' \cdot v \mu \right\rangle_{\Ss} = \left\langle \rectL_\xi \left(u' \cdot v \mu\right), u \cdot v \mu \right\rangle_{\Ss} = 0.
			\end{gather*}
			Furthermore, choosing an orthogonal matrix $\Theta$ mapping $(\xi, u, u')$ onto $(\xi, u', u)$, we have
			\begin{align*}
				\langle \rectL_\xi (u \cdot v \mu), u \cdot v \mu \rangle_{\Ss}  = \langle \rectL_\xi (\Theta u) \cdot v \mu, (\Theta u) \cdot v \mu \rangle_{\Ss} = \left \langle \rectL_\xi \left(u' \cdot \mu\right), u' \cdot v \mu \right \rangle_{\Ss}.
			\end{align*}
			We conclude by applying these two relations to any orthogonal basis of $\Big\{ u \cdot v \mu ~| ~ u \perp \xi \Big \}$.
		\end{proof}
		
		\begin{lem}[\textit{\textbf{Expansion in matrix form}}]
			\label{lem:expansion_rectified_matrix}
			Recall that $\kappa_\star$ and $c$ are defined in Theorem \ref{thm:spectral_study}. With the notations of Lemma \ref{lem:rectified_block_matrix},
			the ``incompressible'' eigenvalue expands~as
			\begin{equation*}
				\lambda_\Inc(\xi) = - \kappa_\Inc | \xi |^2 + \OO\left( | \xi |^3 \right).
			\end{equation*}
			Furthermore, the matrix $\rectsL(\xi)$ in \eqref{eq:LLXIMXI} can be written in the basis $\psi_{\Bou}, \psi_{- \Wave}\left( \widetilde{\xi} \right), \psi_{ + \Wave}\left( \widetilde{\xi} \right)$ as  
			\begin{equation}
				\label{eq:expansion_rectified_matrix}
				\rectsL(\xi) = i | \xi |
				\left(
				\begin{matrix}
					0 & 0 & 0 \\
					0 & c & 0 \\
					0 & 0 & - c
				\end{matrix}
				\right) + | \xi |^2 \left(
				\begin{matrix}
					- \kappa_\Bou & * & * \\
					* & - \kappa_{\Wave} & * \\
					* & * & - \kappa_{\Wave}
				\end{matrix}
				\right) + \OO\left( | \xi |^3 \right).
			\end{equation}
		\end{lem}
		
		\begin{proof}
			Straightforward calculations show that, in the basis $\psi_{\Bou}, \psi_{- \Wave}\left( \widetilde{\xi} \right), \psi_{+\Wave}\left( \widetilde{\xi} \right)$, the first order coefficient is diagonal:
			\begin{equation*}
				\rectsL(\xi) = -i \PP (v \cdot \xi) + \OO\left( | \xi |^2 \right) = i | \xi |
				\left(
				\begin{matrix}
					0 & 0 & 0 \\
					0 & c & 0 \\
					0 & 0 & - c
				\end{matrix}
				\right) + \OO\left( | \xi |^2 \right).
			\end{equation*}
			We then turn to the second order coefficient $ \left( v \cdot \widetilde{\xi} \right)\mathsf{R}_{0} \left( v \cdot \widetilde{\xi} \right)$.
			
			\step{1}{The diffusion coefficient $\kappa_\Inc$}
			In this step, we will use the following identity:
			\begin{equation}
				\label{eq:invariance_L_inv_A}
				\left\la \LL^{-1}\BurA u \cdot u', \BurA   u \cdot u' \right\ra_{\Ss} = \left\la \LL^{-1}\BurA (\Theta u) \cdot (\Theta u'),\BurA  (\Theta u) \cdot (\Theta u') \right\ra_{\Ss},
			\end{equation}
			which holds for any $u, u' \in \R^d$ and any orthogonal matrix $\Theta$, and is a consequence of the identities (where we recall $\mathsf{R}_0 = -\LL^{-1} (\Id - \PP)$ )
			$$\Theta\mathsf{R}_{0} =\mathsf{R}_{0} \Theta, \qquad \BurA (v) u \cdot u' = \BurA (\Theta v) (\Theta u) \cdot (\Theta u').$$
			The diffusion coefficient $\kappa_\Inc$ is given for any $\sigma \in \S^{d-1}$ orthogonal to $\widetilde{\xi}$ by
			$$
			\kappa_\Inc = \frac{ d}{E} \left\langle\mathsf{R}_{0} \left( v \cdot \widetilde{\xi} \right)  (v \cdot \sigma) \mu, \left( v \cdot \widetilde{\xi} \right) \left(v \cdot \sigma\right) \mu \right\rangle_{\Ss} = \left\langle \LL^{-1}\BurA \widetilde{\xi} \cdot \sigma, \BurA  \widetilde{\xi} \cdot \sigma \right\rangle_{\Ss},
			$$	which, for any orthogonal pair $\omega, \sigma \in \S^{d-1}$, rewrites using \eqref{eq:invariance_L_inv_A} as
			\begin{equation}
				\label{eq:rel_non_diag}
				\kappa_\Inc = -  \left\langle \LL^{-1}\BurA \omega \cdot \sigma,  \BurA  \omega \cdot \sigma \right\rangle_{\Ss}.
			\end{equation}
			Choosing in particular $\omega = \frac{1}{\sqrt{2}}(u - u')$ and $\sigma = \frac{1}{\sqrt{2}}(u + u')$, where $u, u' \in \S^{d-1}$ are orthogonal, we have
			$$\kappa_{\Inc}=- \frac{1}{2} \left\langle \LL^{-1} \big(\BurA u \cdot u - \BurA  u' \cdot u'\big),  \BurA u \cdot u - \BurA  u' \cdot u' \right\rangle_{\Ss}.$$
			%	\begin{align*}
			%		\kappa_\Inc & = - \frac{1}{2} \left\langle \LL^{-1}\BurA (u-u') \cdot (u+u'),  A (u-u') \cdot (u+u') \right\rangle_{\Ss} \\
			%		&=- \frac{1}{2} \left\langle \LL^{-1} \big(Au \cdot u - A u' \cdot u'\big),  Au \cdot u - A u' \cdot u' \right\rangle_{\Ss},
			%	\end{align*}
			where we used the fact that $A$ is symmetric. Consequently, we have by \eqref{eq:invariance_L_inv_A}
			\begin{align*}
				\kappa_\Inc  = & - \left\langle \LL^{-1} \BurA u \cdot u,  \BurA u \cdot u \right\rangle_{\Ss} + \left\langle \LL^{-1} \BurA u \cdot u,  \BurA  u' \cdot u' \right\rangle_{\Ss},
			\end{align*}
			and since $\BurA$ is trace-free, one can get rid of the term involving $u'$ by averaging over some orthonormal family~$u' = u_1', \dots, u_{d-1}' \in \left( \R u \right)^\perp$:
			\begin{align*}
				\kappa_\Inc  = & - \left\langle \LL^{-1} \BurA u \cdot u,  \BurA u \cdot u \right\rangle_{\Ss} + \frac{1}{d-1} \left\langle \LL^{-1} \BurA u \cdot u , \sum_{i=1}^{d-1} \BurA u'_i \cdot u'_i \right\rangle_{\Ss} \\
				= & - \left\langle \LL^{-1} \BurA u \cdot u,  \BurA u \cdot u \right\rangle_{\Ss} - \frac{1}{d-1} \left\langle \LL^{-1} \BurA u \cdot u , \BurA  u \cdot u \right\rangle_{\Ss},
			\end{align*}
			and therefore
			\begin{equation}
				\label{eq:rel_diag}
				\kappa_\Inc = - \frac{d}{d-1} \left\langle \LL^{-1} \BurA u \cdot u,  \BurA u \cdot u \right\rangle_{\Ss}.
			\end{equation}
			Summing \eqref{eq:rel_non_diag} and \eqref{eq:rel_diag} over all pairs of vectors in the canonical basis of $\R^d$, we rewrite the coefficient $\kappa_\Inc$ as a Hilbert-Schmidt norm of matrices:
			\begin{equation*}
				\kappa_\Inc = -\frac{1}{(d-1)(d+1)} \left\langle \LL^{-1}\BurA , \BurA  \right\rangle_{\Ss}.
			\end{equation*}
			\step{2}{The diffusion coefficient $\kappa_\Bou$}
			The coefficient $\kappa_\Bou$ writes
			\begin{align*}
				\kappa_\Bou & = -  \left \langle\mathsf{R}_{0} \left( v \cdot \widetilde{\xi} \right)  \psi_{\Bou} , \left( v \cdot \widetilde{\xi} \right) \psi_{\Bou} \right\rangle_{\Ss}= - \left \langle \LL^{-1}\BurB \cdot \widetilde{\xi}, \BurB  \cdot \widetilde{\xi} \right \rangle_{\Ss} = - \frac{1}{d} \left \langle \LL^{-1} \BurB , \BurB  \right \rangle_{\Ss},
			\end{align*}
			where we used the invariance of $\LL^{-1}$ again, as well as the identity $\BurB  (v) \cdot u = \BurB  (\Theta v) \cdot (\Theta u)$, allowing to sum over $\widetilde{\xi}$ taken in the canonical basis of $\R^d$.
			
			\step{3}{The diffusion coefficient $\kappa_{\Wave}$}
			The coefficient $\kappa_{\Wave}$ writes
			\begin{align*}
				\kappa_\Wave & = - \left \langle\mathsf{R}_{0} \left( v \cdot \widetilde{\xi} \right)  \psi_{\pm \Wave} \left(\widetilde{\xi}\right), \left( v \cdot \widetilde{\xi} \right) \psi_{\pm \Wave}\left(\widetilde{\xi}\right) \right\rangle_{\Ss} \\
				& = - \frac{1}{2} \left \langle \LL^{-1}\BurA \widetilde{\xi} \cdot \widetilde{\xi}, \BurA   \widetilde{\xi} \cdot \widetilde{\xi} \right \rangle_{\Ss} - \frac{E^2(K-1)}{2} \left \langle \LL^{-1}\BurB \cdot \widetilde{\xi}, \BurB   \cdot \widetilde{\xi} \right \rangle_{\Ss},
			\end{align*}
			where we have used the fact that $\BurA   \widetilde{\xi} \cdot \widetilde{\xi}$, and thus  $\LL^{-1}\BurA \widetilde{\xi} \cdot \widetilde{\xi}$ again by the invariance of~$\LL^{-1}$, is even in the direction $\widetilde{\xi}$, whereas $B \cdot \widetilde{\xi}$ is odd in this direction, and therefore these functions are $\Ss$-orthogonal. Using the results of the previous steps, we actually have
			$$
			\kappa_\Wave = \frac{d-1}{2 d} \kappa_\Inc + \frac{E^2(K-1)}{2} \kappa_\Bou.$$
			The lemma is proved.
		\end{proof}
		
		\newcommand{\rectP}{\mathbb{P}}
		\begin{lem}[\textit{\textbf{Second order diagonalization and decomposition of the rectified operator}}]
			\label{lem:diagonalization_rectified} With the notations of Lemma \ref{lem:rectified_block_matrix},
			the rectified operator $\rectL_\xi$ has four distinct eigenvalues which expand as
			\begin{gather*}
				\lambda_\Inc(\xi) = - \kappa_\Inc | \xi |^2 + \OO\left( |\xi|^3 \right), \\
				\lambda_\Bou(\xi) = - \kappa_\Bou | \xi |^2 + \OO\left( |\xi|^3 \right), \quad
				\lambda_{\pm \Wave}(\xi) = \pm i c | \xi| - \kappa_\Wave |\xi|^2 + \OO\left( |\xi|^3 \right)\,.
			\end{gather*}
			Furthermore, the ``incompressible'' eigenvalue is associated with the following spectral projector:
			\begin{equation*}
				\rectP_\Inc\left( \xi \right) \left[ a \mu + u \cdot v \mu + c |v|^2 \mu \right] =  {\left(\Pi_{\widetilde{\xi}}u\cdot v\right)} \mu,
			\end{equation*}
			and the spectral projectors associated with the ``Boussinesq'' and ``waves'' eigenvalues expand in the basis $\psi_{\Bou}, \psi_{- \Wave}\left( \widetilde{\xi} \right), \psi_{+ \Wave}\left( \widetilde{\xi} \right)$ as
			\begin{equation}\label{eq:textbook_expansion_projector}\rectP_{\star}(\xi)=\rectP_{\star}^{(0)} + |\xi|\rectP^{(1)}_{\star}+\OO\left(|\xi|^{2}\right), \qquad \star=\Bou,\pm\Wave\end{equation}
			where
			\begin{equation}\label{eq:textbook_rectP0}	
				\rectP_{\Bou}^{(0)}= \left(\begin{matrix}
					1 & 0 & 0 \\ 0 & 0 & 0 \\ 0 & 0 & 0
				\end{matrix}\right), \quad \rectP_{-\Wave}^{(0)} = \left(\begin{matrix}
					0 & 0 & 0 \\ 0 & 1 & 0 \\ 0 & 0 & 0
				\end{matrix}\right), \quad \rectP_{\Wave}^{(0))} = \left(\begin{matrix}
					0 & 0 & 0 \\ 0 & 0 & 0 \\ 0 & 0 & 1
				\end{matrix}\right)\end{equation}
			and $\rectP_{\star}^{(1)}$ is a constant $3\times 3$-matrix for $\star=\Bou,\pm\Wave$. 
		\end{lem}
		
		\begin{proof}
			Because the operator $\rectsL(\xi)$ has the asymptotic expansion
			$$\rectsL(\xi) = |\xi| \Big( \rectsL^{(0)}\left( \widetilde{\xi} \right) + | \xi | \rectsL^{(1)}\left( \widetilde{\xi} \right) + \OO\left( |\xi|^2 \right) \Big)$$
			where $\rectsL^{(0)}\left( \widetilde{\xi} \right)$ has three distinct eigenvalues by Lemma \ref{lem:expansion_rectified_matrix}, we know from \cite[Chapter 2, Theorem 5.4]{K1995} that $\rectsL(\xi)$ has three distinct eigenvalues for $\xi$ small, and thus admits the following spectral decomposition:
			$$\rectsL(\xi) = \lambda_\Bou(\xi) \rectP_\Bou(\xi) +  \lambda_{+ \Wave}(\xi) \rectP_{+\Wave}(\xi) +  \lambda_{- \Wave}(\xi) \rectP_{-\Wave}(\xi).$$
			The expansions of the eigenvalues is given by \cite[Chapter II-(5.12)]{K1995} applied to \eqref{eq:expansion_rectified_matrix}:
			\begin{gather*}
				\lambda_\Inc(\xi) = - \kappa_\Inc | \xi |^2 + \OO\left( |\xi|^3 \right), \quad \lambda_\Bou(\xi) = - \kappa_\Bou | \xi |^2 + \OO\left( | \xi |^2 \right),\\
				\lambda_{\pm \Wave}(\xi) = \pm i c | \xi| - \kappa_{\Wave} | \xi |^2 + \OO \left( | \xi |^3 \right),
			\end{gather*}
			and the expansions of the spectral projectors are given by \cite[Chapter II-(5.9)]{K1995} applied to~\eqref{eq:expansion_rectified_matrix} yielding
			\eqref{eq:textbook_expansion_projector} and \eqref{eq:textbook_rectP0} and we point out that the first order coefficients $\rectP^{(1)}_{\star}$ in matrix form has coefficients independent of $\widetilde{\xi}$ and can be explicitly computed, although their expression will not be needed.
		\end{proof}

		\begin{figure}[h]
			\centering
			
			\begin{tikzpicture}
				% Spectre cinétique
				\fill [pattern=north west lines, pattern color=red] (-6, -2.5) rectangle (-2.7, 2.5);
				\draw [color=red] (-2.7, -2.5) -- (-2.7, 2.5);
				\node[red] at (-4, -3) {$\Re z \le - \lambda$};
				
				% Seuils cinétique
				\draw[ultra thick, dash pattern=on 5pt, color=red] (-3.4, -2.5) -- (-3.4, 2.5);
				\node[red] at (-5.2, 3) {$\Re z = - \lambda_\BB$};
				\draw[ultra thick, dash pattern=on 5pt, color=red] (-4.7, -2.5) -- (-4.7, 2.5);
				\node[red] at (-3, 3) {$\Re z = - \lambda_\LL$};

				% Valeurs propres ondes
				\draw[black, variable=\y, dash pattern=on 5pt] plot[smooth,domain=-2.25:2.25] ( {-0.5 * (\y)^2 }, \y );
				\node[blue] at ( {-0.5 * (1.7)^2} , {1.7 + 0.7} ) {$\lambda_{+ \Wave}(\xi)$};
				\fill [color=blue] ( {-0.5 * (2.2)^2} , 2.2) circle (0.05);
				\node[blue] at ( {-0.5 * (1.7)^2} , {-1.7 - 0.7} ) {$\lambda_{- \Wave}(\xi)$};
				\fill [color=blue] ( {-0.5 * (2.2)^2} , -2.2) circle (0.05);
				
				% Valeur propre Boussinesq
				\fill [color=blue] (-1.9, 0) circle (0.05);
				\node[blue] at (-1.9, -0.35) {$\lambda_\Bou(\xi)$};
				
				% Valeur propre incompressible
				\fill [color=blue] (-1.1, 0) circle (0.05);
				\node[blue] at (-1, 0.4) {$\lambda_\Inc(\xi)$};
				
				% Repère
				\draw[-stealth] (-6, 0) -- (1, 0);
				\draw[-stealth] (0, -2.5) -- (0, 2.5);
				
			\end{tikzpicture}
			
			\caption{Localization of the spectrum of $\LL - i (v \cdot \xi) $ for $| \xi | \le \alpha_0 $.}
			
		\end{figure}
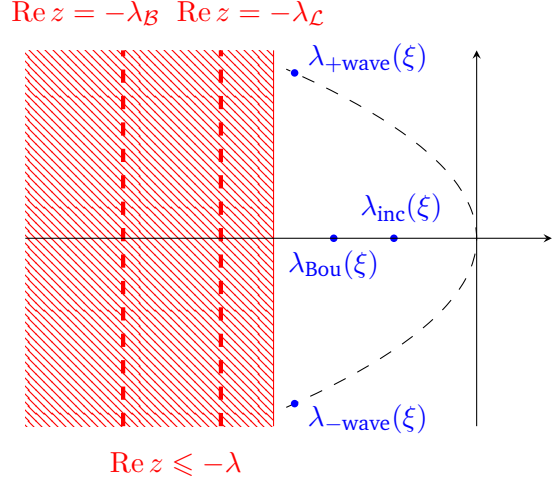
		
		\begin{lem}[\textit{\textbf{Expansion of the spectral projectors}}]
			\label{lem:expansion_projectors}
			For $\alpha_{0} >0$ small enough, the following spectral decomposition holds for any $| \xi | \le \alpha_{0}$:
			\begin{equation*}\begin{split}
					\LL_\xi \PP(\xi) &= \PP(\xi) \LL_\xi = \lambda_{\Bou}(\xi)\PP_{\Bou}(\xi)+\lambda_{\Inc}(\xi)\PP_{\Inc}(\xi)\\
					&\phantom{++++} +\lambda_{+\Wave}(\xi)\PP_{+\Wave}(\xi)+\lambda_{-\Wave}(\xi)\PP_{-\Wave}(\xi)
			\end{split}\end{equation*}
			where the projector operators $\PP_\star(\xi)$ $(\star = \Bou, \Inc, \pm\Wave)$  expand in $\BBB(\Ssm ; \Ssp)$ as \eqref{eq:PPstar}
			with the zeroth order coefficients being defined in Theorem \ref{thm:spectral_study}.
			%	\begin{gather*}
			%		\PP_\Inc^{(0)}\left( \omega \right) \varphi = \frac{d}{E} \left(\Id - \omega \otimes \omega\right) \langle \varphi, v \mu \rangle_{\Ss} \cdot v \mu,\\
			%		\PP_\star(\omega) = \langle \varphi, \psi_\star(\omega) \rangle_{\Ss} \psi_\star(\omega), \quad \star = \Bou, \pm \Wave,
			%	\end{gather*}
			%	and the first order terms are defined as
			%	\begin{gather*}
			%		\PP^{(1)}_\Inc(\omega) \varphi = \sqrt{\frac{d}{E}} \left\langle \varphi, (\Id - \omega \otimes \omega) \LL^{-1}\BurA \omega \right\rangle_{\Ss} \cdot v \mu,\\
			%		\PP^{(1)}_{\Bou}(\omega) \varphi = - \langle \varphi, \LL^{-1} \BurB \omega \rangle_{\Ss} \psi_\Bou,\\
			%		\PP^{(1)}_{\pm \Wave}(\omega) \varphi = \left(\mp \frac{1}{\sqrt{2}  } \langle \varphi , \LL^{-1}  A(v) \omega \cdot \omega \rangle_{\Ss} - E \sqrt{\frac{K-1}{2}} \langle \varphi, \LL^{-1}  B(v) \cdot \omega \rangle_{\Ss} \right) \psi_{\pm \Wave}(\omega).
			%	\end{gather*}
		\end{lem}
		
		\begin{proof}
			Recall that $\rectL_\xi$ and $\LL_\xi$ are related through
			$\rectL_\xi = \left(\UU_\xi^{-1} \LL_{\xi} \UU_\xi\right)_{| \nul(\LL)},$
			thus, using the fact that $\UU_\xi \PP= \PP(\xi) \UU_{\xi}$, we deduce that
			\begin{equation*}
				\PP(\xi) \LL_{\xi} = \LL_\xi \PP(\xi)= \UU_\xi \rectL_{\xi} \UU_\xi^{-1}.
			\end{equation*}
			This lemma is therefore a lifted version of Lemma \ref{lem:diagonalization_rectified}, and the corresponding projectors are related through
			\begin{equation*}
				\PP_\star(\xi) = \UU_{\xi} \rectP_\star(\xi) \UU_\xi^{-1} = \big(\UU_{\xi} \PP\big) \rectP_\star(\xi) \big(\PP \UU_\xi^{-1}\big).
			\end{equation*}
			We then deduce the expansion of each $\PP_\star(\xi)$ from those of $\rectP_\star(\xi)$ in $\BBB\left( \nul(\LL) \right)$ established in Lemma~\ref{lem:diagonalization_rectified} and those of $\UU_\xi \PP$ and $\PP \UU_\xi^{-1}$ in $\BBB(\Ssm ; \Ssp)$ from \eqref{eq:expansion_Kato_isomorphismPP}:
			$$
			\PP_\star(\xi) =  \Big( \PP + i \xi\cdot \mathsf{R}_{0}  v \PP + \PP r_{7}(\xi) \Big) 
			\Big( \rectP_\star^{(0)}\left( \widetilde{\xi} \right) + |\xi| \rectP^{(1)} \left( \widetilde{\xi} \right) + \mathbb{S}(\xi)\Big)\Big( \PP + i \xi \cdot\PP v\mathsf{R}_{0} + r_{6}(\xi)\PP \Big)
			$$
			where we recall that $\|r_{j}(\xi)\|_{\BBB(\Ssm;\Ssp)} \lesssim |\xi|^{2}$ while $\mathbb{S}(\xi) \in \MMM_{(d+2)\times (d+2)}(\R)$ with norm of order~$\OO\left( |\xi|^2 \right).$ We can expand further to deduce
			%\begin{align*}
			%	\PP_\star(\xi) = & \rectP^{(0)}\left( \widetilde{\xi} \right) + i \xi \cdot \rectP^{(0)} \left( \widetilde{\xi} \right) v \mathsf{R}_{0} 
			%	+ \Big( | \xi | \rectP^{(1)}\left( \widetilde{\xi} \right) + i \xi \cdot\mathsf{R}_{0}  v \rectP^{(0)}\left( \widetilde{\xi} \right) \Big)
			%	+ \OO(| \xi |^2 ) \\
			%	& =: \PP_\star^{(0)}\left(\widetilde{\xi}\right) + i \xi \cdot \PP^{(1)}_\star\left(\widetilde{\xi}\right) + | \xi |\,S_{\star}^{(0)}(\xi) + |\xi|^{2}S_{\star}^{(1)}(\xi),
			%\end{align*}
			$$\PP_\star(\xi) =: \PP_\star^{(0)}\left(\widetilde{\xi}\right) + i \xi \cdot \PP^{(1)}_\star\left(\widetilde{\xi}\right)  + \mathsf{S}_{\star}(\xi),$$
			where we have denoted \footnote{Of course, $\PP\rectP^{(0)}_{\star}(\cdot)\PP$ can be identified with $\rectP^{(0)}_{\star}(\cdot)$ but we make here the (slight) distinction  between operators defined on the finite dimensional space  and the associated matrices.}
			\begin{gather*}%\begin{cases}
				\PP^{(0)}_\star\left( \widetilde{\xi} \right) := \PP\rectP^{(0)}_\star\left( \widetilde{\xi} \right)\PP,		
				\qquad
				\PP_\star^{(1)}\left( \widetilde{\xi} \right) := \PP^{(0)}_\star(\widetilde{\xi}) v \mathsf{R}_{0} +\mathsf{R}_{0}v \PP^{(0)}_{\star}\left(\widetilde{\xi}\right)-i\widetilde{\xi}\PP\,\rectP^{(1)}_{\star}(\widetilde{\xi})\PP\,.%\end{cases}
			\end{gather*}
			We notice that both $\mathsf{R}_{0}v\PP^{(0)}_{\star}(\widetilde{\xi})$ and $\PP\,\rectP^{(1)}_{\star}(\widetilde{\xi})\PP$ vanish on $\nul(\LL)^\perp$:
			$$\varphi \in \nul(\LL)^\perp \Rightarrow \PP^{(1)} \left( \widetilde{\xi} \right) \varphi = \PP^{(0)} \left( \widetilde{\xi} \right) v \mathsf{R_{0}} \varphi =  \PP^{(0)} \left( \widetilde{\xi} \right) v \varphi.$$
			Therefore, the first order term of the projector associated with the ``incompressible'' eigenvalue writes explicitly for any~$\varphi\in \nul(\LL)^{\perp}$, and $\omega \in \S^{d-1}$
			\begin{align*}
				\PP^{(1)}_\Inc(\omega) \varphi
				& = \frac{d}{E} \Bigg( \Big[ \left(\Id - \omega \otimes \omega \right) \left\la v_j \mathsf{R}_0 \varphi , v \mu \right\ra_{\Ss} \Big] \cdot v \mu \Bigg)_{j=1}^d \\
				&= \frac{d}{E} \Bigg(  \left\la v_j \mathsf{R}_0 \varphi , v \mu \right\ra_{\Ss} \cdot \Big[ \left(\Id - \omega \otimes \omega \right) v \Big] \mu \Bigg)_{j=1}^d \\
				& = \sqrt{\frac{d}{E}} \left\la \varphi , \LL^{-1}\BurA \right\ra_{\Ss} \Big[ \left(\Id - \omega \otimes \omega \right) v \Big] \mu,
			\end{align*}
			and, in particular,
			\begin{align*}
				\omega \cdot \PP^{(1)}_\Inc(\omega) \varphi
				& = \frac{d}{E} \Big[ \left(\Id - \omega \otimes \omega \right) \left\la v \cdot \omega \mathsf{R}_0 \varphi , v \mu \right\ra_{\Ss} \Big] \cdot v \mu \\
				& = \sqrt{\frac{d}{E}} \Big[ \left(\Id - \omega \otimes \omega \right) \left\la \varphi , \LL^{-1}\BurA \omega \right\ra_{\Ss} \Big] \cdot v \mu,
			\end{align*}
			the one associated with the ``Boussinesq'' eigenvalue writes for any $\varphi \in \nul(\LL)^\perp$
			$$
			\PP^{(1)}_\Bou(\omega) \varphi = \left \langle \mathsf{R}_{0} \varphi, v \psi_\Bou \right \rangle_{\Ss} \psi_\Bou= \langle \varphi, \LL^{-1} \BurB \rangle_{\Ss} \psi_\Bou,
			$$ 
			and the ones associated with the ``waves'' eigenvalues write, for $\varphi\in \nul(\LL)^{\perp}$,
			\begin{align*}
				\PP^{(1)}_{\pm \Wave}(\omega) \varphi&= \left \langle v \mathsf{R}_{0}  \varphi, \psi_{\pm \Wave}(\omega) \right \rangle_{\Ss} \psi_{\pm \Wave}(\omega)\\
				&= \left(\pm \frac{1}{\sqrt{2}  } \left\langle \varphi , \LL^{-1}  A(v) \omega \right\rangle_{\Ss} + E \sqrt{\frac{K-1}{2}} \left\langle \varphi, \LL^{-1}  B(v) \right\rangle_{\Ss} \right) \psi_{\pm \Wave}(\omega).
			\end{align*}
			This concludes the proof.
		\end{proof}
		We recall now the following \emph{hypocoercivity} result extracted from \cite{D2011}:
		\begin{lem}[\textit{\textbf{Hypocoercivity  \cite[Lemma 4.1]{D2011}}}]\label{lem:lemhypocoercivity}
			Assume that $\LL\::\:\dom(\LL) \subset \Ss \to \Ss$ satisfies Assumptions \ref{L1}--\ref{L4}. Then, for any $\xi \in\R^{d}$, there exists some~$\xi$-dependent bilinear symmetric form $\Phi_\xi[\cdot, \cdot]\::\:\Ss\times \Ss \to \R^{d}$ defined through
			$$\Phi_\xi\left[f, f\right] = \frac{\xi}{1 + | \xi |^2} \cdot \Big\la \PP f, \, \TT_1 \PP f + \TT_2 \PP v \left( \Id - \PP \right) f \Big\ra_{\Ss},$$
			where $\TT_1 \in \BBB\left( \nul(\LL) ; \R^d \right)$ and $\TT_2 \in \BBB\Big( \big(\nul(\LL)\big)^d ; \R^d \Big)$,
			such that there holds for some $c > 0$
			\begin{equation}
				\label{eq:hypocoercivity}
				\Phi_\xi\Big[\LL_{\xi} f, f\Big] \le - \frac{ c | \xi |^2 }{1 + | \xi |^2 } \| \PP f \|^2_{\Ss} + \frac{1}{c} \| (\Id - \PP) f \|_{\Ss}^2,
			\end{equation}
			uniformly in $\xi \in \R^d$ and $f \in \dom(\LL_\xi)$.
		\end{lem}
		With this at hands, we may turn to the proof of the decay estimates of Theorem \ref{thm:spectral_study}. 
		\begin{lem}[\textit{\textbf{Resolvent bounds and decay estimates of the semigroup}}]  
			\label{lem:spectral_decay} With the notations of Lemma \ref{lem:localization_spectrum}, let $0 < \lambda < \lambda_\LL$. 
			There exist  some constants $C, \gamma > 0$ such that, for any $|\xi| \ge \alpha_{0}$, the spectrum is localized as follows:
			$$\mathfrak{S}_{\Ss}(\LL_\xi) \cap \Delta_{-\gamma} = \varnothing,$$
			and the resolvent satisfies
			$$ \sup_{z \in \Delta_{-\gamma}} \| \RR(z,\LL_{\xi}) \|_{\BBB(\Ss)} \le C.$$
			Furthermore, for any $0 < \sigma < \sigma_0 := \min\{\lambda, \gamma\}$, the decay estimates
			$$
			\sup_{t \ge 0} \,e^{2 \sigma_0 t} \left\| U_{\xi}(t) {\left( \Id - \PP(\xi) \right)} f \right\|^2_{\Ss} + \int_0^\infty e^{2 \sigma t} \left\| U_{\xi}(t) {\left( \Id - \PP(\xi) \right)}f \right\|_{\Ssp}^2 \, \d t \le C_{\sigma}\| \left( \Id - \PP(\xi) \right)f \|^2_{\Ss}$$
			and
			$$\int_{0}^{\infty} e^{2\sigma t} \left\|U_{\xi}(t) {\left(\Id-\PP(\xi)\right)}f\right\|_{\Ss}^{2}\d t \le C_{\sigma}\|{\left(\Id-\PP(\xi)\right)}f\|_{\Ssm}^{2}$$
			hold uniformly in $\xi \in \R^{d}$ and $f \in {\Ss}$, where $\left(U_{\xi}(t)\right)_{t\geq0}$ denotes the $C^{0}$-semigroup in $\Ss$ generated by $(\LL_{\xi},\dom(\LL_{\xi}))$. 
			%Moreover
			%$$\int_{0}^{\infty}e^{2\sigma t}\left\|U_{\xi}(t)\textcolor{blue}{\left(\Id-\PP(\xi)\right)}f\right\|_{\Ss}^{2}\d t \le C_{\sigma}\|\textcolor{blue}{\left(\Id-\PP(\xi)\right)}f\|_{\Ssm}^{2}$$	
		\end{lem}
		
		\begin{proof}
			In a first step, we prove resolvent bounds using the above hypocoercivity result  as well as the uniform decay estimate in $\Ss$. In a second step, we prove the $\Ssp-\SS$ integral decay estimate using an energy method, from which we deduce the $\Ss-\Ssm$ one in a third step. Let us fix $0<\lambda < \lambda_\LL$. 
			
			\step{1}{Resolvent bounds and uniform decay estimate} Using the above Lemma \ref{lem:lemhypocoercivity}, we define, for any $|\xi| \ge \alpha_{0}$  the equivalent inner product on $\Ss$
			$$( \cdot , \cdot )_{\Ss, \xi} := \la \cdot, \cdot \ra_{\Ss} + \eta \Phi_\xi[\cdot, \cdot]$$
			with some small $\eta > 0$. By combining the control of $(\Id - \PP)f$ from \ref{L3} and the control of $\PP f$ from \eqref{eq:hypocoercivity}, we have for any $| \xi | \ge \alpha_0$
			\begin{align*}
				\left( \LL_\xi f, f \right)_{\Ss, \xi} & \le \left(\frac{\eta}{c} - \lambda_\LL\right) \| (\id - \PP) f \|_{\Ss}^2 - \frac{c  |\xi|^2}{1 +  |\xi| ^2} \| \PP f \|^2_{\Ss} \\
				& \le -\min\left\{ \frac{\lambda_\LL}{2}, \frac{c  \alpha_{0} ^2 }{1 + \alpha_{0} ^2} \right\} \| f \|^2_{\Ss},
			\end{align*}
			where we chose $\eta \le \frac{c}{2} \lambda_\LL$. Assuming $\eta$ small enough so that the norm induced by~$( \cdot, \cdot )_{\Ss, \xi}$ is equivalent to $\| \cdot \|_{\Ss}$ uniformly in $\xi \in \R^d$, we have for some $\gamma > 0$
			\begin{equation*}
				( \LL_\xi f, f )_{\Ss, \xi} \le -\gamma \| f \|_{\Ss}^2, \qquad 
				\forall | \xi | \ge \alpha_0, ~ f \in \dom(\LL_{\xi})\,. 		\end{equation*}
			We thus deduce that for $| \xi | \ge \alpha_{0}$
			\begin{equation*}
				\| U_{\xi}(t) \|_{\BBB(\Ss)} \lesssim e^{-\gamma t},
			\end{equation*}
			as well as (up to a reduction  of $\gamma$ for the resolvent bound)
			\begin{equation*}
				\mathfrak{S}(\LL_{\xi}) \cap \Delta_{-\gamma} = \varnothing, \quad \sup_{z \in \Delta_{-\gamma}} \| \RR(z,\LL_{\xi}) \|_{\BBB(\Ss)} \lesssim 1.
			\end{equation*}
			On the other hand, for $|\xi| \le \alpha_{0}$, the resolvent of $(\Id - \PP(\xi)) \LL_\xi$ is given for $z \in \Delta_{-\lambda}$ by
			$$\RR(z,\LL_{\xi}(\Id-\PP(\xi)))=\RR(z,\LL_{\xi}) - \sum_{\star = \Inc, \Bou, \pm \Wave}(z-\lambda_{\star}(\xi))^{-1} \PP_\star(\xi) $$
			which is therefore holomorphic in $z \in \Delta_{-\lambda}$ and thus can be explicitly bounded using the bound \eqref{eq:bound_L_xi} and the maximum principle. We deduce that the semigroup it generates is bounded by the Gearhart-Pruss theorem \cite[Theorem V.1.10]{engel}, i.e.
			\begin{equation*}
				\left\| U_{\xi}(t)\left(\Id-\PP(\xi)\right) f \right\|_{\BBB(\Ss)} \lesssim e^{-\lambda t} \| f \|_{\Ss} \qquad \quad
				\forall | \xi | \le \alpha_0, ~ f \in \Ss\,.
			\end{equation*}
			To sum up, putting together both decay estimates and denoting $\sigma_0 := \min\{ \lambda, \gamma\}$, there holds
			\begin{equation*}
				\sup_{\xi \in \R^d}\left\| U_{\xi}(t)\left(\Id-\PP(\xi)\right)f \right\|_{\BBB(\Ss)} \lesssim e^{-\sigma_0 t} \| f \|_{\Ss}, \qquad \forall f \in H\,,
			\end{equation*}
			where we recall that $\PP(\xi) = 0$ for $|\xi| > \alpha_0$. This concludes this step.
			
			\step{2}{Integral decay estimates}
			Let us prove both integral decay estimates.
			
			\step{2a}{The $\Ssp-\Ss$-integral decay estimate}
			Let $f \in \range (\Id - \PP(\xi)) {\cap \dom(\LL_{\xi})}$ and denote~$f(t) =U_{\xi}(t)f$ the unique solution to
			$$
			\begin{cases}
				\partial_{t} f(t)=\LL_{\xi}f(t)=\LL_{\xi}\left(\Id-\PP(\xi)\right)f(t),\\
				f(0) = f.
			\end{cases}
			$$
			Using that $\LL$ is self-adjoint in $\Ss$ and that the multiplication by $i v \cdot \xi$ is skew-adjoint, we have the energy estimate
			$$\frac{\d}{\d t} \| f(t) \|^2_{\Ss} = 2 \Re \lla \LL_\xi f(t) , f(t) \rra_{\Ss} = 2 \lla \LL f(t) , f(t) \rra_{\Ss}.$$
			Furthermore, using the dissipativity estimate of $\LL$ from \ref{L3}, we get
			$$\frac{\d}{\d t} \| f(t) \|^2_{\Ss} + 2 \lambda_\LL \| (\Id - \PP) f(t) \|^2_{\Ssp} \le 0,$$
			which we complete using \eqref{eq:degenerate_pythagora_X_star} and $\PP \in \BBB(\Ss; \Ssp)$ as
			\begin{equation}
				\label{eq:EnergyHH_Hh}
				\frac{\d}{\d t} \| f(t) \|^2_{\Ss} + \lambda_\LL \| f(t) \|^2_{\Ssp}   \lesssim \| \PP f(t) \|_{\Ssp}^2 \lesssim \| f(t) \|_{\Ss}^2 \lesssim e^{-2 \sigma_0 t} \| f \|_{\Ss}^2,
			\end{equation}
			where we also used the decay estimate established in the previous step. Multiplying this by~$e^{2 \sigma t}$ and integrating, one easily deduces
			%	\begin{align*}
			%		\frac{\d}{\d t} \Big( e^{2 \sigma t} \| f(t) \|^2_{\Ss} \Big) + \lambda_\LL \, e^{2 \sigma t } \| f(t) \|^2_{\Ssp} & \lesssim \lambda_\LL e^{-2 (\sigma_0 - \sigma) t} \| f \|_{\Ss}^2 + \sigma e^{2 \sigma t} \| f(t) \|_{\Ss}^2 \\
			%		& \lesssim (\lambda_\LL + \sigma) e^{-2 (\sigma_0 - \sigma) t} \| f \|^2_{\Ss},
			%	\end{align*}
			%	which once integrated yields
			\begin{equation*}
				\sup_{t \ge 0} e^{2 \sigma t} \| f(t) \|_{\Ss}^2 + \lambda_\LL \int_0^\infty e^{2 \sigma t} \| f(t) \|_{\Ssp}^2 \d t \lesssim \| f \|^2_{\Ss}.
			\end{equation*}
			This identity holds for any $f \in \left(\Id-\PP(\xi)\right) \cap \dom(\LL_{\xi})$ and we conclude the proof using that~$\dom(\LL_{\xi})$ is dense in $\Ss$.
			
			\step{2b}{The $\Ss-\Ssm$-integral decay estimate} As before, we assume now $f \in \left(\Id-\PP(\xi)\right) \cap \Ssm$ and set $f(t)=U_{\xi}(t)f$ for any $t\geq0.$
			We use a duality argument together with a density argument and prove that
			$$\la e^{\sigma t} f(t) , \phi \ra_{L^2_t(\Ss)} \lesssim \| f\|_{\Ssm} \| \phi \|_{L^2_t(\Ss)}.$$
			To perform the duality argument, we need to point out that $\LL$ is self-adjoint, thus
			$$\left(U_{\xi}(t) \left( \Id - \PP(\xi) \right) \right)^\star = U_{-\xi}(t) \left( \Id - \PP(-\xi) \right).$$
			Since the step functions span a dense subspace of $L^2_t(\Ss)$, it is enough to check that the dual estimate holds for such a following function:
			\begin{gather*}
				\phi(t) = \begin{cases}
					\phi_0 \in \Ss, & t \in [t_1 , t_2] ,\\
					0, & t \notin [t_1, t_2],
				\end{cases}, \quad \| \phi \|_{L^2_t(\Ss)} = \sqrt{t_2 - t_1} \| \phi_0 \|_{\Ss}.
			\end{gather*}
			For such a step function, the inner product writes explicitly as
			\begin{align*}
				\la e^{\sigma t} f(t) , \phi \ra_{L^2_t(\Ss)}  = \int_0^\infty \big\la e^{\sigma t} f(t) , \phi(t) \big\ra_{\Ss} \d t= \int_{t_1 }^{ t_2 } \left\la e^{\sigma t} U_{\xi}(t)\left( \Id - \PP(\xi) \right) f , \phi_0 \right\ra_{\Ss} \d t.
			\end{align*}
			We then get by duality and using Cauchy-Schwarz's inequality
			\begin{equation*}\begin{split}
					\la e^{\sigma t} f(t) , \phi \ra_{L^2_t(\Ss)}  & = \int_{t_1 }^{ t_2 } \big\la f , e^{\sigma t} U_{-\xi}(t) \left( \Id - \PP(-\xi) \right) \phi_0 \big\ra_{\Ss} \d t \\
					& \le \| f \|_{\Ssm} \, \int_{ t_1 }^{ t_2 } e^{\sigma t} \| U_{-\xi}(t) \left( \Id - \PP(-\xi) \right) \phi_0 \|_{\Ssp} \d t \\
					& \le \| f \|_{\Ssm} \, \sqrt{t_2-t_1}  \left(\int_0^\infty e^{2 \sigma t} \| U_{-\xi}(t) \left( \Id - \PP(-\xi) \right) \phi_0 \|_{\Ssp}^2 \d t\right)^{\frac{1}{2}}
			\end{split}\end{equation*}
			thus, using the $\Ssp-\Ss$ estimate from \textit{Step 2a}
			\begin{align*}
				\la e^{\sigma t} f(t) , \phi \ra_{L^2_t(\Ss)} & \lesssim \| f  \|_{\Ssm} \sqrt{t_2-t_1} \| \phi_0 \|_{\Ss} \lesssim \| f \|_{\Ssm} \| \phi \|_{L^2_t(\Ss)}.
			\end{align*}
			\color{black}
			This concludes the proof.\end{proof}

		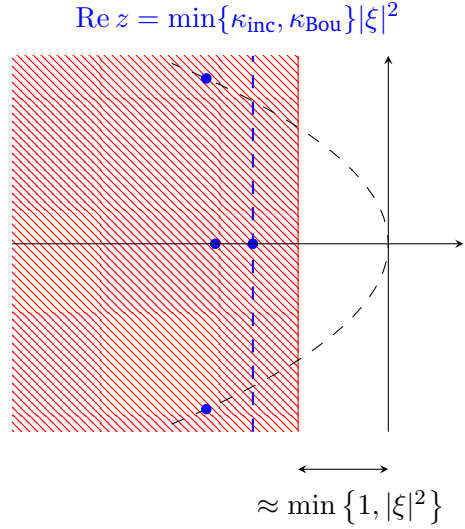
\begin{figure}[h]
			\centering
			
			\begin{tikzpicture}
				% Spectre cinétique
				\fill [pattern=north west lines, pattern color=red] (-5, -2.5) rectangle (-1.2, 2.5);
				\draw [color=red] (-1.2, -2.5) -- (-1.2, 2.5);
				\draw[stealth-stealth] (-1.2, -3) -- (0, -3);
				\node at (-0.5, -3.5) {$\approx \min\left\{ 1, |\xi|^2 \right\}$};
				
				% Seuil hydrodynamique
				\draw [color=blue, dash pattern=on 5pt, thick] (-1.8, -2.5) -- (-1.8, 2.5);
				\node [color=blue] at (-2, 3) {$\Re z = \min\{ \kappa_\Inc, \kappa_\Bou \} | \xi |^2$};
				
				% Valeurs propres ondes
				\draw[black, variable=\y, dash pattern=on 5pt] plot[smooth,domain=-2.4:2.4] ( {-0.5 * (\y)^2 }, \y );
				\fill [color=blue] ( {-0.5 * (2.2)^2} , 2.2) circle (0.07);
				\fill [color=blue] ( {-0.5 * (2.2)^2} , -2.2) circle (0.07);
				
				% Valeur propre Boussinesq
				\fill [color=blue] (-2.3, 0) circle (0.07);
				
				% Valeur propre incompressible
				\fill [color=blue] (-1.8, 0) circle (0.07);
				
				% Repère
				\draw[-stealth] (-5, 0) -- (1, 0);
				\draw[-stealth] (0, -2.5) -- (0, 2.5);
				
			\end{tikzpicture}
			
			\caption{Localization of the spectrum of $\LL_\xi $ provided by the hypocoercivity lemma \ref{lem:lemhypocoercivity} for $\xi \in \R^d$, compared with the localization of the hydrodynamic eigenvalues defined for $|\xi| \le \alpha_0$. }
			
		\end{figure}
		
		\subsection{Proof of Theorem \ref{thm:enlarged_thm}}
		
		We present here the full proof a the ‘‘enlarged'' version of Theorem \ref{thm:spectral_study} as provided in  Theorem \ref{thm:enlarged_thm}. We present in a first step how to extend the resolvent bounds, in a second step how to extend the decay estimate. In a third step, we  extend the projector bounds from Lemma \ref{lem:expansion_projection} as this is enough to deduce the same bounds on the expansion of $\UU_\xi$ and in turn of $\PP_\star(\xi)$.
		
		\step{1}{Resolvent bounds}
		Using the factorization formulae
		\begin{equation}
			\label{eq:shrinkage_Y}
			\begin{split}
				\RR(z,\LL_{\xi}) & = \RR(z,\BB^{(0)}_{\xi}) + \RR(z,\LL_{\xi}) \AA^{(0)} \RR(z,\BB^{(0)}_{\xi}) \\ & =\RR(z,\BB^{(0)}_{\xi})+ \RR(z,\BB^{(0)}_{\xi}) \AA^{(0)} \RR(z,\LL_{\xi})
			\end{split}
		\end{equation}
		and the fact that the function $z \in \Delta_{-\lambda} \mapsto \RR(z,\BB^{(0)}_{\xi}) \in  \BBB(\Ss) \cap \BBB(\Sl)$ is holomorphic, as well as $\Ss \hookrightarrow \Sl$ and $\AA^{(0)}\in \BBB(\Sl;\Ss)$,
		we deduce that, for any $z \in \Delta_{-\lambda}$
		$$ \RR(z,\LL_{\xi}) \in \BBB\left( \Sl \right) \iff \RR(z,\LL_{\xi}) \in \BBB(\Ss)\,,$$
		or in other words, the spectrum of $\LL_\xi$ in $\Delta_{-\lambda}$ does not depend on the space $\Ss$ or $\Sl$:
		$$\mathfrak{S}_{\Ss}(\LL_{\xi}) \cap \Delta_{-\lambda} = \mathfrak{S}_{\Sl}(\LL_\xi) \cap \Delta_{-\lambda}.$$
		More precisely, since $\RR\left( z, \BB^{(0)}_\xi \right) \in \BBB(\Ss) \cap \BBB(\Sl)$ uniformly in $\xi \in \R^d$ and $z \in \Delta_{-\lambda}$, there holds for any $z \in \Delta_{-\lambda} \setminus \mathfrak{S}(\LL_\xi)$
		$$1 + \|\RR(z,\LL_{\xi})\|_{\BBB(\Ss)} \lesssim 1 + \|\RR(z,\LL_{\xi})\|_{\BBB(\Sl)} \lesssim 1 +  \|\RR(z,\LL_{\xi})\|_{\BBB(\Ss)}.$$
		This implies that the bounds in $\BBB(\Ss)$ on the resolvent from Theorem \ref{thm:spectral_study} also hold in~$\BBB\left( \Sl \right)$. Actually, also the bounds \eqref{eq:bound_L_xi} in Lemma \ref{lem:localization_spectrum} can be refined for $|\xi| \le \alpha_{0}$ small enough as
		\begin{equation}
			\label{eq:bound_L_xi_Y}
			\sup_{ | z | = r } \| \RR(z,\LL_{\xi}) \|_{\BBB(\Sl;\Slp)}  + \sup_{ | z | = r } \| \RR(z,\LL_{\xi}) \|_{\BBB(\Slm;\Sl)} 
			+ \sup_{z\in\Omega} \| \RR(z,\LL_{\xi}) \|_{\BBB(\Sl)} \le C_1,
		\end{equation}
		where $r$ is small enough and $\Omega := \Delta_{-\lambda} \cap \{ | z | \ge r \}$.  Indeed, starting from the dissipativity estimate involving $\Slp$ and $\AA^{(0)} \in \BBB(\Sl)$ from \ref{LE}:
		\begin{align*}
			\Re \left \la (\LL_\xi - z) f, f \right\ra_{\Sl} & = \Re \left \la \BB^{(0)} f, f \right\ra_{\Sl} + \Re \left \la \AA^{(0)} f, f \right\ra_{\Sl} - z \| f \|^2_{\Sl} \\
			& \le - \lambda_\BB \| f \|^2_{\Slp} + \left(\| \AA^{(0)}\|_{\BBB(\Sl)} - z \right)\| f \|_{\Sl}^2
		\end{align*}
		which gives for some $z_0 \ge \| \AA^{(0)} \|_{\BBB(\Sl)}$
		$$\| \RR(z_{0},\LL_{\xi})\|_{\BBB(\Sl,\Slp)} \le \lambda_\BB^{-1}.$$
		Performing the same computations with the decomposition
		$$\LL^\star_{\xi} = (\BB^{(0)})^{\star} + i v \cdot \xi + (\AA^{(0)})^{\star},$$
		where $(\BB^{(0)})^{\star}$ satisfies the same dissipativity as $\BB^{(0)}$, we get $\|\RR(z_{0},\LL_{\xi})^{\star} \|_{\BBB(\Sl;\Slp)} \le \lambda_\BB^{-1}$ and thus by the adjoint identity \eqref{eq:twisted_adjoint_identity_one_sided}
		$$\| \RR(z_{0},\LL_{\xi})\|_{\BBB(\Slm;\Sl)} \le \lambda_\BB^{-1}.$$ Using the resolvent identity as in the proof of Lemma \ref{lem:localization_spectrum}, we deduce \eqref{eq:bound_L_xi_Y}.
		
		\step{2}{Decay estimates}
		We first prove the uniform estimates, and then the integral ones.
		
		\step{2a}{The uniform decay estimate}
		To improve the decay estimate of \eqref{eq:decay} to \eqref{eq:decay-EE}, we apply the Duhamel formula to the decomposition $\LL_\xi = \BB_{\xi}^{(0)}+ \AA^{(0)}$
		$$U_{\xi}(t) = V_{\xi}^{(0)}(t) + \int_0^t U_{\xi}(t-\tau)\AA^{(0)} V_{\xi}^{(0)}(\tau)\d \tau,$$
		where $\left(V_{\xi}^{(0)}(t)\right)_{t\geq0}$ denotes the $C^{0}$-semigroup in $\Sl$ generated by $(\BB^{(0)}_{\xi},\dom(\LL_{\xi}))$. After 
		%	$$U_{\xi}(t) = V_{\xi}^{(0)}(t) + \int_0^t \mathcal{V}_{\xi}(t-\tau)\AA^{(0)} V_{\xi}^{(0)}(\tau)\d \tau.$$
		composing with $\Id - \PP(\xi)$ from the left, we get
		\begin{align*}
			\left( \Id - \PP(\xi) \right)U_{\xi}(t) = & \left( \Id - \PP(\xi) \right)V_{\xi}^{(0)}(t)  + \int_0^t \left( \Id - \PP(\xi) \right) U_{\xi}(t-\tau)\AA^{(0)} V_{\xi}^{(0)}(\tau) \d \tau.
		\end{align*}
		Since $\| \PP(\xi) \|_{\BBB(\Sl)} \lesssim 1$ from \eqref{eq:bound_L_xi_Y}, and using $\AA^{(0)} \in \BBB( \Sl ; \Ss )$, we have
		\begin{align*}
			\big\| ( \Id  - \PP(\xi) ) U_{\xi}(t) \big\|_{\BBB(\Sl)} 
			\lesssim   \left\|V_{\xi}^{(0)}(t)\right\|_{\BBB(\Sl)} 
			+ \int_0^t \left\| \left( \Id - \PP(\xi) \right) U_{\xi}(t - \tau) \right\|_{\BBB(\Ss)} \left\| V_{\xi}^{(0)}(\tau)\right\|_{\BBB(\Sl)} \d \tau,
		\end{align*}
		and using the decay estimate of $\left( \Id - \PP(\xi) \right)U_{\xi}(t)$ in $\BBB(\Ss)$ from Theorem \ref{thm:spectral_study}, as well as the dissipativity hypothesis for $\BB^{(0)}_{\xi}$ from \ref{LE}, we then get (recall that $\sigma_0 \le \lambda$)
		\begin{equation*}
			\left\| \left( \Id - \PP(\xi) \right) U_{\xi}(t)\right\|_{\BBB(\Sl)} \lesssim e^{-\lambda\,t}
			+ \int_0^t e^{-\sigma_0 (t - \tau)} e^{-\lambda \tau} \d \tau \lesssim e^{-\sigma_0 t}.
		\end{equation*}
		This proves the uniform in time decay. 
		
		\step{2b}{The integral $\Slp-\Sl$ and $\Sl-\Slm$ decay estimates}
		From then on, the proof of the integral decay estimate follows the same strategy as the one adopted for the proof of Lemma \ref{lem:spectral_decay}, starting from the decomposition $\LL = \BB^{(0)} + \AA^{(0)}$ and, resuming the computations of Lemma \ref{lem:spectral_decay}. Typically, estimate \eqref{eq:EnergyHH_Hh} can be adapted to give now
		\begin{align*}
			\frac{\d}{\d t} \| f(t) \|^2_{\Sl} + 2 \lambda_\BB \| f(t) \|^2_{\Slp} & \lesssim \| \AA^{(0)} \|_{\BBB(\Sl)} \| f(t) \|_{\Sl}^2 \lesssim \| \AA^{(0)} \|_{\BBB(\Sl)} e^{-2 \sigma_0 t} \| f \|^2_{\Sl}.
		\end{align*}
		After integration, one obtains easily \eqref{eq:decay-EE} as in Lemma \ref{lem:spectral_decay}, as well as the corresponding estimate for $\Big(U_\xi(t) (\Id - \PP(\xi)) \Big)^\star$, from which we deduce \eqref{eq:decay_Ee'} by a similar duality argument.
		
		\step{3}{Expansion of the projectors}
		To establish  the uniform bounds in $\BBB(\Slm ;\Ssp)$ on the expansion of the spectral projectors
		$$\PP_\star(\xi) = \PP^{(0)}_\star\left( \widetilde{\xi} \right) + i \xi \cdot \PP^{(1)}_\star\left( \widetilde{\xi} \right) + S_{\star}\left(  \xi \right),$$
		we use a similar bootstrap strategy as in the proofs of Lemma \ref{lem:expansion_projection} and Theorem~\ref{thm:regularized_thm}. More precisely, in each step, we will prove uniform bounds in $\BBB\left(\Slm ; \Sl \right)$ and then combine with those in $\BBB\left( \Ss ; \Ssp\right)$ from Lemma \ref{lem:expansion_projection} to conclude. Similarly, we will need the following regularization properties for $\PP(\xi)$:
		\begin{equation}
			\label{eq:regularization_P_Y}
			\| \PP(\xi) \|_{ \BBB(\Slm ; \Sl_2 ) } + \| \PP(\xi) \|_{ \BBB(\Slm ; \Ssp ) }  \lesssim 1,
		\end{equation}
		which, as in the original proof, comes from combining the identity $\PP(\xi)^2 = \PP(\xi)$ with the resolvent bound \eqref{eq:bound_L_xi_Y} (for the $\Slm-\Sl$ bound), with Lemma \ref{lem:expansion_projection} (for the $\Ss-\Ssp$ bound), and with the regularization hypothesis \ref{assumption_large_bounded_A} for $\AA^{(0)}$ (for the $\Sl-\Sl_2$ and $\Sl-\Ss$ bounds) applied to the representation
		\begin{gather*}
			\PP(\xi) = \frac{1}{2 i \pi} \oint_{ | z | = r } \left( \RR\left( z, \BB^{(0)}_\xi \right) \AA^{(0)} \right)^2 \RR\left( z , \LL_\xi \right) \d z.
		\end{gather*}
		Similarly, we will need the regularization properties
		\begin{equation}
			\label{eq:regularization_P_Y_adjoint}
			\| \PP(\xi)^\star \|_{ \BBB(\Slm ; \Sl_j ) } \lesssim 1, \quad j = 1, 2,
		\end{equation}
		which comes from the regularization hypothesis \ref{assumption_large_bounded_A} for $\left( \AA^{(0)} \right)^\star$ applied to the representation
		$$\PP(\xi)^\star = \frac{1}{2 i \pi} \oint_{ | z | = r } \left[ \RR\left( z, \BB^{(0)}_\xi \right)^\star \left(\AA^{(0)}\right)^\star \right]^j \left(\RR\left( z , \LL_\xi \right)\right)^\star \d z.$$
		Finally, we will the need the resolvent bounds
		\begin{equation}
			\label{eq:bound_resolvent_Y_j}  \| \RR(z,\LL) \|_{\BBB(\Sl_{j})} \lesssim 1 + \frac{1}{|z|}, \quad j = 0,1,2, \qquad z \in \Delta_{-\lambda} \setminus \{0\}, 
		\end{equation}
		which also come from a factorization strategy as in the original proof.
		
		\step{3a}{First order expansion of $\PP(\xi)$}
		We use the the bootstrap formula \eqref{eq:bootstrap_P_1_A}:
		$$\PP^{(1)}(\xi) = \PP(\xi) \PP^{(1)}(\xi) + \PP^{(1)}(\xi) \PP,$$
		and  the adjoint identity $\| T \|_{\BBB(\Slm;\Sl)} = \| T^\star \|_{\BBB(\Sl;\Slp)}$ from \eqref{eq:twisted_adjoint_identity_one_sided} to establish the bound
		\begin{align*}
			\| \PP^{(1)}(\xi) \|_{\BBB(\Slm;\Sl)}
			&\le \| \PP^{(1)}(\xi)^\star \PP(\xi)^\star \|_{\BBB(\Sl;\Slp)} + \| \PP^{(1)}(\xi) \PP \|_{\BBB(\Slm;\Sl)} \\
			&\le  \| \PP^{(1)}(\xi)^\star \|_{\BBB(\Sl_1;\Slp)} \| \PP(\xi)^\star \|_{\BBB(\Sl;\Sl_{1})} +  \| \PP^{(1)}(\xi) \|_{\BBB(\Sl_1;\Sl)} \| \PP \|_{\BBB(\Slm;\Sl_1)}\\
			&\lesssim \| \PP^{(1)}(\xi)^\star \|_{\BBB(\Sl_1;\Slp)} +  \| \PP^{(1)}(\xi) \|_{\BBB(\Sl_1;\Sl)}
		\end{align*}
		where we used the the regularization property \eqref{eq:regularization_P_Y} for $\PP=\PP(0)$ and \eqref{eq:regularization_P_Y_adjoint} for $\PP^\star(\xi)$ in the last estimate. We now turn to the first term $\| \PP^{(1)}(\xi)^\star \|_{\BBB(\Sl_{1};\Slp)}$:
		\begin{align*}
			\| \PP^{(1)}(\xi)^\star \|_{\BBB(\Sl_{1};\Slp)} & \le \frac{1}{2 \pi} \oint_{|z|=r} \left\| \Big(\RR(z,\LL) v \RR(z,\LL_{\xi})\Big)^\star \right\|_{\BBB(\Sl_{1};\Slp) } \d |z| \\
			&\le\frac{1}{2 \pi} \oint_{|z|=r} \left\| \RR(z,\LL_{\xi})^\star v^{\star}\RR(z,\LL)^\star  \right\|_{\BBB(\Sl_{1};\Slp) } \d |z| \\
			& \lesssim \oint_{|z|=r} \| \RR(z,\LL_{\xi})^\star \|_{\BBB(\Sl;\Slp)} \| \RR(z,\LL)^\star \|_{\BBB(\Sl_{1}) } \d |z|,
		\end{align*}
		where we used the fact that the adjoint of the multiplication by $v$ is in $\BBB(\Sl_1 ; \Sl)$ according to \ref{assumption_large_multi-v}. We can rewrite this estimate without the adjoints and estimate it using the resolvent bounds \eqref{eq:bound_L_xi_Y} and \eqref{eq:bound_resolvent_Y_j}:
		\begin{align*}
			\| \PP^{(1)}(\xi)^\star \|_{\BBB(\Sl_{1};\Slp)} \lesssim \oint_{|z|=r} \| \RR(z,\LL_{\xi}) \|_{\BBB(\Slm;\Sl) } \| \RR(z,\LL) \|_{\BBB(\Sl_{1})} \d |z|
			\lesssim 1.
		\end{align*}
		The second term $\| \PP^{(1)}(\xi)\|_{\BBB(\Sl_1;\Sl)} \lesssim 1$ is estimated in the same way, thus we obtain
		$$\|\PP^{(1)}(\xi)\|_{\BBB(\Slm;\Sl)} \lesssim 1.$$
		Integrating the formula \eqref{eq:shrinkage_Y} and combining with the $\Ss-\Ssp$ resolvent bound \eqref{eq:bound_L_xi}, one proves the estimate
		$$\|\PP(\xi)\|_{\BBB(\Sl;\Ssp)} \lesssim 1,$$
		which then allows to perform another simpler (duality-free) bootstrap argument by combining the above estimates with the bounds of Lemma \ref{lem:expansion_projection}:
		\begin{align*}
			\| \PP^{(1)}(\xi) \|_{ \BBB(\Slm ; \Ssp ) } & \le \| \PP(\xi) \PP^{(1)}(\xi) \|_{ \BBB(\Slm ; \Ssp ) } + \| \PP^{(1)}(\xi) \PP \|_{ \BBB(\Slm ; \Ssp ) } \\
			& \le \| \PP(\xi) \|_{ \BBB(\Sl ; \Ssp ) } \| \PP^{(1)}(\xi) \|_{ \BBB(\Slm ; \Sl ) }  + \| \PP^{(1)}(\xi) \|_{ \BBB(\Sl ; \Ssp ) } \| \PP \|_{ \BBB(\Slm ; \Sl ) } \\
			& \lesssim 1.
		\end{align*}
		This concludes this step.
		
		\step{3b}{Second order expansion}
		We use this time the bootstrap formula \eqref{eq:bootstrap_P_2_A} and the first order estimates, together with the duality identity \eqref{eq:twisted_adjoint_identity_one_sided}
		\begin{align*}
			\| \PP^{(2)}(\xi) \|_{\BBB(\Slm;\Sl)} & \lesssim  1 
			+ \| \PP \PP^{(2)}(\xi) \|_{\BBB(\Slm;\Sl) } + \| \PP^{(2)}(\xi) \PP(\xi) \|_{\BBB(\Slm;\Sl)} \\
			& \lesssim 1 
			+ \| \PP^{(2)}(\xi)^\star \PP^\star \|_{\BBB(\Sl;\Slp) } + \| \PP^{(2)}(\xi) \PP(\xi) \|_{\BBB(\Slm;\Sl)}.
		\end{align*}
		Using the regularization estimate \eqref{eq:regularization_P_Y} on $\PP$, we obtain
		\begin{align*}
			\| \PP^{(2)}(\xi) \|_{\BBB(\Slm;\Sl)} \lesssim 1 & + \| \PP^{(2)}(\xi)^\star \|_{\BBB(\Sl_2;\Slp)} \| \PP^\star \|_{\BBB(\Sl;\Sl_{2})}  + \| \PP^{(2)}(\xi) \|_{\BBB(\Sl_{2};\Sl) } \| \PP \|_{\BBB(\Slm;\Sl_2) }  \\
			\lesssim 1 & + \| \PP^{(2)}(\xi)^\star \|_{\BBB(\Sl_2;\Slp)} + \| \PP^{(2)}(\xi) \|_{\BBB(\Sl_{2};\Sl) }.
		\end{align*}
		We conclude as in the previous step that $\| \PP^{(2)}(\xi) \|_{\BBB(\Slm;\Sl)} \lesssim 1$, and then perform a second bootstrap to deduce $\| \PP^{(2)}(\xi) \|_{\BBB(\Slm;\Ssp)} \lesssim 1$ from the estimates of Lemma \ref{lem:expansion_projection}. This concludes the proof.
		%\end{proof}
		
		\begin{rem}
			Notice that, since 
			$$\left\|\left[\left(\Id-\PP(\xi)\right)U_{\xi}(t)\right]^{\star}\right\|_{\BBB(\Sl)}=\left\|\left(\Id-\PP(\xi)\right)U_{\xi}(t)\right\|_{\BBB(\Sl)}$$
			the decay estimates \eqref{decay-semigroup-EE} extends easily to the adjoint $ U_{\xi}(t)^{\star} {\left( \Id - \PP(\xi) \right)^{\star}}$.
		\end{rem}
		
		\subsection{Regularized version of the spectral result}
		\label{sec:LR}
		We present here yet another improved version of Theorem \ref{thm:spectral_study}, taking now advantage of possible alternative splittings of the linearized operator $\LL$. In order to prove a ``regularized'' version of our main result, we will need the following extra assumption.
		\begin{enumerate}[label=\hypst{LR}]
			\item \label{LR} 
			Besides Assumptions \ref{L1}--\ref{L4},  assume that the operator can be decomposed in a way $\LL = \BB^{(1)} + \AA^{(1)}$ compatible with a hierarchy of \textbf{Banach} spaces $\left(\Sr_j\right)_{j=-\ell}^{2}$, where $\ell \ge 0$, such that
			\begin{enumerate}
				\item the spaces $\Sr_j$ embed into one another and the regular space embeds into the original space:
				$$\Sr_2 \hookrightarrow \Sr_1 \hookrightarrow \Sr = \Sr_{0} \hookrightarrow \Sr_{-1} \hookrightarrow \dots \hookrightarrow \Sr_{-\ell}, \quad \Sr \hookrightarrow \Ss,$$
				\item the multiplication by $v$ is bounded from $\Sr_{j}$ to $\Sr_{j-1}$ for some $j$:
				$$\|vf\|_{\Sr_{j}} \lesssim \|f\|_{\Sr_{j+1}}, \quad \quad j=0,1,$$
				\item \label{ass:reg_A_reg} the operator $\AA^{(1)}$ is bounded from $\Sr_{j}$ to $\Sr_{j+1}$ and from $\Ss$ to $\Sr_{-\ell}$:
				$$\AA^{(1)} \in \BBB(\Sr_{j};\Sr_{j+1}) \cap \BBB(\Ss;\Sr_{-\ell}), \qquad j=1,\ldots,-\ell,$$
				\item \label{ass:reg_diss} the part $\BB^{(1)}_{\xi}$ is dissipative on $\Sg = \Sr_{-\ell}, \dots, \Sr_2, \Ss$ in the sense that
				$$\mathfrak{S}_{\Sg}(\BB^{(1)}_{\xi}) \cap \Delta_{-\lambda_{\BB}} = \varnothing, \qquad \sup_{\xi \in \R^d}\left\|\RR(z,\BB^{(1)}_{\xi})\right\|_{\BBB(\Sg)} \lesssim | \re z + \lambda_\BB|^{-1},$$
				uniformly in $z \in \Delta_{-\lambda_{\BB}}$.
			\end{enumerate}	
		\end{enumerate}
		Under this new set of Assumptions, we derive the following version of Theorem \ref{thm:spectral_study}:	
		\begin{theo}[\textit{\textbf{Regularized result}}]
			\label{thm:regularized_thm}
			If Assumptions \ref{LR} are in force, then the spectral projectors from Theorem \ref{thm:spectral_study} are regularizing in the sense that in the decomposition \eqref{eq:PPstar}
			\begin{equation*}
				\PP_\star(\xi) = \PP^{(0)}_\star\left( \widetilde{\xi} \right) + i \xi \cdot \PP_\star^{(1)}\left( \widetilde{\xi} \right)  + \mathsf{S}_{\star}(\xi),
			\end{equation*}
			each term belongs to $\BBB\left(\Ssm ; \Sr \right)$ uniformly in $|\xi| \le \alpha_{0}$, and $\| \mathsf{S}_\star(\xi) \|_{ \BBB(\Ssm ; \Sr) } \lesssim | \xi |^2$.
		\end{theo}
		
		\begin{rem}
			Once again, we illustrate this set of assumptions in the case of the Boltzmann equation for hard spheres. In this context the hierarchy of spaces can be taken to be
			$$\Sr_j = L^\infty\left( \mu^{-1/2} \la v \ra^{j + \ell} \d v \right), \qquad j=2,\ldots,-\ell\,,$$
			for some integer $\ell > \frac{d}{2}$, and the splitting is also Grad's splitting (see Remark \ref{rem:L1B1}).
		\end{rem}
		
		\begin{proof}[Proof of Theorem \ref{thm:regularized_thm}]
			Let us prove that the coefficients of the expansion
			$$\PP(\xi) = \PP + \xi \cdot \PP^{(1)} + \xi \otimes \xi : \PP^{(2)}(\xi)$$
			belong to $\BBB(\Ssm;\Sr)$ uniformly in $\xi$ small enough. As pointed out in the proof of Theorem \ref{thm:enlarged_thm}, this will be enough to deduce it also holds for $\PP_\star(\xi)$.
			
			\step{1}{Estimate for the resolvent in the regular space $\Sr$}
			Starting from the factorization formula
			\begin{equation}
				\label{eq:factorization_L_xi_reg}
				\RR(z,\LL_{\xi}) = \sum_{n=0}^{\ell-1} \left( \RR\left(z,\BB^{(1)}_{\xi} \right) \AA^{(1)} \right)^n \RR\left(z,\BB_{\xi}^{(1)}\right) + \left(\RR\left(z,\BB^{(1)}_{\xi} \right) \AA^{(1)} \right)^{\ell} \RR(z,\LL_{\xi}),
			\end{equation}
			one gets from \eqref{eq:bound_L_xi}, the embedding $\Sr \hookrightarrow \Ss$, as well as the bounds \ref{ass:reg_diss} on $\RR\left(z,\BB_{\xi}^{(1)}\right)$ and the regularization hypothesis \ref{ass:reg_A_reg} for $\AA^{{(1)}}$ that, for any $0 < \lambda < \lambda_\LL$, there are some $\alpha_{0},r >0$ small enough such that
			\begin{equation}
				\label{eq:bound_L_xi_reg}
				\sup_{ z \in \Omega } \| \RR(z,\LL_{\xi}) \|_{\BBB(\Sr)} \le C, \qquad |\xi| \le \alpha_0
			\end{equation}
			with, as before, $\Omega=\Delta_{-\lambda}\cap \{ |z|\geq r\}$.
%			One also shows
%			\begin{equation*}
%				\label{eq:bound_resolvent_L_X_j_reg}
%				\| \RR(z,\LL) \|_{\BBB(\Sr_{j})} \lesssim 1 + \frac{1}{|z|}, \quad j = -\ell, \dots, 2,  \quad z \in \Delta_{-\lambda} \setminus \{0\}\,. 
%			\end{equation*}
			
			\step{2}{Behavior of the spectral projector as $\xi \to 0$}
			We use a similar bootstrap strategy. It is simpler because no duality argument is involved, in exchange we replace the use of adjoint operators by estimates in $\Sr_{-1}$ and $\Sr_{-2}$.
			
			The first step is to extend the bound \eqref{eq:bound_L_xi_reg} from $\Sr$ to $\Sr_j$ for $j=-2, \dots, 2$ using similar factorization arguments as in the previous proofs.
			
			The second step is then to deduce, using the bounds \ref{ass:reg_A_reg}--\ref{ass:reg_diss} and the representation formula
			\begin{equation*}
				\PP(\xi) = \frac{1}{2 i \pi} \int_{ | z | = r } \left( \RR\left(z,\BB^{(1)}_{\xi} \right) \AA^{(1)} \right)^{2} \RR(z,\LL_{\xi}) \d z
			\end{equation*}
			the regularization bounds
			$$\| \PP(\xi) \|_{ \BBB\left( \Sr ; \Sr_2 \right)  } + \| \PP(\xi) \|_{ \BBB\left( \Sr_{-2} ; \Sr \right)  } \lesssim 1, \quad | \xi | \le \alpha_0.$$
			From then, we follow a simplified version of the bootstrap procedures used in the proofs of Lemma \ref{lem:expansion_projection} and Theorem \ref{thm:enlarged_thm}. This concludes this proof.
		\end{proof}
		
		\section{Properties of the linearized semigroup in the physical space}
		
		\label{scn:study_physical_space}
		
		If one assumes only \ref{L1}--\ref{L4}, we denote in this section $\Sl = \Ss$. Under the extra assumption \ref{LE}, $\Sl$ is the space from \ref{LE}.
		
		\medskip
		In this section, we exploit the spectral description of the previous Section to study the main properties of the semigroup $(U^\eps(t))_{t\geq0}$.
		We adopt the notations and definitions introduced in Section \ref{sec:detail}. We only recall that 
		$$U^\eps(t)=U^\eps_\kin(t)+ U^\eps_\hyd(t), \qquad U^\eps_\hyd(t)=U^\eps_\ns(t) + U^\eps_\Wave(t)$$
		where the various semigroups are defined in Definitions \ref{def:hydro_semigroups} and \ref{def:kinetic_semigroups}. As explained in the Introduction and in Section \ref{sec:detail}, it is important for the definition of the various stiff terms $\Psi^\eps_\hyd[f,g]$ and $\Psi^\eps_\kin[f,g]$ to study suitable bounds on the semigroups $U^\eps_\kin(t) $ and $U^\eps_\hyd(t) $ as well as their convolution with suitable time-dependent functions.
		
		We begin with   the following uniform estimates on $U^{\eps}_{\hyd}(\cdot) $
		\begin{lem}[\textit{\textbf{Bounds for the hydrodynamic semigroup}}]
			\label{lem:decay_semigroups_hydro}
			For $\star=\ns,\Wave,\disp$, the hydrodynamic semigroups~$U^\eps_\star(\cdot)$ are bounded from $\SSl$ to $\SSSs$:
			\begin{equation}
				\label{eq:decay_semigroups_hydro}
				\Nt U^\eps_\star(\cdot) g \Nt_\SSSs \lesssim \| g \|_{\SSl} +\|g\|_{\dot{\mathbb{H}}_x^{-\alpha}(\Slm)}.
			\end{equation}
			Furthermore, if $\varphi \in L^2\left( [0, T) ; \SSlm \right)$ where $T \in (0, \infty]$ is such that $\PP \varphi(t) = 0$, then for~$\star=\ns, \Wave$ and thus $\star=\hyd$
			\begin{subequations}
				\label{eq:decay_semigroups_hydro_orthogonal}
				\begin{equation}
					\label{eq:decay_semigroups_hydro_orthogonal_negative}
					\frac{1}{\eps} \Nt U^\eps_\star(\cdot) * \varphi \Nt_\SSSs \lesssim \sup_{0 \le t < T} \left\{w_{\phi, \eta}(t) \left(\int_0^t \|  \varphi(\tau) \|_{  \SSlm \cap \dot{\mathbb{H}}^{-\alpha}_x (\Slm_v) }^2 \d \tau \right)^{\frac{1}{2}} \right\},
				\end{equation}
				and
				\begin{equation}
					\label{eq:decay_semigroups_hydro_orthogonal_positive}
					\frac{1}{\eps} \Nt U^\eps_\star(\cdot) * \varphi \Nt_\SSSs \lesssim \sup_{0 \le t < T} \left\{w_{\phi, \eta}(t) \left(\int_0^t \|  \varphi(\tau) \|_{  \SSlm }^{ \frac{2}{1+\alpha} } \d \tau \right)^{ \frac{1+\alpha}{2} } \right\},
				\end{equation}
			\end{subequations}
		\end{lem}
		\begin{rem} The above estimates still hold  true for $\star=\hyd$ since $$U^\eps_\hyd(t)=U^\eps_\Wave(t)+U^\eps_\ns(t).$$ 
			Notice also that \eqref{eq:decay_semigroups_hydro_orthogonal_positive} shows that, with respect to \eqref{eq:decay_semigroups_hydro_orthogonal_negative}, no use of Sobolev space of negative order is required under the stronger integrability assumption $\varphi  \in L^{\frac{2}{1+\alpha}}\left( [0, T) ; \SSlm \right).$\end{rem}
		
		\begin{proof}
			Let us fix $\star=\Bou,\Inc,\pm\Wave$. First of all, notice that for $\eps | \xi | \le \alpha_{0}$, where we recall from \eqref{eq:lambda_star} that $\alpha_{0}$ can be taken small enough
			$$\re \left(\eps^{-2} \lambda_\star(\eps \xi) \right) = - \kappa_\star | \xi |^2 + \OO\left( \eps | \xi |^3 \right) \le - \frac{\kappa_\star}{2} | \xi |^2,$$
			and thus, the following estimate holds:
			$$\exp\left(\eps^{-2} t \lambda_\star(\eps \xi) \right) \le \exp(- t \kappa_\star \frac{|\xi|^2}{2}).$$
			Let us prove in the first step \eqref{eq:decay_semigroups_hydro} and in the second step \eqref{eq:decay_semigroups_hydro_orthogonal}.
			
			\step{1}{Proof of \eqref{eq:decay_semigroups_hydro}}
			Using the Fourier representation from Definition \ref{def:hydro_semigroups} of $U^\eps_\star f$, together with the boundedness $\| \PP_\star(\eps \xi) \|_{\BBB( \Slm ; \Ssp )}$ of Theorems \ref{thm:spectral_study}  and \ref{thm:enlarged_thm}, we easily have for $d \ge 2$
			\begin{equation*}\begin{split}
					\| U^\eps_\star(t) g \|_{\SSsp}^2 &=\int_{\R^d}\left\|\FF_x\left[U^{\eps}_\star(t)g\right](\xi)\right\|_{\Ssp}^2\,\langle \xi\rangle^{2s}\d \xi\\
					&\lesssim \int_{\R^{d}} e^{- t \kappa_\star | \xi |^2 } \| \widehat{g}(\xi) \|_{\Slm}^2 \la \xi \ra^{2s} \d \xi \lesssim \| g \|_{\SSlm}^2,  
				\end{split}
			\end{equation*}
			as well as
			\begin{equation*}\begin{split}
					\int_0^T \left\| |\nabla_x|^{1-\alpha} U^\eps_\star(t) g \right\|_{\SSsp}^2 \d t &=\int_0
					^T\d t \int_{\R^d}|\xi|^{2-2\alpha}\left\|\FF_x\left[U^\eps_\star(t) g\right](\xi)\right\|^2_{\Ssp}\langle\xi\rangle^{2s}\d \xi\\
					&\lesssim \int_0^T \int_{\R^{d}} | \xi |^{2-2\alpha} e^{- t \kappa_\star | \xi |^2  } \| \widehat{g}(\xi) \|_{\Slm}^2 \la \xi \ra^{2s} \d \xi \d t \\
					& \lesssim \int_{\R^{d}} \| \widehat{g}(\xi) \|_{\Slm}^2 \la \xi \ra^{2s} \left(\int_0^T | \xi |^{2 - 2\alpha} e^{- t \kappa_\star \frac{|\xi|^2}{2} } \d t\right) \d \xi\,.
			\end{split}\end{equation*}
			Consequently,
			\begin{equation}\label{eq:graU_*}
				\int_0^T \left\| |\nabla_x|^{1-\alpha} U^\eps_\star(t) g \right\|_{\SSsp}^2 \d t\lesssim \int_{\R^{d}} \| \widehat{g}(\xi) \|_{\Slm}^2 | \xi |^{-2\alpha} \la \xi \ra^{2s} \left(\int_0^T | \xi |^{2} e^{- t \kappa_\star \frac{|\xi|^2}{2} } \d t\right) \d \xi.\end{equation}
			Using that
			\begin{equation}\label{eq:XIALPHA}
				|\xi|^{-2\alpha}\langle \xi\rangle^{2s} \lesssim \langle \xi\rangle^{2s}\mathbf{1}_{|\xi|\ge 1} +|\xi|^{-2\alpha}\mathbf{1}_{|\xi|\le 1} \qquad \text{ and } \quad \int_0^T | \xi |^{2} e^{- t \kappa_\star \frac{|\xi|^2}{2} } \d t \lesssim 1\end{equation}
			we deduce that
			$$\int_0^T \left\| |\nabla_x|^{1-\alpha} U^\eps_\star(t) g \right\|_{\SSsp}^2 \d t \lesssim \|g\|_{\SSlm}^{2}+\| g \|_{\dot{\mathbb{H}}^{-\alpha}_x \left(\Slm_v\right)}^2.$$
			This concludes this step thanks to \eqref{eq:roughNt}.
			
			\step{2}{Proof of \eqref{eq:decay_semigroups_hydro_orthogonal}}
			Recall from the expansion \eqref{eq:PPstar} of $\PP_\star$  that
			\begin{equation}
				\label{eq:overPP_star}
				\begin{split}
					\PP_\star(\eps \xi) &= \PP_\star^{(0)}\left( \widetilde{\xi} \right) +i \eps \xi \cdot \PP_\star^{(1)}( \widetilde{\xi} ) + S_{\star}(\eps \xi)\\
					&=:\PP_{\star}^{(0)}\left(\widetilde{\xi}\right) + \eps \xi \cdot \overline{\PP}_{\star}^{(1)}(\eps \xi),
				\end{split}
			\end{equation}
			where the remainder $\overline{\PP}_{\star}^{(1)}(\eps \xi)$ satisfies 
			$$\sup_{\eps|\xi| \le \alpha_0}\|\overline{\PP}_{\star}^{(1)}(\eps \xi)\|_{\BBB(\Slm;\Ssp)} \lesssim 1$$ by virtue of  Theorems \ref{thm:spectral_study} and \ref{thm:enlarged_thm} whereas $\PP_{\star}^{(0)}\left( \widetilde{\xi} \right)$ is an $\Ss$-orthogonal projection on a subspace of $\nul(\LL)$. In particular, we have
			\begin{equation}
				\label{eq:KernelP}   
				\PP_\star(\eps \xi) \widehat{\varphi}(t, \xi) = \eps \xi \cdot \overline{\PP}_\star^{(1)}(\eps \xi) \widehat{\varphi}(t, \xi)\end{equation}
			and thus there holds, for any $t\ge 0,\tau \ge 0$
			\begin{align*}
				\left\| \exp\left( \eps^{-2} t \lambda_\star(\eps \xi) \right) \PP_\star(\eps \xi) \, \widehat{\varphi}(\tau, \xi) \right\|_{ \Ssp }
				& \le \eps |\xi| e^{- t \kappa_\star \frac{|\xi|^2}{2}} \| \overline{\PP}^{(1)}_\star(\eps \xi) \, \widehat{\varphi}(\tau,\xi) \|_{\Ssp } \\
				& \lesssim \eps |\xi| e^{- t \kappa_\star \frac{|\xi|^2}{2}} \| \widehat{\varphi}(\tau,\xi) \|_{\Slm }.
			\end{align*}
			Therefore, 
			\begin{equation}\label{eq:convVarphi}\begin{split}
					\frac{1}{\eps^2} \| U^\eps_\star(\cdot) * \varphi(t) \|_{\SSsp}^2 &=\frac{1}{\eps^2}\int_{\R^d}\langle \xi\rangle^{2s}\left\|\int_0^t \FF_x\left[U^\eps_\star(t-\tau)\varphi(\tau)\right](\xi)\d \tau\right\|_{\Ssp}^2\d \xi\\
					&   \leq \frac{1}{\eps^2}\int_{\R^d}\langle \xi\rangle^{2s}\left(\int_0^t \eps|\xi|e^{- (t-\tau) \kappa_\star \frac{|\xi|^2}{2}} \| \widehat{\varphi}(\tau,\xi)\|_{\Slm }\d \tau\right)^2\d \xi.
			\end{split}\end{equation}
			Using Cauchy-Schwarz's inequality to estimate the integral over $[0,t]$, we have
			\begin{align*}
				\frac{1}{\eps^2} \| U^\eps_\star(\cdot) * \varphi(t) \|_{\SSsp}^2 & \lesssim \int_{\R^{d}} \la \xi \ra^{2 s}  \left(\int_0^t \left[| \xi | e^{- (t-\tau) \kappa_\star \frac{|\xi|^2}{2} } \right]^2\d\tau\right)\left(\int_0^t \| \widehat{\varphi}(\tau, \xi) \|_{\Slm}^2 \d \tau\right) \d \xi \\
				& \lesssim \int_{\R^{d}} \la \xi \ra^{2 s} \d\xi \int_0^t \| \widehat{\varphi}(\tau, \xi) \|_{\Slm}^2 \d \tau    \lesssim \int_0^t \| \varphi(\tau) \|_{\SSlm}^2 \d \tau.
			\end{align*}
			In the same way,
			\begin{multline}\label{eq:younIne}
				\frac{1}{\eps^2}  \int_0^T \| |\nabla_x|^{1-\alpha} U^\eps_\star(\cdot)  * \varphi(t) \|_{\SSsp}^2 \d t 
				\\
				\lesssim \frac{1}{\eps^2}\int_0^T\d t\int_{\R^d}\langle \xi\rangle^{2s}|\xi|^{2-2\alpha}
				\left(\int_0^t \eps|\xi|e^{-(t-\tau)\kappa_{\star}\frac{|\xi|^2}{2}}\|\widehat{\varphi}(\tau,\xi)\|_{\Slm}\d\tau\right)^2\d\xi\\
				\lesssim \int_{\R^d}\langle \xi\rangle^{2s}|\xi|^{-2\alpha}\d \xi\int_0^T \left(\int_0^t |\xi|^2e^{-(t-\tau)\kappa_{\star}\frac{|\xi|^2}{2}}\|\widehat{\varphi}(\tau,\xi)\|_{\Slm}\d\tau\right)^2\d t.
			\end{multline}
			Recalling \eqref{eq:XIALPHA} and using  Young's convolution inequality in the form $L^1\left( [0, T] \right) \ast L^2\left([0, T]\right) \hookrightarrow  L^2\left([0, T]\right)$ we deduce that
			\begin{align*}
				\frac{1}{\eps^2}  \int_0^T \| |\nabla_x|^{1-\alpha} U^\eps_\star(t) * \varphi \|_{\SSsp}^2 \d t & \lesssim \int_{\R^{d}} \la \xi \ra^{2 s} | \xi |^{-2\alpha} \int_0^T \| \widehat{\varphi}(t, \xi) \|_{\Slm}^2 \d t \d \xi \\
				& \lesssim \int_0^T \| \varphi(t) \|_{\SSlm}^2 \d t + \int_0^T \| \varphi(t) \|_{ \dot{\mathbb{H}}^{-\alpha}_x \left(\Slm_v\right) }^2 \d t\,
			\end{align*}
			which easily prove \eqref{eq:decay_semigroups_hydro_orthogonal_negative}. To prove \eqref{eq:decay_semigroups_hydro_orthogonal_positive}, we rewrite \eqref{eq:younIne} as
			\begin{align*}
				\frac{1}{\eps^2}  \int_0^T \| |\nabla_x|^{1-\alpha} & U^\eps_\star(\cdot)  * \varphi(t) \|_{\SSsp}^2 \d t 
				\\
				& \lesssim   \int_{\R^d}\langle \xi\rangle^{2s}\d \xi\int_0^T \left(\int_0^t |\xi|^{2-\alpha}e^{-(t-\tau)\kappa_{\star}\frac{|\xi|^2}{2}}\|\widehat{\varphi}(\tau,\xi)\|_{\Slm}\d\tau\right)^2\d t.
			\end{align*}
			Using now  Young's convolution inequality in the form  $L^{\frac{2}{2-\alpha}}\left( [0, T] \right) \ast L^{\frac{2}{1+\alpha}}\left([0, T]\right) \hookrightarrow L^2\left( [0, T] \right)$ we deduce that
			\begin{align*}
				\frac{1}{\eps^2}  \int_0^T \| |\nabla_x|^{1-\alpha} U^\eps_\star(\cdot) * \varphi(t) \|_{\SSsp}^2 \d t 
				& \lesssim \int_{\R^{d}} \la \xi \ra^{2 s} \left(\int_0^T \| \widehat{\varphi}(t, \xi) \|_{\Slm}^{\frac{2}{1+\alpha}} \d t\right)^{1+\alpha} \d \xi \\
				& \lesssim \left(\int_0^T \| \varphi(t) \|_{\SSlm}^{\frac{2}{1+\alpha}} \d t\right)^{1+\alpha},
			\end{align*}
			where we used Minkowski's integral inequality for the last estimate. Since the estimates established are uniform in $T$, this concludes the proof.
		\end{proof}
		
		We now make precise the asymptotic equivalence between the semigroup  $\left(U^{\eps}_{\Wave}(t)\right)_{t\ge 0}$ and its leading order $\left(U^{\eps}_{\disp}(t)\right)_{t\geq0}$.
		\begin{lem}[\textit{\textbf{Asymptotic equivalence of the oscillating semigroups}}]
			\label{lem:asymptotic_equiv_oscillating_semigroup} Given $s \ge0$ and  some regularity parameter $r \in (s, s+1]$,  it holds
			\begin{equation}
				\label{eq:asymptotic_equiv_oscillating_reg}
				\Nt U^\eps_\Wave(\cdot) f - U^\eps_\disp(\cdot) f \Nt_{\rSSSs{s}} \lesssim \eps^{r-s} \left(\| f \|_{\rSSlm{r}} + \| f \|_{\dot{\mathbb{H}}^{-\alpha}_x (\Slm_v) } \right),
			\end{equation}
			for any $f \in \rSSlm{r} \cap \dot{\mathbb{H}}_x^{-\alpha}(\Slm_v)$ while  there holds
			\begin{equation}
				\label{eq:asymptotic_equiv_oscillating_general}
				\lim_{\eps \to 0}\Nt U^\eps_\Wave(\cdot) f - U^\eps_\disp(\cdot) f \Nt_{\SSSs} =0
			\end{equation}
			for any $f \in \SSlm \cap \dot{\mathbb{H}}^{-\alpha}_x (\Slm_v),$ i.e. whenever $r=s$.
		\end{lem}

		\begin{proof}
			We start by expanding the symbol of $U^\eps_{\pm \Wave}(t)$ using the decomposition of $\PP_{\pm \Wave}(\eps \xi)$ from \emph{Step 2} of the proof of Lemma \ref{lem:decay_semigroups_hydro}, we obtain
			\begin{equation*}\begin{split}
					\exp\left( \eps^{-2} t \lambda_{\pm \Wave}(\eps \xi) \right)&\PP_{\pm \Wave}(\eps \xi) =  \exp\left( \eps^{-2} t \lambda_{\pm \Wave}(\eps \xi) \right)\left( \PP_{\pm \Wave}^{(0)}\left( \widetilde{\xi} \right)  + i\eps \xi \cdot \overline{\PP}^{(1)}_{\pm \Wave}(\eps \xi) \right)\\
					= & \exp\left( \pm i c \eps^{-1}t| \xi |  - t \kappa_\Wave | \xi |^2 \right) \PP_{\star}^{(0)}\left( \widetilde{\xi} \right) \\
					& + \Bigg[ \exp\left( \eps^{-2} t \lambda_{\star}(\eps \xi) \right) - \exp\left( \pm i c \eps^{-1}t| \xi | - t \kappa_\Wave | \xi |^2 \right) \Bigg] \PP_{\pm \Wave}^{(0)}\left( \widetilde{\xi} \right) \\
					& + \eps  \exp\left( \eps^{-2} t \lambda_{\star}(\eps \xi) \right) \xi \cdot  \overline{\PP}^{(1)}_{\pm \Wave}(\eps \xi).
			\end{split}\end{equation*}
			Thus, the symbol of the difference $U^\eps_\Wave(t) - U^\eps_\disp(t)$ writes, for $\eps |\xi| \le \alpha_{0},$ as the sum of the two terms (corresponding to $\star=\pm\Wave$):
			\begin{align*}
				\Big[\exp\left( \eps^{-2} t \lambda_{\star}(\eps \xi) \right)   - \exp\left( \pm i c \eps^{-1}t | \xi | - t \kappa_\Wave | \xi |^2 \right) \Big]& \PP_{\star}^{(0)}\left( \widetilde{\xi} \right)\\
				+ \eps \exp\left( \eps^{-2} t \lambda_{\star}(\eps \xi) \right) \xi \cdot  & \overline{\PP}^{(1)}_{\star}(\eps \xi).
			\end{align*}
			On the one hand, when $\eps |\xi| > \alpha_{0}$, since $\overline{\PP}^{(1)}_{\star}(\eps \xi)$ is supported in $\{ \eps |\xi | \le \alpha_{0} \}$, the symbol reduces to 
			$$-  \exp\left(  i c \eps^{-1}t | \xi | - t \kappa_\Wave | \xi |^2 \right)  \PP_\Wave^{(0)}\left( \widetilde{\xi} \right)-\exp\left( - i c \eps^{-1}t | \xi | - t \kappa_\Wave | \xi |^2 \right)  \PP_{-\Wave}^{(0)}\left( \widetilde{\xi} \right).$$
			On the other hand, when $\eps | \xi | \le \alpha_{0}$, we estimate the difference of exponentials using the inequality $\left|1 - e^a\right| \le a e^{|a|}$ as well as the expansion \eqref{eq:lambda_star} of 
			$\lambda_{\pm\Wave}(\xi)$:
			\begin{align*}
				\Big| \exp\left( \eps^{-2} t \lambda_{\star}(\eps \xi) \right) & - \exp\left( \pm i c \eps^{-1} |\xi| - t \kappa_\Wave | \xi |^2 \right) \Big| \\
				& = \left| \exp\left( \pm i c t \eps^{-1}|\xi| - t \kappa_\Wave | \xi |^2 \right) \right| \Big| \exp \left( \OO\left( t \eps | \xi |^3 \right) \right) -1 \Big| \\
				& \lesssim \left(\eps | \xi|\right) \left(t | \xi |^2\right) \exp\left( - t \kappa_\Wave | \xi |^2 \right) \exp\Big( \OO\left( t \eps | \xi |^3 \right) \Big),
			\end{align*}
			thus, using $re^{-r} \lesssim e^{-\frac{1}{2}r}$ and assuming $\alpha_{0}$ small enough so that $\OO\left( \eps | \xi |^3 \right) \le \frac{1}{4}\kappa_\Wave | \xi |^2$,  we obtain
			\begin{align*}
				\Big| \exp\left( \eps^{-2} t \lambda_{\star}(\eps \xi) \right) & - \exp\left( \pm i c t \eps^{-1}|\xi| - t \kappa_\Wave | \xi |^2 \right) \Big|\\
				& \lesssim \eps | \xi | \exp\left( - \frac{t}{2} \kappa_\Wave | \xi |^2  \right) \exp\Big( \OO\left( t \eps  | \xi |^3 \right) \Big) \\
				& \lesssim \eps | \xi | \exp\left( - \frac{t}{4} \kappa_\Wave | \xi |^2 \right).
			\end{align*}
			Putting together the previous estimates, we then bound the operator norm in $\BBB(\Slm ;\Ssp )$ of the symbol of the difference~$U^\eps_\Wave(t) - U^\eps_\disp(t)$. It is controlled by
			\begin{align*}
				\mathbf{1} _{\eps | \xi | \le \alpha_0} \eps | \xi | \exp\left( - \frac{t}{4} \kappa_\Wave | \xi |^2 \right) & +  \mathbf{1} _{\eps | \xi | > \alpha_0} \exp\left( - \frac{t}{4} \kappa_\Wave | \xi |^2 \right) \\
				& \lesssim \left(\eps | \xi |\right)^{r - s} \exp\left( - \frac{t}{4} \kappa_\Wave | \xi |^2 \right),
			\end{align*}
			where we used the comparison $u\mathbf{1}_{u\le \alpha_{0}} + \mathbf{1}_{u >\alpha_{0}} \lesssim u^{r-s}$ for any $u \geq0$ since $r-s \in [0, 1]$.
			As in the proof of Lemma \ref{lem:decay_semigroups_hydro}, such an estimate on the symbol of $U^\eps_\Wave(t) - U^\eps_\disp(t)$ yields the controls \eqref{eq:asymptotic_equiv_oscillating_reg}, from which we deduce \eqref{eq:asymptotic_equiv_oscillating_general} by density.
		\end{proof}
		
		A similar result holds for the difference between $U^{\eps}_{\ns}(t)-U_{\ns}(t)$.
		\begin{lem}[\textit{\textbf{Asymptotic equivalence of the Navier-Stokes semigroup}}]
			\label{lem:asymptotic_equiv_NS_semigroup}
			Given $s \ge0$ and consider some regularity parameter $r \in (s, s+1]$, the part $U^\eps_\ns(\cdot)$ of the hydrodynamic semigroup is such that
			\begin{equation}
				\label{eq:asymptotic_equiv_NS_reg}
				\Nt U^\eps_\ns(\cdot) f - U_\ns(\cdot) f \Nt_{\rSSSs{s}} \lesssim \eps^{r - s} \left(\| f \|_{\rSSlm{r}} + \| f \|_{ \dot{\mathbb{H}}^{-\alpha}_x (\Slm_v) } \right),
			\end{equation}
			for any $f \in \rSSlm{r} \cap \dot{\mathbb{H}}^{-\alpha}_x(\Slm_v)$, while, for $f \in \SSlm \cap \dot{\mathbb{H}}^{-\alpha}_x (\Slm_v)$ (i.e. $r=s$), there holds
			\begin{equation}
				\label{eq:asymptotic_equiv_NS_general}
				\lim_{\eps \to 0}\Nt U^\eps_\ns(\cdot)  f - U_\ns(\cdot)  f \Nt_\SSSs = 0.
			\end{equation}
			Furthermore, if $\varphi \in L^2\left( [0, T) ; \SSlm \right)$ where $T \in (0, \infty]$ is such that $\PP \varphi(t) = 0$, then
			\begin{equation}
				\label{eq:asymptotic_equiv_NS_orthogonal_reg}
				\frac{1}{\eps} \Nt U^\eps_\ns(\cdot)  * \varphi - \eps \nabla_x \cdot V_\ns(\cdot)  * \varphi \Nt_{\rSSSs{s}} \lesssim \eps^{r - s} \sup_{0 \le t < T} \left\{w_{\phi, \eta}(t) \left(\int_0^t \|  \varphi(\tau) \|_{  \rSSlm{r} \cap \dot{\mathbb{H}}^{-\alpha}_x (\Slm_v) }^2 \d \tau \right)^{\frac{1}{2}} \right\}.
			\end{equation}
		\end{lem}
		
		\begin{proof}
			Let us fix $\star=\Bou,\Inc$. As in the previous proof, we start by expanding the symbol of $U^\eps_\ns(t)$ so as to compare it with those of $U_\ns(t)$ and $V_\ns(t)$. We first prove \eqref{eq:asymptotic_equiv_NS_reg} and \eqref{eq:asymptotic_equiv_NS_general}, and then \eqref{eq:asymptotic_equiv_NS_orthogonal_reg}.
			
			\step{1}{Proof of \eqref{eq:asymptotic_equiv_NS_reg} and \eqref{eq:asymptotic_equiv_NS_general}}
			For $\eps | \xi | \le \alpha_0$, using the decomposition \eqref{eq:overPP_star}, there holds
			\begin{align*}
				\exp\left( \eps^{-2} t \lambda_\star(\eps \xi) \right) \PP_\star(\eps \xi)
				= & \exp\left( - t \kappa_\star | \xi |^2 \right) \PP_\star^{(0)}\left( \widetilde{\xi} \right) + \eps \exp\left( \eps^{-2} t \lambda_\star(\eps \xi) \right)\xi \cdot   \overline{\PP}_\star^{(1)}(\eps \xi) \\
				& + \Big[\exp\left( \eps^{-2} t \lambda_\star(\eps \xi) \right) - \exp\left( - t \kappa_\star | \xi |^2 \right) \Big] \PP^{(0)}_\star\left( \widetilde{\xi} \right)
			\end{align*}
			whereas, for $\eps | \xi| > \xi_0$, since $\PP_\star(\eps \xi)$ vanishes, the symbol of the difference $U^\eps_\ns(t) - U_\ns(t)$ reduces to that of $-U_\ns(t)$ given by
			$$-\exp\left( - t \kappa_\Bou | \xi |^2 \right) \PP_\Bou^{(0)}\left( \widetilde{\xi} \right) - \exp\left(-t\kappa_{\Inc}|\xi|^{2}\right)\PP_{\Inc}^{(0)}\left(\widetilde{\xi}\right).$$
			To sum up, the symbol of the difference $U_\ns^\eps(t) - U_\ns (t)$ writes as the sum over $\star=\Bou,\Inc$ of the symbols
			\begin{align*}
				\mathbf{1}_{\eps | \xi | \le \alpha_0} \Big( \eps \exp\left( \eps^{-2} t \lambda_\star(\eps \xi) \right)\xi \cdot  \overline{\PP}_\star^{(1)}(\eps \xi) & + \Big[\exp\left( \eps^{-2} t \lambda_\star(\eps \xi) \right) - \exp\left( - t \kappa_\star | \xi |^2 \right) \Big] \PP^{(0)}_\star\left( \widetilde{\xi} \right) \Big) \\
				& - \mathbf{1}_{\eps | \xi | > \alpha_0} \exp\left( - t \kappa_\star | \xi |^2 \right) \PP_\star^{(0)}\left( \widetilde{\xi} \right),
			\end{align*}
			and its operator norm in $\BBB( \Slm ; \Ssp )$ is controlled as in the proof of Lemma \ref{lem:decay_semigroups_hydro} by
			$$\eps | \xi | \exp\left( - t \kappa_\star \frac{| \xi |^2}{4} \right).$$
			We then deduce \eqref{eq:asymptotic_equiv_NS_reg} as well as \eqref{eq:asymptotic_equiv_NS_general} by density as in the proof of Lemma \ref{lem:asymptotic_equiv_oscillating_semigroup}.
			
			\step{2}{Proof of \eqref{eq:asymptotic_equiv_NS_orthogonal_reg}}
			In the case $\PP \varphi(t)=0$, the projector $\PP^{(0)}_\star \varphi(t)$ vanishes, and we use the second order expansion of $\PP_{\star}(\eps \xi)$ provided in \eqref{eq:PPstar} in  Theorem \ref{thm:spectral_study}:
			$$\PP_{\star}(\eps \xi) =   i\eps \xi \cdot \PP^{(1)}\left( \widetilde{\xi} \right) + S_{\star}(\eps \xi),$$
			where we recall that $\| S_\star(\eps \xi) \|_{ \BBB(\Slm ; \Ssp) } \lesssim \eps^2 | \xi|^2$ uniformly in $\eps | \xi | \le \alpha_0$.
			Similarly, the symbol of $U^\eps_\ns(t) - \eps \nabla_x \cdot V_\ns(t)$ restricted to $\nul(\PP)$ then writes
			\begin{align*}
				& \mathbf{1}_{\eps | \xi | \le \alpha_0} \Big( \eps^2 \exp\left( \eps^{-2} t \lambda_\star(\eps \xi) \right) |\xi|^{2} S_\star(\eps \xi)   + i\eps  \Big[\exp\left( \eps^{-2} t \lambda_\star(\eps \xi) \right) - \exp\left( - t \kappa_\star | \xi |^2 \right) \Big] \xi \cdot \PP^{(1)}_\star\left( \widetilde{\xi} \right) \Big) \\
				& \phantom{++++} - \mathbf{1}_{\eps | \xi | > \alpha_0} i\eps \exp\left( - t \kappa_\star | \xi |^2 \right)  \xi \cdot  \PP_\star^{(1)}\left( \widetilde{\xi} \right),
			\end{align*}
			which is similarly controlled by
			$$\eps | \xi | (\eps | \xi |)^{r-s} \exp\left( - t \kappa_\star \frac{| \xi |^2}{4} \right).$$
			These representations allow to proceed as in the proofs of Lemmas \ref{lem:decay_semigroups_hydro} and \ref{lem:asymptotic_equiv_oscillating_semigroup} to get the desired conclusion.
		\end{proof}
		We present now a dispersive estimate for the semigroup $\left(U_{\disp}^{\eps}(t)\right)_{t\geq0}$ which is deduced from a general result about the decay rate for solutions to the wave equation.
		\begin{lem}[\textit{\textbf{Dispersive estimate}}]
			The part $U^\eps_\disp$ of the hydrodynamic semigroup satisfies the dispersive estimate
			\begin{gather}
				\label{eq:dispersive}
				\left\| U^\eps_\disp(t) g \right\|_{ W^{s, \infty}_x \left( \Ssp_v \right) } \lesssim \left(\frac{\eps}{t}\right)^{\frac{d-1}{2}} \| g \|_{ \dot{\mathbb{B}}^{ \frac{d+1}{2}+ s }_{1,1} \left( \Slm_v \right) }.
			\end{gather}
		\end{lem}	
		\begin{proof}
			In virtue of the macroscopic representation of $U^\eps_\disp$ from Proposition \ref{prop:macro_representation_spectral} and the continuity of the heat semigroup on $L^1$, we can deduce \eqref{eq:dispersive} directly from Lemma \ref{lem:wave-equation}.\end{proof}
		
		\begin{lem}[\textit{\textbf{Vanishing estimate for the convoluted oscillating semigroup}}]
			\label{lem:convolution_wave}
			Suppose $\varphi \in L^{\infty}([0,T)\,;\,\SSlm)$ is such that $|\nabla_{x}|^{1-\alpha}\varphi \in L^{2}([0,T)\;;\;\SSlm)$ and
			$\PP \varphi(t) = 0$ for any $t\ge0$ together with
			$$\partial_t\varphi \in  L^2 \cap L^{ \frac{2}{1+\alpha} } \left( [0, T) ; \rSSlm{s-1}\right) \bigcap  L^{\frac{4}{3}} \cap L^{\frac{4}{3 + 2 \alpha}} \left( [0, T) ; \dot{\mathbb{H}}^{-\frac{1}{2}}_x \left(\Slm_v\right) \right).$$
			Then,  there holds
			\begin{equation*}
				\begin{aligned}
					\frac{1}{\eps^2} \Nt U^\eps_\Wave(\cdot) * \varphi \Nt_{\SSSs} \lesssim \| \varphi(0) \|_{ \dot{\mathbb{H}}^{-\alpha}_x \left( \Slm_v \right) } & + \| \varphi \|_{ L^\infty \left( [0, T) ; \SSlm \right) } + \| | \nabla_x|^{1 - \alpha} \varphi \|_{ L^2 \left( [0, T) ; \SSlm \right) } \\
					& + \| \partial_t \varphi \|_{ \left(L^2 \cap L^{ \frac{2}{1+\alpha} }\right) \left( [0, T) ; \rSSlm{s-1} \right) } \\
					& + \| \partial_t \varphi \|_{ \left(L^{ \frac{4}{3} } \cap L^{\frac{4}{3 + 2 \alpha}}\right) \left( [0, T) ; \dot{\mathbb{H}}^{-\frac{1}{2}}_x \left(\Slm_v\right) \right)  }.
				\end{aligned}
			\end{equation*}
		\end{lem}
		
		\begin{rem}
			Note that if $T < \infty$, we have $L^{\frac{4}{3}} \cap L^{\frac{4}{3+2\alpha}}=L^{\frac{4}{3}}$ and $L^{\frac{2}{1+\alpha}}\cap L^2=L^2$.
		\end{rem}

		\begin{proof}
			In the first step, we establish a preparatory estimate for any $\xi \in \R^d$ satisfying $\eps | \xi | \le \alpha_0$, which we will use in the following step to prove the lemma. Since $w_{\phi, \eta} \le 1$, we neglect it for the estimates on $\partial_t \varphi$.
			
			\step{1}{Preparatory estimate}
			Recall that $\| U^{\eps}_{\Wave}(\cdot) * \varphi(t)\|_{\SSsp}^{2}$ is given by \eqref{eq:convVarphi} which allows us to
			%$$\| U^{\eps}_{\Wave}(\cdot) * \varphi(t)\|_{\SSsp}^{2}= \int_{\R^{d}}\la \xi\ra^{2s} %\left\|\int_{0}^{t} \FF [U^{\eps}_{\Wave}(t)\varphi(t-\tau)](\xi) \d \tau\right\|_{\Ssp}^{2}\d\xi$$
			work, as in the proof of Lemma \ref{lem:asymptotic_equiv_oscillating_semigroup},  on the two parts of the symbol of $U^{\eps}_{\Wave}(t).$ Recalling \eqref{eq:KernelP}, for any fixed $t \ge0$ and any $\tau \in [0,t]$ one has
			\begin{equation*}\begin{split}
					\exp\left( \eps^{-2} \tau \lambda_{\pm \Wave}(\eps \xi) \right) & \PP_{\pm \Wave}(\eps \xi) \widehat{\varphi}(t-\tau, \xi) \\
					& = \eps\exp\left( \eps^{-2} \tau \lambda_{\pm \Wave}(\eps \xi) \right) \xi \cdot   \overline{\PP}^{(1)}_{\pm \Wave}(\eps \xi) \widehat{\varphi}(t-\tau, \xi) \\
					& = \eps\exp\left( \eps^{-2} \tau \lambda_{\pm \Wave}(\eps \xi) \right)  \xi \cdot  {\phi^{\pm}
					}(\tau, \xi)
			\end{split}\end{equation*}
			where we denoted ${\phi^{\pm}}(\tau, \xi) := \overline{\PP}^{(1)}_{\pm \Wave}(\eps \xi) \widehat{\varphi}(t-\tau, \xi)$. We now integrate with respect to $\tau \in [0, t]$ using integration by parts:
			\begin{align*}
				\int_0^t \exp\left( \eps^{-2} \tau \lambda_{\pm \Wave}(\eps \xi) \right) & \PP_{\pm \Wave}(\eps \xi) \widehat{\varphi}(t-\tau, \xi) \d \tau\\
				= & \eps \xi \cdot \int_0^t \exp\left( \eps^{-2} \tau \lambda_{\pm \Wave}(\eps \xi) \right) \phi^{\pm}(\tau, \xi) \d \tau \\
				= & - \frac{\eps^3 \xi }{\lambda_{\pm \Wave}(\eps \xi)} \cdot \int_0^t \exp\left( \eps^{-2} \tau\lambda_{\pm \Wave}(\eps \xi) \right) \partial_\tau {\phi}^{\pm}(\tau, \xi) \d \tau \\
				& + \frac{\eps^3 \xi }{\lambda_{\pm \Wave}(\eps \xi)} \cdot \left[ {\phi}^{\pm}(t, \xi)\exp\left( \eps^{-2} t \lambda_{\pm \Wave}(\eps \xi ) \right) 
				- {\phi}^{\pm}(0, \xi)\right].
			\end{align*}
			As in Lemma \ref{lem:asymptotic_equiv_oscillating_semigroup}, we can choose $\alpha_0$ small enough so that
			$$|\lambda_{\pm \Wave}(\eps \xi)| \approx \eps | \xi |, \qquad \re\left( \eps^{-2} \lambda_{\pm \Wave} (\eps \xi) \right) \le - \frac{1}{2}\kappa_{ \Wave} | \xi |^2,$$
			uniformly in $|\xi| \le \alpha_{0}$, and one notices
			$$
			\|\phi^{\pm}(t,\xi)\|_{\Ssp}=\|\overline{\PP}^{(1)}_{\pm \Wave}(\eps \xi) \widehat{\varphi}(0, \xi)\|_{\Ssp} \lesssim \|\widehat{\varphi}(0,\xi)\|_{\Slm},$$
			while, in the same way,
			$$\|\phi^{\pm}(0,\xi)\|_{\Ssp}\lesssim \|\widehat{\varphi}(t,\xi)\|_{\Slm},\qquad 
			\|\partial_{\tau}\phi^{\pm}(\tau,\xi)\|_{\Ssp}
			\lesssim \|\partial_{\tau}\widehat{\varphi}(t-\tau,\xi)\|_{\Slm}\,.$$
			Those considerations lead to		
			\begin{equation*}\begin{split}
					\frac{1}{\eps^2} \Big\|\int_0^t & \exp\left( \eps^{-2} \tau \lambda_{\pm \Wave}(\eps \xi) \right) \PP_{\pm \Wave}(\eps \xi) \widehat{\varphi}(t-\tau, \xi) \d \tau\Big\|_{\Ssp}\\
					\le & ~ \int_{0}^{t} \exp\left(- \frac{\tau}{2} \kappa_\Wave |\xi |^2\right) \|\partial_{\tau}\widehat{\varphi}(t-\tau,\xi)\|_{\Slm}\d \tau  + \|\widehat{\varphi}(t,\xi)\|_{\Slm}\\
					&\phantom{+++++} + \|\widehat{\varphi}(0,\xi)\|_{\Slm}\exp\left(-\frac{t}{2}\kappa_{\Wave}|\xi|^{2}\right).
			\end{split}\end{equation*}
			In other words, we have shown that
			\begin{equation}\label{eq:FxUwave}\begin{split}
					\frac{1}{\eps^2} \Big\| \FF_x \big[ & U^\eps_\Wave(\cdot) * \varphi \big](t, \xi) \Big\|_{\Ssp} \le  ~ \int_{0}^{t} \exp\left(- \frac{\tau}{2} \kappa_\Wave |\xi |^2 \right) \|\partial_{\tau}\widehat{\varphi}(t-\tau,\xi)\|_{\Slm}\d \tau 
					\\ &+ \|\widehat{\varphi}(t,\xi)\|_{\Slm}+ \|\widehat{\varphi}(0,\xi)\|_{\Slm}\exp\left(-\frac{t}{2}\kappa_{\Wave}|\xi|^{2}\right).
			\end{split}\end{equation}
			
			\step{2}{Completion of the proof} We first deduce from \eqref{eq:FxUwave} that
			$$\frac{1}{\eps^2} \| U^\eps_\Wave(\cdot) * \varphi(t) \|_{\SSsp} \lesssim  \sup_{0 \le \tau \le t} \| \varphi(\tau) \|_{ \SSlm } + \left\| \la \xi \ra^{s} \int_0^t e^{- \tau \kappa_{\Wave} \frac{|\xi|^2}{2}} \| \partial_\tau \widehat{\varphi}(t-\tau, \xi) \|_{\Slm} \d \tau \right\|_{ L^2_\xi }.$$
			For notations simplicity, we call $J=J(t,\varphi)$ the above $L^2_\xi$-norm and split it according to $|\xi| \leq 1$ or $|\xi| >1$, i.e. $J=J_1+J_2$
			where
			$$J_1^2=\int_{|\xi| \leq 1}\langle \xi\rangle^{2s}\left(\int_0^t e^{- \tau \kappa_{\Wave} \frac{|\xi|^2}{2}} \| \partial_\tau \widehat{\varphi}(t-\tau, \xi) \|_{\Slm} \d \tau \right)^2 \d\xi$$
			and
			\begin{multline*}
				J_2^2=\int_{|\xi| \geq 1}\langle \xi\rangle^{2s}\left(\int_0^t e^{- \tau \kappa_{\Wave} \frac{|\xi|^2}{2}} \| \partial_\tau \widehat{\varphi}(t-\tau, \xi) \|_{\Slm} \d \tau \right)^2 \d\xi\\
				\leq \int_{|\xi| \geq 1}\langle \xi\rangle^{2s-2}\left(\int_0^t \left[|\xi|\,e^{- \tau \kappa_{\Wave} \frac{|\xi|^2}{2}}\right] \| \partial_\tau \widehat{\varphi}(t-\tau, \xi) \|_{\Slm} \d \tau \right)^2 \d\xi.    \end{multline*}
			On the one hand, using Cauchy-Schwarz inequality and the second estimate in \eqref{eq:XIALPHA}, one has
			$$J_2^2 \lesssim \int_{\R^d}\langle\xi\rangle^{2s-2}\d\xi\int_0^t\|\partial_\tau\widehat{\varphi}(\tau,\xi)\|_{\Slm}^2\d\tau=\|\partial_t \varphi\|_{L^2_t H^{s-1}_x(\Slm_v)}^2.$$
			On the other hand,
			invoking H\"older's inequality (with exponents $p=4,q=\frac{4}{3}$) to estimate the time integral, we deduce that
			\begin{equation*}\begin{split}
					J_1^2 &\lesssim  \int_{|\xi| \leq 1} \left(\int_0^t\left[e^{- \tau \kappa_{\Wave} \frac{|\xi|^2}{2}}\right]^4\d \tau\right)^{\frac{1}{2}}\,\left(\int_0^t\| \partial_\tau \widehat{\varphi}(t-\tau, \xi) \|_{\Slm}^{\frac{4}{3}}\d\tau\right)^{\frac{3}{2}}\d\xi \\
					&\lesssim  \int_{\R^d} |\xi|^{-1}\left(\int_0^t\| \partial_\tau \widehat{\varphi}(\tau, \xi) \|_{\Slm}^{\frac{4}{3}}\d\tau\right)^{\frac{3}{2}}\d\xi 
			\end{split}\end{equation*}
			which, thanks to Minkowski's integral inequality, yields
			$$J_1^{\frac{4}{3}}\lesssim \int_{0}^t\left(\int_{\R^d} |\xi|^{-1}\| \partial_\tau \widehat{\varphi}(\tau, \xi) \|_{\Slm}^{2}\d\xi\right)^{\frac{2}{3}}\d\tau=\int_0^t  \|\partial_t \varphi(t)\|_{\dot{\mathbb{H}}^{-\frac{1}{2}}_x(\Slm_v)}^{\frac{4}{3}}\d\tau.$$
			Therefore,
			$$J \lesssim \|\partial_t \varphi\|_{L^2_t H^{s-1}_x(\Slm_v)} + \| \partial_t \varphi \|_{ L^{\frac{4}{3}}_t \dot{\mathbb{H}}^{-\frac{1}{2}}_x \Slm_v } $$
			i.e.
			$$\frac{1}{\eps^2} \| U^\eps_\Wave(\cdot) * \varphi(t) \|_{\SSsp} \lesssim \sup_{0 \le \tau \le t} \| \varphi(\tau) \|_{ \SSlm } + \| \partial_t \varphi \|_{ L^{ \frac{4}{3} }_t \dot{\mathbb{H}}^{-\frac{1}{2}}_x \Slm_v } + \| \partial_t \varphi \|_{ L^2_t H^{s-1}_x \Slm_v }.$$
			Furthermore, coming back to \eqref{eq:FxUwave}, 
			\begin{multline*}
				\frac{1}{\eps^4} \int_0^T | \xi |^{2-2\alpha} \Big\| \FF_x \big[ U^\eps_\Wave(\cdot) * \varphi \big](t, \xi) \Big\|_{ \Ssp }^2 \d t \\
				\lesssim  \int_0^T \left(\int_{0}^{t} \left[ | \xi |^{1-\alpha} e^{-\frac{\tau}{2} \kappa_\Wave | \xi |^2} \right] \|\partial_{\tau}\widehat{\varphi}(t-\tau,\xi)\|_{\Slm} \d \tau\right)^2 \d t\\
				+ \int_0^T | \xi |^{2-2\alpha} \big\| \widehat{\varphi}(t, \xi) \big\|_{\SSlm}^2 \d t  +  \| \widehat{\varphi}(0, \xi) \|_{\Slm}^2  \left(\int_0^T |\xi|^{2-2\alpha} e^{-t \kappa_{\Wave}|\xi|^{2}} \d t\right)
			\end{multline*}
			Using now Young's convolution inequality in the form $L^{\frac{2p}{3 p - 2}}\left([0, T]\right) \ast  L^p\left( [0, T] \right) \hookrightarrow L^2\left([0, T]\right)$ ($p \in [1, 2]$) in the first time integral, 
			we deduce that
			\begin{multline*}
				\frac{1}{\eps^4} \int_0^T | \xi |^{2-2\alpha} \Big\| \FF_x \big[ U^\eps_\Wave(\cdot) * \varphi \big](t, \xi) \Big\|_{ \Ssp }^2 \d t \lesssim \int_0^T | \xi |^{2-2\alpha} \big\| \widehat{\varphi}(t, \xi) \big\|_{\Slm}^2 \d t  + | \xi |^{-2\alpha} \| \widehat{\varphi}(0, \xi) \|_{\Slm}^2 \\
				+  \left( \int_0^T \left(| \xi |^{-\left(2 + \alpha - \frac{2}{p}\right)} \|\partial_{\tau}\widehat{\varphi}(\tau,\xi)\|_{\Slm}\right)^p \d \tau\right)^{\frac{2}{p}}.
			\end{multline*}		
			We integrate this inequality against $\la \xi \ra^{2s}$ with the choice $p=\frac{2}{1+\alpha} \in \left[\frac{4}{3}, 2\right]$ on the region $|\xi| \geq 1$ and with~$p=\frac{4}{3 + 2\alpha } \in \left[1, \frac{4}{3} \right]$ on the region $|\xi| \le 1$, to obtain, after a simple use of Minkowski's integral inequality, 
			\begin{align*}
				\frac{1}{\eps^2} \Bigg(\int_0^T \| |\nabla_x|^{1-\alpha} U^\eps_\Wave(\cdot) * \varphi(t) \|_{\SSsp}^2 \d t\Bigg)^{\frac{1}{2}}
				\lesssim &
				\left\| | \nabla_x|^{1-\alpha} \varphi \right\|_{ L^2\left( [0, T) ; \SSlm \right) } + \| \varphi(0) \|_{ \dot{\mathbb{H}}^{-\alpha}_x \left( \Slm_v \right) } \\
				& + \| \partial_t \varphi \|_{ L^{\frac{4}{3 + 2 \alpha}}_t \dot{\mathbb{H}}^{-\frac{1}{2}}_x \Slm_v }
				+ \| \partial_t \varphi \|_{ L^{ \frac{2}{1+\alpha} }_t H^{s-1}_x \Slm_v }.
			\end{align*}
			Since the estimates established are uniform in $T$ and $w_{\phi, \eta} \le 1$, this concludes the proof.
		\end{proof}
		
		The decay and regularization estimates for $\left(U^\eps_\kin(t)\right)_{t\ge0}$ are given by scaling the estimates from Theorem \ref{thm:spectral_study}, or under the enlargement assumptions \ref{LE}, Theorem \ref{thm:enlarged_thm}.
		\begin{lem}[\textit{\textbf{Decay and regularization of the kinetic semigroup}}]
			\label{lem:decay_regularization_kinetic_semigroup}
			For any fixed decay rate~$\sigma \in (0, \sigma_0)$, the kinetic part~$\left(U^\eps_\kin(t)\right)_{t\ge0}$ of the semigroup satisfies the decay and regularization estimates
			$$		\sup_{t \ge 0} \, e^{2 \sigma_0 t / \eps^2} \| U^\eps_\kin(t) f \|_{\SSl}^2 + \frac{1}{\eps^2} \int_0^\infty e^{2 \sigma t / \eps^2} \| U^\eps_\kin(t) f \|_{\SSlp}^2 \, \d t \lesssim \| f \|_{\SSl}^2$$
			as well as
			$$		\frac{1}{\eps^2} \int_0^\infty e^{2 \sigma t / \eps^2} \| U^\eps_\kin(t) f \|_{\SSl}^2 \, \d t \lesssim \| f \|^2_{\SSlm},$$
			with exactly the same estimate satisfied by the adjoint $\left((U^{\eps}_{\kin}(t))^{\star}\right)_{t\geq0}$.
		\end{lem}
		
		As for the hydrodynamic semigroup $U^\eps_\hyd(\cdot)$, we establish now suitable convolution estimates:
		
		\begin{lem}[\textit{\textbf{Decay and regularization of the convoluted kinetic semigroup}}]
			\label{lem:decay_regularization_convolution_kinetic_semigroup}
			Consider $T \in (0, \infty]$. For any $\varphi \in L^2\left( [0, T) ; \SSsm \right)$, there holds uniformly in $\eps$
			\begin{subequations}
				\begin{equation}
					\label{eq:decay_convolution_semigroup_no_exp}
					\begin{aligned}
						\frac{1}{\eps} \Nt U^\eps_\kin(\cdot) * \varphi \Nt_{\SSSm(T, \phi, \eta, \eps)} \lesssim \sup_{0 \le t < T} \left\{w_{\phi, \eta}(t) \left(\int_0^t \|  \varphi(\tau) \|_{  \SSsm }^2 \d \tau \right)^{\frac{1}{2}} \right\}.
					\end{aligned}
				\end{equation}
				Furthermore, consider $\sigma \in [0, \sigma_0)$, there holds uniformly in $\eps$ 
				\begin{equation}
					\label{eq:decay_convolution_semigroup_exp}
					\begin{aligned}
						\frac{1}{\eps} \Nt U^\eps_\kin(\cdot) * \varphi \Nt_{\SSSl(\sigma, \eps)} \lesssim \left(\int_0^T e^{2 	\sigma t / \eps^2 } \| \varphi(t) \|_{\SSlm }^2 \d t\right)^{\frac{1}{2}}
					\end{aligned}
				\end{equation}
				for any $\varphi$ for which the right-hand-side is finite. 
				%			\begin{aligned}
				%				\sup_{0 \le t < T} e^{2 \sigma t / \eps^2} & \| U^\eps_\kin(\cdot) * \varphi (t) \|_{\SSg}^2 \\
				%					& + \int_0^T e^{2 \sigma t / \eps^2} \| U^\eps_\kin(\cdot) * \varphi (t) \|_{\SSgp}^2 \d t \lesssim \eps^2 \int_0^T e^{2 	\sigma t / \eps^2 } \| \varphi(t) \|_{\SSgm }^2 \d t,
				%			\end{aligned}
			\end{subequations}
			%		as well as the corresponding estimate for an algebraic decay of order $\alpha \ge 0$:
			%		\begin{equation}
			%			\label{eq:decay_convolution_semigroup_alg}
			%			\begin{aligned}
			%				\sup_{0 \le t < T} \la t \ra^{2 \alpha} & \| U^\eps_\kin(\cdot) * \varphi (t) \|_{\SSg}^2 \\
			%				& + \int_0^T \la t \ra^{2 \alpha} \| U^\eps_\kin(\cdot) * \varphi (t) \|_{\SSgp}^2 \d t \lesssim \eps^2 \int_0^T \la t \ra^{2 	\alpha} \| \varphi(t) \|_{\SSgm }^2 \d t.
			%			\end{aligned}
			%		\end{equation}
			
		\end{lem}
		
		\begin{proof}
			Denote by $(\Sg, \SSg, \SSgp, \SSgm)$ either $(\Ss, \SSs, \SSsp, \SSsm)$ or $(\Sl, \SSl, \SSlp, \SSlm)$.
			In a first step, we use a duality argument to prove the $\SSg-\SSgm$-integral decay:
			\begin{equation*}
				\frac{1}{\eps^2} \int_0^T e^{2 \sigma t / \eps^2} \| U^\eps_\kin(\cdot) * \varphi (t) \|_{\SSgp}^2 \, \d t \lesssim \eps^2 \int_0^T 	\| \varphi(t) \|_{ \SSgm }^2 \, \d t.
			\end{equation*}
			and we deduce from it the $\SSg-\SSgm$-uniform decay together with the stronger $\SSgp-\SSgm$-integral decay using an energy method in a second~step:
			\begin{equation*}
				\sup_{0 \le t < T} e^{2 \sigma t / \eps^2} \| U^\eps_\kin(\cdot) * \varphi (t) \|_{\SSg}^2 + \frac{1}{\eps^2} \int_0^T e^{2 \sigma t / \eps^2} \| U^\eps_\kin(\cdot) * \varphi (t) \|_{\SSgp}^2 \lesssim \eps^2 \int_0^T \| \varphi(t) \|_{ \SSgm }^2 \, \d t.
			\end{equation*}
			Note that this proves \eqref{eq:decay_convolution_semigroup_exp}, and it is enough to prove \eqref{eq:decay_convolution_semigroup_no_exp} as it follows from the particular case~$\sigma = 0$ and $\Sg = \Ss$.
			
			\step{1}{Integral decay in $\SSg-\SSgm$}
			We will prove the following estimate uniformly in $T \in (0, \infty]$ and~$\phi \in L^2\left( [0, T) ; \SSg\right)$:
			$$\left\la e^{\sigma t / \eps^2 } (U^\eps_\kin(\cdot) * \varphi), \phi \right\ra_{L^2\left( [0, T) ; \SSg\right) } \lesssim \eps^2 \| \phi \|_{L^2\left( [0, T) ; \SSg\right) } \left(\int_0^T e^{2\sigma \tau / \eps^2 } \| \varphi(\tau) \|_{\SSgm}^2 \d \tau\right)^{\frac{1}{2}},$$
			and as in the proof of Lemma \ref{lem:spectral_decay}, it is enough to check that it holds for $\phi$ of the form
			\begin{gather*}
				\phi(t) = \begin{cases}
					\phi_0 \in \SSg, & t \in [t_1 , t_2]\\
					0, & t \notin [t_1, t_2]
				\end{cases}, \qquad \| \phi \|_{ L^2\left( [0, T) ; \SSg \right) } = \sqrt{t_2 - t_1} \| \phi_0 \|_{\SSg}.
			\end{gather*}
			By duality, we have
			\begin{equation*}\begin{split}
					\big\la e^{\sigma t / \eps^2}  & (U^\eps_\kin(\cdot) * \varphi), \phi \big\ra_{ L^2\left( [0, T) ; \SSg \right) } \\
					& = \int_{t_1}^{t_2} \int_0^t \lla e^{\sigma t / \eps^2} U^\eps_\kin(t-\tau) \varphi(\tau) , \phi_0 \rra_{\SSg} \d \tau \, \d t \\
					& = \int_{t_1}^{t_2} \int_0^t \lla e^{\sigma \tau / \eps^2} \varphi(\tau) , e^{\sigma (t-\tau) / \eps^2 } U^\eps_\kin(t-\tau)^\star \, \phi_0 \rra_{\SSg} \d \tau \, \d t  \\
					& \le \int_{t_1}^{t_2} \int_0^t \left\| e^{\sigma \tau / \eps^2 } \varphi(\tau) \right\|_{\SSgm } \left\| e^{\sigma (t-\tau) / \eps^2 } U^\eps_\kin(t-\tau)^\star \, \phi_0 \right\|_{\SSgp} \d \tau \, \d t,
			\end{split}\end{equation*}
			and so, using first Cauchy-Schwarz's inequality and then Young's convolution inequality in the form~$L^2\left( [0, T] \right) \ast  L^1\left([0, T]\right) \hookrightarrow L^2\left([0, T]\right)$, there holds
			\begin{equation*}\begin{split}
					\big\la e^{\sigma t / \eps^2}  & (U^\eps_\kin(\cdot) * \varphi), \phi \big\ra_{ L^2\left( [0, T) ; \SSg \right) } \\
					& \le \sqrt{t_1-t_1} \left( \int_{0}^{T} \left( \int_0^t \left\| e^{\sigma \tau / \eps^2} \varphi(\tau) \right\|_{\SSgm } \left\| e^{\sigma (t-\tau) / \eps^2 } U^\eps_\kin(t-\tau)^\star \phi_0 \right\|_{\SSgp} \d \tau \right)^2 \d \tau \right)^{\frac{1}{2}} \\
					& \le \sqrt{t_1-t_1} \left(\int_0^T \left\| e^{\sigma t / \eps^2 } \varphi(t) \right\|_{\SSgm }^2 \d t\right)^{\frac{1}{2}} \int_0^T \left\| e^{\sigma t / \eps^2 } U^\eps_\kin(t)^\star \phi_0 \right\|_{\SSgp} \d t. 
			\end{split}\end{equation*}
			Furthermore, using Cauchy-Schwarz's inequality, we deduce for some $\sigma' \in (\sigma, \sigma_0)$
			\begin{equation*}\begin{split}
					\la e^{\sigma t / \eps^2}  & (U^\eps_\kin(\cdot) * \varphi), \phi \ra_{ L^2\left( [0, T) ; \SSg \right) } \\
					& \lesssim \eps \sqrt{t_1-t_1} \left(\int_0^T \| e^{\sigma t / \eps^2 } \varphi(t) \|_{\SSgm }^2 \d t\right)^{\frac{1}{2}} \left(\int_0^T \| e^{\sigma' t / \eps^2 } U^\eps_\kin(t)^\star \phi_0 \|_{\SSgp}^2 \d t\right)^{\frac{1}{2}}.
			\end{split}\end{equation*}
			from which, using the $\SSgp-\SSg$-integral estimate for $\left((U^\eps_\kin(t))^\star\right)_{t\geq0}$ obtained in Lemma \ref{lem:decay_regularization_kinetic_semigroup}, we obtain
			\begin{equation*}\begin{split}
					\big\la e^{\sigma t / \eps^2} & (U^\eps_\kin(\cdot) * \varphi), \phi \big\ra_{ L^2\left( [0, T) ; \SSg \right) } \\
					& \lesssim \eps^2 \sqrt{t_1-t_1} \left(\int_0^T \| e^{\sigma t / \eps^2 } \varphi(t) \|_{\SSgm }^2 \d t\right)^{\frac{1}{2}} \| \phi_0 \|_{\SSg}.
			\end{split}\end{equation*}
			This concludes this step.

			\step{2}{Regularized uniform and integral decay}
			Denote $u(t) := U^\eps_\kin * \varphi(t) = U^\eps * \PP^\eps_\kin \varphi(t)$, it satisfies the evolution equation
			$$\partial_t u = \frac{1}{\eps^2}  \left(\LL - \eps v \cdot \nabla_x \right) u + \PP^\eps_\kin \varphi, \quad u(0) = 0.$$
			Note that, considering the decomposition $\LL = \left(\LL - \PP\right) + \PP$ in the case $\Sg = \Ss$ (from Assumption \ref{L1}--\ref{L4}), or $\LL = \BB+ \AA$ in the case $\Sg = \Sl$ ( from Assumption \ref{LE}), the following degenerate dissipativity estimate holds for some $\lambda > 0$:
			$$\Re \la \LL f, f \ra_\SSg + \lambda \| f \|_{\SSgp}^2 \lesssim \| f \|^2_\SSg.$$
			We now write an energy estimate, using the skew-adjointness of $v \cdot \nabla_{x}$:
			$$\frac{1}{2} \frac{\d}{\d t} \| u \|^2_{\SSg} + \frac{\lambda}{\eps^2} \| u \|_{\SSgp}^2 \lesssim \| u \|^2_{\SSg} + \left| \la \PP^\eps_\kin \varphi, u \ra_\SSg \right|.$$
			Recall that $\PP^\eps_\kin = \Id - \PP^\eps_\hyd$, where we know from the spectral analysis performed in Theorems \ref{thm:spectral_study} or \ref{thm:enlarged_thm} that $\PP^\eps_\hyd \in \BBB(\SSgm ; \SSgp) \subset \BBB(\SSgm)$ uniformly in $\eps$ from the embedding $\SSgp \hookrightarrow \SSg \hookrightarrow \SSgm$. Thus, we have $\PP^\eps_\kin \in \BBB(\SSgm)$ uniformly in $\eps$, from which we deduce
			$$\frac{1}{2} \frac{\d}{\d t} \| u \|^2_{\SSg} + \frac{\lambda}{\eps^2} \| u \|_{\SSgp}^2 \lesssim \| u \|^2_{\SSg} + \| \PP^\eps_\kin \varphi \|_{\SSgm} \| u \|_{\SSgp} \lesssim \| u \|^2_{\SSg} + \| \varphi \|_{\SSgm} \| u \|_{\SSgp}.$$
			Therefore, multiplying by $e^{2\sigma t / \eps^2}$, we obtain
			\begin{align*}
				\frac{1}{2} \frac{\d}{\d t} \left(e^{2 \sigma t / \eps^2} \| u \|^2_{\SSg}\right) + & \frac{\lambda}{\eps^2} e^{2 \sigma t / \eps^2} \| u \|_{\SSgp}^2 \\
				& \lesssim \frac{1}{\eps^2} e^{2 \sigma t / \eps^2} \| u \|^2_{\SSg} + \eps \left(e^{\sigma t / \eps^2} \| \varphi \|_{\SSgm}\right) \left( \frac{1}{\eps} e^{\sigma t / \eps^2} \| u \|_{\SSgp}\right),
			\end{align*}
			or more simply by Young's inequality
			\begin{align*}
				\frac{1}{2} \frac{\d}{\d t} \left(e^{2 \sigma t / \eps^2} \| u \|^2_{\SSg}\right) + & \frac{\lambda}{2 \eps^2} e^{2 \sigma t / \eps^2} \| u \|_{\SSgp}^2
				\lesssim \frac{1}{\eps^2} e^{2 \sigma t / \eps^2} \| u \|^2_{\SSg} + \eps^2 e^{2 \sigma t / \eps^2} \| \varphi \|_{\SSgm}^2.
			\end{align*}
			Integrating in time, we finally deduce from the previous step
			$$\Nt u \Nt^2_{\SSSl(\sigma, \eps)} \lesssim \eps^2 \int_0^T e^{2 \sigma t / \eps^2} \| \varphi(t) \|_{\SSgm}^2\d t .$$
			This concludes the proof.
		\end{proof}
		
		\section{Bilinear theory}\label{sec:Bilin}
		
		We come now to the main nonlinear estimates involbing the various stiff terms
		$$\Psi^\eps[f,g]=\frac{1}{\eps^2}U^\eps \ast \QQ^\sym(f,g).$$
		We will exploit the decomposition of $ U^\eps(t) $ given in \eqref{eq:decompUeps} and the associated nonlinear decomposition 
		$$		\Psi^\eps[f,g](t) = \Psi^\eps_\hyd [f, g](t) + \Psi^\eps_\kin [f, g](t),$$
		with
		$$		\Psi^\eps_\star [f, g](t) := \PP^\eps_\star \Psi^\eps[f,g](t) = \frac{1}{\eps} \int_0^t U^\eps_\star(t - \tau) \QQ^\sym (f(\tau), g(\tau)) \d \tau.$$
		We first need the following spatially inhomogeneous nonlinear estimates of $\QQ$.
		\begin{lem}[\textit{\textbf{Nonlinear Sobolev estimates for $\QQ$}}]
			\label{lem:Q_sobolev}
			Denote $\Sg = \Ss$ under assumption \ref{Bbound}, or $\Sg = \Sl$ under assumption~\ref{BE}.
			Consider $s > \frac{d}{2}$ and recall that~$\alpha \in \left(0, \frac{1}{2}\right)$ if $d = 2$, or $\alpha = 0$ if $d \ge 3$. There holds
			\begin{subequations}
				\label{eq:Q_refined_sobolev_algebra}
				\begin{gather}
					\label{eq:Q_refined_sobolev_negative_algebra_inequality}
					\begin{aligned}
						\| \QQ(f, g) \|_{ \dot{\mathbb{H}}^{-\alpha}_x \left( \Sgm_v \right) } & + \| \QQ(f, g) \|_{\rSSgm{s} } \\
						\lesssim& \| f \|_{ \rSSg{s} } \| | \nabla_x |^{1-\alpha} g \|_{ \rSSgp{s-(1-\alpha)} }
						+ \| | \nabla_x |^{1-\alpha} f \|_{ \rSSgp{s-(1-\alpha)} } \| g \|_{ \rSSg{s} },
					\end{aligned} \\
					\label{eq:Q_refined_sobolev_negative_algebra_inequality_same}
					\begin{aligned}
						\| \QQ(f, g) \|_{ \dot{\mathbb{H}}^{-\alpha}_x \left( \Sgm_v \right) } & + \| \QQ(f, g) \|_{\rSSgm{s} } \\
						\lesssim & \| f \|_{ \rSSg{s} } \| | \nabla_x |^{1-\alpha} g \|_{ \rSSgp{s-(1-\alpha)} }
						+ \| f \|_{ \rSSgp{s} } \| | \nabla_x |^{1-\alpha} g \|_{ \rSSg{s- (1-\alpha) } }.
					\end{aligned}
				\end{gather}
			\end{subequations}
			Furthermore, we have the following control when $g \in W^{s', \infty}_x \left( \Sgp_v \right)$ for $s' > s$:
			\begin{equation}
				\label{eq:Q_sobolev_algebra_Holder}
				\| \QQ(f, g) \|_{\rSSgm{s} } \lesssim \| f \|_{ \rSSg{s} } \| g \|_{ W^{s', \infty}_x \left( \Sgp_v \right) } + \| f \|_{ \rSSgp{s} } \| g \|_{ W^{s', \infty}_x \left( \Sg_v \right) }.
			\end{equation}
		\end{lem}
	\begin{rem} Because the estimates provided in the Lemma are involving fractional Sobolev spaces in the variable $x$, and due to the locality in $x$ of the bilinear operator $\QQ$, the proof requires some estimates reminiscent to paradifferential calculus. This is not the case when dealing with mere $\mathbb{H}^{k}_{x}(X_{v})$ spaces with $k \in \N$ as in Lemma \ref{lem72} where $L^{p}$-estimates and Sobolev embeddings will allow to recover the needed estimates.\end{rem}		
\begin{proof} As just said, due to the locality in $x$ of the bilinear operator $\QQ$, one can adapt classical results from paradifferential calculus, replacing the multiplication $(u, v) \mapsto uv$ (resp. the modulus $|\cdot|$) by the collision operator $(u, v) \mapsto \QQ(u, v)$ (resp. the $\Sgm$-norm). In particular, we redefine the homogeneous paraproduct and remainder (see Appendix \ref{scn:littlewood-paley}) as
			$$\dot{T}_u v= \sum_{j} \QQ\left( \dot{S}_{j-1} u , \dot{\Delta}_j v \right), \qquad \dot{R}(u, v) = \sum_{|j - k| \le 1} \QQ\left( \dot{\Delta}_k u , \dot{\Delta}_j v \right),$$
			which satisfy, under the assumption \ref{Bbound} or \ref{BE}, the estimates
			$$  \left\| \QQ\left( \dot{S}_{j-1} u , \dot{\Delta}_j v \right) \right\|_{ L^p_x \left( \Sgm_v \right) } \lesssim
			\left\| \| \dot{S}_{j-1} u \|_{ \Sgp_v } \| \dot{\Delta}_j v \|_{ \Sg_v } \right\|_{ L^p_x } +
			\left\| \| \dot{S}_{j-1} u \|_{ \Sg_v } \| \dot{\Delta}_j v \|_{ \Sgp_v } \right\|_{ L^p_x },$$
			$$	
			\left\| \QQ\left( \dot{\Delta}_k u , \dot{\Delta}_j v \right) \right\|_{ L^p_x \left( \Sgm_v \right) } \lesssim
			\left\| \| \dot{\Delta}_k u \|_{ \Sgp_v } \| \dot{\Delta}_j v \|_{ \Sg_v } \right\|_{ L^p_x } +
			\left\| \| \dot{\Delta}_k u \|_{ \Sg_v } \| \dot{\Delta}_j v \|_{ \Sgp_v } \right\|_{ L^p_x }.
			$$
			One then checks that \eqref{eq:Q_sobolev_algebra_Holder} is the $\QQ$-version of Proposition \ref{prop:product_sobolev_holder}. Furthermore, denoting for compactness $\boldsymbol{\alpha} = 1-\alpha \in \left(0, \frac{d}{2}\right)$, one gets from the $\QQ$-version of Proposition~\ref{prop:product_homogeneous_sobolev} with $s_1 = \frac{d}{2} -1 \in \big[ 0, \frac{d}{2} \big)$ and $s_2 = \boldsymbol{\alpha}$, so that $s_1 + s_2 - \frac{d}{2} = - \alpha$:
			\begin{equation*}\begin{split}
					\| \QQ(f, g) \|_{\dot{\mathbb{H}}^{-\alpha}_x \left( \Sgm_v \right) }
					&\lesssim   \| f \|_{ \dot{\mathbb{H}}_x^{\frac{d}{2} - 1} \left( \Sg_v \right) } \left\| | \nabla_{x} |^{\boldsymbol{\alpha}} g \right\|_{ L^2_x \left( \Sgp_v \right) }
					+  \| | \nabla_x |^{\boldsymbol{\alpha}} f \|_{ L^2_x\left( \Sgp_v \right) } \|  g \|_{ \dot{\mathbb{H}}^{\frac{d}{2} - 1}_x \left( \Sg_v \right) } \\
					&\lesssim  \| f \|_{ \mathbb{H}^s_x \left( \Sg_v \right) } \left\| | \nabla_{x} |^{\boldsymbol{\alpha}} g \right\|_{ L^2_x \left( \Sgp_v \right) }
					+ \left\| | \nabla_{x} |^{\boldsymbol{\alpha}} f \right\|_{ L^2_x \left( \Sgp_v \right) } \| g \|_{ \mathbb{H}^s_x \left( \Sg_v \right) }.
			\end{split}\end{equation*}
			
			We now turn to the estimate in $\rSSgm{s} = \mathbb{H}^s_x \left( \Sgm_v \right)$. Using \ref{Bbound} or \ref{BE}, we have
			\begin{align*}
				\| \QQ(f, g) \|_{ \mathbb{H}^s_x \left( \Sgm_v \right) } = & \left\| \la \xi \ra^s \int_{\R^{d}} \QQ\left( \widehat{f}(\xi - \zeta) , \widehat{g}(\zeta) \right) \d \zeta \right\|_{ L^2_\xi \left( \Sgm_v \right) } \\
				\lesssim & \left\| \la \xi \ra^s \int_{\R^{d}} \| \widehat{f}(\xi - \zeta) \|_{ \Sgp_v } \| \widehat{g}(\zeta) \|_{ \Sg_v } \d \zeta \right\|_{ L^2_\xi } \\
				& + \left\| \la \xi \ra^s \int_{\R^{d}} \| \widehat{f}(\xi - \zeta) \|_{ \Sg_v } \| \widehat{g}(\zeta) \|_{ \Sgp_v } \d \zeta \right\|_{ L^2_\xi } =: I_1 + I_2.
			\end{align*}
			We split the frequency weight as
			$$\la \xi \ra^s \approx 1 + | \xi |^s \lesssim 1 + | \xi - z |^s + | \zeta |^s \lesssim | \xi - z|^{\boldsymbol{\alpha}} \la \xi - z \ra^{s - \boldsymbol{\alpha}} + \la z\ra^s,$$
			since $s-\boldsymbol{\alpha} >0,$
			which allows to control the term $I_1$ as follows:
			\begin{equation*}\begin{split}
					I_1 &\lesssim  \left\|   \int_{\R^{d}} \Big[ | \xi - z |^{\boldsymbol{\alpha}} \la \xi - z \ra^{s - \boldsymbol{\alpha}} \| \widehat{f}(\xi -z) \|_{ \Sgp_v } \Big] \| \widehat{g}(z) \|_{ \Sg_v } \d z \right\|_{ L^2_\xi } \\
					&\phantom{+++++*} +  \left\|  \int_{\R^{d}} \| \widehat{f}(\xi -z) \|_{ \Sgp_v } \Big[\la \zeta \ra^s \| \widehat{g}(z) \|_{ \Sg_v }\Big] \d z \right\|_{ L^2_\xi }.
			\end{split}\end{equation*}
			Using Young's convolution inequality $L^2_\xi  \ast  L^1_\xi \hookrightarrow L^2_\xi$   
			we deduce that
			\begin{equation*}\begin{split}
					I_1 &\lesssim \left\|\,|\xi|^{\boldsymbol{\alpha}}\langle\xi\rangle^{\boldsymbol{\alpha}}\,
					\|\widehat{f}\|_{\Sgp_v}\,\right\|_{L^2_\xi}
					\left\|\,\|\widehat{g}(\xi)\|_{\Sg_v}\right\|_{L^1_\xi}+
					\left\|\,\|\widehat{f}(\xi)\|_{\Sgp_v}\right\|_{L^1_\xi}\,\left\|\,\langle\xi\rangle^{s}\,\|\widehat{g}\|_{\Sg_v}\right\|_{L^2_\xi}\\
					&\lesssim \left\|\,|\nabla_x|^{\boldsymbol{\alpha}}\,f\right\|_{H^{s-\boldsymbol{\alpha}}_x(\Sgp_v)}\,\|\widehat{g}\|_{L^1_\xi(\Sg_v)}+\left\|\widehat{f}\right\|_{L^1_\xi(\Sgp_v)}\,\left\|g\right\|_{\mathbb{H}^s_x(\Sg_v)}.
			\end{split}\end{equation*}
			Using the fact that~$\la \xi \ra^{ -s } \in L^2_\xi$ (resp. $| \xi |^{ -\boldsymbol{\alpha} } \la \xi \ra^{-(s - \boldsymbol{\alpha}) } \in L^2_\xi$), a simple use of Cauchy-Schwarz's inequality allows to estimate $L^1_\xi$-norms with weighted $L^2_\xi$-norms resulting in		
			$$ I_1 	\lesssim  \| | \nabla_x |^{\boldsymbol{\alpha}} f \|_{ H^{ s-\boldsymbol{\alpha} }_x \left( \Sgp_v \right) } \| g \|_{ \mathbb{H}^s_x \left( \Sg_v \right) }.$$
			We prove in the exact same way that
			$$I_2 \lesssim   \| f \|_{ \mathbb{H}^s_x \left( \Sg_v \right) } \| | \nabla_x |^{\boldsymbol{\alpha}} g \|_{ H^{ s-\boldsymbol{\alpha} }_x \left( \Sgp_v \right) },$$
			thus \eqref{eq:Q_refined_sobolev_negative_algebra_inequality} is proved, and the proof of \eqref{eq:Q_refined_sobolev_negative_algebra_inequality_same} is similar. The proofs of \eqref{eq:Q_sobolev_algebra_degenerate} and \eqref{eq:Q_sobolev_algebra_degenerate_closed} are also similar. This concludes the proof.
		\end{proof}
		
		\subsection{Bilinear and linear hydrodynamic estimates}
		We have all in hands to estimate the bilinear ‘‘hydrodynamic'' operator $\Psi^\eps_\hyd(f,g)=\frac{1}{\eps}U_\hyd^\eps(\cdot) \ast \QQ^\sym(f,g)$. \textbf{The results of this section hold under any assumption \ref{Bbound}, \ref{BE} or \ref{BED}}. They are based upon the above properties of $\QQ$ as well as the results of Section \ref{scn:study_physical_space} on the various semigroups involved:
		\begin{prop}[\textit{\textbf{General bilinear hydrodynamic estimates}}]
			\label{prop:bilinear_hydrodynamic}
			The bilinear operator~$\Psi^\eps_\hyd$ satisfies the following continuity estimates in $\SSSs$ when at least one argument is in $\SSSs$:
			\begin{subequations}
				\label{eq:bilinear_hyd_hyd}
				\begin{gather}
					\label{eq:bilinear_hyd_hyd_hyd}
					\Nt \Psi^\eps_\hyd [f, g] \Nt_\SSSs \lesssim w_{\phi, \eta}(T)^{-1} \Nt f \Nt_\SSSs \Nt g \Nt_\SSSs,\\
					\label{eq:bilinear_hyd_hyd_mix}
					\Nt \Psi^\eps_\hyd [f, g] \Nt_\SSSs \lesssim \eps w_{\phi, \eta}(T)^{-1} \Nt f \Nt_\SSSs \Nt g \Nt_\SSSm,\\
					\label{eq:bilinear_hyd_hyd_kin}
					\Nt \Psi^\eps_\hyd [f, g] \Nt_\SSSs \lesssim \eps \Nt f \Nt_\SSSs \Nt g \Nt_\SSSl,
				\end{gather}
			\end{subequations}
			as well as as the following ones when at least one argument is in $\SSSm$:
			\begin{subequations}
				\label{eq:bilinear_hyd_mix}
				\begin{gather}
					\label{eq:bilinear_hyd_mix_mix}
					\Nt \Psi^\eps_\hyd [f, g] \Nt_\SSSs \lesssim \eps w_{\phi, \eta}(T)^{-1} \Nt f \Nt_\SSSm \Nt g \Nt_\SSSm,\\
					\label{eq:bilinear_hyd_mix_kin}
					\Nt \Psi^\eps_\hyd [f, g] \Nt_\SSSs \lesssim \eps \Nt f \Nt_\SSSm \Nt g \Nt_\SSSl,
				\end{gather}
			\end{subequations}
			and the following one when both arguments are in $\SSSl$:
			\begin{equation}
				\label{eq:bilinear_hyd_kin_kin}
				\Nt \Psi^\eps_\hyd [f, g] \Nt_\SSSs \lesssim \eps \Nt f \Nt_\SSSl \Nt g \Nt_\SSSl.
			\end{equation}
			Furthermore, it is strongly continuous at $t=0$:
			$$\lim_{t \to 0}\| \Psi_\hyd^\eps(f, g)(t) \|_{\SSs}=0$$
			in all cases considered above.
		\end{prop}
		
		\begin{proof}
			We recall the definition of $\Psi^\eps_\hyd$ and the orthogonality property of $\QQ$:
			$$\Psi^\eps_\hyd [f, g](t) = \frac{1}{\eps} \int_0^t U^\eps_\hyd(t-\tau) \QQ(f(\tau), g(\tau)) \d \tau, \qquad \PP \QQ = 0,$$
			thus, denoting for compactness $w = w_{\phi, \eta}$, the convolution estimate \eqref{eq:decay_semigroups_hydro_orthogonal} gives
			\begin{equation}
				\label{eq:pre_estimate_hyd}
				\Nt \Psi^\eps_\hyd [f, g] \Nt_{\rSSSs{s}}^2 \lesssim \sup_{0 \le t < T} \left\{w(t)^2 \int_0^t \| \QQ(f(\tau), g(\tau) ) \|_{  \rSSgm{s} \cap \dot{\mathbb{H}}^{-\alpha}_x \left(\Sgm_v \right) }^2 \d \tau\right\},
			\end{equation}
			where $(\SSg, \Sg) = (\SSs, \Ss), (\SSl, \Sl)$ or $(\SSl_{-1}, \Sl_{-1})$, and we recall that $w$ is non-increasing and bounded from above and below:
			\begin{equation}
				\label{eq:bound_w}
				\forall 0 \leq t_1 \le t_2 < T, \quad 0 < w(T) \le w(t_2) \le w(t_1) \le 1.
			\end{equation}
			The continuity at $t = 0$ will be an easy consequence of the estimate \eqref{eq:pre_estimate_hyd}.
			
			\step{1}{Proof of \eqref{eq:bilinear_hyd_hyd} for $f \in \SSSs$}
			When $g \in \SSSs$, we combine \eqref{eq:pre_estimate_hyd} with the bilinear estimate~\eqref{eq:Q_refined_sobolev_negative_algebra_inequality} for $\QQ$, to deduce the following, where $\boldsymbol{\alpha} := 1 - \alpha$
			\begin{align*}
				\Nt \Psi^\eps_\hyd [f, g] \Nt_\SSSs^2
				\lesssim & \sup_{ 0 \le t < T }\left\{w(t)^2\int_0^t   \Big[ \|f(\tau) \|_{ \rSSs{s} }^2 \left\| |\nabla_x|^{\boldsymbol{\alpha}} g(\tau) \right\|_{\rSSsp{s-{\boldsymbol{\alpha}}}}^2  \right.\\
				& \phantom{+++++} \left.+ \left\| |\nabla_x|^{\boldsymbol{\alpha}} f(\tau) \right\|_{\rSSsp{s-{\boldsymbol{\alpha}}}}^2 \| g(\tau) \|_{ \rSSs{s} }^2 \Big] \d \tau\right\}\\
				\lesssim & \sup_{ 0 \le t < T }\int_0^t  \Big[w^2(\tau)\|f(\tau) \|_{ \rSSs{s} }^2 \left\| |\nabla_x|^{\boldsymbol{\alpha}} g(\tau) \right\|_{\rSSsp{s-{\boldsymbol{\alpha}}}}^2 \\
				&\phantom{+++++} +\left\| |\nabla_x|^{\boldsymbol{\alpha}} f(\tau) \right\|_{\rSSsp{s-{\boldsymbol{\alpha}}}}^2 w(\tau)^2\| g(\tau) \|_{ \rSSs{s} }^2 \Big] \d \tau
			\end{align*}
			where we used \eqref{eq:bound_w} in the second inequality. Recalling that
			$$\Nt h \Nt_{\rSSSs{s}}^2 := \sup_{0 \le t < T} \left\{ w(t)^2 \| h(t) \|_{\rSSsp{s}}^2 + w(t)^2 \int_0^t \left\| |\nabla_x|^{1-\alpha} h( \tau ) \right\|_{\rSSsp{s}}^2 \d \tau \right\},$$
			and using $\SSsp \hookrightarrow \SSs$ and \eqref{eq:bound_w}, we deduce
			\begin{equation*}\begin{split}
					\Nt \Psi^\eps_\hyd(f,g) \Nt_\SSSs^2 &\lesssim
					\Nt f \Nt_{ \SSSs }^2 \int_0^T \left\| |\nabla_x|^{\boldsymbol{\alpha}} g(t) \right\|_{\rSSsp{s-{\boldsymbol{\alpha}}}}^2 \d t
					+ \Nt g \Nt_{ \SSSs }^2 \int_0^T \left\| |\nabla_x|^{\boldsymbol{\alpha}} f(t) \right\|_{\rSSsp{s-{\boldsymbol{\alpha}}}}^2 \d t\\
					& \lesssim w(T)^{-2} \Nt f \Nt_\SSSs^2 \Nt g \Nt_\SSSs^2
			\end{split}\end{equation*}
			which is exactly ~\eqref{eq:bilinear_hyd_hyd_hyd}.
			When $g \in \SSSm$, using  furthermore $\SSsp \hookrightarrow \SSs$, we similarly have \eqref{eq:bilinear_hyd_hyd_mix}:
			\begin{equation*}\begin{split}
					\Nt \Psi^\eps_\hyd [f, g] \Nt_\SSSs^2  &\lesssim \eps^2 \int_0^T \left( w(\tau)\|f(\tau) \|_{ \SSsp } \right)^2  \left(\frac{1}{\eps} \| g(\tau) \|_{\SSsp}\right)^2 \d \tau \\
					&\lesssim  \eps^2 \Nt f \Nt_\SSSs^2 \int_0^T \left(\frac{1}{\eps} \| g(\tau) \|_{\SSsp}\right)^2 \d \tau\lesssim  \eps^2 w(T)^{-2} \Nt f \Nt_\SSSs^2 \Nt g \Nt_\SSSm^2,
			\end{split}\end{equation*}
			which gives now \eqref{eq:bilinear_hyd_hyd_mix}. In the same way, using also $\SSsp \hookrightarrow \SSlp$, we have \eqref{eq:bilinear_hyd_hyd_kin}.
			
			\step{2}{Proof of \eqref{eq:bilinear_hyd_mix} and \eqref{eq:bilinear_hyd_kin_kin}}
			When $f \in \SSSm$ and $g \in \SSSl$, we combine~\eqref{eq:pre_estimate_hyd} with the bilinear estimate \eqref{eq:Q_refined_sobolev_negative_algebra_inequality} for $\QQ$. Using that $\SSsp \hookrightarrow \SSlp$ and the property \eqref{eq:bound_w} of~$w$, we have:
			\begin{equation*}\begin{split}
					\Nt \Psi^\eps_\hyd (f, g) \Nt_\SSSs^2  
					\lesssim &  \eps^2 \sup_{0 \le t < T} \Bigg\{w(t)^2 \int_0^t \left[ \| f(\tau) \|_{ \SSs }^2 \left( \frac{1}{\eps} \| g(\tau) \|_{\SSlp}\right)^2 \right. \\
					& \left.\phantom{+++++++++}  + \left( \frac{1}{\eps} \| f(\tau) \|_{\SSsp} \right)^2 \| g(\tau) \|_{ \SSl }^2 \, \right] \d \tau \Bigg\} \\
					\lesssim & \eps^2 \int_0^T w(\tau)^2 \|f(\tau) \|_{ \SSs }^2 \left( \frac{1}{\eps} \| g(\tau) \|_{\SSlp}\right)^2 \d t\\
					& + \eps^2 \sup_{0 \le t < T} \Bigg\{w(t)^2 \int_0^t \left( \frac{1}{\eps} \| f(\tau) \|_{\SSsp} \right)^2 \| g(\tau) \|_{ \SSl }^2 \d \tau \Bigg\} \\
					\lesssim & \eps^2 \Nt f \Nt_\SSSm^2 \Nt g \Nt_\SSSl^2.
			\end{split}\end{equation*}
			This proves  ~\eqref{eq:bilinear_hyd_mix_kin}
			The proofs of \eqref{eq:bilinear_hyd_mix_mix} and \eqref{eq:bilinear_hyd_kin_kin} are similar and omitted. 
		\end{proof}
		
		\begin{prop}[\textit{\textbf{Special bilinear hydrodynamic estimates}}]
			\label{prop:special_bilinear_hydrodynamic}
			When $f \in \SSSs$ and $\phi$ is the parameter defining the $\SSSs$-norm
			\begin{equation}
				\label{eq:bilinear_hyd_phi}
				\Nt \Psi^\eps_\hyd[f, \phi] \Nt_\SSSs \lesssim \eta \Nt f \Nt_\SSSs\,.
			\end{equation}
			Furthermore, when $g^\eps_\disp = U^\eps_\disp(\cdot) g$ where $g = \PP_{\Wave} g \in \rSSs{s} \cap \dot{\mathbb{H}}^{-\alpha}_x \left( \Ss_v \right)$, there holds
			\begin{equation}
				\label{eq:bilinear_hyd_disp}
				\Nt \Psi^\eps_\hyd[ f , g^\eps_\disp] \Nt_\SSSs \lesssim  {\beta_{\disp}}(g, \eps) \Nt f \Nt_{ \SSSs }, \qquad \lim_{\eps \to 0} {\beta_{\disp}}(g, \eps)=0,
			\end{equation}
			and in the case $d \ge 3$, the rate of convergence is explicit assuming $g \in \dot{\mathbb{B}}^{s' + (d+1)/2}_{1, 1} \left( \Ss_v \right) \cap \rSSs{s}$ for some $s'>s$:
			\begin{equation*}
				 {\beta_{\disp}}(g, \eps) \lesssim \sqrt{\eps}
				\left( \| g \|_{ \rSSs{s} } + \| g \|_{ \dot{\mathbb{B}}^{s' + (d+1)/2 }_{1, 1} \left( \Ss_v \right) }\right).
			\end{equation*}
		\end{prop}
		
		\begin{proof}
			We start by proving \eqref{eq:bilinear_hyd_phi} and then prove \eqref{eq:bilinear_hyd_disp}. We use again the shorthand notation~$w = w_{\phi, \eta}$ and recall that, besides \eqref{eq:pre_estimate_hyd}, the convolution estimate
			\eqref{eq:decay_semigroups_hydro_orthogonal} also leads to
			\begin{equation}
				\label{eq:pre_estimate_2_pos}
				\Nt \Psi^\eps_\hyd [f, g] \Nt_{\rSSSs{s}}^2 \lesssim \sup_{0 \le t < T} \left\{w(t)^2 \left(\int_0^t \|  \QQ(f(\tau), g(\tau) ) \|_{  \SSsm }^{ \frac{2}{1+\alpha} } \d \tau \right)^{  1+\alpha  } \right\},
			\end{equation}	
			
			%\begin{equation}
			%	\label{eq:bound_w_2}
			%	\forall t_1 \le t_2, \quad 0 < w(T) \le w(t_2) \le w(t_1) %\le 1.
			%	\end{equation}
			
			\step{1}{Proof of \eqref{eq:bilinear_hyd_phi}}
			We combine the convolution estimate \eqref{eq:pre_estimate_hyd} with the nonlinear bound \eqref{eq:Q_refined_sobolev_negative_algebra_inequality_same}, and use $\SSsp \hookrightarrow \SSs$ to obtain (where we denote for compactness ${\boldsymbol{\alpha}} = 1-\alpha$)
			\begin{equation*}\begin{split}
					\Nt \Psi^\eps_\hyd [f, \phi] \Nt_\SSSs^2 
					&\lesssim  \sup_{0 \le t < T} \left\{ w(t)^2 \int_0^t \Big[w(\tau)\|  f(\tau) \|_{ \rSSsp{s} }\Big]^2 \Big[w(\tau)^{-1} \| |\nabla_x|^{\boldsymbol{\alpha}} \phi(\tau) \|_{\rSSsp{s-{\boldsymbol{\alpha}}}}\Big]^2 \d \tau \right\}\\
					& \lesssim \Nt f \Nt_\SSSs^2 \sup_{0 \le t < T} \Bigg\{ w(t)^2 \int_0^t \Big[w(\tau)^{-1} \| |\nabla_x|^{\boldsymbol{\alpha}} \phi(\tau) \|_{\rSSsp{s-\boldsymbol{\alpha}}}\Big]^2 \d \tau \Bigg\}\,
			\end{split}\end{equation*}
			and then, using \eqref{eq:NS_exponential_weight}, we finally get
			$
			\Nt \Psi^\eps_\hyd (f, \phi) \Nt_\SSSs^2
			\lesssim \eta^2 \Nt f \Nt_\SSSs^2.
			$
			which is exactly \eqref{eq:bilinear_hyd_phi}.	
			
			\step{2}{Proof of \eqref{eq:bilinear_hyd_disp}}
			As in the previous step, but using the nonlinear bound~\eqref{eq:Q_refined_sobolev_negative_algebra_inequality_same} for $\QQ$, one has
			\begin{align*}
				\Nt \Psi^\eps_\hyd[ f , g^\eps_\disp]   \Nt^2_\SSSs & \lesssim \Nt f \Nt_{ \SSSs}^2 \int_0^T \left\| | \nabla_{x} |^{\boldsymbol{\alpha}} g^\eps_\disp(t) \right\|^2_{ \rSSsp{s-{\boldsymbol{\alpha}}} } \d t,
			\end{align*}
			which, according to \eqref{eq:graU_*} and \eqref{eq:XIALPHA}, satisfies for some universal $\kappa > 0$
			\begin{equation}
				\label{eq:continuity_hyd_disp}
				\Nt \Psi^\eps_\hyd[ f , g^\eps_\disp]  \Nt^2_\SSSs \lesssim \Nt f \Nt_{ \SSSs}^2 \left( \| g \|_{ \rSSs{s} } + \| g \|_{ \dot{\mathbb{H}}^{-\alpha}_x \left( \Ss_v \right) } \right)^2 \sup_{\xi \in \supp \| \widehat{g} \|_{\Ssm_v} } \int_0^T | \xi |^2 e^{-t \kappa | \xi |^2 } \d t.
			\end{equation}
			Since the supremum term is bounded uniformly in $T \in (0, \infty)$ and $g \in \rSSsm{s} \cap \dot{\mathbb{H}}^{-\alpha}_x \left( \Ss_v \right)$, it is enough to prove \eqref{eq:bilinear_hyd_disp} in the case where $g \in \rSSsm{s} \cap \dot{\mathbb{B}}^{s' + (d+1)/2 }_{1, 1} \left( \Ssm_v \right)$ for some $s' > s$ and~$\xi \mapsto \| \widehat{g}(\xi) \|_{ \Ssm_v }$ is supported away from $0$, as it will allow to conclude by a density argument. 
			We assume therefore there is some $\nu > 0$ such that
			$$\forall | \xi | \le \nu, \quad \widehat{g}(\xi) = 0,$$
			and we point out that when $T < \infty$, using $e^{-r} \lesssim r^{-1}$, for any $R > 0$
			$$\sup_{|\xi | \ge \nu} \int_R^T | \xi|^2 e^{-\kappa t | \xi |^2} \d t \lesssim \int_{ R }^{ T } \frac{\d t}{t} = \log(T) - \log(R),$$
			and when $T = \infty$
			$$\sup_{|\xi | \ge \nu} \int_R^\infty | \xi|^2 e^{-\kappa t | \xi |^2} \d t = \sup_{|\xi | \ge \nu } \int_{ R | \xi |^2 }^\infty e^{- \kappa t } \d t \lesssim \exp\left( - R \nu^2 \right),$$
			so that, in both cases, we have
			$$C_\nu(R, T) := \sup_{|\xi | \ge \nu } \int_R^T | \xi|^2 e^{-\kappa t | \xi |^2} \d t \xrightarrow[]{R \to T} 0.$$
			We split for some $0 < R < T$ the nonlinear term:
			$$\QQ(g^\eps_\disp, f) = \mathbf{1}_{0 \le t \le R} \QQ(g^\eps_\disp, f) + \mathbf{1}_{R \le t < T} \QQ(g^\eps_\disp, f) =: \varphi^-(t) + \varphi^+(t),$$
			so that, using an estimate analogous to \eqref{eq:continuity_hyd_disp} we have
			\begin{align*}
				\left(\frac{1}{\eps} \Nt U^\eps_\hyd(\cdot) * \varphi^+ \Nt_\SSSs\right)^2 \lesssim & 
				\Nt f \Nt_{ \SSSs}^2 \left( \| g \|_{ \rSSs{s} } + \| g \|_{ \dot{\mathbb{H}}^{-\alpha}_x \left( \Ss_v \right) } \right)^2 	C_\nu(R, T).
			\end{align*}
			Furthermore, combining this time \eqref{eq:pre_estimate_2_pos} with the boundedness estimate \eqref{eq:decay_semigroups_hydro} (resp. the dispersive estimate \eqref{eq:dispersive}) $U^\eps_\disp(\cdot)$, together with the corresponding nonlinear bound \eqref{eq:Q_refined_sobolev_algebra} (resp.~\eqref{eq:Q_sobolev_algebra_Holder}) for $\QQ$, we have for $\varphi^-(t)$
			\begin{equation*}
				\label{eq:continuity_hyd_disp_vanish}
				\left(\frac{1}{\eps} \Nt U^\eps_\hyd(\cdot) * \varphi^- \Nt_\SSSs \right)^2\lesssim \Nt f \Nt_{\SSSs}^2 \left( \| g \|_{ \rSSsm{s} }^2 + \| g \|_{ \dot{\mathbb{B}}^{s + (d+1)/2 }_{1, 1} \left( \Ssm_v \right) }^2 \right) \left(\int_0^R 1 \land \left(\frac{\eps}{t}\right)^{\frac{d-1}{1+\alpha}} \d t\right)^{1+\alpha}.
			\end{equation*}
			Put together, the two previous controls yield uniformly in $R \in (0, T)$
			\begin{align*}
				\Nt \Psi^\eps_\hyd[ f , g^\eps_\disp]  \Nt_\SSSs^2 \lesssim & 
				\Nt f \Nt_{\SSSs}^2 \left( \| g \|_{ \rSSsm{s} }^2 + \| g \|_{ \dot{\mathbb{B}}^{s + (d+1)/2 }_{1, 1} \left( \Ssm_v \right) }^2 \right) \left(\int_0^{R} 1 \land \left(\frac{\eps}{t}\right)^{\frac{d-1}{1+\alpha}} \d t\right)^{1+\alpha} \\
				& + \Nt f \Nt_{ \SSSs}^2 \left( \| g \|_{ \rSSs{s} } + \| g \|_{ \dot{\mathbb{H}}^{-\alpha}_x \left( \Ss_v \right) } \right)^2 C_\nu(R, T).
			\end{align*}
			Letting $\eps \to 0$ and then $R \to T$, one deduces \eqref{eq:bilinear_hyd_disp} for $g \in \rSSsm{s} \cap \dot{\mathbb{B}}^{s' + (d+1)/2 }_{1, 1} \left( \Ssm_v \right)$ with~$s' > s$ and whose Fourier transform is supported away from $0$. We conclude to the general case of~$g \in \rSSsm{s}$ by density thanks to \eqref{eq:continuity_hyd_disp}. Note that in dimension $d \ge 3$, there holds $\alpha = 0$, thus one has
			$$\int_0^T 1 \land \left(\frac{\eps}{t}\right)^{d-1} \d t =\int_0^\eps \d t+ \eps^{d-1}\int_{\eps}^T t^{1-d}\d t \lesssim \eps.$$
			This concludes the proof.
		\end{proof}
		
		\subsection{Bilinear kinetic and mixed estimates}
		The results of this section hold assuming \ref{Bbound}. We point out that only one estimate, namely \eqref{eq:bilinear_kin_kin_kin}, holds assuming \ref{Bbound} or \ref{BE} \textbf{but not \ref{BED}}, which is why it has to be treated using the alternative strategy of Section \ref{scn:hydrodynamic_limit_BED}.
		We have the analogue of Proposition \ref{prop:bilinear_hydrodynamic} for the kinetic bilinear operator $\Psi^\eps_\kin(f, g) = \frac{1}{\eps} U^\eps_\kin \ast \QQ^\sym(f, g)$.
		\begin{prop}[\textit{\textbf{General bilinear kinetic and mixed estimates}}]
			\label{prop:bilinear_kinetic}
			The bilinear operator $\Psi^\eps_\kin$ satisfies the following continuity estimates in the mixed space $\SSSm$ when at least one argument is in $\SSSm$:
			\begin{gather}
				\label{eq:bilinear_mix_mix_mix}
				\Nt \Psi^\eps_\kin  [f, g] \Nt_\SSSm \lesssim \eps w_{\phi, \eta}(T)^{-1} \Nt f \Nt_\SSSm \min\left\{\Nt g \Nt_\SSSm\,,\,\Nt g \Nt_\SSSs\right\}\,,
			\end{gather}
			and the following one when $f, g \in \SSSs$ (note the absence of a factor $\eps$):
			\begin{equation}
				\label{eq:bilinear_mix_hyd_hyd}
				\Nt \Psi^\eps_\kin  [f, g] \Nt_\SSSm \lesssim w_{\phi, \eta}(T)^{-1} \Nt f \Nt_\SSSs \Nt g \Nt_\SSSs.
			\end{equation}
			Furthermore, considering $\Sl = \Ss$ in the definition of $\SSSl$ under Assumption \ref{Bbound}, or considering $\Sl$ to be the space from \ref{BE} under this assumption, there holds in the kinetic space $\SSSl$ when at least one argument is in $\SSSl$
			\begin{subequations}
				\label{eq:bilinear_kin}
				\begin{gather}
					\label{eq:bilinear_kin_kin_kin}
					\Nt \Psi^\eps_\kin [f, g] \Nt_\SSSl \lesssim \eps \Nt f \Nt_\SSSl \Nt g \Nt_\SSSl,\\
					\label{eq:bilinear_kin_kin_mix}
					\Nt \Psi^\eps_\kin  [f, g] \Nt_\SSSl \lesssim \eps w_{\phi, \eta}(T)^{-1} \Nt f \Nt_\SSSl \min\left\{ \Nt g \Nt_\SSSm\,,\,\Nt g \Nt_\SSSs\right\}.
				\end{gather}
			\end{subequations}
			Finally, it is strongly continuous at $t = 0$ in the sense that in the corresponding cases
			$$\lim_{t\to 0}\| \Psi^\eps_\kin [f, g](t) \|_\SSl= 0, \qquad \lim_{t\to 0}\| \Psi^\eps_\kin [f, g](t) \|_\SSs = 0.$$
		\end{prop}
		
		\begin{proof}
			Recall the definition of $\Psi^\eps_\kin$:
			$$\Psi^\eps_\kin [f, g](t) = \frac{1}{\eps} \int_0^t U^\eps_\kin(t-\tau)  \QQ(f(\tau), g(\tau)) \d \tau,$$
			thus denoting for compactness $w_{\phi, \eta}(t) = w(t)$, the convolution estimates \eqref{eq:decay_convolution_semigroup_exp} and \eqref{eq:decay_convolution_semigroup_no_exp} give respectively
			\begin{equation}
				\label{eq:pre_estimate_kin}
				\Nt \Psi^\eps_\kin [f, g] \Nt_{ \SSSl }^2 \lesssim \int_0^T e^{2\sigma t / \eps^2} \left\| \QQ( f(t), g(t) ) \right\|^2_{ \SSlm } \d t,
			\end{equation}
			and
			\begin{equation}			\label{eq:pre_estimate_mix}
				\Nt \Psi^\eps_\kin [f, g] \Nt_{ \SSSm }^2 \lesssim \sup_{0 \le t < T} \left\{ w(t)^2 \int_0^T \left\| \QQ( f(\tau), g(\tau) ) \right\|^2_{ \SSsm } \d t \right\}.
			\end{equation}
			We also recall the bound \eqref{eq:bound_w} for $w$. The continuity at $t = 0$ will be immediate from the estimates below by letting $T \to 0$. For the reader convenience, we also recall the definitions of $\Nt\cdot\Nt_{\SSSl}$ and $\Nt \cdot \Nt_{\SSSm}$:
			$$	\Nt f \Nt_{ \SSSl}^2 := \sup_{0 \le t < T} \, e^{2 \sigma t / \eps^2 } \| f(t) \|_{\SSl}^2 + \frac{1}{\eps^2} \int_0^T e^{2 \sigma t / \eps^2} \| f(t) \|_{\SSlp}^2 \d t,$$
			$$\Nt g \Nt_{\SSSm}^2 := \sup_{0 \le t < T } \left\{ w(t)^2 \| g(t) \|_{\SSs}^2 + \frac{w(t)^2}{\eps^2}   \int_0^t \| g(\tau) \|_{\SSsp}^2 \d \tau \right\}.$$
			\step{1}{Proof of \eqref{eq:bilinear_kin} for $f \in \SSSl$}
			On the one hand, if $g \in \SSSl$, combining the estimate \eqref{eq:pre_estimate_kin} with the bilinear estimate \eqref{eq:Q_refined_sobolev_negative_algebra_inequality} for $\QQ$, one has:
			\begin{equation*}\begin{split}
					\Nt \Psi_\kin^\eps [f, g] \Nt_\SSSl^2
					& \lesssim \eps^2 \int_0^T \left\{ \left( \frac{1}{\eps} e^{\sigma t / \eps^2} \| f(t) \|_{\SSlp}\right)^2 \| g(t) \|_{ \SSl }^2 + \| f(t) \|_{ \SSl }^2 \left( \frac{1}{\eps} e^{\sigma t / \eps^2} \| g(t) \|_{\SSlp}\right)^2 \right\} \d t \\
					& \lesssim \eps^2 \Nt f \Nt_{\SSSl}^2 \Nt g \Nt_{\SSSl}^2,
			\end{split}\end{equation*}
			which is ~\eqref{eq:bilinear_kin_kin_kin}. Similarly, when $g \in \SSSm$, we have
			\begin{equation*}\begin{split}
					\Nt \Psi_\kin^\eps [f, g] \Nt_\SSSl^2 
					&\lesssim \eps^2 \int_0^T \left\{ \left( \frac{1}{\eps} e^{\sigma t / \eps^2} \| f(t) \|_{\SSlp}\right)^2 \| g(t) \|_{ \SSs }^2 + \left( e^{\sigma t / \eps^2}  \| f(t) \|_{ \SSl } \right)^2 \left( \frac{1}{\eps} \| g(t) \|_{\SSsp}\right)^2 \right\} \d t \\
					&\lesssim \eps^2 w(T)^{-2} \Nt f \Nt_{\SSSl}^2 \Nt g \Nt_{\SSSs}^2\,,	
			\end{split}\end{equation*}
			where we used \eqref{eq:bound_w}. This proves \eqref{eq:bilinear_kin_kin_mix} for $g \in \SSsp$. On the other hand, if $g \in \SSSs$, using furthermore $\SSlp \hookrightarrow \SSl$ and $\SSsp \hookrightarrow \SSlp$, we have 
			$$
			\Nt \Psi_\kin^\eps [f, g] \Nt_\SSSl^2
			\lesssim \eps^2\int_0^T \left( \frac{1}{\eps} e^{\sigma t / \eps^2} \| f(t) \|_{\SSlp}\right)^2 \| g(t) \|_{ \SSsp }^2 \d t \lesssim \eps^2 w(T)^{-2} \Nt f \Nt_{\SSSl}^2 \Nt g \Nt_{\SSSs}^2
			$$  which gives \eqref{eq:bilinear_kin_kin_mix}.

			\step{2}{Proof of \eqref{eq:bilinear_mix_mix_mix} for $f \in \SSSm$}
			In the case $f, g \in \SSSm$, combining the estimate \eqref{eq:pre_estimate_mix} with the bilinear estimate \eqref{eq:Q_refined_sobolev_negative_algebra_inequality}, and using the bound \eqref{eq:bound_w} for $w$, we have
			\begin{align*}
				\Nt & \Psi_\kin^\eps [f, g] \Nt_\SSSm^2 \\
				& \lesssim \eps^2 \int_0^T \left\{ \left(\frac{1}{\eps} \| f(t) \|_{\SSsp}\right)^2 \Big[ w(t) \|g(t) \|_{ \SSs }\Big]^2 + \Big[w(t)\| f(t) \|_{ \SSs }\Big]^2 \left(\frac{1}{\eps} \| g(t) \|_{\SSsp}\right)^2 \right\}  \d t 
			\end{align*}
			which readily gives  \eqref{eq:bilinear_mix_mix_mix} for $f,g\in \SSSm.$ 
			In the case $f \in \SSSm$ and $g \in \SSSs$, using furthermore $\SSsp \hookrightarrow \SSs$, we have  
			$$
			\Nt \Psi_\kin^\eps [f, g] \Nt_\SSSm^2 \lesssim  \eps^2 \int_0^T \left(\frac{1}{\eps} \| f(t) \|_{\SSsp}\right)^2 \| w(t) g(t) \|_{ \SSsp }^2 \d t 
			\lesssim \eps^2 w(T)^{-2} \Nt f \Nt_{\SSSm}^2 \Nt g \Nt_{\SSSs}^2.$$
			This shows \eqref{eq:bilinear_mix_mix_mix}.
			Similarly, using the nonlinear estimate~\eqref{eq:Q_refined_sobolev_negative_algebra_inequality_same} for $\QQ$, denoting for compactness ${\boldsymbol{\alpha}} = {1-\alpha}$, we have  
			$$
			\Nt \Psi_\kin^\eps [f, g]\Nt_\SSSm^2
			\lesssim  \int_0^T \| w(t) f(t) \|_{\SSsp}^2 \left\| |\nabla_x|^{\boldsymbol{\alpha}} g(t) \right\|_{ \rSSsp{s-{\boldsymbol{\alpha}}} }^2 \d t \\
			\lesssim  w(T)^{-2} \| f \|_{\SSSs}^2 \| g \|_{\SSSs}^2.
			$$
			This proves \eqref{eq:bilinear_mix_hyd_hyd} and concludes the proof.
		\end{proof}
		
		This next proposition is proved as Proposition \ref{prop:special_bilinear_hydrodynamic} and its proof is omitted.
		\begin{prop}[\textit{\textbf{Special bilinear mixed estimates}}]
			\label{prop:special_bilinear_mix}
			When $f \in \SSSs$ and $\phi$ is the parameter defining the $\SSSs$-norm
			\begin{equation}
				\label{eq:bilinear_mix_phi}
				\Nt \Psi^\eps_\kin [f, \phi]  \Nt_\SSSm \lesssim \eta \Nt f \Nt_\SSSs,
			\end{equation}
			furthermore, when $g^\eps_\disp = U^\eps_\disp(\cdot) g$ where $g = \PP_{\disp} g \in \rSSs{s} \cap \dot{\mathbb{H}}^{-\alpha}_x \left( \Ss_v \right)$, there holds
			\begin{equation}
				\label{eq:bilinear_mix_disp}
				\Nt \Psi^\eps_\kin [f,  g^\eps_\disp] \Nt_\SSSm \lesssim {\beta_{\disp}}(g , \eps) \Nt f \Nt_{ \SSSs }, \qquad \lim_{\eps\to0} {\beta_{\disp}}(g, \eps)=0,
			\end{equation}
			and in the case $d \ge 3$, the rate of convergence is explicit assuming $g \in \dot{\mathbb{B}}^{s' + (d+1)/2}_{1, 1} \left( \Ssm_v \right) \cap \rSSsm{s}$ for some $s'>s$:
			\begin{equation*}
				 {\beta_{\disp}}(g, \eps) \lesssim \sqrt{\eps} \left( \| g \|_{ \rSSs{s} } + \| g \|_{ \dot{\mathbb{B}}^{s' + (d+1)/2 }_{1, 1} \left( \Ss_v \right) }\right).
			\end{equation*}
		\end{prop}
		
		%	\section{Main tools for the hydrodynamic limit}
		
		\section{Proof of Theorems \ref{thm:hydrodynamic_limit} and \ref{thm:hydrodynamic_limit-gen_symmetric}}
		
		\label{scn:proof_hydrodynamic_limit_symmetrizable}
		
		In this section, we denote $\Sl = \Ss$ under the sole assumptions \ref{L1}--\ref{L4} and \ref{Bortho}--\ref{Bbound}, and $\Sl$ is the space from assumptions \ref{LE} and \ref{BE} under these extra assumptions.
		\medskip
		
		In this section, we construct a solution of the perturbed equation and then show it must be unique. We follow the approach described in Section \ref{sec:detail}. We refer more specifically to Section \ref{sec:detail-sum} that we briefly resume here. Recall that we look for a solution of the form
		\begin{equation}\label{eq:ansatz}\begin{split}
				f^\eps(t)&=f^\eps_\kin(t)+f^\eps_\mix(t)+f^\eps_\hyd(t)\\
				&=f^\eps_\kin(t)+f^\eps_\mix(t)+f^\eps_\disp(t)+ f_\ns(t)+g^\eps(t)
			\end{split}
		\end{equation}  
		where $f_\ns(\cdot)$ as well as $f_\ini$ (and thus $f^\eps_\disp(t) = U^\eps_\disp(t) f_\ini$) are functions to be considered as fixed parameters  since they depend only on the initial datum $f_\ini$ (and $\eps$). We point out that by Lemma \ref{lem:decay_regularization_kinetic_semigroup}, Lemma \ref{lem:decay_semigroups_hydro} and Lemma \ref{lem:NS_parabolic_space} respectively
		$$\Nt U^\eps_\kin(\cdot) f_\ini \Nt_{\SSSl} \lesssim 1, \qquad \Nt f^\eps_\disp \Nt_{\SSSs} \lesssim 1, \qquad \Nt f_\ns \Nt_{\SSSs} \lesssim 1,$$
	without those quantities being necessarily small. The smallness necessary to our fixed-point argument is coming from Proposition \ref{prop:special_bilinear_hydrodynamic} which yields
		$$\| \Psi^\eps_\hyd(f_\ns, \cdot) \|_{ \BBB\left( \SSSs  \right) } + \| \Psi^\eps_\kin(f_\ns, \cdot) \|_{ \BBB\left( \SSSs ; \SSSm \right) } \lesssim \eta,$$
		and from Proposition \ref{prop:special_bilinear_mix} which yields
		$$\lim_{\eps \to 0} \left( \| \Psi^\eps_\hyd(f_\disp^\eps, \cdot) \|_{ \BBB\left( \SSSs  \right) } + \| \Psi^\eps_\kin(f_\disp^\eps, \cdot) \|_{ \BBB\left( \SSSs ; \SSSm \right) } \right) = 0.$$
		
		From now on, we work with the space $\SSSm$ and $\SSSs$ associated to the solution $f_\ns$ which corresponds, in Eq. \eqref{eq:wphieta}, to the choice of the weight function 
		$$w_{f_\ns,\eta}(t)=\exp\left( \frac{1}{2 \eta^2} \int_0^t\| |\nabla_x|^{1-\alpha} f_\ns(\tau) \|_{\rSSsp{s}}^2 \d \tau \right) \qquad t \ge0,$$
		with $\eta >0$ still to be chosen.  
		
		We recall that we showed in Section \ref{sec:detail-sum} that solving equation
		$$f^\eps(t)=U^\eps(t)f_\ini+\Psi^\eps[f^\eps,f^\eps](t)$$
		can be reformulated, under the above \emph{ansatz}, into the system of coupled nonlinear equations
		\begin{equation}\label{eq:systemKinMixG-1}
			\begin{cases}
				f^\eps_\kin(t) &= U^\eps_\kin(t) f_\ini + \Psi^\eps_\kin \left[ f^\eps_\kin, f^\eps_\kin \right](t) + 2 \Psi^\eps_\kin\left[ f^\eps_\kin, f^\eps_\hyd + f^\eps_\mix\right](t),\\
				\\
				f^\eps_\mix(t)    &=\Psi^\eps_\kin\left[\left(f_\ns + f^\eps_\disp\right)+ f_\mix^\eps+g^\eps\,;\,\left(f_\ns + f^\eps_\disp\right)+ f_\mix^\eps+g^\eps \right](t),\\
				\\	
				g^\eps(t) &= \Phi^\eps[ f^\eps_\kin, f^\eps_\mix ](t) g^\eps(t)+
				\Psi^\eps_\hyd\left[ g^\eps, g^\eps\right] + \SS^\eps(t),
			\end{cases}
		\end{equation}
		where the source term $\SS^\eps(t)$ is defined through \eqref{eq:source}. We construct a solution $\left( f^\eps_\kin, f^\eps_\mix, g^\eps \right)$ of this system in the space
		$ \SSSl \times \SSSm \times \SSSs $	and more specifically, in a product of the following balls for some small radii $ c_2, c_3 >0$:
		\begin{equation*}\begin{split}
				B_1 := \bigg\{ U^\eps_\kin(\cdot) f_\ini + \phi \, : \, &\Nt \phi  \Nt_{\SSSl} \le 1 \bigg\}, \qquad
				B_2 := \bigg\{  \varphi \in \SSSm \, : \, \| \varphi \Nt_{\SSSm} \le c_2 \bigg\}, \\ 
				B_3 &:= \bigg\{ \psi \in \SSSs \, : \, \| \psi \Nt_{\SSSs} \le c_3 \bigg\},
		\end{split}\end{equation*}
		where $\SSSs=\SSSs(T, f_\ns, \eta)$ with $T$ being the lifespan of $f_\ns$. To do so, we reformulate the system \eqref{eq:systemKinMixG-1} as a fixed point problem of the type
		$$(f^\eps_\kin, f^\eps_\mix, g^\eps)=\bm{\Xi}\left[f^\eps_\kin, f^\eps_\mix, g^\eps\right]$$
		where the mapping $\bm{\Xi}=\left( \Xi_1, \Xi_2, \Xi_3 \right)$
		\begin{equation*}%\begin{cases}
			\bm{\Xi}  \, :  \,  B_1\times B_2 \times B_3 \longrightarrow B_1\times B_2\times B_3
		\end{equation*}
		is defined through its components:
		\begin{equation}\label{eq:systemXIi}\begin{cases}
				\Xi_1\left[\phi, \varphi, \psi \right]  &= U^\eps_\kin(\cdot) f_\ini + \Psi^\eps_\kin\left[\phi, \phi\right]  +2 \Psi^\eps_\kin\left[\phi, f_\ns + f^\eps_\disp\right]+ 2 \Psi^\eps_\kin\left[\phi, \psi + \varphi\right],\\
				\\
				\Xi_2\left[\phi, \varphi, \psi\right]  &= \Psi^\eps_\kin\left[\left(f_\ns + f^\eps_\disp\right) + \varphi   + \psi , \left(f_\ns + f^\eps_\disp\right)  + \varphi  + \psi \right] ,\\
				\\
				\Xi_3\left[\phi, \varphi, \psi\right] &=  \Phi^\eps[\phi,\varphi]\psi   + \Psi^\eps_\hyd\left[\psi , \psi\right] + \SS^\eps[\phi, \varphi]\,,
		\end{cases}\end{equation}
		for any
		$$(\phi,\varphi,\psi) \in \bm{B} := B_1\times B_2 \times B_3 \subset  \SSSl \times \SSSm \times \SSSs , $$
		where we recall that $f_\ns$ and $f_\disp^\eps$ are \emph{fixed} parameters for this problem and thus, with a slight  abuse of notation,  the source term $\SS^\eps[\phi,\varphi](t)$ writes
		$$\SS^\eps[\phi,\varphi](t)=\SS_1^\eps(t)+\SS^\eps_2(t)+\SS^\eps_3[\phi,\varphi](t)$$
		where we recall that $\SS_1^\eps(t)$ depends only on $f_\ns(t)$ and $f_\ini$ while $\SS_2^\eps(t)$ depends only on $f_\disp^\eps$ and $f_\ns$. We also defined, for $(\phi,\varphi) \in  \SSSl \times \SSSm$, the linear operator $\Phi^\eps[\phi,\varphi]$ on $\SSSs$ as
		$$\Phi^\eps\left[\phi,\varphi\right]h=2\Psi_\hyd^\eps\left[h,\left(f_\ns+f_\disp^\eps\right)+\varphi+\phi\right], \qquad \quad \forall h\in \SSSs.$$
		\subsection{Linear estimates and source terms estimates}
		The source term $\SS^\eps$ and the linear terms involved in the above system \eqref{eq:systemKinMixG-1} can be estimated with a simple use of the results of Section \ref{sec:Bilin}.   In particular, we recall that the norm $\SSSs$ depends on a parameter $\eta >0$ which can be chosen freely.
		\begin{prop}[\textit{\textbf{Linear hydrodynamic estimate}}]
			\label{prop:linear_hydrodynamic}
			With the notation ${\beta_{\disp}}$ of Proposition~\ref{prop:special_bilinear_hydrodynamic} and assuming that $f_\ini \in \SSl$ and $f^\eps_\disp = U^\eps_\disp f_\ini$ are given, the following continuity estimate holds in $\SSSs=\SSSs(T, f_\ns, \eta)$:
			\begin{align*}
				\big\| \Psi^\eps_\hyd\big[ h, & f_\ns + f^\eps_\disp + g_\mix + g_\kin\big] \big\|_{ \SSSs  } \\
				& \lesssim \left(\eta + {\beta_{\disp}}(\PP_\disp f_\ini, \eps) + \eps w_{f_\ns, \eta}(T)^{-1} \Nt g_\mix \Nt_{\SSSm} + \eps \Nt g_\kin \Nt_\SSSl\right) \Nt h \Nt_\SSSs ,
			\end{align*}
			as well as the following stability estimate:
			\begin{align*}
				\Nt \Psi^\eps_\hyd[ h, f_\ns & + f^\eps_\disp + g _\mix + g _\kin]  - \Psi^\eps_\hyd\left[ h', f_\ns + f^\eps_\disp + g_\mix '+ g_\kin'\right] \Nt_\SSSs \\
				\lesssim & \Nt h - h' \Nt_\SSSs \left( \eta + \beta_{g}(\eps) + \eps w_{f_\ns, \eta}(T)^{-1} \Nt g_\mix \Nt_{\SSSm} + \eps \Nt g_\kin \Nt_\SSSl \right) \\
				& + \Nt h \Nt_\SSSs \left( \eps w_{f_\ns, \eta}(T)^{-1} \Nt g_\mix - g_\mix' \Nt_\SSSm + \eps \Nt g_\kin - g_\kin' \Nt_\SSSl \right).
			\end{align*}
		\end{prop}
		
		\begin{proof}
			The two estimates are direct consequences of Propositions \ref{prop:bilinear_hydrodynamic} and \ref{prop:special_bilinear_hydrodynamic} since \begin{align*}
				\Psi^\eps_\hyd\left[ h, f_\ns + f^\eps_\disp + g _\mix + g _\kin\right]  & - \Psi^\eps_\hyd\left[ h', f_\ns + f^\eps_\disp + g_\mix '+ g_\kin'\right]\\
				= & \Psi^\eps_\hyd\left[h- h', f_\ns + f^\eps_\disp + g_\mix + g_\kin\right] \\
				& + 2 \Psi^\eps_\hyd\left[h, g_\mix-g_\mix'\right] + 2 \Psi^\eps_\hyd\left[h, g_\kin-g_\kin'\right]\,.
			\end{align*}
			This proves the result.
		\end{proof}
		
		The first part $\SS_1^\eps$ of the source term $\SS^\eps$ which depends only on the initial data $f_\ini$ and the Navier-Stokes solution $f_\ns$ (but not on the partial solutions $f^\eps_\kin$, $f^\eps_\mix$ or $g^\eps$) is estimated in this next lemma.
		
		\begin{lem}[\textit{\textbf{Estimate of the first source term $\SS^\eps_1$}}]\label{lem:SourceS1}
			Consider some $f_\ini \in \SSl$. The source term $\SS^\eps_1$ satisfies
			$$\Nt \SS^\eps_1 \Nt_\SSSs \le \beta_\ns(f_\ns, f_\ini, \eps), \qquad \lim_{\eps \to 0}\beta_\ns(f_\ns, f_\ini, \eps)=0.$$
			If we assume additionally that the initial data $f_\ini$ lies in $\rSSl{s+\delta} \cap \dot{\mathbb{H}}^{-\alpha}_x \left( \Sl_v \right)$ for some $\delta \in (0, 1]$, then the rate of convergence can be made explicit as
			\begin{align*}
				\beta_\ns(f_\ns, f_\ini, \eps) \lesssim \eps^{\delta} \Big( & 1 + \| f_\ini \|_{\rSSl{s+\delta}} + \| f_\ini \|_{ \dot{\mathbb{H}}^{-\alpha}_x \left( \Sl_v \right) } \\
			 & +  \| f_\ns \|_{ L^\infty\left( [0, T) ; \rSSs{s+\delta} \right) } + \| \nabla_x f_\ns \|_{ L^2\left( [0, T) ; \rSSs{s+\delta} \right) } \Big)^3.
			\end{align*}
		\end{lem}
		\begin{proof} Recalling that $U_\hyd^{\eps}(t)=U^{\eps}_{\ns}(t)+U_{\Wave}^{\eps}(t)$, we write the source term $\SS^\eps_1(t)$ as
			\begin{align*}
				\SS^\eps_1(t) = & \left(U^\eps_\Wave(t)f_{\ini} - U^\eps_\disp(t)f_{\ini}\right) + \left(U^\eps_\ns(t)f_{\ini} - U_\ns(t)f_{\ini}\right)  \\
				& + \Psi^\eps_\Wave\left[f_\ns, f_\ns\right](t) + \left(\Psi^\eps_\ns\left[f_\ns, f_\ns\right](t)- \Psi_\ns\left[f_\ns, f_\ns\right](t)\right) .
			\end{align*}
			Using Lemmas \ref{lem:asymptotic_equiv_oscillating_semigroup} and \ref{lem:asymptotic_equiv_NS_semigroup}, we have for a smooth initial data $f_\ini$
			\begin{equation*}\begin{split} 
					\Nt U^\eps_\hyd(\cdot)f_\ini &- U_\ns(\cdot)f_\ini - U^\eps_\disp(\cdot) f_\ini  \Nt_\SSSs \\
					& \le \Nt  U^\eps_\Wave(\cdot) f_\ini - U^\eps_\disp f_\ini \Nt_\SSSs + \Nt  U^\eps_\ns(\cdot) f_\ini - U_\ns(\cdot) f_\ini \Nt_\SSSs \\
					& \lesssim  \eps^{\delta} \left( \| f_\ini \|_{\rSSlm{s+\delta}} + \| f_\ini \|_{ \dot{\mathbb{H}}^{-\alpha}_x \left( \Slm_v \right) } \right),
			\end{split}\end{equation*}
			and in general, by a limiting argument
			$$\lim_{\eps\to0}\Nt U^\eps_\hyd(\cdot)f_\ini - U_\ns(\cdot)f_\ini - U^\eps_\disp(\cdot) f_\ini  \Nt_\SSSs=0.$$
			Furthermore, using the estimate of Lemma \ref{lem:convolution_wave} with $\varphi = \QQ(f_\ns, f_\ns)$, and where we point out that, for $d=2$, we have $0 < \alpha < \frac{1}{2}$ and thus
			$$ \frac{4}{3 + 2 \alpha} \in \left(1, \frac{4}{3}\right) , \qquad   \frac{2}{1+\alpha} \in \left(\frac{4}{3}, 2\right),$$
			whereas, for $d \geq3$, we have $\alpha=0$ and thus
			$$\frac{4}{3 + 2 \alpha} = \frac{4}{3}, \qquad \frac{2}{1+\alpha} = 2,$$
			we estimate $ \| \varphi(0) \|_{ \dot{\mathbb{H}}_x^{-\alpha} \left( \Ss_v \right) }$ thanks to \eqref{eq:Q_refined_sobolev_negative_algebra_inequality},  and the other ones using using Lemmas \ref{lem:estimates_derivative_Q_navier_stokes}--\ref{lem:NS_parabolic_space} to deduce that		
			\begin{equation*}
				\Nt \Psi^\eps_\Wave\left[f_\ns, f_\ns\right] \Nt_{\SSSs} \lesssim \eps \Big( 1 + \| f_\ns(0) \|_{ \dot{\mathbb{H}}^{-\alpha}_x \left( \Ss_v \right) } + \| f_\ns \|_{ L^\infty\left( [0, T) ;\rSSs{s} \right) }  + \| \nabla_x f_\ns \|_{ L^2\left( [0, T) ; \rSSs{s} \right) }\Big)^3.
			\end{equation*} 
			Finally, one proves as for \eqref{eq:bilinear_hyd_hyd} using this time \eqref{eq:asymptotic_equiv_NS_orthogonal_reg}, that for any $\delta \in [0, 1]$
			$$\Nt  \Psi^\eps_\ns\left[f_\ns, f_\ns\right] - \Psi_\ns\left[f_\ns, f_\ns\right] \Nt_{ \rSSSs{s} } \lesssim \eps^{\delta} \Nt f_\ns \Nt_{\rSSSs{s+\delta}}^2,$$
			which, on the one hand, implies by Lemma \ref{lem:NS_parabolic_space}
			\begin{align*}
				\Nt  \Psi^\eps_\ns & \left[f_\ns, f_\ns\right]  - \Psi_\ns\left[f_\ns, f_\ns\right] \Nt_{ \rSSSs{s} } \\
				& \lesssim \eps^{\delta} \Big( \| f_\ns(0) \|_{ \dot{\mathbb{H}}^{-\alpha}_x \left( \Ss_v \right) } + \| f_\ns \|_{ L^\infty\left( [0, T) ;\rSSs{s+\delta} \right) }  + \| \nabla_x f_\ns \|_{ L^2\left( [0, T) ; \rSSs{s+\delta} \right) }\Big)^2,
			\end{align*} 
			and on the other hand, since $f_\ns$ can be approximated by elements of $\rSSSs{s+1}$ by Lemma \ref{lem:NS_parabolic_space}, we have in general
			$$\lim_{\eps\to0}\Nt  \Psi^\eps_\ns\left[f_\ns, f_\ns\right] - \Psi_\ns\left[f_\ns, f_\ns\right] \Nt_{ \rSSSs{s} }=0.$$
			This concludes the proof.
		\end{proof}
		
		\begin{prop}[\textit{\textbf{Estimate for the source term $\SS^\eps$}}]
			\label{prop:source_term}
			Consider some $f_\ini \in \SSl$ and denote $f^\eps_\disp = U^\eps_\disp(\cdot) f_\ini$, the source term~$\SS^\eps$ satisfies in $\rSSSs{s} = \rSSSs{s}(T, f_\ns, \eta)$
			\begin{equation*}\begin{split}
				\Nt \SS^\eps[  g_\kin, g_\mix] \Nt_\SSSs \lesssim  & {\beta_{\disp}}(\PP_\disp f_\ini, \eps) \left(\Nt f_\ns \Nt_\SSSs + \| g \|_{\SSs}\right) + \beta_\ns(f_\ns, f_\ini, \eps) \\
				& + \eps w_{f_\ns, \eta}(T)^{-1} \left( \Nt g_\kin \Nt_\SSSl + \Nt g_\mix \Nt_\SSSm \right) \\
				& \quad \times \left( \Nt g_\kin \Nt_\SSSl + \Nt g_\mix \Nt_\SSSm + \Nt f_\ns \Nt_\SSSs + \| \PP_\disp f_\ini \|_{ \SSs } + \| \PP_\disp f_\ini \|_{ \dot{\mathbb{H}}^{-\alpha}_x \left( \Ss_v \right) }\right),
			\end{split}\end{equation*}
			where there holds
			$$\displaystyle \lim_{\eps\to0} {\beta_{\disp}}( \PP_\disp f_\ini, \eps) =\lim_{\eps\to0}\beta_\ns(f_\ns, f_\ini, \eps)=0.$$
			The rate of convergence of the term~$\beta_\ns(f_\ns, f_\ini, \eps)$ can be made explicit if the initial data $f_\ini$ lies in $\rSSl{s+\delta} \cap \dot{\mathbb{H}}^{-\alpha}_x \left( \Sl_v \right)$ for some $\delta \in (0, 1]$:
			\begin{align*}
				\beta_\ns(f_\ns, f_\ini, \eps) \lesssim \eps^{\delta} \Big( & 1 + \| f_\ini \|_{\rSSl{s+\delta}} + \| f_\ini \|_{ \dot{\mathbb{H}}^{-\alpha}_x \left( \Sl_v \right) } \\
				& +  \| f_\ns \|_{ L^\infty\left( [0, T) ; \rSSs{s+\delta} \right) } + \| \nabla_x f_\ns \|_{ L^2\left( [0, T) ; \rSSs{s+\delta} \right) } \Big)^3.
			\end{align*}
			and if $d \ge 3$, the rate of convergence of $ {\beta_{\disp}}(f_\ini, \eps)$ is explicit if $\PP_\disp f_\ini \in \dot{\mathbb{B}}^{s' + (d+1)/2}_{1, 1} \left( \Ss_v \right) \cap \rSSs{s}$ for some $s'>s$:
			\begin{equation*}
				 {\beta_{\disp}}(f_\ini, \eps) \lesssim \sqrt{\eps} \left( \| \PP_\disp f_\ini \|_{ \rSSs{s} } + \| \PP_\disp f_\ini \|_{ \dot{\mathbb{B}}^{s' + (d+1)/2 }_{1, 1} \left( \Ss_v \right) }\right).
			\end{equation*}
			Furthermore, the source term $\SS^\eps$ satisfies the stability estimate
			\begin{align*}
				\Nt \SS^\eps_3[g_\kin, & g_\mix] - \SS^\eps_3[g_\kin', g_\mix']  \Nt_\SSSs \\
				\lesssim \eps & w_{f_\ns, \eta}(T)^{-1} \left( \Nt g_\kin - g_\kin' \Nt_\SSSl + \Nt g_\mix - g_\mix' \Nt_\SSSm \right) \\
				& \times \left( \Nt g_\kin + g_\kin' \Nt_\SSSl + \Nt g_\mix +  g_\mix' \Nt_\SSSm + \Nt f_\ns \Nt_\SSSs + \| \PP_\disp f_\ini \|_{\SSs} + \| \PP_\disp f_\ini \|_{ \dot{\mathbb{H}}^{-\alpha}_x \left( \Ss_v \right) } \right)\,.
			\end{align*}
		\end{prop}
		
		\begin{rem}
			Note that the terms $\beta_{g}$ and $\beta_\ns$ depend respectively on $g$, which stands for~$\PP_\disp f_\ini$, and $f_\ns$ which are to be considered as fixed data of the problem, thus the lack of uniform estimate for their convergence is not an issue for the iterative scheme.
		\end{rem}
		\begin{proof}
			Recalling the definition of $\SS^\eps_2$:
			$$
			\SS^\eps_2 = \Psi^\eps_\hyd\left[ g^\eps_\disp, 2 f_\ns + f^\eps_\disp \right]
			$$
			we easily have thanks to Eq. \eqref{eq:bilinear_hyd_disp} in Proposition \ref{prop:special_bilinear_hydrodynamic} and Lemma \ref{lem:decay_semigroups_hydro}
			$$\Nt \SS^\eps_2 \Nt_\SSSs \lesssim \beta_\disp( \PP_\disp f_\ini, \eps) \Nt 2f_\ns + f^\eps_\disp\Nt_\SSSs \lesssim \beta_\disp( \PP_\disp f_\ini, \eps) \left(\Nt f_\ns \Nt_\SSSs + \| g \|_{\SSs}\right).$$
			Furthermore, recalling the definition of $\SS^\eps_3$:
			$$
			\SS^\eps_3[g_\kin, g_\mix] = \Psi^\eps_\hyd\left[ g_\kin + g_\mix , g_\kin + g_\mix + f_\ns + f^\eps_\disp \right],
			$$
			we easily have thanks to the various estimates of Proposition \ref{prop:bilinear_hydrodynamic} and Lemma \ref{lem:decay_semigroups_hydro} together with the bilinearity of $\Psi^\eps_\hyd$
			\begin{align*}
				\Nt \SS^\eps_3[g_\kin, g_\mix] \Nt_{ \SSSs} 
				\lesssim \eps & w_{\phi, \eta}(T)^{-1} \left( \Nt g_\kin \Nt_\SSSl + \Nt g_\mix \Nt_\SSSm \right) \\
				& \times \left( \Nt g_\kin \Nt_\SSSl + \Nt g_\mix \Nt_\SSSm + \Nt f_\ns \Nt_\SSSs + \| \PP_\disp f_\ini \|_{ \SSs } + \| \PP_\disp f_\ini \|_{ \dot{\mathbb{H}}^{-\alpha}_x \left( \Ss_v \right) } \right).
			\end{align*}
			The stability estimate comes from the identity
			\begin{align*}
				\SS^\eps_3[g_\kin, g_\mix] - & \SS^\eps_3[g_\kin', g_\mix'] \\
				= & \Psi^\eps_\hyd\left[ g_\kin - g_\kin' + g_\mix - g_\mix' , g_\kin + g_\kin' + g_\mix + g_\mix' \right] \\
				& + \Psi^\eps_\hyd\left[ g_\kin - g_\kin' + g_\mix - g_\mix' , f_\ns + f_\disp^\eps \right]
			\end{align*}
			which we control using the same estimates. This concludes the proof.
		\end{proof}

		\subsection{The mapping is a contraction}
		\label{scn:mapping_contraction}
		In what follows, we will simplify some estimates by using the fact that
		$$c_2, c_3, \eta, \eps \lesssim 1, \qquad 1 \le w_{f_\ns, \eta}(T)^{-1}$$
		and that $\beta(\eps) =  {\beta_{\disp}}(\PP_\disp f_\ini, \eps) + \beta_{\ns}(f_\ns, f_\ini, \eps)$ (see Proposition \ref{prop:source_term}) can be assumed to vanish at a slower rate that $\eps$:
		$$\eps \lesssim \beta(\eps).$$ 
		
		To prove the existence and uniqueness of a fixed point for $\bm{\Xi}$, we need to check that $\bm{\Xi}$ is a a contraction on $\bm{B}$. We begin by showing that $\bm{B}$ is stable under the action of $\bm{\Xi}$  under suitable smallness assumption on $\eps,c_3,\eta,c_2$:
		\begin{lem} For a suitable choice of 
			$$\eps \ll c_3 \ll \eta \ll c_2 \ll 1$$
			the mapping $\bm{\Xi}$ is well-defined on $\bm{B}$ and
			$\bm{\Xi}(\bm{B}) \subset \bm{B}.$
		\end{lem}
		\begin{proof} Let us check that the first component $\Xi_1$ is well defined and take values in $B_1$. We assume $(\phi,\varphi,\psi) \in \bm{B}$
			\begin{align*}
				\Nt \Xi_1\left[\phi, \varphi, \psi\right] - U^\eps_\kin(\cdot) f_\ini \Nt_{\SSSl} \le &  \Nt \Psi^\eps_\kin\left[\phi, \phi\right] \Nt_{\SSSl} + 2 \Nt \Psi^\eps_\kin\left[\phi, f_\ns + f^\eps_\disp \right] \Nt_{\SSSl} \\
				& + 2 \Nt \Psi^\eps_\kin\left[\phi, \psi\right] \Nt_{ \SSSl } + 2 \Nt \Psi^\eps_\kin\left[\phi, \varphi\right] \Nt_{ \SSSl }.
			\end{align*}
			Using \eqref{eq:bilinear_kin}, we have the estimates
			$$\Nt \Psi^\eps_\kin\left[ \phi, \phi\right] \Nt_{\SSSl} \lesssim \eps \Nt \phi\Nt_{\SSSl}^2 \lesssim \eps,$$
			\begin{align*} 
				2 \Nt \Psi^\eps_\kin \left[ \phi, \varphi\right] \Nt_{ \SSSl } & \lesssim \eps  w_{f_\ns, \eta }(T)^{-1}\Nt \phi\Nt_{\SSSl}\,\Nt \psi\Nt_{\SSSm}  \lesssim \eps w_{f_\ns, \eta }(T)^{-1},
			\end{align*}
			as well as 
			\begin{equation*} 	\begin{split}	
					2 \Nt \Psi^\eps_\kin\left[\phi, f_\ns + f^\eps_\disp \right] \Nt_{\SSSl} &+ 2 \Nt \Psi^\eps_\kin\left[\phi, \psi\right] \Nt_{ \SSSl }\\ 
					&\lesssim \eps w_{f_\ns,\eta}(T)^{-1}\Nt \phi\Nt_{\SSSl}\left(\Nt f_\ns +f^\eps_\disp \Nt_{\SSSs} + \Nt \psi \Nt_{\SSSs} \right)
					\\ 
					&\lesssim \eps w_{f_\ns, \eta}(T)^{-1}\,.
			\end{split}\end{equation*}
			Consequently
			$$\Nt \Xi_1\left[\phi, \varphi, \psi\right] - U^\eps_\kin(\cdot) f_\ini \Nt_{\SSSl} \lesssim \eps w_{f_\ns, \eta}(T)^{-1},$$
			thus, considering $\eps \ll \eta$, we conclude that $\Xi_1[\phi,\varphi,\psi] \in B_1$.
			
			\medskip
			The second component $\Xi_2$ is also well-defined. We have
			\begin{equation*}\begin{split}
					\Nt \Xi_2\big[\phi, &\varphi, \psi\big] \Nt_{ \SSSm }  \\
					&\le  \Nt \Psi^\eps_\kin\left[ \varphi , \varphi \right] \Nt_{ \SSSm } + \Nt \Psi^\eps_\kin\left[ f_\ns , f_\ns \right] \Nt_{ \SSSm } + \Nt \Psi^\eps_\kin\left[f^\eps_\disp , f^\eps_\disp\right] \Nt_{ \SSSm } + \Nt \Psi^\eps_\kin\left[ \psi , \psi \right] \Nt_{ \SSSm } \\
					& + 2 \Nt \Psi^\eps_\kin\left[ \varphi , f_\ns \right]  \Nt_{ \SSSm } + 2 \Nt \Psi^\eps_\kin\left[ \varphi , f_\disp^\eps \right] \Nt_{ \SSSm } + 2 \Nt \Psi^\eps_\kin\left[ \varphi , \psi \right] \Nt_{ \SSSm } \\
					& + 2 \Nt \Psi^\eps_\kin\left[ f_\ns , f^\eps_\disp \right] \Nt_{ \SSSm } + 2 \Nt \Psi^\eps_\kin\left[ f_\ns , \psi \right] \Nt_{ \SSSm }  + 2 \Nt \Psi^\eps_\kin\left[ f^\eps_\disp , \psi \right] \Nt_{ \SSSm }\,.
			\end{split}\end{equation*}
			Using \eqref{eq:bilinear_mix_mix_mix}, \eqref{eq:bilinear_mix_phi}, \eqref{eq:bilinear_mix_disp} and \eqref{eq:bilinear_mix_hyd_hyd} respectively, we have
			\begin{equation*}\begin{split}
					\Nt \Psi^\eps_\kin\left[ \varphi , \varphi \right] \Nt_{ \SSSm }& + \Nt \Psi^\eps_\kin\left[ f_\ns , f_\ns \right] \Nt_{ \SSSm } + \Nt \Psi^\eps_\kin\left[f^\eps_\disp , f^\eps_\disp\right] \Nt_{ \SSSm } + \Nt \Psi^\eps_\kin\left[ \psi , \psi \right] \Nt_{ \SSSm }  \\
					&\lesssim \eps w_{f_\ns, \eta}(T)^{-1}\Nt \varphi\Nt_{\SSSm}^2 + \eta \Nt f_\ns \Nt_{\SSSs} +\beta(\eps)\Nt f_\disp^\eps\Nt_{\SSSs}+w_{f_\ns, \eta}(T)^{-1}\Nt \psi\Nt_{\SSSs}^2\\
					& \lesssim \eps c_2^2 w_{f_\ns, \eta}(T)^{-1} + \eta + \beta(\eps) + c_3^2 w_{f_\ns, \eta}(T)^{-1} \\ 
					&\lesssim \left(\beta(\eps) + c_3\right) w_{f_\ns, \eta}(T)^{-1} + \eta\,.
			\end{split}\end{equation*}
			Furthermore, using \eqref{eq:bilinear_mix_mix_mix}, we have
			\begin{equation*}\begin{split}
					2 \Nt \Psi^\eps_\kin\left[ \varphi , f_\ns\right] \Nt_{ \SSSm }  &+ 2 \Nt \Psi^\eps_\kin\left[\varphi , f_\disp^\eps\right] \Nt_{ \SSSm }  + 2 \Nt \Psi^\eps_\kin\left[ \varphi , \psi\right]\Nt_{ \SSSm }\\ 
					&\lesssim \eps w_{f_\ns, \eta}(T)^{-1}  \Nt \varphi \Nt_{\SSSm}\bigg(\Nt f_\ns\Nt_{\SSSs} 
					+ \Nt f_\disp^\eps \Nt_{\SSSs} + \Nt \psi \Nt_{\SSSs}\bigg)\\
					&\lesssim \eps c_2 ( 1 + c_3) w_{f_\ns, \eta}(T)^{-1}  \lesssim \eps w_{f_\ns, \eta}(T)^{-1},
			\end{split}\end{equation*}
			whereas,  \eqref{eq:bilinear_mix_phi} and \eqref{eq:bilinear_mix_disp} give respectively
			\begin{align*}
				2 \Nt \Psi^\eps_\kin\left[ f_\ns , f^\eps_\disp \right] \Nt_{ \SSSm }  + 2 \Nt \Psi^\eps_\kin\left[ f_\ns , \psi \right] \Nt_{ \SSSm }  & \lesssim \eta\left(\Nt f^\eps_\disp\Nt_{\SSSs} + \Nt \psi \Nt_{\SSSs}\right) \\
				&\lesssim\eta (1+c_3) \lesssim \eta,
			\end{align*}
			and
			\begin{align*}
				2 \Nt \Psi^\eps_\kin\left[ f^\eps_\disp , \psi \right] \Nt_{ \SSSm }   \lesssim \beta(\eps)\Nt \psi\Nt_{\SSSs}  
				\lesssim \beta(\eps) c_3 \lesssim \beta(\eps).
			\end{align*}
			Gathering these estimates yield
			$$\Nt \Xi_2\left[\phi, \varphi, \psi\right] \Nt_{\SSSl} \lesssim \left(\beta(\eps) + c_3\right) w_{f_\ns, \eta}(T)^{-1} + \eta.$$
			We deduce for $\max\{ \eps , c_3 \} \ll \eta \ll c_2$ that $\Xi_2$ takes value in $B_2$.
			
			\medskip
			Finally, the third component $\Xi_3$ is well defined. We have
			$$\Nt \Xi_3\left[\phi, \varphi, \psi\right]\Nt_{\SSSs} \le \Nt \Phi^\eps[ \phi , \varphi ] \psi \Nt_{\SSSs} + \Nt \Psi^\eps_\hyd [\psi , \psi] \Nt_{\SSSs} + \Nt \SS^\eps[\phi, \varphi] \Nt_{\SSSs}.$$
			By Proposition \ref{prop:linear_hydrodynamic}
			\begin{align*}
				\Nt \Phi^\eps[ \phi , \varphi ] \psi \Nt_{\SSSs} & \lesssim \left(\eta + \beta(\eps) + \eps c_2 w_{f_\ns, \eta}(T)^{-1} + \eps \right) c_3  \lesssim \beta(\eps)  w_{f_\ns, \eta}(T)^{-1} + \eta c_3,
			\end{align*}
			while, from \eqref{eq:bilinear_hyd_hyd} and  Proposition \ref{prop:source_term}
			$$
			\Nt \Psi^\eps_\hyd [\psi, \psi] \Nt_{\SSSs} \lesssim c_3^2 w_{f_\ns, \eta}(T)^{-1}, \qquad \Nt \SS^\eps[\phi, \varphi] \Nt_{ \SSSs } \lesssim \beta(\eps).$$
			All these estimates gathered together give
			\begin{align*}
				\Nt \Xi_3[\phi, \varphi, \psi] \Nt_{\SSSs} \lesssim \beta(\eps) w_{ f_\ns, \eta }(T)^{-1} + \left( \eta + c_3 w_{ f_\ns, \eta }(T)^{-1} \right) c_3
			\end{align*}
			thus $\Xi_3$ takes value in $B_3$ by taking $\eps \ll c_3 \ll \eta \ll 1$.
			This completes the proof.\end{proof}

		We show now that, up to reducing further the parameters $\eps,c_2,c_3,\eta$, the mapping $\bm{\Xi}$ is a contraction on $\bm{B}$.
		\begin{prop} Under the smallness assumption
			$$\max\{ \eps , c_3 \} \ll \eta \ll 1\,,$$
			the mapping $\bm{\Xi}\::\:\bm{B} \to \bm{B} \subset \SSSl\times \SSSm \times \SSSs$ is a contraction. 
		\end{prop}
		\begin{proof} Let us fix $(\phi,\varphi,\psi) \in \bm{B},$ $(\phi',\varphi',\psi') \in \bm{B}$. We prove that each component of $\bm{\Xi}$ is contractive.  We have
			\begin{equation*}\begin{split}
					\Nt \Xi_1[\phi, \varphi, \psi]  & - \Xi_1[\phi', \varphi', \psi' ] \Nt_{ \SSSl } \\
					\le & \Nt \Psi^\eps_\kin[\phi - \phi', \phi + \phi'] \Nt_{\SSSl}  + 2 \Nt \Psi^\eps_\kin[\phi - \phi', f_\ns + f^\eps_\disp] \Nt_{ \SSSl } \\
					& + 2 \Nt \Psi^\eps_\kin[\phi - \phi', \psi] \Nt_{ \SSSl } + 2 \Nt \Psi^\eps_\kin[\phi', \psi - \psi'] \Nt_{ \SSSl } \\
					& + 2 \Nt \Psi^\eps_\kin[\phi - \phi', \varphi] \Nt_{ \SSSl } + 2 \Nt \Psi^\eps_\kin[\phi', \varphi - \varphi']\Nt_{ \SSSl }
			\end{split}\end{equation*}
			As in the previous proof, using \eqref{eq:bilinear_kin} we have
			\begin{equation*}\begin{split}
					\Nt \Psi^\eps_\kin[\phi - \phi', \phi + \phi'] \Nt_{\SSSl}  & + 2 \Nt \Psi^\eps_\kin[\phi - \phi', f_\ns + f^\eps_\disp] \Nt_{ \SSSl } \\
					& \lesssim \eps \left(1 + w_{f_\ns, \eta}(T)^{-1} \right) \Nt \phi - \phi' \Nt_{ \SSSl } \\
					& \lesssim \eps w_{f_\ns, \eta}(T)^{-1} \Nt \phi - \phi' \Nt_{ \SSSl },
			\end{split}\end{equation*}
			and
			\begin{equation*}\begin{split}
					2 \Nt \Psi^\eps_\kin[\phi - \phi', \psi] \Nt_{ \SSSl } &+ 2 \Nt \Psi^\eps_\kin[\phi', \psi - \psi'] \Nt_{ \SSSl } \\
					&\lesssim  \eps c_3 w_{f_\ns, \eta}(T)^{-1} \Nt \phi - \phi' \Nt_{ \SSSl }
					+ \eps w_{f_\ns, \eta}(T)^{-1} \Nt \psi - \psi' \Nt_{ \SSSs } \\
					&\lesssim  \eps w_{ f_\ns, \eta }(T)^{-1} \Nt (\phi, \psi) - (\phi', \psi') \Nt_{ \SSSl \times \SSSs },\end{split}\end{equation*}
			as well as
			\begin{align*}
				2 \Nt \Psi^\eps_\kin[\phi - \phi', \varphi] \Nt_{ \SSSl } & + 2 \Nt \Psi^\eps_\kin[\phi', \varphi - \varphi'] \Nt_{ \SSSl } \\
				\lesssim & \eps c_2 w_{ f_\ns, \eta }(T)^{-1} \Nt \phi - \phi' \Nt_{ \SSSl } + \eps  w_{ f_\ns, \eta }(T)^{-1} \Nt \varphi - \varphi' \Nt_{ \SSSm } \\
				\lesssim & \eps w_{ f_\ns, \eta }(T)^{-1} \Nt (\phi, \varphi) - (\phi', \varphi') \Nt_{ \SSSl \times \SSSm }.
			\end{align*}
			This shows that
			$$\Nt  \Xi_1[\phi, \varphi, \psi]    - \Xi_1[\phi', \varphi', \psi' ] \Nt_{ \SSSl } \lesssim \eps w_{ f_\ns, \eta }(T)^{-1} \Nt (\phi, \varphi, \psi) - (\phi', \varphi', \psi') \Nt_{ \SSSl \times \SSSm \times \SSSs }.$$
			Thus, taking $\eps \ll \eta$, the first component $\Xi_1$ is indeed a contraction. We argue in the same way for the second component. It holds
			\begin{equation*}\begin{split}
					\Nt \Xi_2[\phi, \varphi, \psi] - \Xi_2[\phi', \varphi', \psi' ] \Nt_{ \SSSm }
					&\le  \Nt \Psi^\eps_\kin[\varphi - \varphi' , \varphi + \varphi ] \Nt_{ \SSSm }  \\
					& + 2 \Nt \Psi^\eps_\kin [\varphi - \varphi' , f_\ns + f^\eps_\disp]  \Nt_{ \SSSm } \\
					& + 2 \Nt \Psi^\eps_\kin[ \varphi - \varphi', \psi ] \Nt_{ \SSSm } + \Nt \Psi^\eps_\kin [ \varphi' , \psi - \psi' ] \Nt_{ \SSSm } \\
					& + 2 \Nt \Psi^\eps_\kin [f_\ns , \psi - \psi' ] \Nt_{ \SSSm } + 2 \Nt \Psi^\eps_\kin[  f^\eps_\disp , \psi - \psi'  ] \Nt_{ \SSSm }.
			\end{split}\end{equation*}
			As in the previous proof, resorting to  \eqref{eq:bilinear_mix_mix_mix}, one deduces that
			$$
			\Nt \Psi^\eps_\kin [  \varphi - \varphi' , \varphi + \varphi ] \Nt_{ \SSSm }  \lesssim \eps c_2 w_{f_\ns, \eta}(T)^{-1} \Nt \varphi - \varphi' \Nt_{\SSSm} \\
			\lesssim \eps w_{f_\ns, \eta}(T)^{-1} \Nt \varphi - \varphi' \Nt_{\SSSm}$$
			and
			$$2 \Nt \Psi^\eps_\kin [\varphi - \varphi' , f_\ns + f^\eps_\disp] \Nt_{ \SSSm } \lesssim \eps w_{f_\ns, \eta}(T)^{-1} \Nt \varphi - \varphi' \Nt_{ \SSSm },$$
			whereas \begin{equation*}\begin{split}
					2 \Nt \Psi^\eps_\kin [  \varphi - & \varphi', \psi ] \Nt_{ \SSSm } + \Nt \Psi^\eps_\kin [ \varphi' , \psi - \psi' ] \Nt_{ \SSSm } \\
					& \lesssim \eps c_3 w_{f_\ns, \eta}(T)^{-1} \| \varphi - \varphi' \|_{ \SSSm } + \eps c_2 w_{f_\ns, \eta}(T)^{-1} \| \psi - \psi' \|_{ \SSSs } \\
					& \lesssim \eps w_{f_\ns, \eta}(T)^{-1} \Nt (\varphi, \psi) - (\varphi', \psi' ) \Nt_{ \SSSm \times \SSSs }\,.
			\end{split}\end{equation*}
			Finally, using \eqref{eq:bilinear_mix_phi} and \eqref{eq:bilinear_mix_disp}, one has as previously
			$$2 \Nt \Psi^\eps_\kin [ f_\ns , \psi - \psi' ] \Nt_{ \SSSm }  + 2 \Nt \Psi^\eps_\kin[ f^\eps_\disp , \psi - \psi' ] \Nt_{ \SSSm }  \lesssim \left(\eta + \beta(\eps)\right) \Nt \psi - \psi' \Nt_{ \SSSs }\,$$
			which, together with the previous estimates, yields
			\begin{align*}
				\Nt \Xi_2[\phi, & \varphi, \psi]  - \Xi_2[\phi', \varphi', \psi' ] \Nt_{ \SSSm } \\
				& \lesssim \left[\eps w_{ f_\ns, \eta }(T)^{-1} +\eta + \beta(\eps)\right]\Nt (\phi, \varphi, \psi) - (\phi', \varphi', \psi') \Nt_{ \SSSl \times \SSSm \times \SSSs },
			\end{align*}
			thus, taking $\eps \ll \eta$, the second component $\Xi_2$ is also a contraction. As far as the third component is concerned, one has
			\begin{equation*}\begin{split}
					\Nt \Xi_3[\phi, \varphi, \psi] & - \Xi_3[\phi', \varphi', \psi'] \Nt_{ \SSSs } \\
					& \le \Nt \Phi^\eps[\phi, \varphi] \psi - \Phi^\eps[\phi', \varphi'] \psi' \Nt_{ \SSSs } + \Nt \Psi^\eps_\hyd[\psi - \psi', \psi + \psi'] \Nt_{ \SSSs }\\
					&\phantom{++++} +  \Nt \mathcal{S}^\eps[\phi-\phi',\varphi]\Nt_{\SSSs} + \Nt \mathcal{S}^\eps[\phi',\varphi-\varphi']\Nt_{\SSSs}.
			\end{split}\end{equation*}
			Now, using Proposition \ref{prop:linear_hydrodynamic},
			\begin{equation*}\begin{split}
					\Nt \Phi^\eps[\phi, \varphi] \psi & - \Phi^\eps[\phi', \varphi'] \psi' \Nt_{ \SSSs } \\
					\lesssim & \left( \eta + \beta(\eps) + \eps c_2 w_{ f_\ns, \eta }(T)^{-1} + \eps\right) \Nt \psi - \psi' \Nt_{ \SSSs } \\
					& + c_3 \eps w_{ f_\ns, \eta }(T)^{-1} \Nt \varphi - \varphi' \Nt_{ \SSSm }
					+ c_3 \eps \Nt \phi - \phi' \Nt_{ \SSSl } \\
					\lesssim & \left( \eta + \beta(\eps) + \eps c_2 w_{ f_\ns, \eta }(T)^{-1} + \eps \right) \Nt (\phi , \varphi, \psi) - (\phi' , \varphi', \psi') \Nt_{ \SSSl \times \SSSm \times \SSSs } \\
					\lesssim & \left( \eta + \beta(\eps) w_{ f_\ns, \eta }(T)^{-1} \right) \Nt (\phi , \varphi, \psi) - (\phi' , \varphi', \psi') \Nt_{ \SSSl \times \SSSm \times \SSSs },
			\end{split}\end{equation*}
			while \eqref{eq:bilinear_hyd_hyd_hyd} yields
			$$\Nt \Psi^\eps_\hyd[\psi - \psi' , \psi + \psi' ] \Nt_{ \SSSs } \lesssim w_{f_\ns, \eta}(T)^{-1} c_3 \Nt \psi - \psi' \Nt_{ \SSSs }.$$
			Finally, Proposition \ref{prop:source_term} easily gives
			\begin{align*}
				\Nt \mathcal{S}^\eps[\phi,\varphi] -\mathcal{S}^\eps[\phi',\varphi']\Nt_{\SSSs}
				& \leq \Nt \mathcal{S}_3^\eps[\phi-\phi',\varphi]\Nt_{\SSSs} + \Nt \mathcal{S}_3^\eps[\phi',\varphi-\varphi']\Nt_{\SSSs} \\
				& \lesssim  \eps\,\bigg(\Nt \phi-\phi'\Nt_{\SSSl} + \Nt \varphi-\varphi'\Nt_{\SSSm}\bigg)\,. 
			\end{align*}
			All these estimates yield
			\begin{align*}
				\Nt  \Xi_3[\phi, \varphi, \psi] & - \Xi_3[\phi', \varphi', \psi'] \Nt_{ \SSSs } \\
				& \lesssim \Big[{\eps +} \eta + \left( \beta(\eps) + c_3 \right) w_{ f_\ns, \eta }(T)^{-1} \Big] \Nt (\phi , \varphi, \psi) - (\phi' , \varphi', \psi') \Nt_{ \SSSl \times \SSSm \times \SSSs },
			\end{align*}
			thus, taking $\max\{ \eps, c_3 \} \ll \eta \ll 1$, the third component $\Xi_3$ is indeed a contraction.
		\end{proof}	
		\bigskip
		
		\subsection{Proof of Theorem \ref{thm:hydrodynamic_limit}: existence and convergence}
		
		We have established that $\bm{\Xi}$ is a well-defined contraction on $\bm{B}=B_1 \times B_2 \times B_3$ under the smallness assumption
		$$\eps \ll c_3 \ll \eta \ll c_2 \ll 1,$$
		thus it admits a unique fixed point denoted $(f^\eps_\kin, f^\eps_\mix, g^\eps)$. The part $g^\eps$ satisfies (for some sufficiently small $c > 0$)
		\begin{equation*}
			\Nt g^\eps \Nt_{\SSSs} \lesssim c \Nt g^\eps \Nt_{ \SSSs } + \Nt \SS^\eps \Nt_{ \SSSs },
		\end{equation*}
		and therefore, considering $c$ small enough, we have:
		\begin{equation*}
			\Nt g^\eps \Nt_{\SSSs} \lesssim \Nt \SS^\eps \Nt_{\SSSs} \lesssim \beta(\eps).
		\end{equation*}
		The part $f^\eps_\mix$ satisfies the equation
		\begin{equation*}\begin{split}
			f^\eps_\mix - \Psi^\eps_\kin[ f_\ns ,  f_\ns]
			&=   \Psi^\eps_\kin[f^\eps_\mix , f^\eps_\mix]  +  \Psi^\eps_\kin[ g^\eps , g^\eps ]  +  \Psi^\eps_\kin[ f^\eps_\disp , f^\eps_\disp ]  \\
			& + 2  \Psi^\eps_\kin[ f^\eps_\mix , f_\ns]  + 2  \Psi^\eps_\kin[ f^\eps_\mix , f_\disp^\eps ]  + 2  \Psi^\eps_\kin[ f^\eps_\mix , g^\eps ]  \\
			& + 2  \Psi^\eps_\kin[ f_\ns , f^\eps_\disp ]  + 2  \Psi^\eps_\kin[ f_\ns , g^\eps ]  + 2  \Psi^\eps_\kin[ f^\eps_\disp , g^\eps ],
		\end{split}\end{equation*}
		therefore, from the computations of Section \ref{scn:mapping_contraction}, we have
		$$\Nt f^\eps_\mix - \Psi^\eps_\kin[f_\ns , f_\ns ]  \Nt_{ \SSSm } \lesssim \Nt g^\eps \Nt_{ \SSSs }   + \beta(\eps) \lesssim \beta(\eps).$$
		Furthermore, by a duality argument similar to the one from the proof of \eqref{eq:decay_convolution_semigroup_exp}
		\begin{align*}
			\int_{t_1}^{t_2} \big\la   \Psi^\eps_\kin[f_\ns, f_\ns](t) , \phi_0 \big\ra_{\SSs} \d t 
			& = \frac{1}{\eps} \int_{t_1}^{t_2} \int_0^t \left\la \QQ\left( f_\ns(\tau) , f_\ns(\tau) \right) , U^\eps_\kin(t-\tau)^\star \phi_0 \right\ra_{\SSs} \d \tau \, \d t \\
			& \lesssim \frac{1}{\eps} \| f_\ns \|_{ L^\infty\left( [0, T) ; \SSsp \right) } (t_2 - t_1) \int_0^\infty \left\| U^\eps_\kin(t)^\star \phi_0 \right\|_{ \SSsp } \d t \\
			& \lesssim (t_2 - t_1) \left(\int_0^\infty e^{2 \sigma t / \eps^2} \left\| U^\eps_\kin(t)^\star \phi_0 \right\|_{ \SSsp }^2 \d t\right)^{\frac{1}{2}} \lesssim \eps (t_2 - t_1) \| \phi_0 \|_{\SSs},
		\end{align*}
		where we used that $\| f_\ns \|_{ L^\infty\left( [0, T) ; \SSsp \right) } \lesssim 1$. Thus, we deduce
		$$\left\| \Psi^\eps_\kin[f_\ns, f_\ns] \right\|_{ L^\infty\left( [0, T) ; \SSs \right) } \lesssim \eps,$$
		from which we conclude that
		$\left\| f^\eps_\mix + g^\eps \right\|_{ L^\infty\left( [0, T) ; \SSs \right) } \lesssim \beta(\eps).$
		We conclude to Theorem \ref{thm:hydrodynamic_limit} by letting
		\begin{gather*}
			f^\eps_\err := g^\eps + f^\eps_\mix.
		\end{gather*}
		This concludes the proof.
		
		\subsection{Proof of Theorem \ref{thm:hydrodynamic_limit}: uniqueness}
		\label{scn:uniqueness_simmetrizable}
		
		Consider another solution associated with the same initial data $f_\ini$:
		$$\overline{f}^\eps \in L^\infty\left( [0, T) ; \SSl \right) \cap L^2_{\text{loc}}\left( [0, T) ; \SSlp \right)$$
		satisfying for some universal small $c > 0$ the bound (note that the same bound holds for $f^\eps$ since $\| f^\eps \|_{L^\infty_t (\SSl)} \lesssim 1$)
		$$\| \overline{f}^\eps \|_{ L^\infty \left( [0, T) ; \SSl \right) } \le \frac{c}{\eps}.$$
		Define the difference of solutions
		$$h^\eps = f^\eps - \overline{f}^\eps$$
		and observe it satisfies the  equation
		$$\partial_t h^\eps = \frac{1}{\eps^2} \left( \LL - \eps v \cdot \nabla_x \right) h^\eps + \frac{1}{\eps} \QQ \left( h^\eps, f^\eps + \overline{f}^\eps \right), \quad h^\eps(0) = 0.$$
		We write an energy estimate for $h^\eps$ (see \textit{Step 2} of the proof of Lemma \ref{lem:decay_regularization_convolution_kinetic_semigroup} for the dissipative part):
		$$\frac{1}{2} \frac{\d}{\d t} \| h^\eps \|^2_{ \SSl } + \frac{\lambda}{\eps^2} \| h^\eps \|^2_{ \SSlp } \lesssim \frac{1}{\eps^2} \| h^\eps \|^2_{ \SSl } + \frac{1}{\eps} \| h^\eps \|_{ \SSlp }^2 \| f^\eps + \overline{f}^\eps \|_{ \SSl } + \frac{1}{\eps} \| h^\eps \|_{ \SSl } \| h^\eps \|_{ \SSlp } \| f^\eps + \overline{f}^\eps \|_{ \SSlp }, $$ 
		which gives after integrating on $[0, t]$ in the space $\SSSl(\sigma=0, t, \eps)$:
		\begin{align*}
			\Nt h^\eps \Nt_{ \SSSl }^2 \lesssim \frac{t}{\eps^2} \Nt h^\eps \Nt_{ \SSSl }^2 + c \Nt h^\eps \Nt_{ \SSSl }^2 + \Nt h^\eps \Nt_{ \SSSl }^2 \left(\int_0^t \| f^\eps(\tau) + \overline{f}^\eps(\tau) \|^2_{\SSlp} \d \tau\right)^{1/2}.
		\end{align*}
		Thus, since $c$ is supposed to be small, taking $t$ close enough to $0$ yields (for instance)
		$$\Nt h^\eps \Nt_{ \SSSl } \le \frac{1}{2} \Nt h^\eps \Nt_{ \SSSl },$$
		which in turn implies $h(\tau) = 0$, or equivalently $f^\eps(\tau) = \overline{f}^\eps(\tau)$ for any $\tau \in [0, t]$. Repeating this argument yields the uniqueness of the solution.

		\section{Proof of Theorem \ref{thm:hydrodynamic_limit-gen_degenerate}}
		\label{scn:hydrodynamic_limit_BED}
		
		We prove here Theorem~\ref{thm:hydrodynamic_limit-gen_degenerate} under the assumption \ref{BED}. Note that the following strategy can also be seen as an alternative proof under assumption \ref{BE}.
		
		\subsection{Modification of the strategy}
		Under the assumption \ref{BED}, the arguments of $\QQ(f, g)$ in $\Sl$ no longer play symmetric roles, so \textbf{we no longer consider $\QQ$ in its symmetrized form}. This does not induce any change for the parts $f^\eps_\hyd$ and $f^\eps_\mix$ of the solution since they are constructed using assumption \ref{Bbound} and not \ref{BED}, for which both arguments play symmetric roles. The only modification is therefore the  need to adjust the iterative scheme constructing $f^\eps_\kin$ as well as its space-velocity functional space so as to take into account the assumptions \ref{BED} following the strategy adopted in \cite{GMM2017, CTW2016, CM2017, HTT2020, CG2022}. We detail below this new strategy.
		\begin{itemize}
			\item We no longer consider the equation on $f^\eps_\kin$ in integral form
			$$f^\eps_\kin(t) = U^\eps_\kin(t) f_\ini + \Psi^\eps_\kin \left[ f^\eps_\kin, f^\eps_\kin \right](t)+ 2 \Psi^\eps_\kin\left[f^\eps_\kin, f^\eps_\hyd + f^\eps_\mix\right](t),$$
			but we study the evolution of $f_\kin^\eps$ it in its differential form
			\begin{equation}\label{eq:diffFkin}\begin{split}
					\partial_t f^\eps_\kin = \frac{1}{\eps^2} \left(\LL - \eps v \cdot \nabla_x\right) f^\eps_\kin &+ \frac{1}{\eps} \PP^\eps_\kin \QQ(f^\eps_\kin, f^\eps_\kin)  \\ \phantom{+++}&+ \frac{2}{\eps} \PP^\eps_\kin \QQ^\sym(f^\eps_\kin, f^\eps_\hyd + f^\eps_\mix).
			\end{split}\end{equation}
			This equation can be studied through  a suitable energy method so as to be able to use the ‘‘closing estimate'' of \ref{BED} (which does not translate in integral form).
			\item Since the roles played by both arguments of $\QQ(f, g)$ in $\Sl$ under the assumption \ref{BED} are different, we do not construct $f^\eps_\kin$ using Banach's theorem, which, as far as \eqref{eq:diffFkin} is concerned, would correspond to the convergence an iterative scheme of the form
			\begin{equation*}\begin{split}  \partial_t f^\eps_{\kin, N} = \frac{1}{\eps^2} \left( \LL - \eps v \cdot \nabla_x\right) f^\eps_{\kin, N} &+ \frac{1}{\eps} \PP^\eps_\kin \QQ(f^\eps_{\kin, N-1}, f^\eps_{\kin, N-1}) \\ \phantom{+++}&+ \frac{2}{\eps} \PP^\eps_\kin \QQ^\sym(f^\eps_{\kin, N-1} f^\eps_{\hyd, N-1}  + f^\eps_{\mix,N-1})\end{split}\end{equation*}
			but using a variation of such a scheme which allows to use the the "closing estimate'' of \ref{BED}. Namely, we prove the stability of the scheme \eqref{eq:scheme} hereafter.
			\item We define a new hierarchy of spaces $(\SSh_j)_{j=-2-s}^1$ of the form
			$$\SSh_j = L^2_x \left( \Sh_j \right) \cap \dot{\mathbb{H}}^s_x \left( \Sh_{j-s} \right)$$
			which allows to prove spatially inhomogeneous counterparts of the estimates of \ref{BED}. Notice here that we assume our ``regularity parameter'' $s$ to be integer $s \in \N$ and it is now assumed an additional role in the hierarchy of spaces $\SSh_{-2-s},\ldots,\SSh_{1}.$ 
			\item The operator $\LL-\eps v \cdot \nabla_x$ is not dissipative for the inner product of $\Sh$, but it is \emph{hypo-dissipative} on $\range\left( \PP^\eps_\kin \right)$, so, we introduce an equivalent inner product of the form
			\begin{equation*}
				\dlla f, g \drra_{\SSh_j, \eps} := \delta \la f, g \ra_{\SSh_j} + \frac{1}{\eps^2} \int_0^\infty \lla U^\eps_\kin(t) f, U^\eps_\kin(t) g \rra_{\SSh_{j-1}} \d t,
			\end{equation*}
			for which $\LL-\eps v \cdot \nabla_x$ is dissipative and $\QQ$ satisfies the same estimates as \ref{BED}.
		\end{itemize}
		To summarize, our approach will be aimed at constructing a solution $(f^\eps_\kin, f^\eps_\mix, g^\eps)$ as the limit of a sequence of approximate solutions $\left\{( f^\eps_{\kin, N}, f^\eps_{\mix, N}, g^\eps_N )\right\}_{N \geq 0}$, where the first component $f_{\kin,N}^\eps$  is constructed inductively by solving the following differential equation:
		\begin{equation}\label{eq:scheme}
			\begin{cases}
				\displaystyle
				\begin{aligned}
					\partial_t f^\eps_{\kin, N} = \frac{1}{\eps^2} \left(\LL - \eps v \cdot \nabla_x\right) f^\eps_{\kin, N} & + \frac{1}{\eps} \PP^\eps_\kin \QQ(f^\eps_{\kin, N-1}, f^\eps_{\kin, N})  \\
					& + \frac{2}{\eps} \PP^\eps_\kin \QQ^\sym(f^\eps_{\kin, N}, f^\eps_{\hyd, N-1} + f^\eps_{\mix, N-1})
				\end{aligned}\\
				f^\eps_{\kin, N}(0) = \PP^\eps_\kin f_\ini,\\
				f^\eps_{\kin, 0} = 0,
			\end{cases}
		\end{equation}
		where we naturally denoted $f^\eps_{\hyd, N} = f_\ns + f^\eps_\disp + g^\eps_N$, and the other parts are still constructed as in Section \ref{scn:proof_hydrodynamic_limit_symmetrizable}:
		\begin{equation*}
			\begin{cases}
				g^\eps_N =  \Phi^\eps[f^\eps_{\kin, N-1} , f^\eps_{\mix, N-1} ] g^\eps_{N-1} + \Psi^\eps_\hyd[g^\eps_{N-1}, g^\eps_{N-1}]
				+ \SS^\eps[f^\eps_{\kin, N-1} , f^\eps_{\mix, N-1}], \\
				\\
				f^\eps_{\mix, N} = \Psi^\eps_\kin\left[ f^\eps_{\mix,N-1} + f^\eps_{\kin, N-1} , f^\eps_{\mix,N-1} + f^\eps_{\kin, N-1} \right], \\
				\\
				g^\eps_0 = 0, \qquad f^\eps_{\mix,0} = 0.
			\end{cases}
		\end{equation*}
	
		Let us define the new functional space we will use in this section.
		\begin{defi}
			Suppose $s \in \N$ and satisfies $s \ge 3$ if $d = 2$ or $s \ge \frac{d}{2} + 1$ if $d \ge 3$, we define
			\begin{equation*}
				\| f \|_{ \SSh_j^s } := \| f \|_{ L^2_x \left( \Sh_j \right) } + \| \nabla_x^s f \|_{ L^2_x \left( \Sh_{j-s} \right) },
			\end{equation*}
			and note that, since $\Sh_k \hookrightarrow \Sh_j$ as soon as $j \leq k$, the following equivalence of norms holds:
			\begin{equation*}
%				\label{eq:equivalence_degenerate_norm}
				\| f \|_{ \SSh_j^s } \approx \sum_{k = 0}^{s} \| f \|_{ \mathbb{H}^k_x \left( \Sh_{j-k} \right) } \approx \sum_{k = 0}^{s} \| f \|_{ \SSh_j^k }.
			\end{equation*}
			In particular, this hierarchy of spaces is decreasing in both indexes:
			\begin{equation}
				\label{eq:hierarchy_degenerate_spaces}
				s_1 \le s_2 ~ \text{ and } ~ j_1 \le j_2 \Longrightarrow \SSh_{j_2}^{s_2} \hookrightarrow \SSh_{j_1}^{s_1}.
			\end{equation}
		\end{defi}
	
		\begin{lem}\label{lem72}
			The bilinear operator $\QQ$ satisfies in $\SSh_j$ the estimates
			\begin{gather}
				\label{eq:Q_sobolev_algebra_degenerate}
				\left\| \QQ(f, g) \right\|_{ \SShm_j } \lesssim \| f \|_{ \SSh_j } \| g \|_{ \SShp_{j+1} } + \| f \|_{ \SShp_j } \| g \|_{ \SSh_{j+1} }, \\
				\label{eq:Q_sobolev_algebra_degenerate_closed}
				\la \QQ(f, g), g \ra_{ \SSh_j } \lesssim \| f \|_{ \SSh_j } \| g \|_{ \SShp_j }^2 +  \| f \|_{ \SShp_j } \| g \|_{ \SSh_j } \| g \|_{ \SShp_j }.
			\end{gather}
		\end{lem}
	
		\begin{proof}
			The general non-closed control \eqref{eq:Q_sobolev_algebra_degenerate} is easily obtained from its spatially homogeneous counterpart of \ref{BED} so we only prove the closed control \eqref{eq:Q_sobolev_algebra_degenerate_closed}.
			
			The inner product writes according to Leibniz's formula and the locality of $\QQ$ as
			\begin{equation*}\begin{split}
				\left\la \QQ(f, g), h \right\ra_{ \SSh_j } = & \left\la \QQ(f, g), h \right\ra_{ L^2_x \left( \Sh_{j} \right) } + \sum_{| \beta | = s } \left\la \partial_x^\beta \QQ(f, g), \partial_x^\beta h \right\ra_{ L^2_x \left( \Sh_{j-s} \right) } \\ = & \left\la \QQ(f, g), h \right\ra_{ L^2_x \left( \Sh_{j} \right) } + {\sum_{|\gamma|+|\beta-\gamma|=|\beta|=s}} \left\la \QQ\left( \partial_x^\gamma f, \partial_x^{\beta-\gamma} g \right), \partial_x^{\beta} h \right\ra_{ L^2_x \left( \Sh_{j-s} \right) }.
			\end{split}\end{equation*}
			We first look at the terms which can be controlled using the closed estimate \eqref{eq:BEDclosed} of \ref{BED}.  Using H\"older's inequality in $L^\infty_x \times L^2_x \times L^2_x$ (or some appropriate permutation) together with the embeddings $\mathbb{H}^s_x \hookrightarrow L^\infty_x$ and $\Sh_{j-s} \hookrightarrow \Sh_{-1-s}$ and therefore $\mathbb{H}^s_x \left( \Sh_{j-s} \right) \hookrightarrow L^\infty_x \left( \Sh_{-1-s} \right)$, we immediately have the estimate
			\begin{equation*}
				\begin{split}
					\la \QQ(f, g), g \ra_{ L^2_x \left( \Sh_{j} \right) }
					\lesssim & \sum_{ \{ a, b, c \} = \{j, j, -1-s \} } \left( \int_{\R^d} \| f \|_{\Sh_{ a }} \| g \|_{\Shp_{b}} \| g \|_{\Shp_{c}} \d x + \int_{\R^{d}}\| f \|_{\Shp_{a}} \| g \|_{\Sh_{b}} \| g \|_{\Shp_{c}} \d x \right)\\
					\lesssim & \| f \|_{ \SSh_j } \| g \|_{ \SShp_j }^2 +  \| f \|_{ \SShp_j } \| g \|_{ \SSh_j } \| g \|_{ \SShp_j }
				\end{split}
			\end{equation*}
			\color{black}
			as well as the terms associated with $| \beta | = s$ and $|\gamma| = 0$:
			$$\left\la \QQ\left(f, \partial_x^\beta g\right), \partial_x^\beta g \right\ra_{ L^2_x \left( \Sh_{j-s} \right) } \lesssim \| f \|_{ \SSh_j } \| g \|_{ \SShp_j }^2 +  \| f \|_{ \SShp_j } \| g \|_{ \SSh_j } \| g \|_{ \SShp_j }.$$
			We are thus left with the terms associated with $|\beta|=s$ and  $|\gamma| \ge 1$ which have to be controlled starting from the non-closed estimate \eqref{eq:BEDNon} of \ref{BED}:
			\begin{equation}\label{eq:QQgammabeta}
				\begin{split}
					\Big\la \QQ\big( \partial_x^\gamma f, & \partial_x^{\beta-\gamma} g\big), \partial_x^\beta g \Big\ra_{ L^2_x \left( \Sh_{j-s} \right) } \\
					\lesssim & \int_{\R^{d}} \| \partial_x^\gamma f(x) \|_{ \Sh_{j-s} } \| \partial_x^{\beta-\gamma} g(x) \|_{ \Shp_{j - (s-1) } } \| \partial_x^\beta g(x) \|_{ \Shp_{j - s } }  \d x \\
					& + \int_{\R^{d}} \| \partial_x^\gamma f(x) \|_{ \Shp_{j-s} } \| \partial_x^{\beta-\gamma} g(x) \|_{ \Sh_{j - (s-1) } } \| \partial_x^\beta g(x) \|_{ \Shp_{j - s } }  \d x \, ,
				\end{split}
			\end{equation}
			and we will show in each case that
\begin{equation}
				\label{eq:degenerate_bilinear_goal}
				\begin{split}
\left\|\| \partial_x^\gamma f(x) \|_{ \Sh_{j-s} } \| \partial_x^{\beta-\gamma} g(x) \|_{ \Shp_{j - (s-1) } } \right\|_{L^{2}_{x}} &\lesssim \|f\|_{\SSh_{j}}\|g\|_{\SShp_{j}}\\
\left\|\| \partial_x^\gamma f(x) \|_{ \Shp_{j-s} } \| \partial_x^{\beta-\gamma} g(x) \|_{ \Sh_{j - (s-1) } } \right\|_{L^{2}_{x}} &\lesssim \|f\|_{\SShp_{j}}\|g\|_{\SSh_{j}}\,.\end{split}
\end{equation}

			\step{1}{The case $d = 2$}
			When $| \gamma | = s$, we have $|\beta - \gamma| = 0$, thus, using the fact that $\mathbb{H}^{s-1}_x \hookrightarrow L^\infty_x$ since $s \ge 3$, followed by \eqref{eq:hierarchy_degenerate_spaces}
			$$\| \partial_x^{\beta-\gamma} g \|_{ L^\infty_x \left( \Sh_{j- (s-1) } \right) } \lesssim \| g \|_{ \mathbb{H}^{s-1}_x \left( \Sh_{j- (s-1) } \right) } \lesssim \| g \|_{ \SSh_j }, $$
			thus, \eqref{eq:degenerate_bilinear_goal} holds. When $| \gamma | = s-1$, using the injection $\mathbb{H}^1_x \hookrightarrow L^4_x$, we have
			$$\| \partial_x^\gamma f \|_{ L^4_x \left( \Sh_{j-s} \right) } \lesssim \|  f \|_{ \mathbb{H}^{ s }_x \left( \Sh_{j-s} \right) } \lesssim \| f \|_{ \SSh_j },$$
			similarly, since $| \beta - \gamma | = 1 \le s - 2$, we also have
			$$\| \partial_x^{\beta-\gamma} g \|_{ L^4_x \left( \Sh_{j - (s-1) } \right) } \lesssim \|  f \|_{ \mathbb{H}^{ s-1 }_x \left( \Sh_{j-(s-1) } \right) } \lesssim \| f \|_{ \SSh_j },$$
			thus,  \eqref{eq:degenerate_bilinear_goal} holds. When $| \gamma | \le s-2$, we have
			$$\| \partial_x^\gamma f \|_{ L^\infty_x \left( \Sh_{j-s} \right) } \lesssim \| f \|_{ \mathbb{H}^s_x \left( \Sh_{j-s} \right) }, \lesssim \| f \|_{ \SSh_j }$$
			and similarly, since $| \beta - \gamma | \le s-1$ (recall that $|\beta|=s$ and $| \gamma \ge 1$)
			$$\| \partial_x^{\beta-\gamma} g \|_{ L^2_x \left( \Sh_{j-(s-1)} \right) } \lesssim \| f \|_{ \mathbb{H}^{s-1}_x \left( \Sh_{j-(s-1)} \right) } \lesssim \| f \|_{ \SSh_j },$$
			thus \eqref{eq:degenerate_bilinear_goal} holds. This concludes this step.
			
			\step{2}{The case $d = 3$}
			First, a simple use of Cauchy-Schwarz inequality yields
$$\big\| \| \partial_x^\gamma f \|_{ \Sh_{j-s} } \| \partial_x^{\beta-\gamma} g \|_{ \Shp_{j- (s-1) }  } \big \|_{L^1_x} \leq \| \partial_x^\gamma f \|_{ L^2_x \left( \Sh_{j-s} \right) }	\| \partial_x^{\beta-\gamma}g \|_{ L^2_x \left( \Shp_{j- (s-1) } \right) }$$
and, since $|\beta-\gamma|+|\gamma| \leq s$, we deduce that
\begin{equation}\label{productL1}
\big\| \| \partial_x^\gamma f \|_{ \Sh_{j-s} } \| \partial_x^{\beta-\gamma} g \|_{ \Shp_{j- (s-1) }  } \big \|_{L^1_x} \lesssim \|f \|_{\SSh_j}\,\|g \|_{\SShp_j}.\end{equation}
Now, using Sobolev embeddings, we have
			$$\| \partial_x^\gamma f \|_{ L^p_x \left( \Sh_{j-s} \right) } \lesssim \| f \|_{ \mathbb{H}^s_x \left( \Sh_{j-s} \right) } \lesssim \|f \|_{\SSh_j}, \qquad \frac{1}{p} = \frac{1}{2} - \frac{s - | \gamma| }{d}$$
			as well as
			$$\| \partial_x^{\beta-\gamma}g \|_{ L^q_x \left( \Shp_{j- (s-1) } \right) } \lesssim \| g \|_{ \mathbb{H}^{s-1}_x \left( \Shp_{j-(s-1) } \right) } \lesssim \|g \|_{\SShp_j}, \qquad \frac{1}{q} = \frac{1}{2} - \frac{ (s-1) - | \beta - \gamma| }{d}.$$
			Since $| \beta | = s$ and $s \ge \frac{d}{2} + 1$,  H\"older's inequality implies
			$$\big\|  \partial_x^\gamma f \|_{ \Sh_{j-s} } \| \partial_x^{\beta-\gamma} g \|_{ \Shp_{j- (s-1) }  } \big \|_{L^r_x} \lesssim \| f \|_{\SSh_j} \| g \|_{ \SShp_j }, \qquad \frac{1}{r} = \frac{1}{p} + \frac{1}{q} = 1 - \frac{s - 1}{d} \le \frac{1}{2}\,.$$
			Using \eqref{productL1} and simple
 interpolation, we deduce that
			$$\big\| \| \partial_x^\gamma f \|_{ \Sh_{j-s} } \| \partial_x^{\beta-\gamma} g \|_{ \Sh_{j- (s-1) }  } \big \|_{  L^2_x} \lesssim \| f \|_{\SSh_j} \| g \|_{ \SShp_j },$$
this proves the first estimate in \eqref{eq:degenerate_bilinear_goal} and the second one is done in the same way.
			Combining then  \eqref{eq:degenerate_bilinear_goal}  with \eqref{eq:QQgammabeta} and Cauchy-Schwarz inequality yields
			\begin{equation*}
\Big\la \QQ\big( \partial_x^\gamma f,   \partial_x^{\beta-\gamma} g\big), \partial_x^\beta g \Big\ra_{ L^2_x \left( \Sh_{j-s} \right) }  					\lesssim \| f \|_{ \SSh_j } \| g \|_{ \SShp_j }^2 +  \| f \|_{ \SShp_j } \| g \|_{ \SSh_j } \| g \|_{ \SShp_j }.
			\end{equation*}
			This concludes this step and the proof.
		\end{proof}
	
		As stated above, we will study the equation \eqref{eq:scheme} using an equivalent inner product. The next proposition defines it and presents its properties.
		
		\begin{prop}[\textit{\textbf{Kinetic dissipative inner product}}]
			\label{prop:dissipative_kinetic_inner_product}
			Let $j = -1, 0$ and $\sigma \in (0, \sigma_0)$. For some $\delta > 0$  small enough, the inner product defined for any $f, g \in \range\left( \PP^\eps_\kin \right) \cap \SSh_j$ as
			\begin{equation*}
				\dlla f, g \drra_{\SSh_{j}, \eps} := \delta \la f, g \ra_{\SSh_{j}} + \frac{1}{\eps^2} \int_0^\infty e^{2 \sigma t / \eps^2} \lla U^\eps_\kin(t) f, U^\eps_\kin(t) g \rra_{ \SSh_{j-1} } \d t,
			\end{equation*}
			induces a norm equivalent to that of $\SSh_{j}$ (uniformly in $\eps$), i.e. there exists $C >0$ independent of $\eps$ such that
			\begin{equation}
				\label{eq:equivE}
				\frac{1}{C} \| f \|_{\SSh_{j}} \le \| f \|_{\SSh_{j}, \eps} \le C \| f \|_{\SSh_{j}}.
			\end{equation}
			Moreover, there is $\mu >0$ such that
			$$\re \dlla (\LL - \eps v \cdot \nabla_x)  f, f \drra_{\SSh_{j}, \eps } \le - \sigma \| f \|_{\SSh_{j}, \eps}^2 - \mu \| f \|^2_{\SShp_{j}},$$
			and the nonlinear estimates for $\PP^\eps_\kin \QQ$ are the same as those from \eqref{eq:Q_sobolev_algebra_degenerate} and \eqref{eq:Q_sobolev_algebra_degenerate_closed}: 
			\begin{align*}
				\dlla \PP^\eps_\kin \QQ(f, g) , h \drra_{\SSh_{j} , \eps } \lesssim & \| h \|_{\SShp_{j}} \left( \| f \|_{\SShp_{j}} \| g \|_{\SSh_{j+1}} + \| f \|_{\SSh_{j}} \| g \|_{\SShp_{j+1}} \right),
			\end{align*}
			and
			\begin{equation*}
				\dlla \PP^\eps_\kin \QQ(f, g) , g \drra_{\SSh_{j} , \eps } \lesssim \| f \|_{\SSh_{j}} \| g \|_{\SShp_{j}}^2 + \| g \|_{ \SSh_{j} } \| g \|_{\SShp_{j}} \| f \|_{ \SShp_j }.
			\end{equation*}
		\end{prop}

		\begin{proof}
			
			\step{1}{Proof of the equivalence of norms}
			The norm $\| \cdot \|_{\SSh_{j}, \eps }$ writes
			$$\| f \|_{ \SSh_{j} , \eps }^2  = \delta \| f \|^2_{ \SSh_j} + \frac{1}{\eps^2} \int_0^\infty e^{2 \sigma t / \eps^2 } \left\| U^\eps_\kin(t) f \right\|^2_{\SSh_{j-1}} \d t,$$
			and thus, using $\SSh_j \hookrightarrow \SSh_{j-1}$ and the decay estimate for the semigroup $U^\eps_\kin(t)$ from Lemma~\ref{lem:decay_regularization_kinetic_semigroup}, we have for some $C > 0$
			$$\delta \| f \|_{\SSh_j }^2 \le \| f \|_{\SSh_j , \eps }^2 \le (\delta + C) \| f \|^2_{ \SSh_j}.$$
			This proves \eqref{eq:equivE}.
			
			\step{2}{Proof of the dissipative estimate}
			Using the decomposition $\LL = \BB + \AA$ coming from \ref{LE} in the space $\Sh_{j}$  (recall that, from \ref{BED}, the dissipativity estimate of \ref{LE} is valid in $\Sh_{j}$) together with the estimate for $\AA$ from \ref{BED}, we have for some~$K > 0$
			\begin{equation*}\begin{split}
				\re & \dlla (\LL - \eps v \cdot \nabla_x)  f, f \drra_{\SSh_j , \eps } \\
				&=  \delta \re \lla (\LL - \eps v \cdot \nabla_x)  f , f \rra_{\SSh_j} \\
				& + \re \int_0^\infty e^{2 \sigma t / \eps^2} \lla U^\eps_\kin(t) \left\{ \frac{1}{\eps^2} \left( \LL - \eps v \cdot \nabla_x\right) f \right\} , U^\eps_\kin(t) f \rra_{\SSh_{j-1}} \d t \\
				&\le   - \delta \lambda_\BB \| f \|_{ \SSh_j}^2 + \delta K \| f \|_{ \SSh_{j-1}}^2 \\
				&  -\frac{\sigma}{\eps^2} \int_0^\infty e^{2 \sigma t / \eps^2} \left\| U^\eps_\kin(t) f \right\|_{\SSh_{j-1}}^2 \d t + \frac{1}{2} \int_0^\infty \frac{\d}{\d t} \left[ e^{2 \sigma t / \eps^2} \| U^\eps_\kin(t) f \|^2_{\SSh_{j-1} } \right] \, \d t,
			\end{split}\end{equation*}
			and thus, since $\lim_{t\to\infty}\| U^\eps(t) f \|_{\SSh_{j-1}}=0$ and $\| \cdot \|_{ \SShp_{j}} \ge \| \cdot \|_{\SSh_j}$, we have
			\begin{equation*}\begin{split} 
				\re \dlla (\LL - \eps v \cdot \nabla_x)  f, f \drra_{\SSh_j, \eps }
				\le&   - \delta \lambda_\BB \| f \|_{ \SShp_{j} }^2 + \delta K \| f \|_{ \SSh_{j-1}}^2 \\
				&-\frac{\sigma}{\eps^2} \int_0^\infty e^{2 \sigma t / \eps^2} \left\| U^\eps_\kin(t) f \right\|_{\SSh_{j-1}}^2 \d t - \frac{1}{2} \| f \|^2_{ \SSh_{j-1}} \\
				\le& 
				- \sigma \left( \delta \| f \|^2_{\SSh_j} + \frac{1}{\eps^2} \int_0^\infty e^{2 \sigma t / \eps^2} \left\| U^\eps_\kin(t) f \right\|_{\SSh_{j-1}}^2 \d t  \right) \\
				& - \delta (\lambda_\BB - \sigma )  \| f \|^2_{ \SShp_{j} } - \left( \frac{1}{2} - \delta K \right) \| f \|^2_{ \mathbb{H}^s_x(\Shp_{j-1}) }.
			\end{split}\end{equation*}
			We finally deduce, considering $\delta \le (2K)^{-1}$ and letting $\mu = \delta(\lambda_\BB - \sigma) > 0$
			$$\re \dlla (\LL - \eps v \cdot \nabla_x)  f, f \drra_{\SSh_j , \eps } \le - \sigma \| f \|^2_{ \SSh_j, \eps } - \mu \| f \|_{\SShp_{j}}^2.$$
			This concludes this step.
			
			\step{3}{Proof of the nonlinear estimates}
			Using the definition $\PP^\eps_\kin = \Id - \PP^\eps_\hyd$, we have
			$$
			\dlla \PP^\eps_\kin \QQ(f, g) , h \drra_{\SSh_j, \eps} = \la \QQ(f, g), h \ra_{\SSh_j} + R(f, g, h),$$
			where $$
			R(f, g, h) := - \left\la \PP^\eps_\hyd \QQ(f, g), h \right\ra_{\SSh_j }  + \frac{1}{\eps^2} \int_0^\infty e^{2 \sigma t / \eps^2} \left\la U^\eps_\kin(t) \QQ(f, g),  U^\eps_\kin(t) h \right\ra_{ \SSh_{j-1}} \, \d t.$$
			On the one hand, the boundedness $\PP^\eps_\hyd \in \BBB\left( \SShm_{j-1} ; \SSs \right) \subset \BBB\left( \SShm_{j-1} ;\SShm_{j} \right)$ implies
			\begin{equation*}\begin{split}				\left|\, \la \PP^\eps_\hyd \QQ(f, g), h \ra_{\SSh_j} \right| & \lesssim \| h \|_{\SShp_{j} } \| \PP^\eps_\hyd \QQ(f, g) \|_{ \SShm_{j}} \\
				& \lesssim \| h \|_{ \SShp_{j}} \| \QQ(f, g) \|_{ \SShm_{j-1}},
			\end{split}\end{equation*}
			and on the other hand, using Cauchy-Schwarz inequality and the integral estimates of Lemma~\ref{lem:decay_regularization_kinetic_semigroup}:
			\begin{equation*}\begin{split}
					\frac{1}{\eps^2} \int_0^\infty e^{2 \sigma t / \eps^2 } & \la U^\eps_\kin(t) \QQ(f, g),  U^\eps_\kin(t) h \ra_{ \SSh_{j-1}} \, \d t \\
					& \le \frac{1}{\eps^2} \int_0^\infty \left( e^{\sigma t / \eps^2} \| U^\eps_\kin(t) \QQ(f, g) \|_{ \SSh_{j-1}} \right) \left( e^{\sigma t / \eps^2} \| U^\eps_\kin(t) h \|_{ \SSh_{j-1}} \right) \d t \\
					& \lesssim \| \QQ(f, g) \|_{ \SShm_{j-1}} \| h \|_{  \SSh_{j-1} }  \lesssim \| \QQ(f, g) \|_{ \SShm_{j-1} } \| h \|_{ \SShp_{j-1} }.
			\end{split}\end{equation*}
			Combining  these estimates and using the injection $ \SShp_{j}\hookrightarrow  \SSh_{j-1}$, we deduce that
			$$R(f, g, h) \lesssim \| h \|_{  \SShp_{j} } \| \QQ(f, g) \|_{  \SShm_{j-1} }.$$
			Recalling that $ \SSh_{j-1}$ is defined so that the general non-closed estimate of $\QQ$ in $\la \cdot , \cdot \ra_{\SShp_{j-1} }$ only involves the norms of the space $ \Sh_{j} \hookrightarrow  \Sh_{j-1}$ and $\Shp_{j} \hookrightarrow \Sh_{j-1}$ (but  not~$\Sh_{j+1}$ nor $\Shp_{j+1}$), that is to say, 
			$$ \| \QQ(f, g) \|_{ \SShm_{j-1} } \lesssim \| f \|_{ \SShp_{j} } \| g \|_{ \SSh_j} + \| f \|_{  \SSh_j} \| g \|_{  \SShp_{j}},$$
			we finally end up with
			$$
				R(f, g, h)  \lesssim \| h \|_{ \SShm_{j}  } \left( \| f \|_{  \SShp_{j} } \| g \|_{  \SSh_j} + \| f \|_{  \SSh_j } \| g \|_{  \SShp_{j} } \right).
		$$	This allows to conclude this step and the proof.
		\end{proof}
		
		\subsection{Stability estimates} 
		
		We study here the scheme \eqref{eq:scheme} in the  kinetic-type time-position-velocity space $\SSSh_j$ which corresponds to  the space $\SSSl$ introduced in Definition \ref{defi:NORMSPACES} with $E$ replaced with $\Sh_{j}$, i.e $\SSSh_j$ is characterized by the norm:
		\begin{equation}
			\label{eq:defiTildeSSSl}
			\Nt f \Nt_{ \SSSh_j}^2 := \sup_{0 \le t < T} \, e^{2 \sigma t / \eps^2 } \| f(t) \|_{\SSh_j}^2 + \frac{1}{\eps^2} \int_0^T e^{2 \sigma t / \eps^2} \| f(t) \|_{\SShp_{j}}^2 \d t.
		\end{equation}
		We have the following estimates, valid for any $N\geq0$:
		\begin{lem}\label{lem:stable-scheme} With the notations of Section \ref{scn:proof_hydrodynamic_limit_symmetrizable}, one can choose
			$$\eps \ll c_3 \ll \eta \ll c_2 \ll 1, \qquad C_1 \approx 1$$
			such that, for any $N \ge 0$ and $j = -1, 0$:
			\begin{equation}
				\label{eq:inductive_stability}
				\Nt f^\eps_{\kin, N} \Nt_{\SSSh_j} \le C_1 \| \PP^\eps_\kin f_\ini \|_{   \SSh_j} , \qquad \Nt f^\eps_{\mix,N} \Nt_{\SSSm} \le c_2, \qquad \Nt g^\eps_N \Nt_{\SSSs} \le c_3\,.
			\end{equation}
		\end{lem}
		\begin{proof} The proof of the Lemma is made by induction over $N$. 
			Thanks to \eqref{eq:equivE}, it will turn useful to estimate the various norms rather with $\|\cdot\|_{ \SSh_j, \eps}$. 
			
			The estimates \eqref{eq:inductive_stability} are satisfied for $N = 0$. Let us assume they hold at rank $N-1$ for some~$N \ge 1$ and deduce them at rank $N$. 
			
			\medskip
			We recall that $f^\eps_{\kin,N}$ is a solution to \eqref{eq:scheme} where, thanks to Proposition \ref{prop:dissipative_kinetic_inner_product},  it holds
			\begin{gather*}
				\frac{1}{\eps^2} \dlla (\LL - \eps v \cdot \nabla_x) f^\eps_{\kin, N}, f^\eps_{\kin, N} \drra_{  \SSh_j, \eps} \le - \frac{\mu}{\eps^2} \| f^\eps_{\kin, N} \|^2_{  \SShp_{j}} - \frac{\sigma}{\eps^2} \| f^\eps_{\kin, N} \|^2_{  \SSh_j, \eps}\,,
			\end{gather*}
			as well as
			\begin{equation*}\begin{split}
				\dlla \PP^\eps_\kin &\QQ(  f^\eps_{\kin, N-1} , f^\eps_{\kin, N}), f^\eps_{\kin, N} \drra_{  \SSh_j, \eps} \\
				& \lesssim \| f^\eps_{\kin, N} \|_{  \SShp_{j}} \left(  \| f^\eps_{\kin, N} \|_{  \SShp_{j}} \| f^\eps_{\kin, N-1} \|_{  \SSh_{j-1}} + \| f^\eps_{\kin, N} \|_{  \SSh_j} \| f^\eps_{\kin, N-1} \|_{  \SShp_{j}} \right).
			\end{split}\end{equation*}
			Still using Proposition \ref{prop:dissipative_kinetic_inner_product}, the coupling term is estimated thanks to the non-closed estimate combined with the injections $\SSsp \hookrightarrow \SShp_{j+1}$ and $\SSs \hookrightarrow  \SSh_{j+1}$, and the closed estimate respectively:
			\begin{equation*}\begin{split}
					\dlla \PP^\eps_\kin \QQ( & f^\eps_{\kin, N} , f^\eps_{\hyd, N-1} + f^\eps_{\mix, N-1}), f^\eps_{\kin, N} \drra_{ \SSh_j, \eps} \\
					&\lesssim  \| f^\eps_{\kin, N} \|_{ \SShp_{j}}^2 \left( \| f^\eps_{\hyd, N-1} \|_{\SSs} + \| f^\eps_{\mix, N-1} \|_{\SSs} \right) \\
					& + \| f^\eps_{\kin, N} \|_{ \SShp_{j}} \| f^\eps_{\kin, N} \|_{ \SSh_j} \left( \| f^\eps_{\hyd, N-1} \|_{\SSsp} + \| f^\eps_{\mix, N-1} \|_{\SSsp} \right),
			\end{split}\end{equation*}
			and
			\begin{equation*}\begin{split}
					\dlla \PP^\eps_\kin \QQ( & f^\eps_{\hyd, N-1} + f^\eps_{\mix, N-1} , f^\eps_{\kin, N}), f^\eps_{\kin, N} \drra_{ \SSh_j, \eps} \\
					\lesssim & \| f^\eps_{\kin, N} \|_{ \SShp_{j}}^2 \left( \| f^\eps_{\hyd, N-1} \|_{\SSs} + \| f^\eps_{\mix, N-1} \|_{\SSs} \right) \\
					& + \| f^\eps_{\kin, N} \|_{ \SSh_j} \| f^\eps_{\kin, N} \|_{ \SShp_{j}} \left( \| f^\eps_{\hyd, N-1} \|_{\SSsp} + \| f^\eps_{\mix, N-1} \|_{\SSsp} \right).
			\end{split}\end{equation*}
			Note that the second part of both previous estimates coincide. Put together, and multiplying by $e^{2 \sigma t / \eps^2}$, we have the energy estimate
			\begin{equation}\begin{split}\label{eq:I_1--I_4}
					\frac{1}{2} \frac{\d}{\d t} \left( e^{2 \sigma t / \eps^2} \| f^\eps_{\kin, N} \|_{ \SSh_j, \eps}^2 \right)   + \frac{\mu}{\eps^2} e^{2 \sigma t / \eps^2} \| f^\eps_{\kin, N} \|^2_{ \SShp_{j}}\lesssim \mathscr{I}_{1}(t)+\mathscr{I}_2(t)+\mathscr{I}_3(t)+\mathscr{I}_4(t)\end{split}\end{equation}
			where we introduced  
			\begin{equation*}\begin{split}
					\mathscr{I}_1&=\frac{1}{\eps} e^{2 \sigma t / \eps^2} \| f^\eps_{\kin, N} \|_{ \SShp_{j}}\bigg(  \| f^\eps_{\kin, N} \|_{ \SShp_{j}} \| f^\eps_{\kin, N-1} \|_{\SSh_j} + \| f^\eps_{\kin, N} \|_{\SSh_j} \| f^\eps_{\kin, N-1} \|_{\SShp_{j}} \bigg) \\
					\mathscr{I}_2&= \frac{1}{\eps} e^{2 \sigma t / \eps^2} \| f^\eps_{\kin, N} \|_{\SShp_{j}} \| f^\eps_{\kin, N} \|_{\SSh_j} \left( \| f^\eps_{\hyd, N-1} \|_{\SSsp} + \| f^\eps_{\mix, N-1} \|_{\SSsp} \right) \\
					\mathscr{I}_3&=\frac{1}{\eps} e^{2 \sigma t / \eps^2} \| f^\eps_{\kin, N} \|_{\SShp_{j}}^2 \left( \| f^\eps_{\hyd, N-1} \|_{\SSs} + \| f^\eps_{\mix, N-1} \|_{\SSs} \right) \\
					\mathscr{I}_4&=\frac{1}{\eps} e^{2 \sigma t / \eps^2} \| f^\eps_{\kin, N} \|_{\SSh_j} \| f^\eps_{\kin, N} \|_{\SShp_{j}} \left( \| f^\eps_{\hyd, N-1} \|_{\SSsp} + \| f^\eps_{\mix, N-1} \|_{\SSsp} \right).
			\end{split}\end{equation*}
			One easily sees that
			$$\int_0^T\mathscr{I}_1(t) \d t \lesssim \eps \Nt f_{\kin,N}^\eps\Nt_{\SSSh_j}^2\,\Nt f_{\kin,N-1}^\eps\Nt_{\SSSh_j}$$
			where we recall the definition \eqref{eq:defiTildeSSSl} and the fact that $\|\cdot\|_{\SSh_j} \lesssim \|\cdot\|_{\SShp_{j}}$.
			We write the time integral of the second term as follows
			\begin{multline*}
				\int_0^T \mathscr{I}_2(t)\d t=\eps \int_0^T \left[ \frac{1}{\eps} e^{\sigma t / \eps^2} \| f^\eps_{\kin, N} \|_{\SShp_{j}} \right] \left[ \frac{1}{\eps} e^{\sigma t / \eps^2 } \| f^\eps_{\kin, N} \|_{\SSh_j} \right]  \| f^\eps_{\hyd, N-1} \|_{\SSsp} \d t\\
				+\eps \int_0^T \left[ \frac{1}{\eps} e^{\sigma t / \eps^2} \| f^\eps_{\kin, N} \|_{\SShp_{j}} \right]\left[ \frac{1}{\eps} \| f^\eps_{\mix, N-1} \|_{\SSsp} \right]  \left[ e^{\sigma t / \eps^2 } \| f^\eps_{\kin, N} \|_{\SSh_j} \right]   \d t 
			\end{multline*}
			where, in each integral, the first two terms in brackets belong to $L^2(0,T)$ whereas the third one belongs to $L^\infty(0,T)$. Then, it is easy to deduce that
			$$\int_0^T \mathscr{I}_2(t) \d t\lesssim \eps w_{ f_\ns, \eta }(T)^{-1} \Nt f^\eps_{\kin, N} \Nt^2_{\SSSh_j} \left( \Nt f^\eps_{\hyd, N-1} \Nt_\SSSs + \Nt f^\eps_{\mix, N-1} \Nt_\SSSm \right)\,
			$$
			where we used again that $\|\cdot\|_{\SSh_j} \lesssim \|\cdot\|_{\SShp_{j}}$. One also sees that
			\begin{multline*}
				\int_0^T\mathscr{I}_3(t)\d t \lesssim    
				\eps \int_0^T  \left[ \frac{1}{\eps} e^{\sigma t / \eps^2} \| f^\eps_{\kin, N} \|_{\SShp_{j}} \right]^2 \left( \| f^\eps_{\hyd, N-1} \|_{\SSs} + \| f^\eps_{\mix, N-1} \|_{\SSs} \right) \d t \\
				\lesssim \eps w_{ f_\ns, \eta }(T)^{-1} \Nt f^\eps_{\kin, N} \Nt^2_{\SSSh_j} \left( \Nt f^\eps_{\hyd, N-1} \Nt_{\SSSs} + \Nt f^\eps_{\mix, N-1} \Nt_{\SSSm} \right)\,.\end{multline*} 
			Finally, the term involving $\mathscr{I}_4$ is dealt with as the one involving $\mathscr{I}_2$ writing
			\begin{multline*}
				\int_0^T \mathscr{I}_4(t)\d t = 	\eps \int_0^T \left[ \frac{1}{\eps} e^{\sigma t / \eps^2} \| f^\eps_{\kin, N} \|_{\SSh_j} \right] \left[ \frac{1}{\eps} e^{\sigma t / \eps^2} \| f^\eps_{\kin, N} \|_{\SShp_{j}} \right] \| f^\eps_{\hyd, N-1} \|_{\SSsp} \d t \\
				+ \eps \int_0^T  \left[ \frac{1}{\eps} e^{\sigma t / \eps^2} \| f^\eps_{\kin, N} \|_{\SShp_{j}} \right] \left[ \frac{1}{\eps} \| f^\eps_{\mix, N-1} \|_{\SSsp} \right] \left[ e^{\sigma t / \eps^2} \| f^\eps_{\kin, N} \|_{\SSh_j} \right]\d t 
			\end{multline*}
			where for both integrals, the first two terms in brackets belong to $L^2(0,T)$ whereas the third one belongs to $L^\infty(0,T)$. This gives easily
			$$\int_0^T \mathscr{I}_4(t)\d t \lesssim \eps w_{ f_\ns, \eta }(T)^{-1} \Nt f^\eps_{\kin, N} \Nt^2_{\SSSh_j} \left( \Nt f^\eps_{\hyd, N-1} \Nt_{\SSSs} + \Nt f^\eps_{\mix, N-1} \Nt_{\SSSm} \right).$$
			Coming back to \eqref{eq:I_1--I_4}, we finally deduce the recursive estimate
			\begin{equation*}\begin{split}
				\Nt f^\eps_{\kin, N} \Nt^2_{\SSSh_j} \lesssim  \eps & \Nt f^\eps_{\kin, N} \Nt^2_{ \SSSh_j} \Nt f^\eps_{\kin, N-1} \Nt_{ \SSSh_j}  + \| \PP^\eps_\kin f_\ini \|^2_{\mathbb{H}^s_x(\Sh_j)}\\
				&+ \eps w_{ f_\ns, \eta }(T)^{-1} \Nt f^\eps_{\kin, N} \Nt^2_{\SSSh_j} \left( \Nt f^\eps_{\hyd, N-1} \Nt_{\SSSs} + \Nt f^\eps_{\mix, N-1} \Nt_{\SSSm} \right),
			\end{split}\end{equation*} 
			and using the inductive hypothesis \eqref{eq:inductive_stability}
			\begin{equation*}
				\Nt f^\eps_{\kin, N} \Nt^2_{\SSSh_j} \lesssim   \eps \Big[ C_1 + w_{ f_\ns, \eta }(T)^{-1} \left( 1 + c_2 + c_3 \right)  \Big] \Nt f^\eps_{\kin, N} \Nt^2_{ \SSSh_j} + \| \PP^\eps_\kin f_\ini \|^2_{\tilde{\SSl}_j},
			\end{equation*}
			so that, for $\eps$ small enough, there holds
			$$
			\Nt f^\eps_{\kin, N} \Nt_{\SSSh_j} \lesssim  \| \PP^\eps_\kin f_\ini \|_{\SSSh_j},
			$$
			which allows to deduce the stability estimate at rank $N$ for $f^\eps_{\kin, N}$. 
			
			\medskip
			The proof of the stability estimates for $f^\eps_{\mix, N}$ and $g^\eps_N$ is the same as in Section \ref{scn:proof_hydrodynamic_limit_symmetrizable}. This concludes the proof.\end{proof}
		
		\subsection{Convergence of the scheme}
		
		The convergence will be proved in the larger space $\SSSh_{-1}$ using, of course, the stability estimate in $\SSSh_{-1} $, but also those in $\SSSh_0$ because of the non-closed estimates. To do so, we denote the difference and sum of successive approximate solutions as
		\begin{gather*}
			d^\eps_{\kin, N} := f^\eps_{\kin, N} - f^\eps_{\kin, N-1}, \qquad d^\eps_{\mix, N} := f^\eps_{\mix, N} - f^\eps_{\mix, N-1}\\
			d^\eps_{\hyd, N} := f^\eps_{\hyd, N} - f^\eps_{\hyd, N-1} = g^\eps_N - g^\eps_{N-1}\,. 
		\end{gather*}
		One has the following recursive estimate.
		\begin{prop} With the notations of Lemma \ref{lem:stable-scheme}, up to reducing again 
			$$\eps \ll c_3 \ll \eta \ll c_2 \ll 1,$$
			the following estimate  
			\begin{equation*}\begin{split}
				\Nt d^\eps_{\kin, N} \Nt_{ \SSSh_{-1}} & + \Nt d^\eps_{\hyd, N} \Nt_{ \SSSs } + \Nt d^\eps_{\mix, N} \Nt_{ \SSSm } \\
				&\le  \frac{1}{2} \Bigg( \Nt d^\eps_{\kin, N-1} \Nt_{\SSSh_{-1}} + \Nt d^\eps_{\hyd, N-1} \Nt_{\SSSs} + \Nt d^\eps_{\mix, N-1} \Nt_{\SSSm}\Bigg),
			\end{split}\end{equation*}
			holds for any $N\ge0.$
		\end{prop}
		\begin{proof} As for the proof of  Lemma \ref{lem:stable-scheme}, the difficulty lies in estimating $d^\eps_{\kin,N}$ which solves
			\begin{equation*}
				\begin{cases}
					\displaystyle
					\begin{aligned}
						\partial_t d^\eps_{\kin, N} = & \frac{1}{\eps^2} \big(\LL - \eps v \cdot \nabla_x\big) d^\eps_{\kin, N} \\
						& + \frac{1}{\eps} \PP^\eps_\kin \QQ(f^\eps_{\kin, N-1}, d^\eps_{\kin, N}) + \frac{1}{\eps} \PP^\eps_\kin \QQ(d^\eps_{\kin, N-1}, f^\eps_{\kin, N-1}) \\
						& + \frac{2}{\eps} \PP^\eps_\kin \QQ^\sym(f^\eps_{\hyd, N-1} + f^\eps_{\mix, N-1} , d^\eps_{\kin, N}) \\
						&+ \frac{2}{\eps} \PP^\eps_\kin \QQ^\sym(f^\eps_{\kin, N-1}, d^\eps_{\hyd, N-1} + d^\eps_{\mix, N-1} ), 
					\end{aligned}\\
					d^\eps_{\kin, N}(0) = 0.
				\end{cases}
			\end{equation*}
			We use as previously the equivalent norms $\|\cdot\|_{\SSh_{-1}, \eps}$ which allows the use of  dissipativity:
			\begin{gather*}
				\frac{1}{\eps^2} \dlla (\LL - \eps v \cdot \nabla_x) d^\eps_{\kin, N}, d^\eps_{\kin, N}  \drra_{\SSh_{-1}, \eps} \le - \frac{\sigma}{\eps^2} \| d^\eps_{\kin, N} \|^2_{\SSh_{-1}, \eps} - \frac{\mu}{\eps^2} \| d^\eps_{\kin, N} \|^2_{\SShp_{-1}},
			\end{gather*}
			as well as the closed estimate:
			\begin{equation*}\begin{split}
					\dlla \PP^\eps_\kin \QQ( f^\eps_{\kin, N-1} , d^\eps_{\kin, N}), d^\eps_{\kin, N} \drra_{\SSh_{-1}, \eps}
					&\lesssim  \| d^\eps_{\kin, N} \|_{\SShp_{-1}}^2 \| f^\eps_{\kin, N-1} \|_{\SSh_{-1}} \\
					&\phantom{++++}  + \| d^\eps_{\kin, N} \|_{\SShp_{-1}} \| d^\eps_{\kin, N} \|_{\SSh_{-1}} \| f^\eps_{\kin, N-1} \|_{\SShp_{-1}},
			\end{split}\end{equation*}
			and the non-closed estimate involving the $\SSh_{0} $ and $\SShp_{0}$--norms of $f^\eps_{\kin, N-1}$:
			\begin{equation*}\begin{split}
					\dlla \PP^\eps_\kin \QQ( d^\eps_{\kin, N-1} &, f^\eps_{\kin, N-1}), d^\eps_{\kin, N} \drra_{\SSh_{-1}, \eps}\\
					&\lesssim  \| d^\eps_{\kin, N} \|_{\SShp_{-1}}  \| d^\eps_{\kin, N-1} \|_{\SSh_{-1}} \| f^\eps_{\kin, N-1} \|_{\SShp_{0}} \\
					& \phantom{++++} + \| d^\eps_{\kin, N} \|_{\SShp_{-1}}  \| d^\eps_{\kin, N-1} \|_{\SShp_{-1}} \| f^\eps_{\kin, N-1} \|_{\SSh_{0}}.
			\end{split}\end{equation*}
			The coupling term involving $d^\eps_{\kin, N}$ is estimated using both the closed and non-closed estimates, together with the injections $\SSs \hookrightarrow \SSh_{0}$ and $\SSsp \hookrightarrow \SShp_{0}$:
			\begin{equation*}\begin{split}
					\dlla \PP^\eps_\kin \QQ^\sym&\left( f^\eps_{\hyd, N-1} + f^\eps_{\mix, N-1} , d^\eps_{\kin, N} \right), d^\eps_{\kin, N} \drra_{\SSh_{-1}, \eps} \\
					&\lesssim  \| d^\eps_{\kin, N} \|_{ \SShp_{-1}}^2 \left( \| f^\eps_{\hyd, N-1} \|_{ \SSs } + \| f^\eps_{\mix, N-1} \|_{ \SSs } \right) \\
					& + \| d^\eps_{\kin, N} \|_{ \SSh_{-1} } \| d^\eps_{\kin, N} \|_{\SSh_{-1} } \left( \| f^\eps_{\hyd, N-1} \|_{ \SSsp } + \| f^\eps_{\mix, N-1} \|_{ \SSsp } \right),
			\end{split}\end{equation*}
			and the coupling term involving $f^\eps_{\kin, N-1}$ using the non-closed estimate, together with the injections $\SSs \hookrightarrow \SSh_{0} \hookrightarrow \SSh_{-1}$ and $\SSsp \hookrightarrow \SShp_{0} \hookrightarrow \SShp_{-1}$:
			\begin{equation*}\begin{split}
					\dlla \PP^\eps_\kin \QQ^\sym&\left( d^\eps_{\hyd, N-1} + d^\eps_{\mix, N-1} , f^\eps_{\kin, N-1} \right), d^\eps_{\kin, N} \drra_{\SSh_{-1}, \eps} \\
					&\lesssim  \| d^\eps_{\kin, N} \|_{\SShp_{-1} } \| f^\eps_{\kin, N-1} \|_{ \SShp_{0} } \left( \| d^\eps_{\hyd, N-1} \|_{ \SSs } + \| d^\eps_{\mix, N-1} \|_{ \SSs } \right) \\
					&\phantom{+++} + \| d^\eps_{\kin, N} \|_{\SShp_{-1} } \| f^\eps_{\kin, N-1} \|_{\SSh_{0} } \left( \| d^\eps_{\hyd, N-1} \|_{ \SSsp } + \| d^\eps_{\mix, N-1} \|_{ \SSsp } \right)\,.
			\end{split}\end{equation*}
			Multiplying these estimates by $e^{2 \sigma t / \eps^2}$, we get the energy estimate
			\begin{equation*}\begin{split}
				\frac{1}{2} \frac{\d }{\d t} \Big( e^{2 \sigma t / \eps^2}  & \| d^\eps_{\kin, N} \|_{\SSh_{-1}, \eps}^2 \Big) + \frac{\mu}{\eps^2} e^{2 \sigma t / \eps^2} \| d^\eps_{\kin, N} \|^2_{\SShp_{-1}} \\
				\lesssim & \frac{1}{\eps} e^{2 \sigma t / \eps^2} \| d^\eps_{\kin, N} \|_{\SShp_{-1}}^2 \| f^\eps_{\kin, N-1} \|_{\SSh_{-1}}  \\
				&+ \frac{1}{\eps} e^{2 \sigma t / \eps^2} \| d^\eps_{\kin, N} \|_{\SShp_{-1}} \Bigg(\| d^\eps_{\kin, N} \|_{\SSh_{0}} \| f^\eps_{\kin, N-1} \|_{\SSh_{-1}} \\
				&\phantom{++++} + \| d^\eps_{\kin, N-1} \|_{\SSh_{-1}} \| f^\eps_{\kin, N-1} \|_{\SShp_{0}} 
			+ \| d^\eps_{\kin, N-1} \|_{\SShp_{-1}} \| f^\eps_{\kin, N-1} \|_{\SShp_{1}} \Bigg)\\
				& + \frac{1}{\eps} e^{2 \sigma t / \eps^2} \| d^\eps_{\kin, N} \|_{ \SShp_{-1} }^2 \left( \| f^\eps_{\hyd, N-1} \|_{ \SSs } + \| f^\eps_{\mix, N-1} \|_{ \SSs } \right) \\
				& + \frac{1}{\eps} e^{2 \sigma t / \eps^2} \| d^\eps_{\kin, N} \|_{ \SSh_{-1} } \| d^\eps_{\kin, N} \|_{ \SSh_{-1} } \left( \| f^\eps_{\hyd, N-1} \|_{ \SSsp } + \| f^\eps_{\mix, N-1} \|_{ \SSsp } \right) \\
				& + \frac{1}{\eps} e^{2 \sigma t / \eps^2} \| d^\eps_{\kin, N} \|_{\SShp_{-1} } \| f^\eps_{\kin, N-1} \|_{ \SShp_{0} } \left( \| d^\eps_{\hyd, N-1} \|_{ \SSs } + \| d^\eps_{\mix, N-1} \|_{ \SSs } \right) \\
				& + \frac{1}{\eps} e^{2 \sigma t / \eps^2} \| d^\eps_{\kin, N} \|_{ \SShp_{-1} } \| f^\eps_{\kin, N-1} \|_{ \SSh_{0} } \left( \| d^\eps_{\hyd, N-1} \|_{ \SSsp } + \| d^\eps_{\mix, N-1} \|_{ \SSsp } \right).
			\end{split}\end{equation*}
			Arguing as in the proof of Lemma \ref{lem:stable-scheme}, we then obtain
			\begin{align*}
				\Nt d^\eps_{\kin, N} \Nt^2_{ \SSSh_{-1} } &\lesssim  \eps \Nt d^\eps_{\kin, N} \Nt_{\SSSh_{-1}}^2 \Nt f^\eps_{\kin, N-1} \Nt_{\SSSh_{-1} }   + \eps \Nt d^\eps_{\kin, N} \Nt_{\SSSh_{-1}} \Nt d^\eps_{\kin, N-1} \Nt_{\SSSh_{-1}} \Nt f^\eps_{\kin, N-1} \Nt_{ \SSSh_{0} } \\
				& + \eps w_{ f_\ns, \eta }(T)^{-1} \Nt d^\eps_{\kin, N} \Nt_{\SSSh_{-1}}^2 \left( \Nt f^\eps_{\hyd, N-1} \Nt_\SSSs + \Nt f^\eps_{\mix, N-1} \Nt_\SSSm \right) \\
				& + \eps w_{ f_\ns, \eta }(T)^{-1} \Nt d^\eps_{\kin, N} \Nt_{\SSSh_{-1}} \Nt f^\eps_{\kin, N-1} \Nt_{\SSSh_{0}} \left( \Nt d^\eps_{\hyd, N-1} \Nt_\SSSs + \Nt d^\eps_{\mix, N-1} \Nt_\SSSm \right),
			\end{align*}
			and thus using the stability estimates \eqref{eq:inductive_stability} 
			$$\Nt d^\eps_{\kin, N} \Nt_{ \SSSh_{-1} } \lesssim   \eps w_{ f_\ns, \eta }(T)^{-1} \left( \Nt d^\eps_{\kin, N-1} \Nt_{ \SSSh_{-1}} + \Nt d^\eps_{\hyd, N-1} \Nt_{\SSSs} + \Nt d^\eps_{\mix, N-1} \Nt_{\SSSm} \right)$$
			i.e., and assuming $\eps$ small enough,
			$$\Nt d^\eps_{\kin, N} \Nt_{ \SSSh_{-1} } \le  \frac{1}{4} \left( \Nt d^\eps_{\kin, N-1} \Nt_{ \SSSh_{-1}} + \Nt d^\eps_{\hyd, N-1} \Nt_{\SSSs} + \Nt d^\eps_{\mix, N-1} \Nt_{\SSSm} \right).$$
			Arguing in the very same way as in Section \ref{scn:mapping_contraction}, one also shows
			\begin{align*}
				\Nt d^\eps_{\hyd, N} \Nt_{ \SSSs } + \Nt d^\eps_{\mix, N} \Nt_{ \SSSm } \le & \frac{1}{4} \bigg( \Nt d^\eps_{\kin, N-1} \Nt_{\SSSh_{-1}} + \Nt d^\eps_{\hyd, N-1} \Nt_{\SSSs} + \Nt d^\eps_{\mix, N-1} \Nt_{\SSSm} \bigg)\,.
			\end{align*}
			Summing up the last two estimates yields the result.  
		\end{proof}
		The above result allows to prove in a standard way the convergence of the approximate solutions $\left\{( f^\eps_{\kin, N} , f^\eps_{\mix, N} , g^\eps_N )\right\}_N$ to some limit $( f^\eps_{\kin} , f^\eps_{\mix} , g^\eps )$ solving the system
		\begin{equation*}
			\begin{cases}
				\displaystyle 
				\begin{aligned}
					\partial_t f^\eps_\kin = \frac{1}{\eps^2} \left(\LL - \eps v \cdot \nabla_x\right) f^\eps_\kin & + \frac{1}{\eps} \PP^\eps_\kin \QQ(f^\eps_\kin, f^\eps_\kin)  \\
					& + \frac{2}{\eps} \PP^\eps_\kin \QQ^\sym(f^\eps_\kin, f^\eps_\hyd + f^\eps_\mix), \quad f^\eps_\kin(0) = \PP^\eps_\kin f_\ini,
				\end{aligned} \\
				f^\eps_\mix = \Psi^\eps_\kin\left[f^\eps_\mix + f^\eps_\hyd , f^\eps_\mix + f^\eps_\hyd \right],\\
				g^\eps = \Phi^\eps[ f^\eps_\kin, f^\eps_\mix ] g^\eps + \Psi^\eps_\hyd(g^\eps, g^\eps) + \SS^\eps.
			\end{cases}
		\end{equation*}
		
		A similar argument as the one from Section \ref{scn:uniqueness_simmetrizable} performed in $\SSSh_{-1}$ yields the uniqueness, which implies the uniqueness in $\SSSh_{0} = \SSSl$ and achieves the proof of Theorem \ref{thm:hydrodynamic_limit} under Assumptions \ref{BED}.

		\appendix 
		\section{About Assumptions \ref{L1}--\ref{L4} and \ref{Bortho}-\ref{Bbound}}
		\label{sec:Landau-Boltz}
		
		The various assumptions \ref{L1}--\ref{L4} and \ref{Bortho}--\ref{Bisotrop}, as well as the ``enlargement ones'' \ref{LE}, \ref{BE} and \ref{BED} were identified as being the properties shared by the Boltzmann and Landau equations in a close to equilibrium setting. We specify in this section that the aforementioned assumptions are proven in the literature, with precise references.
		
		\subsection{The case of the classical Boltzmann equation}
		
		Let us recall that, in the case of the classical Boltzmann equations, the distribution $\mu$ can be taken assumed to be the centered maxwellian:
		$$\mu(v) := (2 \pi)^{-d/2} \exp\left( - \frac{|v|^2}{2} \right), \qquad E = d, \quad K = 1 + \frac{2}{d} \, .$$
		The linear operator $\LL$ is the linearized nonlinear operator $\QQ$ around $\mu$:
		$$\LL f := \QQ(\mu, f) + \QQ(\mu, f)$$
		where $\QQ$ is defined as
		$$\QQ(f, f) := \int_{\R^3_{v_*} \times \S^{d-1}_\sigma } | v - v_* |^\gamma b(\cos \theta)  \Big( f(v') f(v'_*) - f(v) f(v_*) \Big) \d \sigma \d v_* \, ,$$
		for some parameter $\gamma \in (-3, 1)$ and some positive function $b$ smooth on $(0, 1]$, and the pre-collisional velocities $v', v_*'$ as well as the deviation angle are defined as
		$$v' = \frac{v+v_*}{2} + \sigma \frac{|v - v_* |}{2}, \qquad v' = \frac{v+v_*}{2} - \sigma \frac{|v - v_* |}{2}, \qquad \cos \theta := \sigma \cdot \frac{v-v_*}{ | v - v_* | }$$
		and illustrated in Figure \ref{fig:collisional geometry}.
		
		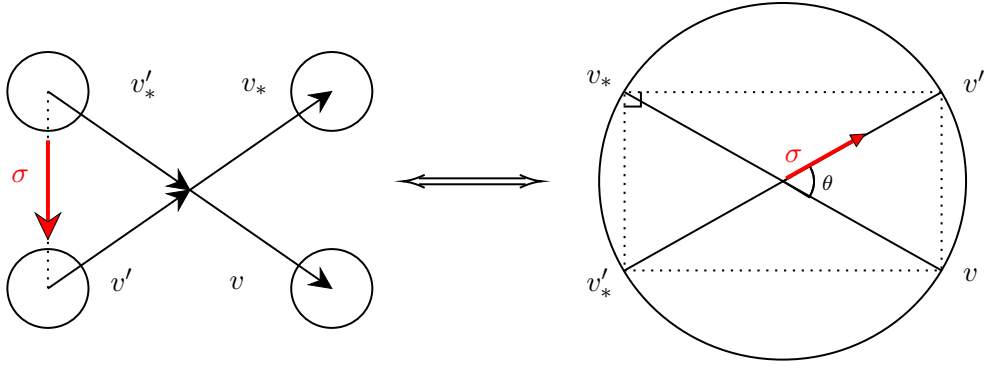
\begin{figure}[h]
			\label{fig:collisional geometry}

			\tikzset{every picture/.style={line width=0.75pt}} %set default line width to 0.75pt        
			
			\begin{tikzpicture}[x=0.75pt,y=0.75pt,yscale=-1,xscale=1]
				%uncomment if require: \path (0,256); %set diagram left start at 0, and has height of 256
				
				%Shape: Arc [id:dp08086903276301782] 
				\draw  [draw opacity=0] (424.93,82.17) .. controls (426.46,84.49) and (427.33,87.25) .. (427.29,90.19) .. controls (427.26,93.08) and (426.35,95.76) .. (424.81,98.02) -- (411.68,90) -- cycle ; \draw   (424.93,82.17) .. controls (426.46,84.49) and (427.33,87.25) .. (427.29,90.19) .. controls (427.26,93.08) and (426.35,95.76) .. (424.81,98.02) ;  
				%Shape: Ellipse [id:dp33313376884218915] 
				\draw   (318.91,90) .. controls (318.91,40.29) and (360.33,0) .. (411.44,0) .. controls (462.54,0) and (503.96,40.29) .. (503.96,90) .. controls (503.96,139.71) and (462.54,180) .. (411.44,180) .. controls (360.33,180) and (318.91,139.71) .. (318.91,90) -- cycle ;
				%Shape: Rectangle [id:dp39177993456247306] 
				\draw  [dash pattern={on 0.84pt off 2.51pt}] (331.73,45) -- (491.63,45) -- (491.63,135) -- (331.73,135) -- cycle ;
				%Straight Lines [id:da16398714686909577] 
				\draw    (331.73,45) -- (491.63,135) ;
				%Straight Lines [id:da976177831632348] 
				\draw    (491.63,45) -- (331.73,135) ;
				%Straight Lines [id:da9993825881026177] 
				\draw [color={rgb, 255:red, 255; green, 0; blue, 0 }  ,draw opacity=1 ][line width=1.5]    (413.26,88.5) -- (449.59,68.08) ;
				\draw [shift={(453.07,66.13)}, rotate = 150.66] [fill={rgb, 255:red, 255; green, 0; blue, 0 }  ,fill opacity=1 ][line width=0.08]  [draw opacity=0] (6.97,-3.35) -- (0,0) -- (6.97,3.35) -- cycle    ;
				%Shape: Right Angle [id:dp7010173847068983] 
				\draw   (331.25,52.27) -- (339.98,52.27) -- (339.98,44.72) ;
				%Shape: Ellipse [id:dp8957805779859759] 
				\draw   (20.47,44.64) .. controls (20.47,33.66) and (29.63,24.75) .. (40.92,24.75) .. controls (52.22,24.75) and (61.37,33.66) .. (61.37,44.64) .. controls (61.37,55.63) and (52.22,64.53) .. (40.92,64.53) .. controls (29.63,64.53) and (20.47,55.63) .. (20.47,44.64) -- cycle ;
				%Shape: Ellipse [id:dp14649222028588338] 
				\draw   (20.47,144.1) .. controls (20.47,133.11) and (29.63,124.21) .. (40.92,124.21) .. controls (52.22,124.21) and (61.37,133.11) .. (61.37,144.1) .. controls (61.37,155.08) and (52.22,163.99) .. (40.92,163.99) .. controls (29.63,163.99) and (20.47,155.08) .. (20.47,144.1) -- cycle ;
				%Shape: Ellipse [id:dp37696478075353723] 
				\draw   (163.62,44.64) .. controls (163.62,33.66) and (172.78,24.75) .. (184.07,24.75) .. controls (195.37,24.75) and (204.52,33.66) .. (204.52,44.64) .. controls (204.52,55.63) and (195.37,64.53) .. (184.07,64.53) .. controls (172.78,64.53) and (163.62,55.63) .. (163.62,44.64) -- cycle ;
				%Shape: Ellipse [id:dp17801263892278962] 
				\draw   (163.62,144.1) .. controls (163.62,133.11) and (172.78,124.21) .. (184.07,124.21) .. controls (195.37,124.21) and (204.52,133.11) .. (204.52,144.1) .. controls (204.52,155.08) and (195.37,163.99) .. (184.07,163.99) .. controls (172.78,163.99) and (163.62,155.08) .. (163.62,144.1) -- cycle ;
				%Straight Lines [id:da5673059339208225] 
				\draw    (40.92,44.64) -- (110.03,92.66) ;
				\draw [shift={(112.5,94.37)}, rotate = 214.79] [fill={rgb, 255:red, 0; green, 0; blue, 0 }  ][line width=0.08]  [draw opacity=0] (10.72,-5.15) -- (0,0) -- (10.72,5.15) -- (7.12,0) -- cycle    ;
				%Straight Lines [id:da16503831114169865] 
				\draw    (40.92,144.1) -- (110.03,96.08) ;
				\draw [shift={(112.5,94.37)}, rotate = 145.21] [fill={rgb, 255:red, 0; green, 0; blue, 0 }  ][line width=0.08]  [draw opacity=0] (10.72,-5.15) -- (0,0) -- (10.72,5.15) -- (7.12,0) -- cycle    ;
				%Straight Lines [id:da19452873521396308] 
				\draw    (112.5,94.37) -- (181.61,142.39) ;
				\draw [shift={(184.07,144.1)}, rotate = 214.79] [fill={rgb, 255:red, 0; green, 0; blue, 0 }  ][line width=0.08]  [draw opacity=0] (10.72,-5.15) -- (0,0) -- (10.72,5.15) -- (7.12,0) -- cycle    ;
				%Straight Lines [id:da945316616273675] 
				\draw    (112.5,94.37) -- (181.61,46.35) ;
				\draw [shift={(184.07,44.64)}, rotate = 145.21] [fill={rgb, 255:red, 0; green, 0; blue, 0 }  ][line width=0.08]  [draw opacity=0] (10.72,-5.15) -- (0,0) -- (10.72,5.15) -- (7.12,0) -- cycle    ;
				%Straight Lines [id:da16010344955841704] 
				\draw [line width=0.75]    (228,88.5) -- (281.84,88.5)(228,91.5) -- (281.84,91.5) ;
				\draw [shift={(289.84,90)}, rotate = 180] [color={rgb, 255:red, 0; green, 0; blue, 0 }  ][line width=0.75]    (10.93,-3.29) .. controls (6.95,-1.4) and (3.31,-0.3) .. (0,0) .. controls (3.31,0.3) and (6.95,1.4) .. (10.93,3.29)   ;
				\draw [shift={(220,90)}, rotate = 0] [color={rgb, 255:red, 0; green, 0; blue, 0 }  ][line width=0.75]    (10.93,-3.29) .. controls (6.95,-1.4) and (3.31,-0.3) .. (0,0) .. controls (3.31,0.3) and (6.95,1.4) .. (10.93,3.29)   ;
				%Straight Lines [id:da7601680116162869] 
				\draw  [dash pattern={on 0.84pt off 2.51pt}]  (40.92,44.64) -- (40.92,144.1) ;
				%Straight Lines [id:da44585102363180307] 
				\draw [color={rgb, 255:red, 255; green, 0; blue, 0 }  ,draw opacity=1 ][line width=1.5]    (40.92,69.37) -- (40.92,115.37) ;
				\draw [shift={(40.92,119.37)}, rotate = 270] [fill={rgb, 255:red, 255; green, 0; blue, 0 }  ,fill opacity=1 ][line width=0.08]  [draw opacity=0] (13.4,-6.43) -- (0,0) -- (13.4,6.44) -- (8.9,0) -- cycle    ;
				
				% Text Node
				\draw (311,32.4) node [anchor=north west][inner sep=0.75pt]    {$v_{*}$};
				% Text Node
				\draw (501,132.4) node [anchor=north west][inner sep=0.75pt]    {$v$};
				% Text Node
				\draw (501,32.4) node [anchor=north west][inner sep=0.75pt]    {$v'$};
				% Text Node
				\draw (311,132.4) node [anchor=north west][inner sep=0.75pt]    {$v'_{*}$};
				% Text Node
				\draw (411,72.4) node [anchor=north west][inner sep=0.75pt]  [color={rgb, 255:red, 255; green, 0; blue, 0 }  ,opacity=1 ]  {$\sigma $};
				% Text Node
				\draw (429.97,85.46) node [anchor=north west][inner sep=0.75pt]  [font=\scriptsize]  {$\theta $};
				% Text Node
				\draw (137,36.4) node [anchor=north west][inner sep=0.75pt]    {$v_{*}$};
				% Text Node
				\draw (131,136.4) node [anchor=north west][inner sep=0.75pt]    {$v$};
				% Text Node
				\draw (81,32.4) node [anchor=north west][inner sep=0.75pt]    {$v'_{*}$};
				% Text Node
				\draw (71,132.4) node [anchor=north west][inner sep=0.75pt]    {$v'$};
				% Text Node
				\draw (21,82.4) node [anchor=north west][inner sep=0.75pt]  [color={rgb, 255:red, 255; green, 0; blue, 0 }  ,opacity=1 ]  {$\sigma $};

			\end{tikzpicture}

			\caption{On the left, the representation of a binary collision. On the right, an illustration of the so-called \emph{collisional geometry}, representing the distribution of the pre-collisional velocities $(v', v'_*)$ and the post-collisional ones $(v, v_*)$  on the circle centered about (the conserved mean velocity) $\frac{v + v_*}{2}$ of radius (the conserved relative velocity) $|v-v_*|$ in the plane $\Span(v-v_*, \sigma)$. This representation allows to visualize the deviation angle $\theta$.}
			
		\end{figure}
		
		\subsubsection{The cutoff case}
		One talks about an \emph{angular cutoff assumption} when the function $b$ is assumed to be well-behaved enough close to $\theta = 0$ to satisfy some integrability property. Concerning our basics assumptions \ref{L1}--\ref{L4} and \ref{Bortho}--\ref{Bisotrop}, one considers hard or Maxwell potentials under a mild cutoff assumption:
		$$\gamma \ge 0, \qquad \sin(\theta) b(\cos \theta) \in L^1_\theta\left( [-1, 1] \right) .$$
		The hierarchy of spaces $\Ss_j$ and the dissipation space $\Ssp$ are defined as
		$$\Ss_j := L^2\left( \mu^{-1} \la v \ra^{2j} \d v \right), \quad \Ssp := L^2\left( \mu^{-1} \la v \ra^{\gamma/2} \d v\right),$$
		and the dissipation estimate is given by the interpolation of the following two estimates coming respectively from Hilbert-Grad's splitting (\cite{H1912, G1963}) and the spectral gap estimate \cite[Theorem 1.1]{BM2005}:
		$$\la \LL g, g \ra_\Ss \le - c \| g \|^2_{\Ssp} + C \| g \|_{\Ss}^2, \qquad \la \LL g, g \ra_\Ss \le - c \| g \|^2_{\Ss}.$$
		The splitting we consider is
		$$\BB g := \LL g- M \mathbf{1}_{| v | \le M} g, \qquad \AA g := M \mathbf{1}_{| v | \le M} g,$$
		for a large enough $M > 0$, and it satisfies \ref{L4} thanks to the following weighted estimates coming from \cite[Proof of Lemma 3.3 for $\beta = 0$]{G2006}:
		$$\la \LL g, g \ra_{\Ss_j} \le - c \| \la v \ra^{\gamma/2} g \|^2_{\Ss_j} + C \| \la v \ra^{\gamma / 2} f \|_{\Ss_0}^2,$$
		for some $c, C > 0$. The nonlinear estimate \ref{Bbound} is given by \cite[Lemma 4.1]{G2004}.
		
		\medskip
		The ``enlarged'' assumptions \ref{LE} and \ref{BE} are satisfied for some restricted versions of the cutoff model. The linear assumptions \ref{LE} are proved under the strong cutoff assumption:
		$$b =b(\cos \theta) \in \CC^1\left( [-1, 1] \right)$$
		in the larger energy space $\Sl$ and dissipation space $\Slp$ defined as
		$$\Sl := L^2\left( \la v \ra^k \d v \right), \quad \Slp := L^2\left( \la v \ra^{k+\gamma/2} \d v\right), \quad k > 2.$$
		The splitting considered is of the form
		\begin{gather*}
			\LL = (\LL - \AA) + \AA =: \BB+ \AA,\\
			\AA g(v) := \int_{\R^d} a(v, v_*) g(v_*) \d v_*, \quad a \in \CC^\infty_c\left( \R^d_v \times \R^d_{v_*} \right),
		\end{gather*}
		and the dissipation estimate is given in \cite[Lemma 3.4]{BMM2019}, whereas the regularization estimate follows from the form of $\AA$. The nonlinear estimate \ref{BE} is given by \cite[Lemma 4.4]{BMM2019}.
		
		\subsubsection{The non cutoff case}
		One talks about a \emph{non cutoff case} when the function $b$ has a non-integrable singularity close to $\theta = 0$, and more precisely when it behaves as follows:
		$$\sin (\theta)^{d-2} b(\cos \theta) \approx \theta^{-2 s}, \quad s \in (0, 1).$$
		In this situation, the angular singularity improves the dissipation in the sense that $\LL$ presents a spectral gap even for some negative values of $\gamma$, at least in the suitable space $\Ss$. The basics assumptions \ref{L1}--\ref{L4} are satisfied for Maxwel, hard and moderately soft potentials:
		$$0 < s < 1, \qquad \gamma + 2 s \ge 0.$$
		The energy space is defined as in the cutoff case, whereas this time, the dissipation space $\Ssp$ is defined as (see \cite{AMUXY2011, AHL2019}, we also mention \cite{GS2011})
			\begin{align*}
				\| g \|^2_{\Ssp} & \approx \| \la v \ra^{\gamma/2 + s} g \|_{\Ss}^2 + \int_{\R^6_{v, v_*} \times \S^{d-1}_\sigma } \la v_* \ra^{\gamma} \mu_* \left( g(v') - g(v) \right)^2 \d \sigma \d v_* \d v \\
				& \approx \| \la v \ra^{\gamma/2 + s} g \|_{\Ss}^2 + \left\| \la v \ra^{\gamma/2} | v \land \nabla_v |^s g \right\|^2_{\Ss},
			\end{align*}
			where $| v \land \nabla_v |^s$ is defined as a pseudo-differential operator. The dissipativity estimate \ref{L4} is given by \cite[Lemma 2.6]{AMUXY2011}, and the nonlinear one \ref{Bbound} is given by~\cite[Theorem 2.1]{GS2011}. 	The intermediate spaces $\Ss_j$ are defined as in the cutoff case, but this time the associated splitting is
			$$\BB g := \LL g- M \chi\left( \frac{|v|}{M} \right) g, \qquad \AA g := M \chi\left( \frac{|v|}{M} \right) g,$$
			where $\chi$ is some smooth bump function and $M$ is large enough. This splitting satisfies \ref{L4} thanks to the following weighted estimates from \cite[Lemmas 2.4-2.5]{GS2011}:
			$$\la \LL g, g \ra_{\Ss_j} \le - c \| g \|^2_{\Ss_j} + C \| \mathbf{1}_{|v| \le C} f \|_{\Ss_j}^2.$$
		The enlargement assumptions are known to hold only for hard and Maxwell potentials ($\gamma \ge 0$) and the enlarged spaces are defined as
		$$\Sl_j := L^2\left( \la v \ra^{2k+2j} \d v \right), \quad k > 3 + \gamma/2 + 2 s,$$
		and the dissipation space $\Slp$ is defined in a similar fashion as in the Gaussian case. The decomposition from \ref{LE} is defined as in the Gaussian case and the dissipativity estimates follow from the weighted estimates
		\begin{gather*}
			\la \LL g, g \ra_{\Sl_j} \le - c \| g \|^2_{\Sl_j} + C \| f \|_{L^2}^2,
\quad 			\la \LL g, g \ra_{\Sl} \le - c \| g \|^2_{\Slp} + C \| f \|_{L^2}^2,
		\end{gather*}	
		which are proved in \cite[Lemma 4.2]{HTT2020}. The nonlinear estimates are proved in \cite[Lemma 2.12]{CDL2022} (see also \cite{HTT2020, AMSY2021}).
		
		\subsection{The case of the classical Landau equation}
		
		The Landau equation corresponds in some sense to the non-cutoff Boltzmann equation for $s = 1$. This time, the nonlinear operator $\QQ$ is defined as
		\begin{equation*}
			\QQ(f, f) = \nabla_v \cdot \int_{\R^3_{v_*} \times \S^{d-1}_\sigma } | v - v_* |^{\gamma+2} \Pi(v-v_*) \Big( f(v_*) \nabla_v f(v) - \nabla_{v_*} f(v_*) f(v) \Big) \d v_*,
		\end{equation*}
		where $\Pi(z) = \Id - |z|^{-2} z \otimes z$ is the orthogonal projection onto $z^\perp$. The basics assumptions \ref{L1}--\ref{L4} and \ref{Bortho}--\ref{Bisotrop} are satisfied for hard, Maxwell and moderately soft potentials ($\gamma + 2 \ge 0$) and the energy space remains the same as for the Boltzmann equation, but the dissipation space is defined as
		\begin{align*}
			\| g \|^2_{\Ssp} & \approx \| \la v \ra^{\gamma/2 + 1} g \|_\Ss^2 + \| \la v \ra^{\gamma/2} (v \land \nabla_v) g \|^2_\Ss \\
			& \approx \| \la v \ra^{\gamma/2 + 1} g \|_\Ss^2 + \| \la v \ra^{\gamma/2} \nabla_v g \|^2_\Ss + \| \la v \ra^{\gamma / 2 + 1} \Pi(v) \nabla_v g \|_\Ss^2.
		\end{align*}
		The dissipativity estimate \ref{L4} is given by \cite[(24)]{G2002}, and the nonlinear one \ref{Bbound} is given by~\cite[Lemma 2.2]{R2021}. The intermediate spaces $\Ss_j$ and splitting are the same as for the non-cutoff Boltzmann equation, which satisfies \ref{L4} using this time the weighted estimates which can be proved as \cite[Lemma 6]{G2002}:
		$$\la \LL g, g \ra_{\Ss_j} \le - c \| g \|^2_{\Ss_j} + C \| \mu f \|_{\Ss}^2.$$
		The enlarged assumptions \ref{LE} and \ref{BED} hold for Maxwell and hard potentials ($\gamma \ge 0$) and the enlarged spaces are defined as
		$$\Sl_j = L^2\left( \la v \ra^{2k + 2 j} \d v \right), \quad k > \gamma + 17/2,$$
		and the dissipation space is defined as
		\begin{align*}
			\| g \|^2_{\Slp} = \| \la v \ra^{\gamma/2 + 1} g \|_\Sl^2 + \| \la v \ra^{\gamma/2} \nabla_v g \|^2_\Sl + \| \la v \ra^{\gamma / 2 + 1} \Pi(v) \nabla_v g \|_\Sl^2.
		\end{align*}
		The decomposition from \ref{LE} is defined as in the gaussian case and the dissipativity is proved in \cite[(2.22)-(2.23)]{CTW2016}. The nonlinear bounds \ref{BED} are proved in \cite[Lemma 3.5]{CTW2016}.

		\subsection{Towards more general collision operators}
		\label{scn:general_collisions}
		The method tailored in the present contribution should be robust enough to deal with more general models. As in \cite{BaGoLe1}, one can tackle the derivation of the Navier-Stokes-Fourier system from a generic collisional kinetic equation conserving mass, momentum and energy, and dissipating entropy of the form
		\begin{equation}
			\label{eq:Kin-general}
			( \partial_t F + v \cdot \nabla_x ) F(t, x, v) = \CC[ F(t, x, \cdot) ](v) \, .
		\end{equation}
		Neglecting for a while any functional analytic issues, let us  explain formally how our framework would adapt to such model.
	 
\subsection*{General collisional equation.}
		The macroscopic conservation property reads
		$$\int_{\R^d} \CC[F](v) \left( \begin{array}{c}
			1 \\ v \\ |v|^2
		\end{array} \right) \d v = \left( \begin{array}{c}
			0 \\ 0 \\ 0
		\end{array} \right) \,$$
while we assume there exists some convex $\Phi\::\:[0,\infty) \to\R^{+}$ such that the following dissipation property (H-theorem) holds
		$$\DD[F] = \int_{ \R^d } \Phi(F(v)) \CC[F](v) \d v \le 0$$
		together with the equivalence for some universal profile $\MM : \R^{d+2}_{\bm m} \times \R^d_v \to [0, \infty)$
		$$\CC[F] = 0 \iff \DD[F] = 0 \iff F(v) = \MM \left( {\bm m} ; v \right)$$
		where the macroscopic quantities ${\bm m} = (R_F, U_F, T_F) \in \R^{d+2}$ are defined as
		$$R_F = \int_{ \R^d } F(v) \d v, \qquad  U_F = \frac{1}{R_F} \int_{\R^d} v F(v) \d v, \qquad T_F = \frac{1}{d R_F} \int_{ \R^d } |v-U_F|^2 F(v) \d v \, .$$
		Let us also assume that $\MM$ is related to $\Phi$ through
		\begin{equation}\label{eq:MMstar}
		\Phi(\MM(\bm{m} ;  v)) = a_{\bm m} + b_{\bm m} \cdot v + c_{\bm m} | v |^2 \quad \text{for some} \quad (a_{\bm m}, c_{\bm m}) \in \R^{2}\,,\,\:b_{\bm m} \in \R^{d} .\end{equation}	and that, denoting ${\bm m}_\star = (1,0,1)$ and $\MM_\star(v) = \MM({\bm m}_\star ; v)$, the macroscopic fluctuations can be factored as
		$$\MM\left( {\bm m}_\star + \eps (\rho, u, \theta) ; v \right) = \MM_\star(v) + \eps \left( \rho + u \cdot v + \frac{ \theta }{2} (|v|^2 - E) \right) \mu(v) + o(\eps) \, ,$$
		where we denoted
		$$\mu(v) := \frac{1}{ \Phi'(\MM_\star(v)) } \, .$$
		\noindent
\subsection*{The hydrodynamic limit problem} With the above premises, the scaling leading to \eqref{eq:Kin-Intro} is then
		$$F(t, x, v) = \MM_\star(v) + \eps f^\eps(\eps^2 t, \eps x, v)$$
		which, plugged in \eqref{eq:Kin-general}, yields
		$$\partial_t f^\eps + \frac{1}{\eps} v \cdot \nabla_x f^\eps = \frac{1}{\eps^2} \LL f^\eps + \frac{1}{\eps} \QQ(f^\eps, f^\eps) + \frac{1}{\eps} \RR[\eps f^\eps]$$
		where the operators $\LL$, $\QQ$ and $\RR^\eps$ are related to $\CC$ through its Taylor expansion about $\MM_\star$:
		$$\CC\left[ \MM_\star + \eps g \right] = \eps \LL g + \eps^2 \QQ(g, g) + \eps^2 \RR[\eps g]\quad \text{where} \quad \lim\limits_{\eps \to 0} \RR[\eps g] = 0 \, .$$
		The assumptions \ref{LE} and \ref{BE} are partially satisfied at a formal level. Denote the spaces
		$$N := \Span( \mu, v_1 \mu, \dots, v_d \mu, |v|^2 \mu ) , \quad H := L^2(\mu(v)^{-1} \d v )\quad \text{where} \quad \mu(v) := \frac{1}{\Phi'(\MM_\star(v))} \, .$$
		The macroscopic conservation property implies the orthogonality properties for $\LL$ and $\QQ$:
		\begin{equation}\label{eq:Nvarphi}
			\la \LL f , \varphi \ra_H = \la \QQ(f, g) , \varphi \ra_H = 0 \, , \qquad \forall \varphi \in N,
	\end{equation}
		and the H-theorem implies, at the quadratic order, the dissipativity of $\LL$ for the inner product of $H$:
		$$\lim_{\eps \to0}\frac{1}{\eps^{2}} \DD[\MM_{\star}+\eps g] = \la \LL f , f \ra_H \le 0  \,$$
		thanks to \eqref{eq:MMstar}--\eqref{eq:Nvarphi}. Finally, the assumption made on the form of the Maxwellian fluctuations implies the inclusion
		$$N \subset \nul(\LL) \, $$
		and the reverse inclusion holds as soon as the dissipativity inequality is strict for $f \perp_H N$.
		
		\medskip
		To handle the presence of the remainder $\RR$ taking into account how $\CC$ fails to be quadratic, note that it satisfies the macroscopic conservation property
		$$\la \RR[g] , \varphi \ra_H = 0 \, , \qquad \forall \varphi \in N \, .$$
		As a consequence, we expect it can be handled in a similar way as $\QQ$, thus it is enough to prove uniform bounds $ \RR [\eps f^\eps] = o(1)$ in $\mathbb{H}^s_x( H^{\circ}_v )$.

\subsection*{Examples.}
		The above general set of assumptions holds for the classical Boltzmann or Landau equation, as well as the BGK or non-linear Fokker-Planck models with
		$$\MM(R, U, T ; v) = \frac{R}{ (2 \pi T)^{d/2} } \exp\left( - \frac{|v-U|^2}{2 T} \right) \quad \text{and} \quad \Phi(r) = \log(r) \, .$$
		The above formalism allows to also encompass other models of physical interest. For instance, the BGK operator, defined as
		$$\CC[F](v) = \MM(R_F, U_F, T_F ; v) -F(v),$$
	or the nonlinear Fokker-Planck operator (see \cite[Section 1.6]{V2002}) as defined as
		$$\CC[F](v) = \nabla_v \cdot (T_F \nabla_v + (v - U_F) F) \, .$$
		Another set of examples are given by the quantum Boltzmann or Landau equations for which
		$$a_{\bm m} \exp\left( -\frac{|v-b_{\bm m} |^2}{c_{\bm m}} \right) = \frac{\MM( {\bm m} , v )}{ \FF[ \MM( {\bm m}  ; v ) ] } \quad \text{and} \quad \Phi(r) = \log\left( \frac{r}{\FF[r] } \right) \, .$$
		In the case of Fermi-Dirac ($\delta > 0$) or Bose-Einstein ($\delta < 0$) statistics (see \cite{D1994, ABL}), the function $\FF$ is given by
		$$\FF[r] = 1 -\delta r \, ,$$
		and in the case of Haldane statistics (see \cite{AN2015}), $\FF$ is given by
		$$\FF[r] = (1-\alpha r)^{\alpha} \left( 1 + (1-\alpha) r\right)^{1-\alpha} \quad \text{for some} \quad \alpha \in (0, 1) \, .$$
		The quantum Boltzmann operator is then given by
		$$\CC[F](v) = \int_{\R^{d}\times\S^{d-1}}B(|v-v_{*}|,\sigma) \Big\{  f(v')f(v_{*}') \FF[f(v)] \FF[f(v_*)] -f(v)f(v_{*}) \FF[f(v')] \FF[f(v_*')] \Big\} \d v_{*}\d\sigma \, ,$$
		whereas the quantum Landau operator is given by
		$$\CC[F](v) = 
		\nabla_{v}\cdot \int_{\R^{d}} |v-v_{*}|^{\gamma+2} \, \Pi_{v-v_{*}}
		\Big\{f(v_{*}) \FF[f(v_*)] \nabla_{v} f(v) - f(v) \FF[f(v)]{\nabla_{v_{*}} f}(v_{*}) \Big\}
		\, \d v_{*} \, .$$
		
		At the linear level, the quantum Boltzmann and Landau equations fit within our framework. The various assumptions \ref{L1}--\ref{L4} have   been shown to hold for the Boltzmann-Bose-Einstein equation $(\delta < 0)$ in \cite{YZ22, Z22} in the case of very soft potentials ($\gamma + 2 s < 0$) for which there does not hold $\| \cdot \|_{\Ssp} \ge \| \cdot \|_\Ss$, but is more intricate than the one for which \ref{L4} would be satisfied ($\gamma + 2 s \ge 0$). 
		
		For Fermi-Dirac statistics, a spectral gap estimate can be found in \cite{jiang} in the case  $\gamma = b(\cos \theta) = 1$. Concerning the Landau equation in the case $\delta > 0$, the linear assumptions \ref{L1}--\ref{L4} as well as \ref{LE} have been partially checked in \cite{ABL}. We refer the reader to the work in preparation \cite{GL2023} for more details on the quantum Boltzmann equation.
		\color{black}
		
		\section{Technical toolbox}\label{sec:toolbox}
		
		\subsection{Littlewood-Paley theory}
		\label{scn:littlewood-paley}
		For some appropriate $\varphi \in \CC^\infty\left( \R^d \right)$ supported in an annulus centered about $0$ and $\chi \in \CC^\infty\left( \R^d \right)$ supported in a ball centered about $0$ such that
		$$0 \le \varphi, \chi \le 1, \qquad \chi(\xi) + \sum_{j=0}^\infty \varphi\left( 2^{-j} \xi \right) = \sum_{j=-\infty}^\infty \varphi\left( 2^{-j} \xi \right) = 1,$$
		one defines the \textit{homogeneous Littlewood-Paley projectors} for any $j \in \Z$ (see \cite[Section 2.2]{BCD2011}):
		\begin{gather*}
			\dot{\Delta}_j u := \FF^{-1}_\xi \left[ \varphi\left( 2^{-j} \xi \right) \widehat{u}(\xi) \right], \qquad \dot{S}_j u := \FF^{-1}_\xi \left[ \chi\left( 2^{-j} \xi \right) \widehat{u}(\xi) \right],
		\end{gather*}
		as well as \textit{Bony's homogeneous decomposition} (see \cite[Section 2.6.1]{BCD2011}):
		$$u v = \dot{T}_u v + \dot{T}_v u + \dot{R}(u, v),$$
		where the homogeneous paraproduct $\dot{T}_f g$ and the homogeneous remainder $\dot{R}(f, g)$ are defined as
		$$\dot{T}_f g := \sum_{j=-\infty}^\infty \dot{S}_{j-1} f \dot{\Delta}_j g, \qquad \dot{R}(f, g) := \sum_{|j - k| \le 1} \dot{\Delta}_j f \dot{\Delta}_k g.$$
		This decomposition allows to prove the following product rule, which follows from the combination of \cite[Corollary 2.55]{BCD2011} and the embedding $\dot{\mathbb{B}}^s_{2, 1} \hookrightarrow \dot{\mathbb{B}}^s_{2, 2} = \dot{\mathbb{H}}^s$.
		\begin{prop}
			\label{prop:product_homogeneous_sobolev}
			For any $s_1, s_2 \in \left( -\frac{d}{2} , \frac{d}{2} \right)$ such that $s_1 + s_2 > 0$, there holds
			$$\| u v \|_{ \dot{\mathbb{H}}^{-\frac{d}{2} + s_1 + s_2} } \lesssim \| u \|_{ \dot{\mathbb{H}}^{s_1} } \| v \|_{ \dot{\mathbb{H}}^{s_2} }.$$
		\end{prop}
		One also defines the \textit{inhomogeneous Littlewood-Paley projectors} for $j \ge -1$ (see \cite[Section 2.2]{BCD2011}):
		$$\Delta_j u := \begin{cases}
			\FF^{-1}_\xi \left[ \chi(\xi) \widehat{u}(\xi) \right], & j = -1, \\
			\\
			\dot{\Delta}_j u, & j \ge 0,
		\end{cases} \qquad S_j := \sum_{k = -1}^{j-1} \Delta_k u,$$
		as well as Bony's inhomogeneous decomposition (see \cite[Section 2.8.1]{BCD2011}):
		$$uv = T_u v + T_v u + R(u, v),$$
		where the inhomogeneous paraproduct $T_f g$ and remainder $R(f, g)$ are defined as
		$$T_f g := \sum_{j=-1}^\infty S_{j-1} f \Delta_j g, \qquad R(f, g) := \sum_{|j - k| \le 1} \Delta_j f \Delta_k g.$$
		This decomposition allows to prove the following product rule, which follows from the combination of \cite[Corollary 2.86]{BCD2011} and the embedding $H^{\sigma} = B^{\sigma}_{2, 2} \hookrightarrow L^\infty$ whenever $\sigma > d/2$.
		\begin{prop}
			For any $s > 0$ and $\sigma > d/2$, there holds
			$$\| u v \|_{H^s} \lesssim \| u \|_{H^s} \| v \|_{H^\sigma} + \| u \|_{H^\sigma} \| v \|_{H^s}.$$
		\end{prop}
		We also present the following Sobolev-H\"older product rule.
		\begin{prop}
			\label{prop:product_sobolev_holder}
			For any $s' > s > 0$, there holds
			$$\| u v \|_{H^s} \lesssim \| u \|_{H^s} \| v \|_{W^{s', \infty}}.$$
		\end{prop}
		\begin{proof}
			When $s \in \N$, the estimate follows from the combination of Leibniz's formula and H\"older's inequality in $L^2 \times L^\infty$: for any $\sigma \le s$
			\begin{align*}
				\| \nabla^\sigma (u v) \|_{L^2}  \lesssim \sum_{r = 0}^\sigma \| \nabla^r u \nabla^{\sigma - r} v \|_{L^2}  
				 \lesssim \sum_{r = 0}^\sigma \| \nabla^r u \|_{L^2} \| \nabla^{\sigma - r} v \|_{L^\infty} \lesssim \| u \|_{H^\sigma} \| v \|_{W^{\sigma, \infty}},
			\end{align*}
			and thus $\| uv \|_{H^s} \lesssim \| u \|_{H^s} \| v \|_{W^{s, \infty}}$.
			
			When $s \notin \N$, then $W^{s, \infty} = B^{s}_{\infty, \infty}$, and we rely on Bony's inhomogeneous decomposition. On the one hand, \cite[Theorem 2.82]{BCD2011} and \cite[Theorem 2.85]{BCD2011} yield
			$$\| T_v u \|_{H^s} + \| R(u, v) \|_{H^s} \lesssim \| u \|_{H^s} \| v \|_{W^{s, \infty}},$$
			so we are left with estimating $T_u v$. Following the proof of \cite[Theorem 2.82]{BCD2011} (or more precisely \cite[Theorem 2.82]{BCD2011}), there holds
			\begin{equation*}\begin{split}
				\| T_u v \|_{H^s}^2 & \lesssim \sum_{j = -1}^\infty 2^{2 j s} \| S_{j-1} u \Delta_j v \|_{L^2}^2 \\
				& \lesssim \sum_{j = -1}^\infty 2^{-2 j (s'-s)} \left( 2^{ j s'} \| \Delta_j v \|_{L^\infty} \right)^2 \| S_{j-1} u \|_{L^2}^2 \\
				& \lesssim \| u \|_{W^{s', \infty}}^2 \| u \|_{L^2}^2 ,
			\end{split}\end{equation*}
			where we used the definition of the norm $B^{s}_{\infty, \infty}=W^{s, \infty}$ and the fact that $\| S_{j-1} u \|_{L^2} \lesssim \| u \|_{L^2}$.
		\end{proof}
		
		\subsection{About the wave equation} On the basis of the above Littlewood-Paley decomposition, we also establish the following decay in time of the wave semigroup. Such a result is very likely to be part of the folklore knowledge for dispersive equations and we give a full proof here:
		\begin{lem}\label{lem:wave-equation} Given $u \in \dot{\mathbb{B}}_{1,1}^{\frac{d+1}{2}}(\R^d)$ one has
			$$\|e^{it|D_x|}u\|_{L^\infty} \lesssim t^{-\frac{d-1}{2}}\|u\|_{\dot{\mathbb{B}}_{1,1}^{\frac{d+1}{2}}}, \qquad t >0$$
			whereas, for $s \geq 0$ and $u \in \dot{\mathbb{B}}_{1,1}^{\frac{d+1}{2}+s}$,
			$$\|e^{it|D_x|}u\|_{W^{s,\infty}}\lesssim t^{-\frac{d-1}{2}}\|u\|_{\dot{\mathbb{B}}^{\frac{d+1}{2}+s}_{1,1}} \qquad t >0.$$
		\end{lem}
		\begin{proof}
			For any $j \in \Z$, we denote $u_j(x) = u\left( 2^j x \right)$ and notice
			$$e^{i t | D_x | } \dot{\Delta}_j u = \left(e^{i 2^j t | D_x | } \dot{\Delta}_0 u_{-j} \right)_j, \qquad \dot{\Delta}_0 u_{-j} = \left(\dot{\Delta}_j u \right)_{-j}.$$
			Using the scaling properties of the $L^1$ and $L^\infty$-norms, we deduce
			$$\| e^{i t | D_x | } \dot{\Delta}_j u \|_{L^\infty} = \| e^{i 2^j t | D_x | } \dot{\Delta}_0 u_{-j} \|_{L^\infty}, \qquad \| \dot{\Delta}_0 u_{-j} \|_{L^1_x} = 2^{j d } \| \dot{\Delta}_j u \|_{L^1_x}, $$
			and thus, using the dispersive estimate for functions whose frequencies are localized in an annulus (see \cite[Proposition 8.15]{BCD2011})
			\begin{equation*}\begin{split}
				\| e^{i t | D_x | } u \|_{L^\infty_x} & \le \sum_{j \in \Z} \| e^{i t | D_x | } \dot{\Delta}_j u \|_{L^\infty_x}  \lesssim \sum_{j \in \Z} \left(2^j t\right)^{-\frac{d-1}{2}} \| \dot{\Delta_0} u_{-j} \|_{L^1_x} \\
				& \lesssim t^{-\frac{d-1}{2}} \sum_{j \in \Z} \left(2^j \right)^{\frac{d+1}{2}} \| \Delta_j u \|_{L^1_x} = t^{-\frac{d-1}{2}} \| u \|_{ \dot{\mathbb{B}}^{(d+1)/2}_{1,1} }.
			\end{split}\end{equation*}
			Furthermore, there holds
			\begin{align*}
				\| e^{i t | D_x | } u \|_{W^{s, \infty}_x} & \approx \| e^{i t | D_x | } u \|_{L^\infty_x} + \sup_{j \ge 0} 2^{j s} \| e^{i t | D_x | } \dot{\Delta}_j u \|_{L^\infty_x} \\
				& \lesssim t^{-\frac{d-1}{2}} \left( \| u \|_{ \dot{\mathbb{B}}^{(d+1)/2}_{1,1} } + \sup_{j \ge 0} 2^{j \left(s+\frac{d+1}{2}\right) } \| \dot{\Delta}_j u \|_{L^1_x}\right) \\
				& \lesssim t^{-\frac{d-1}{2}} \| u \|_{ \dot{\mathbb{B}}^{(d+1)/2 + s }_{1,1} }.
			\end{align*}
			This concludes the proof.
		\end{proof}
		
		\subsection{Duality}
		\label{scn:duality}
		
		Consider some surjective isometry $\Lambda : \Sgp \to \Sg$, that is to say
		$$\| f \|_{\Sgp} = \| \Lambda f \|_{\Sg}.$$
		Its adjoint $\Lambda^\star : \Sg \to \Sgm$ for the inner product of $\Sg$ is then an isometry as well:
		\begin{equation*}
			\| \Lambda^\star f \|_{\Sgm} = \sup_{\| g \|_{\Sgp} = 1} \la f, \Lambda g \ra_{\Sg} = \sup_{\| g' \|_{\Sg} = 1} \la f, g' \ra_{\Sg}  = \| f \|_{\Sg},
		\end{equation*}
		and it extends naturally to a surjective isometry since $\Sgm$ was defined as the completion of $\Sg$.
		This allows to write the~$\BBB( \Sgm ; \Sgp )$-norm in term of the $\BBB(\Sg)$-norm as well as the isometries $\Lambda$ and $\Lambda^\star$:
		$$\| T \|_{ \Sgm \to \Sgp } = \| \Lambda T \Lambda^\star \|_{ \Sg \to \Sg},$$
		and thus deduce the identity
		\begin{equation}
			\label{eq:twisted_adjoint_identity}
			\| T \|_{ \Sgm \to \Sgp } = \| T^\star \|_{ \Sgm \to \Sgp },
		\end{equation}
		from the classical one $\| S \|_{\Sg \to \Sg} = \| S^\star \|_{\Sg \to \Sg}$ and $(\Lambda^\star)^\star=\Lambda$. Similarly, we have
		\begin{equation}
			\label{eq:twisted_adjoint_identity_one_sided}
			\| T \|_{\Sg \to \Sgp } = \| T^\star \|_{ \Sgm \to \Sg },
		\end{equation}
		\subsection{Bootstrap formula for projectors}
		\label{scn:boostrap_projectors}
		
		We present some formulas relating the remainder of Taylor expansions for projectors with the lower order terms and remainders. This will allow to prove inductively regularizing properties on each term of said expansion.
		
		\begin{lem}
			Consider a projector $P(r) \in \BBB(E)$ depending on a parameter $r \in [0, 1]$ and its Taylor expansion at order $N \ge 0$:
			\begin{gather*}
				P(r) = \sum_{n=0}^{N-1} r^n P^{(n)} + r^N P^{(N)}(r),
			\end{gather*}
			whose (constant) coefficients belong to $\BBB(E)$ and satisfy the identities
			\begin{equation}
				\label{eq:convolution_polynom_projector}
				\sum_{n=0}^{M} P^{(n)} P^{(M-n)} =P^{(M)}, \quad 0 \le M \le N-1,
			\end{equation}
			then the remainder satisfies a similar one:
			\begin{align}
				\label{eq:bootstrap_P_N_A}
				P^{(N)}(r) & = P(r) P^{(N)}(r) + \sum_{n=1}^{N} P^{(n)} (r) P^{(N-n)} \\
				\label{eq:bootstrap_P_N_B}
				& = P^{(N)}(r) P(r) + \sum_{n=1}^{N} P^{(N-n)} P^{(n)} (r).
			\end{align}
		\end{lem}
		
		\begin{proof}
			We only take care of \eqref{eq:bootstrap_P_N_A}. Note that when $N=0$, this reduces to $P(r) = P(r) P(r)$, which is true since $P(r)$ is a projector. We prove the case $N \ge 1$ by induction. We start from the induction hypothesis at order $N$:
			$$P^{(N)}(r) = P(r) P^{(N)}(r) + \sum_{n=1}^{N} P^{(n)} (r) P^{(N-n)},$$
			and inject the expansions $P^{(N)}(r) = P^{(N)} + r P^{(N+1)}(r)$ and $P^{(n)}(r) = P^{(n)} + r P^{(n+1)}(r)$:
			\begin{equation*}\begin{split}
				P^{(N)}(r) = & P(r) \Big( P^{(N)} + r P^{(N+1)}(r) \Big) + \sum_{n=1}^{N} \Big(P^{(n)} + r P^{(n+1)}(r)\Big) P^{(N-n)} \\
				=&  P(r) P^{(N)} + \sum_{n=1}^{N} P^{(n)} P^{(N-n)} + r \Big( P(r) P^{(N+1)}(r)  +  \sum_{n=1}^{N} P^{(n+1)}(r) P^{(N-n)} \Big).
			\end{split}\end{equation*}
			Next, we expand $P(r) = P^{(0)} + r P^{(1)}(r)$ in the first line:
			\begin{align*}
				P^{(N)}(r) =&  \Big( P^{(0)} + r P^{(1)}(r) \Big) P^{(N)} + \sum_{n=1}^{N} P^{(n)} P^{(N-n)} \\
				&+ r \Big( P(r) P^{(N+1)}(r)  +  \sum_{n=1}^{N} P^{(n+1)}(r) P^{(N-n)} \Big) \\
				= & \sum_{n=0}^{N} P^{(n)} P^{(N-n)} + r \Big( P(r) P^{(N+1)}(r)  +  \sum_{n=0}^{N} P^{(n+1)}(r) P^{(N-n)} \Big).
			\end{align*}
			Since we assumed \eqref{eq:convolution_polynom_projector}, we replace the first term and thus have
			$$P^{(N)}(r) = P^{(N)} + r \Big( P(r) P^{(N+1)}(r)  +  \sum_{n=0}^{N} P^{(n+1)}(r) P^{(N-n)} \Big),$$
			from which we conclude using $P^{(N)}(r) = P^{(N)} + r P^{(N+1)}(r)$.
		\end{proof}
		
		In the case of the spectral projector from Lemma \ref{lem:expansion_projection}, we use the following corollary.
		\begin{cor}
			The following identities hold for $N=1$:
			\begin{align}
				\label{eq:bootstrap_P_1_A}
				\PP^{(1)}(\xi) & = \PP(\xi) \PP^{(1)}(\xi) + \PP^{(1)}(\xi) \PP,\\
				\label{eq:bootstrap_P_1_B}
				& = \PP^{(1)}(\xi) \PP(\xi) + \PP \PP^{(1)}(\xi),
			\end{align}
			and, assuming $\PP^{(1)} = \PP \PP^{(1)} + \PP^{(1)} \PP$, for $N=2$:
			\begin{align}
				\label{eq:bootstrap_P_2_A}
				\PP^{(2)}(\xi) & = \PP(\xi) \PP^{(2)}(\xi) + \PP^{(1)}(\xi) \otimes \PP^{(1)} + \PP^{(2)}(\xi) \PP,\\
				\label{eq:bootstrap_P_2_B}
				& = \PP^{(2)}(\xi) \PP(\xi) + \PP^{(1)} \otimes \PP^{(1)}(\xi) + \PP \PP^{(2)}(\xi).
			\end{align}
		\end{cor}
		
		\section{Properties of the Navier-Stokes equations}
		\label{sec:N-S}
		
		We recall here some classical results on the Navier-Stokes equations and refer to \cite{R2016} and references therein.
		
		\begin{theo}[\textit{\textbf{Cauchy theory for Navier-Stokes}}]
			\label{thm:cauchy_NSF}
			Let $s \ge \frac{d}{2} - 1$ and consider a triple of initial conditions $(\varrho_\ini, u_\ini, \theta_\ini) \in \mathbb{H}^s_x$ satisfying
			$$\nabla_x (\varrho_\ini + \theta_\ini) = 0, \qquad \nabla_x \cdot u_\ini = 0.$$
			There exists a unique maximal lifespan $T_* \in (0, \infty]$ such that, for any $T < T_*$, the initial data~$(\varrho_\ini, u_\ini, \theta_\ini)$ generates a unique solution
			$$(\varrho, u, \theta) \in \CC\left( [0, T] ; \mathbb{H}^s_x\right) \cap L^2\left( [0, T] ; H^{s+1}_x \right)$$
			to the incompressible Navier-Stokes-Fourier system
			\begin{equation}
				\label{eq:NSF}
				\begin{cases}
					\partial_t u + u \cdot \nabla_x u = \kappa_\Inc \Delta_x u - \nabla_x p, & \nabla_x \cdot u = 0,\\
					\partial_t \theta +  u \cdot \nabla_x \theta = \kappa_\Bou \Delta_x \theta, & \nabla_x (\varrho + \theta) = 0,
				\end{cases}
			\end{equation}
			and it satisfies for some universal constant $C >0$
			\begin{align*}
				\| (\varrho, u, \theta) \|_{ L^\infty \left( [0, T] ; H^s \right) } + \| \nabla_x & (\varrho, u, \theta) \|_{ L^2 \left( [0, T] ; H^s \right) } \\
				& \le C \| (\varrho_\ini, u_\ini, \theta_\ini) \|_{\mathbb{H}^s_x } \exp\left( C \| \nabla_x u \|_{ L^2 \left( [0, T] ; H^{\frac{d}{2}-1}_x \right) } \right).
			\end{align*}
			If the solution is global (i.e. $T_* = \infty$), the solution vanishes for large times:
			$$\lim_{t \infty} \left\| (\rho(t), u(t), \theta(t)) \right\|_{\mathbb{H}^s} = 0,$$
			this is the case if $d = 2$, or if $d \ge 3$ and $\| u \|_{ \mathbb{H}^{\frac{d}{2} - 1}_x }$ is small.
		\end{theo}
		
		Note that, on the one hand, $\varrho(t), \theta(t) \in L^2_x$ thus the Boussinesq condition $\nabla_x (\varrho + \theta) = 0$ is equivalent to $\varrho + \theta = 0$, and on the other hand, since $u$ is incompressible, the pressure (which is to be interpreted as a Lagrange multiplier) can be eliminated using Leray's projector $\mathbb{P}$ on incompressible fields:
		\begin{equation}
			\label{eq:NSF_1}
			\begin{cases}
				\partial_t u + \mathbb{P} \left( u \cdot \nabla_x u \right) = \kappa_\Inc \Delta_x u,\\
				\partial_t \theta = \kappa_\Bou \Delta_x \theta + u \cdot \nabla_x \theta, \\
				\varrho =- \theta,
			\end{cases}
		\end{equation}
		or, equivalently,
		\begin{equation}
			\label{eq:NSF_2}
			\begin{cases}
				\partial_t u + \mathbb{P} \big[ \nabla_x \cdot (u \otimes u) \big] = \kappa_\Inc \Delta_x u,\\
				\partial_t \theta = \kappa_\Bou \Delta_x \theta + \nabla_x \cdot (u \theta) , \\
				\varrho = -\theta.
			\end{cases}
		\end{equation}
	 
		The next two results detail in what sense the Navier-Stokes-Fourier system is equivalent to \eqref{eq:reduction_NS_kin}, proving Proposition \ref{prop:equivalence_kinetic_hydrodynamic_INSF}.
		\begin{lem}
			The following identities hold.
			\label{lem:Q_A}
			\begin{enumerate} 
				\item For the Burnett function $\BurA $, one has
				\begin{gather*}
					\left\la \QQ^\sym( v_i \mu, v_j \mu), \LL^{-1}\BurA \right\ra_{\Ss} = \frac{\vartheta_1}{2} \left( E_{i, j} + E_{j, i} - \frac{2}{d} \delta_{i, j} \Id \right),
				\end{gather*}
				where $(E_{i, j})_{i, j = 1}^d$ is the canonical basis of $\mathscr{M}_{d\times d}$, and the coefficient $\vartheta_1$ is defined as
				\begin{equation*}
					\vartheta_1 := - d \sqrt{ \dfrac{d}{E} } \left\la \QQ^\sym(v_1 \mu, v_1 \mu), \LL^{-1} \left( \Id - \PP \right) v_2^2 \mu \right\ra_{\Ss}.
				\end{equation*}
				Moreover, there holds for $\varphi = \mu, (|v|^2-E) \mu$
				\begin{equation*}
					\left\la \QQ^\sym( \varphi , v \mu ) , \LL^{-1}\BurA \right\ra_{\Ss} = \left\la \QQ^\sym( \varphi , \varphi ) , \LL^{-1}\BurA \right\ra_{\Ss} = 0.
				\end{equation*}
				\item Regarding the Burnett function $\BurB $, one has
				\begin{gather*}
					\left\la \QQ^\sym( v_i \mu, \mu ), \LL^{-1}\BurB \right\ra_{\Ss} = \vartheta_2 \, \mathbf{e}_i, \\
					\left\la \QQ^\sym( v_i \mu,  \left( |v|^2 - E \right) \mu ), \LL^{-1}\BurB \right\ra_{\Ss} = \vartheta_3 \, \mathbf{e}_i,
				\end{gather*}
				where $(\bm{e}_i)_{i=1}^d$ is the canonical basis of $\R^d$ and the coefficients $\vartheta_i$ are defined as
				\begin{equation*}\begin{cases}
						\vartheta_2 &:= - \dfrac{1}{E \sqrt{K (K-1)}} \left\la \QQ^\sym(v_1 \mu, \mu), \LL^{-1} \left( \Id - \PP \right) v_1 |v|^2 \mu \right\ra_{\Ss},\\
						\\
						\vartheta_3 &:= - \dfrac{1}{E \sqrt{K (K-1)}} \left\la \QQ^\sym(v_1 \mu, \left( |v|^2 - E \right) \mu), \LL^{-1} \left( \Id - \PP \right) v_1 |v|^2 \mu \right\ra_{\Ss}\,.
				\end{cases}\end{equation*}
				Furthermore, for $\varphi, \psi = \mu, \, v \mu, \, \left( |v|^2 - E \right) \mu$, one has
				\begin{equation*}
					\left\la \QQ^\sym( \varphi , \psi  ), \LL^{-1}\BurB \right\ra_{\Ss} = 0 \qquad \text{ and } \quad
					\left\la \QQ^\sym( v_i \mu , v_j \mu ), \LL^{-1}\BurB \right\ra_{\Ss} = 0.
				\end{equation*}
			\end{enumerate}
		\end{lem}
		
		\begin{proof}
			We recall the notation from Section \ref{scn:spectral_study}
			$$\mathsf{R}_0 := \LL^{-1} \left( \Id - \PP\right) \quad \text{ as  well as the identity } \quad \LL^{-1}\BurA = \sqrt{ \frac{d}{E} } \mathsf{R}_0 \left[ v \otimes v \mu \right].$$
			We also recall that both $\LL$ and $\QQ^\sym$ commute with orthogonal matrices, and in particular preserve the evenness/oddity.
			
			In this proof, we only prove the first identity which is the most intricate, that is we compute for all $i, j, k, \ell$
			$$\left\la \QQ^\sym(v_i \mu, v_j \mu ), \mathsf{R}_0 \left[ v_k v_\ell \mu \right] \right\ra.$$
			The other identities are proved in a similar yet simpler manner.
			
			\step{1}{The case $i \neq j$}
			If $\{ i, j \} \neq \{ k, \ell \}$, then $\QQ^\sym(v_i \mu, v_j \mu )$ is odd in the variables $v_i$ and $v_j$, however~$\mathsf{R}_0 \big[ v_k v_\ell \mu \big]$ is even in at least one of these variables, thus
			$$\left\la \QQ^\sym(v_i \mu, v_j \mu) , \mathsf{R}_0 \big[ v_k v_\ell \mu \big] \right\ra = 0.$$
			If $\{ i, j \} = \{ k, \ell \}$, we use the isometric change of variables $(v_i, v_j) \rightarrow \left(\frac{v_1 + v_2}{\sqrt{2} }, \frac{v_1 - v_2}{\sqrt{2} } \right)$, which is compatible with the invariance of $\LL$ and $\QQ^\sym$: 
			\begin{align*}
				\big\la \QQ^\sym(v_i \mu, v_j \mu) , \mathsf{R}_0 \big[ v_i v_j \mu \big] \big\ra
				&= \frac{1}{4} \left\la \QQ^\sym( (v_1 + v_2) \mu, (v_1 - v_2) \mu) , \mathsf{R}_0 \big[ (v_1-v_2) (v_1 + v_2) \mu \big] \right\ra \\
				&= \frac{1}{4} \left\langle \QQ^\sym( v_1 \mu, v_1 \mu) - \QQ^\sym( v_2 \mu, v_2 \mu) , \mathsf{R}_0 \big[ v_1^2 \mu \big] - \mathsf{R}_0 \big[ v_2^2 \mu \big] \right\rangle \\
				&= \frac{1}{2} \left( \left\la \QQ^\sym( v_1 \mu, v_1 \mu) , \mathsf{R}_0 \big[ v_1^2 \mu \big] \right\ra - \left\la \QQ^\sym( v_1 \mu, v_1 \mu) , \mathsf{R}_0 \big[ v_2^2 \mu \big] \right\ra\right)
			\end{align*}
			where we used the change of variables $(v_1, v_2) \leftrightarrow (v_2, v_1)$ in the last identity. Using that
			$$\mathsf{R}_0 \big[ v_1^2 \mu \big] + \sum_{j = 2}^d \mathsf{R}_0 \big[ v_j^2 \mu \big] = \mathsf{R}_0 \left[ |v|^2 \mu \right] = 0$$
			together with the change of variables $v_j \rightarrow v_2$, we can rewrite the previous identity as
			\begin{equation*}\begin{split}
				\left\la \QQ^\sym(v_i \mu, v_j \mu) , \mathsf{R}_0 \big[ v_i v_j \mu \big] \right\ra & = \frac{1}{2} \Bigg( - \sum_{j = 2}^d \left\la \QQ^\sym( v_1 \mu, v_1 \mu) , \mathsf{R}_0 \big[ v_j^2 \mu \big] \right\ra \\
				&\phantom{+++} - \left\la \QQ^\sym( v_1 \mu, v_1 \mu) , \mathsf{R}_0 \big[ v_2^2 \mu \big] \right\ra\Bigg) \\
				& = - \frac{d}{2} \left\la \QQ^\sym( v_1 \mu, v_1 \mu) , \mathsf{R}_0 \big[ v_2^2 \mu \big] \right\ra.
			\end{split}\end{equation*}
			To sum up, if $i \neq j$, we have 
			$\displaystyle \left\la \QQ^\sym(v_i \mu, v_j \mu) , \LL^{-1}\BurA \right\ra = \frac{ \vartheta_1 }{2} \left( E_{i, j} + E_{j, i} \right).$
			
			\step{2}{The case $i = j$}
			If $k \neq \ell$, then $\mathsf{R}_0 \big[ v_ k v_\ell \mu \big]$ is odd in both $v_k$ and $v_\ell$, whereas $\QQ^\sym(v_i \mu, v_i \mu)$ is even in all directions, thus
			$$\left\la \QQ^\sym(v_i \mu, v_i \mu ) , \mathsf{R}_0\big[ v_k v_\ell \mu \big] \right\ra = 0.$$
			When $k = \ell$, arguing as in \textit{Step 1}, we have
			$$\left\la \QQ^\sym(v_i \mu, v_i \mu), \mathsf{R}_0 \big[ v_k^2 \mu \big] \right\ra = 
			\begin{cases}
				-(d-1) \left\la \QQ^\sym(v_1 \mu, v_1 \mu), \mathsf{R}_0 \big[ v_1^2 \mu \big] \right\ra, & k = i, \\
				\left\la \QQ^\sym(v_1 \mu, v_1 \mu), \mathsf{R}_0 \big[ v_2^2 \mu \big] \right\ra, & k \neq i.
			\end{cases}$$
			To sum up, when $i = j$, we have
			$\displaystyle \left\la \QQ^\sym(v_i \mu, v_i \mu), \LL^{-1}\BurA \right\ra = \vartheta_1 \left( E_{i, i} - \frac{1}{d} \Id \right)$ and this proves the result.
			
			\step{3}{Comments on the other coefficients}
			Regarding the Burnett function $\BurA $, one proves similarly for $\varphi = \mu, (|v|^2 - E) \mu$
			$$\left\la \QQ^\sym( \varphi, \varphi ) , \LL^{-1}\BurA \right\ra_{\Ss} = \sqrt{ \dfrac{d}{E} } \left\la \QQ^\sym(\varphi, \varphi), \LL^{-1} \left( \Id - \PP \right) v_1^2 \mu \right\ra_{\Ss} \Id,$$
			and observing that the $\QQ(\varphi, \varphi)$ is radial and that $|v|^2 \mu = \sum_{i=1}^d v_i^2 \mu \in \nul(\LL)$, we deduce using the change of variable $v_1 \rightarrow v_i$ that these coefficients vanish.
			
			\medskip
			Regarding the Burnett function $\BurB $, one proves similarly for $\varphi = \mu, (|v|^2 - E) \mu$
			$$\left\la \QQ^\sym( v_i \mu, \varphi ), \LL^{-1}\BurB \right\ra_{\Ss} = - \dfrac{1}{E \sqrt{K (K-1)}} \left\la \QQ^\sym(v_1 \mu, \varphi), \LL^{-1} \left( \Id - \PP \right) v_1 |v|^2 \mu \right\ra_{\Ss},$$
			and for $\varphi, \psi = \mu, (|v|^2 - E) \mu$, we have that
			$$\left\la \QQ^\sym( \varphi, \psi ), \LL^{-1}\BurB \right\ra_{\Ss} = - \dfrac{1}{E \sqrt{K (K-1)}} \left\la \QQ^\sym( \varphi, \psi ), \LL^{-1} \left( \Id - \PP \right) v_1 |v|^2 \mu \right\ra_{\Ss} = 0,$$
			where we used that $\QQ(\varphi, \psi)$ is radial and $\LL^{-1} \left( \Id - \PP \right) v_1 |v|^2 \mu$ is odd in $v_1$. Finally, there holds
			$$\la \QQ^\sym(v_i \mu, v_j \mu) , \LL^{-1}\BurB \ra_\Ss = 0$$
			because $\QQ^\sym(v_i \mu, v_j \mu)$ is odd in both $v_i$ and $v_j$ if $i \neq j$, or even in $v_i = v_j$ otherwise, and~$\LL^{-1} B$ is odd. This concludes the proof.
		\end{proof}
		
		\begin{rem}
			Note that in the case of the classical Boltzmann and Landau equations, the operator $\LL$ is related to $\QQ$ through a linearization procedure, and one can show the identity
			$$\forall f \in \nul(\LL), \quad \QQ(f, f) = - \frac{1}{2} \LL \left(f^2 \mu^{-1} \right),$$
			which implies that $\vartheta_2 = 0$ and the coefficients $\vartheta_1$ and $\vartheta_3$ can be computed explicitly.
		\end{rem}

		Thanks to the above result, we are in position to prove the Proposition \ref{prop:equivalence_kinetic_hydrodynamic_INSF}.
		
		\begin{proof}[Proof of Proposition \ref{prop:equivalence_kinetic_hydrodynamic_INSF}]
			We recall the integral formulation of the incompressible Navier-Stokes system:
			$$
			\begin{cases}
				\displaystyle u(t) = e^{ \kappa_\Inc t \Delta_x } \mathbb{P} u_\ini - \vartheta_\Inc \mathbb{P} \int_0^t e^{(t-\tau) \kappa_\Inc \Delta_x} \, \nabla_x \cdot \left( u \otimes u \right) (\tau) \, \d \tau, \\
				\displaystyle \theta(t) = e^{ t \kappa_\Bou \Delta_x} \theta_\ini - \vartheta_\Bou \int_0^t e^{(t-\tau) \kappa_\Bou \Delta_x } \nabla_x \cdot  (u \theta)(\tau) \d \tau, \\
				\varrho = - \theta,
			\end{cases}
			$$
			where we recall that $\mathbb{P}$ is Leray's projector on incompressible fields, and we point out the equivalence between $\nabla_x(\varrho +\theta) = 0$ and $\varrho + \theta = 0$ since $\varrho(t), \theta(t) \in L^2_x$.
			
			\medskip
			Regarding the kinetic integral equation, we recall the definitions of $U_\ns$ and $V_\ns$:
			$$U_\ns(t) = e^{t \kappa_\Inc \Delta_x} \PP^{(0)}_{\Inc} + e^{t \kappa_\Bou \Delta_x} \PP^{(0)}_{\Bou}, \qquad V_\ns(t) = e^{t \kappa_\Inc \Delta_x} \PP^{(1)}_{\Inc} + e^{t \kappa_\Bou \Delta_x} \PP^{(1)}_{\Bou},$$
			and point out the equivalence coming from  Proposition \ref{prop:macro_representation_spectral}:
			\begin{equation}
				\label{eq:inc_bou_conditions}
				\left( \Id - \PP^{(0)}_\Inc \right) f = \left( \Id - \PP_\Bou^{(0)} \right) f = 0 \iff \nabla_x (\varrho + \theta) = 0 ~\text{ and }~ \nabla_x \cdot u = 0,
			\end{equation}
			and since \ref{Bortho} assumes $\QQ(f, f) \perp \nul(\LL)$, we have
			$$\PP^{(0)}_\Inc \PP^{(1)}_\Inc \QQ(f, f) = \PP^{(1)}_\Inc \QQ(f, f), \qquad\PP^{(0)}_\Bou \PP^{(1)}_\Bou \QQ(f, f) = \PP^{(1)}_\Bou \QQ(f, f),$$
			thus \textit{we assume \eqref{eq:inc_bou_conditions} from now on}.
			
			\medskip
			Since $U_\ns$ and $V_\ns$ both take values in macroscopic distributions, it is enough to consider their macroscopic components.
			
			\step{1}{Description of $\PP^{(1)}_\Inc \QQ(f, f)$ and $\PP^{(1)}_\Bou \QQ(f, f)$}
			Plugging the expression \eqref{eq:f-macro} of $f$ into the nonlinearity $\QQ(f, f)$, we have
			\begin{align*}
				\QQ(f, f) = & \varrho^2 \QQ( \mu, \mu ) + \QQ( u \cdot v \mu, u \cdot v \mu ) +   \frac{\theta^2}{E^2(K-1)^2} \QQ\left( \left( |v|^2-E \right) \mu , \left( |v|^2-E \right) \mu \right) \\
				& + 2 \varrho u \cdot  \QQ(\mu, v \mu) +    \frac{2\varrho \theta}{E(K-1)} \QQ\left( \mu, (|v|^2 - E) \mu \right) 
				+  \frac{2\theta u}{E(K-1)} \cdot  \QQ\left( v \mu, (|v|^2 - E) \mu \right).
			\end{align*}
			On the one hand, Lemma \ref{lem:Q_A} yields
			$$\left\la \QQ( f , f ) , \LL^{-1}\BurA \right\ra_{\Ss} = \frac{\vartheta_1}{2} \left( u \otimes u - \frac{2}{d} |u|^2 \Id \right),$$
			and, since for any $g : \R^d \rightarrow \R$, there holds $\nabla_x \cdot (g \Id) = \nabla_x g$ and thus $\mathbb{P} \left( \nabla_x \cdot (g \Id) \right) = 0$, we deduce from Proposition \ref{prop:macro_representation_spectral} that
			\begin{gather*}
				u\left[ \nabla_x \cdot \PP_\Inc^{(1)} \QQ( f , f ) \right]  = \left( \frac{d}{E} \right)^{\frac{3}{2}} \mathbb{P} \left\{ \nabla_x \cdot \left\la \QQ( f , f ) , \LL^{-1}\BurA \right\ra_{\Ss} \right\} 
				= \left( \frac{d}{E} \right)^{\frac{3}{2}} \frac{\vartheta_1}{2} \mathbb{P} \left\{ \nabla_x \cdot \left( u \otimes u \right) \right\}.
			\end{gather*}
			On the other hand, Lemma \ref{lem:Q_A} and \eqref{eq:inc_bou_conditions} yield
			\begin{equation*}
				\la \QQ(f, f) , \LL^{-1}\BurB \ra_{\Ss} = 2 \vartheta_2 u \varrho + \frac{2 \vartheta_3}{E (K-1)} u \theta
				= \left( 2 \vartheta_2 + \frac{2 \vartheta_3}{E (K-1)} \right) u \theta,
			\end{equation*}
			and thus, according to Proposition \ref{prop:macro_representation_spectral},
			\begin{align*}
				\theta\left[ \nabla_x \cdot \PP^{(1)}_{\Bou} \QQ(f, f) \right] = \frac{1}{K \sqrt{K(K-1)}} \left( 2 \vartheta_2 + \frac{2 \vartheta_3}{E (K-1)} \right) \nabla_x \cdot (u \theta).
			\end{align*}
			\step{2}{The integral formulation in macroscopic variables}
			We are only left with checking that the $u$-part of the kinetic integral system satisfies the Navier-Stokes equations, and that the $\theta$-part satisfies the Fourier equation. 
			
			Indeed, by Proposition \ref{prop:macro_representation_spectral} and the previous step, there holds
			\begin{align*}
				u(t) = u\left[ f(t) \right] & = u\left[ U_\ns(t) f_\ini \right] + u\left[ \int_0^t V_\ns(t-\tau) \QQ(f(\tau), f(\tau) ) \right] \\
				& = e^{t \kappa_\Inc \Delta_x} u_\ini - \vartheta_\Inc \mathbb{P} \int_0^t e^{(t-\tau) \kappa_\Inc \Delta_x} \, \nabla_x \cdot \left( u \otimes u \right) (\tau) \, \d \tau,
			\end{align*}
			as well as
			\begin{align*}
				\theta(t) = \theta\left[ f(t) \right] & = \theta\left[ U_\ns(t) f_\ini \right] + \theta\left[ \int_0^t V_\ns(t-\tau) \QQ(f(\tau), f(\tau) ) \right] \\
				& = e^{t \kappa_\Inc \Delta_x} \theta_\ini - \vartheta_\Bou \int_0^t e^{(t-\tau) \kappa_\Bou \Delta_x } \nabla_x \cdot  (u \theta)(\tau) \d \tau.
			\end{align*}
			This concludes the proof.
		\end{proof} 
		The next two lemmas provide estimates related to the kinetic version of the Navier-Stokes-Fourier solution, and the first one is essentially a quantitative version of those proved in \cite{GT2020}.
		\color{black}
		
		\begin{lem}[\textit{\textbf{Estimates for Navier-Stokes squared}}]
			\label{lem:estimates_derivative_Q_navier_stokes}
			Suppose $s > \frac{d}{2}$ and denote the bilinear term $\varphi = \QQ(f, f)$ where~$f$ is defined as
			$$
			f(t, x, v) = \varrho(t, x) \mu(v) + u(t, x) \cdot v \mu(v) + \frac{1}{E(K-1)} \theta(t, x) \left( |v|^2 - E \right) \mu(v),$$
			and the coefficients $(\varrho, u, \theta)$ are a solution to the Navier-Stokes-Fourier equations given by Theorem \ref{thm:cauchy_NSF}. Then $\varphi$ satisfies
			\begin{equation*}\begin{split}
					\| \varphi & \|_{ L^\infty \left( [0, T) ; \mathbb{H}^s_x \left( \Ssm_v \right) \right) } + \| | \nabla_x |^{1-\alpha} \varphi \|_{ L^2 \left( [0, T) ; \mathbb{H}^s_x \left( \Ssm_v \right) \right) } \\
					& \lesssim \Big(1 + \| (\varrho_\ini, u_\ini, \theta_\ini) \|_{ \dot{\mathbb{H}}^{-\alpha} } + \| (\varrho, u, \theta) \|_{ L^\infty\left( [0, T) ; \mathbb{H}^{s}_x \right) } + \| \nabla_x (\varrho, u, \theta) \|_{ L^2\left( [0, T) ; \mathbb{H}^{s}_x \right) }\Big)^2,
			\end{split}\end{equation*}
			and for any $p \in (1, 2]$, where $p = 1$ is allowed for $d \ge 3$, its derivative $\partial_{t} \varphi$ satisfies
			\begin{equation*}\begin{split}
				\| \partial_t \varphi \|_{ L^p\left( [0, T) ; \mathbb{H}^{s-1}_x(\Ssm_v) \right) } & + \| \partial_t \varphi \|_{ L^p\left( [0, T) ; \dot{\mathbb{H}}^{-\frac{1}{2}}_x (\Ssm_v)\right) } \\
				& \lesssim \Big(1 + \| (\varrho, u, \theta) \|_{ L^\infty\left( [0, T) ; \mathbb{H}^{s}_x \right) } + \| \nabla_x (\varrho, u, \theta) \|_{ L^2\left( [0, T) ; \mathbb{H}^{s}_x \right) }\Big)^3.
			\end{split}\end{equation*}
		\end{lem}
		
		\begin{proof}
			The function $\varphi(t, x) \in \Ssm$ writes for some $\varphi_j \in \Ssm$ as (note that $\varrho = - \theta$)
			$$\varphi(t, x) = u(t, x) \otimes u(t, x) : \varphi_1 + u(t, x) \varrho(t, x) \cdot \varphi_2 + \varrho^2(t, x) \varphi_3,$$
			and its derivative writes
			\begin{align*}
				\partial_t \varphi = & 2 \left(\partial_t u\right) \otimes u : \varphi_1
				+ \big[\left(\partial_t u\right) \varrho + u \left( \partial_t \varrho \right)\big] \cdot \varphi_2
				+ 2 \varrho \left(\partial_t \varrho\right) \varphi_3
			\end{align*}
			We only prove that the term $\varrho \left(\partial_t u\right)$ satisfies the estimates of the lemma, the other ones being treated the same way. Since $(\varrho, u, \theta)$ is a solution of the Navier-Stokes-Fourier system, using the formulation \eqref{eq:NSF_1}, the term writes (omitting constants)
			$$\varrho \left(\partial_t u\right) = \varrho \Big( \mathbb{P} \left(u \cdot \nabla_x u\right) \Big) + \varrho \Delta_x u.$$
			We will require the product rules recalled in Appendix \ref{scn:littlewood-paley}:
			\begin{gather}
				\label{eq:homogeneous_product}
				\forall s_1, s_2 \in \left(-\frac{d}{2}, \frac{d}{2}\right), \quad s_1 + s_2 > 0, \quad \| g h \|_{ \dot{\mathbb{H}}^{s_1 + s_2 - \frac{d}{2}} } \lesssim \| g \|_{\dot{\mathbb{H}}^{s_1} } \| h \|_{\dot{\mathbb{H}}^{s_2} },\\
				\label{eq:inhomogeneous_product}
				\forall s_1 > 0, ~\forall s_2 > \frac{d}{2}, \quad \| g h \|_{ H^s_1 } \lesssim \| g \|_{\dot{\mathbb{H}}^{s_1} } \| h \|_{\dot{\mathbb{H}}^{s_2} } + \| g \|_{\dot{\mathbb{H}}^{s_2} } \| h \|_{\dot{\mathbb{H}}^{s_1} },
			\end{gather}
			and this last one which can be proved as \eqref{eq:Q_refined_sobolev_negative_algebra_inequality}:
			\begin{equation}
				\label{eq:inhomogeneous_product_2}
				\forall s > \frac{d}{2}, ~\forall r \in \left[0, \frac{d}{2}\right), \quad \| g h \|_{ H^s } \lesssim \| | \nabla_x |^r g \|_{ H^{s-r}_x } \| h \|_{H^{s} } .
			\end{equation}
			\step{1}{The estimates for $\varrho \Big( \mathbb{P} \left(u \cdot \nabla_x u\right) \Big)$}
			Using the algebra structure of $H^s$ and $\mathbb{P} \in \BBB\left( H_x^s \right)$:
			$$\left\| \varrho \Big( \mathbb{P} \left(u \cdot \nabla_x u\right) \Big) \right\|_{H^{s-1}_x} \le \left\| \varrho \Big( \mathbb{P} \left(u \cdot \nabla_x u\right) \Big) \right\|_{H^{s}_x} \lesssim \| \varrho \|_{\mathbb{H}^s_x} \| u \|_{\mathbb{H}^s_x} \| \nabla_x u \|_{\mathbb{H}^s_x} \in L^2_t.$$
			Applying the product rule \eqref{eq:homogeneous_product} a first time with the parameters $(s_1, s_2)= \left( \frac{1}{2}, \frac{d-1}{2} \right)$ and using the boundedness $\mathbb{P} \in \BBB\left( \dot{\mathbb{H}}^{(d-1)/2} \right)$
			\begin{align*}
				\left\| \varrho \Big( \mathbb{P} \left(u \cdot \nabla_x u\right) \Big) \right\|_{ \dot{\mathbb{H}}^{-\frac{1}{2}} } & \lesssim \| \varrho \|_{L^2_x} \left\| \mathbb{P} \left(u \cdot \nabla_x u\right) \right\|_{ \dot{\mathbb{H}}^{(d-1)/2}_x }  \lesssim \| \varrho \|_{L^2_x} \left\| u \cdot \nabla_x u \right\|_{ \dot{\mathbb{H}}^{(d-1)/2}_x },
			\end{align*}
			and then using the product rule \eqref{eq:homogeneous_product} a second time with the parameters $(s_1, s_2) = \left( \nu , d - \frac{1}{2} - \nu \right)$ for some $\nu \in \left( \frac{d-1}{2} , \frac{d}{2} \right)$ so that $s_2 \in \left( \frac{d-1}{2} , \frac{d}{2} \right)$
			$$
			\left\| \varrho \Big( \mathbb{P} \left(u \cdot \nabla_x u\right) \Big) \right\|_{ \dot{\mathbb{H}}^{-\frac{1}{2}} } \lesssim \| \varrho \|_{L^2_x} \| u \|_{ \dot{\mathbb{H}}^\nu_x } \left\| \nabla_x u \right\|_{ \dot{\mathbb{H}}^{ d - \nu - \frac{1}{2} }_x }.$$
			When $d \ge 3$, we have $\nu \in \left( 1, s \right)$ and $d - \nu - \frac{1}{2} < s$, and thus
			$$\| \varrho u \cdot \nabla_x u \|_{ \dot{\mathbb{H}}^{-\frac{1}{2}} } \lesssim  \| \varrho \|_{ \mathbb{H}^s_x } \| \nabla_{x} u \|_{ H^{s-1}_x } \| \nabla_{x} u \|_{ H^{s}_x } \in L^1_t \cap L^2_t,$$
			and when $d = 2$, since $\nu \in (\frac{1}{2}, 1)$, we have by interpolation
			$$\| \varrho u \cdot \nabla_x u \|_{ \dot{\mathbb{H}}^{-\frac{1}{2}} } \lesssim \| \varrho \|_{ \mathbb{H}^s_x } \| u \|_{\mathbb{H}^s_x}^{1-\nu} \| \nabla_{x} u \|_{\mathbb{H}^s_x}^{1+\nu} \in L^2_t \cap L^{\frac{2}{1+\nu}}_t,$$
			thus, taking $\nu$ arbitrarily small to $1$ yields the result.
			
			\step{2}{The estimates for $\varrho \Delta_x u$}
			We rewrite this term as
			$$\varrho \Delta_x u = \nabla_x \cdot \left( \varrho \nabla_x u \right) - \nabla_x \varrho \cdot \nabla_x u,$$
			from which we deduce
			$$\| \varrho \Delta_x u \|_{H^{s-1} } \lesssim \| \varrho \nabla_x u \|_{\mathbb{H}^s_x} + \| \nabla_x \varrho \cdot \nabla_x u \|_{H^{s-1} }.$$
			Using for the first term the product rule \eqref{eq:inhomogeneous_product_2} with $\nu \in (0, 1)$ when $d = 2$ or $\nu = 1$ when $d \ge 3$, and \eqref{eq:inhomogeneous_product} for the second term, we have
			\begin{equation*}\begin{split}
				\| \varrho \Delta_x u \|_{H^{s-1} } \lesssim &
				\| | \nabla_x |^\nu \varrho \|_{H^{s-\nu}_x} \| \nabla_x u \|_{\mathbb{H}^s_x} 
				 + \| \nabla_x \varrho \|_{H^{s-1}_x} \| \nabla_x u \|_{\mathbb{H}^s_x} \\
				& + \| \nabla_x \varrho \|_{\mathbb{H}^s_x} \| \nabla_x u \|_{H^{s-1}_x} \in L^2_t \cap L^{ \frac{2}{1 + \nu } }.
			\end{split}\end{equation*}
			Furthermore, using the product rule \eqref{eq:homogeneous_product} with the parameters $(s_1, s_2) = \left( \nu, \frac{ d + 1 }{2} - \nu\right)$ for some~$\nu \in \left( \frac{1}{2} , \frac{d}{2} \right)$, and with $(s_1, s_2) = \left( \frac{d-1}{2}, 0 \right)$, we have
			\begin{equation*}\begin{split}
				\| \varrho \Delta_x u \|_{ \dot{\mathbb{H}}^{-\frac{1}{2}}_x } & \lesssim \| \varrho \nabla_x u \|_{ \dot{\mathbb{H}}^{\frac{1}{2}}_x } + \| \nabla_x \varrho \cdot \nabla_x u \|_{ \dot{\mathbb{H}}^{-\frac{1}{2}}_x } \\
				& \lesssim \| \varrho \|_{ \dot{\mathbb{H}}^{\nu}_x } \| \nabla_x u \|_{ \dot{\mathbb{H}}^{(d+1)/2 - \nu}_x } + \| \nabla_x \varrho \|_{ \dot{\mathbb{H}}^{(d-1)/2 }_x } \| \nabla_x u \|_{ L^2_x } \\
				& \lesssim \| \nabla_x u \|_{\mathbb{H}^s_x} \left( \| \varrho \|_{ \dot{\mathbb{H}}^\nu_x } + \| \nabla_x \varrho \|_{ H^{s-1}_x } \right).
					\end{split}\end{equation*}
			In the case $d \ge 3$, we deduce taking $\nu = 1$
			$$\| \varrho \Delta_x u \|_{ \dot{\mathbb{H}}^{-\frac{1}{2}}_x } \lesssim \| \nabla_x u \|_{\mathbb{H}^s_x} \| \nabla_x \varrho \|_{ H^{s-1}_x } \in L^1_t \cap L^2_t,$$
			and when $d = 2$, by interpolation,
			$$\| \varrho \Delta_x u \|_{ \dot{\mathbb{H}}^{-\frac{1}{2}}_x } \lesssim \| \nabla_x u \|_{\mathbb{H}^s_x} \| \nabla_x \varrho \|_{ H^{s-1} }^\nu \| \varrho \|_{ \mathbb{H}^s_x }^{1-\nu} \in L^{\frac{2}{1+\nu}}_t \cap L^2_t,$$
			from which we conclude the the result by taking $\nu$ arbitrarily close to $1$. This concludes the proof of the estimates for $\partial_t \varphi$.
			
			For the estimates of $\varphi$, one proves similarly
			$$\| \varrho u \|_{ \mathbb{H}^s_x } \lesssim \| \varrho \|_{ \mathbb{H}^s_x } \| u \|_{ \mathbb{H}^s_x }$$
			and
			$$\| | \nabla_x |^{1 - \alpha} (\varrho u) \|_{ \mathbb{H}^s_x } \lesssim \left( \| | \nabla_x |^{1 - \alpha} \varrho \|_{ \mathbb{H}^s_x } + \| \nabla_x \varrho \|_{ \mathbb{H}^s_x } \right) \| u \|_{ \mathbb{H}^s_x } + \| \varrho \|_{ \mathbb{H}^s_x } \| \nabla_x u \|_{ \mathbb{H}^s_x }$$
			which allows to conclude using Lemma \ref{lem:NS_parabolic_space}. This concludes the proof.
		\end{proof}
		
		\begin{lem}[\thttl{The Navier-Stokes-Fourier solution and the space $\SSSs$}]
			\label{lem:NS_parabolic_space}
			The Navier-Stokes solution in its kinetic form belongs to the space $\rSSSs{s}$ (where the parameter $\alpha$ defines this space):
			\begin{align*}
				\Nt f \Nt_{ \SSSs } \lesssim \| (\varrho_\ini, u_\ini, \theta_\ini) \|_{ \dot{\mathbb{H}}^{-\alpha}_x   } + \| (\varrho, u, \theta) \|_{ L^\infty \left( [0, T) ; \mathbb{H}^s_x \right) } + \| \nabla_x (\varrho, u, \theta) \|_{ L^2 \left( [0, T) ; \mathbb{H}^s_x \right) }\,.
			\end{align*}
			Furthermore, it can be approximated by a smoother sequence; there exists $\left(f_\eps\right)_{ \eps \in (0, 1] } \in \rSSSs{s+1}$ such that
			$$\lim_{\eps \to 0} \Nt f_\eps - f \Nt_{\rSSSs{s}}=0.$$
		\end{lem}
		
		\begin{proof}
			\step{1}{Bound in $\SSSs$}
			We only need to consider the case $d = 2$ and only prove the estimate for $u$. We start by applying Duhamel's principle to \eqref{eq:NSF_2}:
			$$u(t) = e^{t \kappa_\Inc \Delta_x} u_\ini - \int_0^t e^{(t-\tau) \Delta_x} \varphi(\tau) \d \tau, \qquad \varphi := \mathbb{P} \big[ \nabla_x \cdot \left( u \otimes u \right) \big],$$
			from which we obtain
			\begin{align*}
				\int_0^T \| |\nabla_x|^{ 1  - \alpha } u(t) \|^2_{ L^2_x } \d t \lesssim & \int_0^T \left\| |\nabla_x|^{ 1  - \alpha } e^{ t \kappa_\Inc \Delta_x} u_\ini \right\|_{L^2_x}^2 \d t \\
				& + \int_0^T \left\| |\nabla_x|^{ 1  - \alpha } \int_0^t e^{(t-\tau) \Delta_x} \varphi(\tau) \d \tau  \right\|_{L^2_x}^2 \d t,
			\end{align*}
			or, equivalently, in Fourier variables:
				\begin{equation*}\begin{split}
				\int_0^T \| |\nabla_x|^{ 1  - \alpha } u(t) \|^2_{ L^2_x } \d t \lesssim & \int_{\R^{d}} | \xi |^{-2 \alpha} | \widehat{u}_\ini(\xi) |^2 \int_0^T | \xi |^2 e^{ - 2 \kappa_\Inc  t |\xi|^2} \d t \, \d \xi \\
				& + \int_{\R^{d}} | \xi|^{ - 2 \alpha } \int_0^T \left( \int_0^t | \xi | e^{(t-\tau) |\xi|^2} \widehat{\varphi}(\tau, \xi) \d \tau  \right)^2 \d t \, \d \xi.
			\end{split}\end{equation*}
			Using Young's convolution inequality in the form $L^2\left( [0, T] \right) \ast L^1\left( [0, T] \right) \hookrightarrow  L^2\left( [0, T] \right)$ followed by Minkowski's integral inequality for the second term, we thus get
			\begin{equation*}\begin{split}
					\int_0^T \| |\nabla_x|^{ 1  - \alpha } u(t) \|^2_{ L^2_x } \d t &\lesssim  \int_{\R^{d}} | \xi |^{-2 \alpha} | \widehat{u}_\ini(\xi) |^2 \d \xi 
					+ \int_{\R^{d}} | \xi|^{ - 2 \alpha } \left(\int_0^T \left| \widehat{\varphi}(t, \xi) \right| \d t\right)^2 \, \d \xi \\
					&\lesssim  \| u_\ini \|^2_{ \dot{\mathbb{H}}^{-\alpha}_x } + \left(\int_0^T \| \varphi(t) \|_{ \dot{\mathbb{H}}^{-\alpha}_x } \d t\right)^2.
			\end{split}\end{equation*}
			Since $\mathbb{P} \in \BBB\left( \dot{\mathbb{H}}_x^{-\alpha} \right)$, we get using the product rule \eqref{eq:homogeneous_product} (using that $\frac{d}{2} = 1$) and then by interpolation
			$$
				\| \varphi \|_{ \dot{\mathbb{H}}_x^{-\alpha} }  \lesssim \| u \otimes u \|_{\dot{\mathbb{H}}_x^{1-\alpha}}   \lesssim \| u \|_{\dot{\mathbb{H}}_x^{1-\frac{\alpha}{2}}}^2   \lesssim \| \nabla_x u \|_{L^2_x} \| | \nabla_x |^{1-\alpha} u \|_{L^2_x}$$			
				from which we conclude using Cauchy-Schwarz
			$$
				\int_0^T \| |\nabla_x|^{ 1  - \alpha } u(t) \|^2_{ L^2_x } \d t \lesssim  \| u_\ini \|^2_{ \dot{\mathbb{H}}^{-\alpha}_x } + \left(\int_0^T \| |\nabla_x|^{ 1  - \alpha } u(t) \|^2_{ L^2_x } \d t\right)^{\frac{1}{2}} \left(\int_0^T \| \nabla_x u(t) \|^2_{ L^2_x } \d t\right)^{\frac{1}{2}},
			$$
			and thus, by Young's inequality
			$$	\int_0^T \| |\nabla_x|^{ 1  - \alpha } u(t) \|^2_{ L^2_x } \d t \lesssim  \| u_\ini \|^2_{ \dot{\mathbb{H}}^{-\alpha}_x } +\int_0^T \| \nabla_x u(t) \|^2_{ L^2_x } \d t.
			$$
			This concludes this step.
			
			\step{2}{Approximation by functions in $\rSSSs{s+1}$}
			Since the solution is instantly regularized in the sense that 
			$$(\nabla_x \varrho, \nabla_x u, \nabla_x \theta) \in L^2 \left( [0, T) ; H^s_x \right),$$
			and in virtue of the control from Theorem \ref{thm:cauchy_NSF}, there holds for any $\delta \in (0, T)$
				\begin{equation*}\begin{split}
				\| (\varrho, u, \theta) \|_{ L^\infty \left( [\delta, T) ; H^{s+1} \right) } + \| \nabla_x & (\varrho, u, \theta) \|_{ L^2 \left( [\delta, T) ; H^{s+1} \right) } \\
				& \le C \left\| \left(\varrho(\delta), u(\delta), \theta(\delta) \right) \right\|_{\mathbb{H}^{s+1}_x } \exp\left( C \| \nabla_x u \|_{ L^2 \left( [0, T) ; H^{\frac{d}{2}-1}_x \right) } \right),
			\end{split}\end{equation*}
			and thus from \textit{Step 1}, we have that
			$f( \delta + \cdot ) \in \rSSSs{s+1}$. From this observation, considering (by the continuity of $f$ and the density of $H^{s+1}_x$) for any $\eps > 0$ some $\delta_\eps > 0$ and $f_{\ini, \eps} \in H^{s+1}_x \left( \Ss_v \right)$ such that
			$$\sup_{0 \le t \le \delta_\eps} \| f(t) - f_{\ini, \eps} \|_{ H^s_x \left( \Ss_v \right) } \le \eps, \qquad \int_0^{\delta_\eps} \| | \nabla_x |^{1-\alpha} f(t) \|_{ H^s_x \left( \Ss_v \right) }^2 \d t \le \eps,$$
			and up to a reduction of $\delta_\eps$
			$$\| | \nabla_x |^{1 - \alpha} f_{\ini, \eps} \|_{ H^s_x \left( \Ss_v \right) } \le \frac{\eps}{\delta_\eps}.$$
			We can therefore define
			$$f_\eps(t) =
			\begin{cases}
				f_{\ini, \eps}, & t \in [0, \delta_\eps), \\
				f(t), & t \in [\delta_\eps, T),
			\end{cases}
			$$
			so as to have $f_\eps \in \rSSSs{s+1}$ and
			$$\Nt f_\eps - f \Nt_{\rSSSs{s}}^2 \lesssim \sup_{0 \le t \le \delta_\eps} \| f(t) - f_{\ini, \eps} \|_{ H^s_x \left( \Ss_v \right) }^2  + \int_0^{\delta_\eps} \left\| | \nabla_x |^{1-\alpha} f(t) - | \nabla_x |^{1-\alpha} f_{\ini, \eps} \right\|_{ H^s_x \left( \Ss_v \right) }^2 \d t \lesssim \eps^2,$$			which concludes the proof.
		\end{proof}

		\bibliographystyle{plainnat-linked}

		%	\bibliography{biblio}

\begin{thebibliography}{99}
			
			\providecommand{\natexlab}[1]{#1}
			\providecommand{\url}[1]{\texttt{#1}}
			\expandafter\ifx\csname urlstyle\endcsname\relax
			\providecommand{\doi}[1]{doi: #1}\else
			\providecommand{\doi}{doi: \begingroup \urlstyle{rm}\Url}\fi
			
			
			\bibitem[Alexandre et al. (2019)]{AHL2019}
			{\sc R. Alexandre, F. H{\'e}rau,  W.-X. Li,}
			\newblock Global hypoelliptic and symbolic estimates for the linearized
			{Boltzmann} operator without angular cutoff,
			\newblock {\em J. Math. Pures Appl.},  {\bf 126\/} (2019), 1--71.
			
			\bibitem[Alexandre et al. (2011)]{AMUXY2011}
			{\sc R. Alexandre, Y. Morimoto, S. Ukai, C-J. Xu,  T. Yang,}
			\newblock Global existence and full regularity of the {Boltzmann} equation
			without angular cutoff,
			\newblock {\em Commun. Math. Phys.} {\bf 304} (2011), 513--581.
			
			\bibitem[Albritton et al (2022)]{albri}
			{\sc D. Albritton, E. Bru\'e, M. Colombo, }
			\newblock Non-uniqueness of Leray solutions of the forced Navier-Stokes equations,
\newblock \textit{Ann. of Math.} {\bf 196} (2022),  415--455.
			
			
			\bibitem[Alonso et~al. (2010)]{ABL}
			\textsc{R. Alonso, V. Bagland, B. Lods,}
			\newblock {Long time dynamics for the Landau-Fermi-Dirac equation with hard potentials, }
			\newblock \textit{J. Differential Equations} {\bf 270} (2021),  596--663.
			
			%\bibitem[Alonso et~al. (2010)]{ACG}
			%\textsc{R. Alonso, E. Carneiro, I. M. Gamba,}
			%\newblock {Convolution inequalities for the Boltzmann collision operator, }
			%\newblock \textit{Comm. Math. Phys.} {\bf 298} (2010),  293--322.
			
			\bibitem[Alonso et~al. (2021)]{AMSY2021}
			\textsc{R. Alonso, Y. Morimoto, W. Sun, T. Yang}
			\newblock {Non-cutoff Boltzmann equation with polynomial decay perturbations, }
			\newblock \textit{Rev. Mat. Iberoam.} {\bf 37} (2021),  189--292.
			
			\bibitem[Arkeryd et~al. (2015)]{AN2015}
			\textsc{L. Arkeryd, A. Nouri}
			\newblock {Well posedness of the Cauchy problem for a space dependent anyon Boltzmann equation, }
			\newblock \textit{SIAM J. Math. Anal.} {\bf 47} (2015),  4720--4742.
			
			
			\bibitem[Arsenio and Saint-Raymond (2019)]{arsenio}
			\textsc{D. Arsenio, L. Saint-Raymond,}
			\newblock \textbf{From the Vlasov-Maxwell-Boltzmann system to incompressible viscous electro-magneto-hydrodynamics. Vol. 1.}
			\newblock EMS Monographs in Mathematics. European Mathematical Society (EMS), Z\"urich, 2019.
			
			\bibitem[Bahouri et al. (2011)]{BCD2011}
			{\sc H. Bahouri, J-Y. Chemin and R. Danchin}
			\newblock {\bf Fourier analysis and nonlinear partial differential equations},
			vol.~343 of {\em Grundlehren Math. Wiss.}
			\newblock Berlin: Heidelberg, 2011.
			
			\bibitem[Baranger \& Mouhot (2005)]{BM2005}
			{\sc C. Baranger, and C. Mouhot,}
			\newblock Explicit spectral gap estimates for the linearized {Boltzmann} and
			{Landau} operators with hard potentials.
			\newblock {\em Rev. Mat. Iberoam.}, {\bf 21} (2005), 819--841.
			
			
			\bibitem[Bardos et~al. (1991)]{BaGoLe1}
			\textsc{C. Bardos, F. Golse,  D. Levermore,} 
			\newblock{Fluid dynamic limits of kinetic equations. I. Formal derivations,}
			\newblock \textit{J. Stat. Phys.}, {\bf 63} (1991),  323--344.
			
			\bibitem[Bardos et~al. (1993)]{BaGoLe2}
			\textsc{C. Bardos, F. Golse,  D. Levermore,} 
			\newblock {Fluid dynamic limits of kinetic equations. II. Convergence proofs for the Boltzmann equation, }
			\newblock \textit{Comm. Pure Appl. Math.}, {\bf 46} (1993), 667--753.
			
			\bibitem[Bardos et~al. (1984)]{BSS}
			\textsc{C. Bardos, R. Santos,  R. Sentis,}
\newblock Diffusion approximation and computation of the critical size,
\newblock \textit{Trans. Amer. Math. Soc.} {\bf 284} (1984), 617--649.


			\bibitem[Bardos \& Ukai (1991)]{BU1991}
			\textsc{C. Bardos, S. Ukai,}
			\newblock { The classical incompressible Navier-Stokes limit of the Boltzmann equation,}  
			\newblock \textit{Math. Models Methods Appl. Sci.}, {\bf 1} (1991), 235--257.
			
			\bibitem[Bensoussan et al. (1979)]{BLP1979}
			\textsc{A. Bensoussan, J. L. Lions and G. C. Papanicolaou,}
			\newblock Boundary layers and homogenization of transport processes,
			\newblock \textit{Publ. Res. Inst. Math. Sci.}, {\bf 15} (1979), 53--157.
			
			%\bibitem[Bernou et~al. (2021)]{BCMT}
			%\textsc{A. Bernou, K. Carrapatoso, S. Mischler, I. Tristani,}
			%\newblock {Hypocoercivity for kinetic linear equations in bounded domains with general Maxwell boundary condition,}
			%\newblock \textit{Ann. Institut Henri Poincar\'e - Analyse Non Lin\'eaire,} in press.
			
			\bibitem[Bouin \& Mouhot (2022)]{BM}
			\textsc{E. Bouin, C. Mouhot},
			\newblock Quantitative fluid approximation in transport theory: a unified approach,
\newblock \textit{Probab. Math. Phys.}, {\bf 3} (2022), 491--542.
			
			\bibitem[Briant (2015)]{B2015}
			\textsc{M. Briant}, 
			\newblock From the Boltzmann equation to the incompressible Navier-Stokes equations on the torus: a quantitative error estimate, 
			\newblock \textit{J. Differential Equations}, {\bf 259} (2015) 6072--6141. 
			
			\bibitem[Briant et~al.(2019)]{BMM2019}
			\textsc{M. Briant, S. Merino-Aceituno, C. Mouhot,} 
			\newblock From Boltzmann to incompressible Navier-Stokes in Sobolev spaces with polynomial weight,
			\newblock \textit{Anal. Appl. (Singap.)}, {\bf 17} (2019), 85--116. 
			
			
			\bibitem[Caflisch (1980)]{caflisch}
			\textsc{R. E. Caflisch,} 
			\newblock The fluid dynamic limit of the nonlinear Boltzmann equation, 
			\newblock \textit{Comm. Pure Appl. Math.}, {\bf 33} (1980), 651--666.
			
			
			\bibitem[Cao \& Carrapatoso (2023)]{CC2023}
			\textsc{C. Cao, K. Carrapatoso,}
			\newblock {Hydrodynamic limit for the non-cutoff Boltzmann equation,}
			\newblock \textit{preprint} (2023), \texttt{https://arxiv.org/abs/2304.06362}.
			
			\bibitem[Cao et~al.(2022)]{CDL2022}
			\textsc{C. Cao, D. Deng, X. Li,}
			\newblock {The Vlasov-Poisson-Boltzmann/Landau systems with polynomial perturbation near Maxwellian}
			\newblock \textit{preprint} (2022), \texttt{https://arxiv.org/abs/2111.05569}.
			
			
			\bibitem[Carrapatoso et~al.(2022a)]{CRT2022}
			\textsc{K. Carrapatoso, M. Rachid, I. Tristani,}
			\newblock {Regularization estimates and hydrodynamical limit for the Landau equation}
			\newblock \textit{J. Math. Pures Appl.}, {\bf 163} (2022), 334--432.  
			
			%\bibitem[Carrapatoso et al. (2022b)]{CDHMMS2022}
			%{\sc K. Carrapatoso, J. Dolbeault, F. H\'erau, S. Mischler, C. Mouhot,  and
			%	C.  Schmeiser,}
			%\newblock Special macroscopic modes and hypocoercivity, 
			%\newblock \textit{preprint} (2022), \texttt{https://arxiv.org/abs/2105.04855.}
			
			\bibitem[Carrapatoso et al.(2016)]{CTW2016}
			{\sc K. Carrapatoso, I. Tristani, and K.-C. Wu,}
			\newblock Cauchy problem and exponential stability for the inhomogeneous
			{Landau} equation.
			\newblock {\em Arch. Ration. Mech. Anal.} {\bf 221} (2016), 363--418.
			
			\bibitem[Carrapatoso \& Gervais (2023)]{CG2022}
			{\sc K. Carrapatoso, P. Gervais,}
			\newblock Non-cutoff Boltzmann equation with soft potentials in the whole space, 
			\newblock  \textit{Pure and Applied Analysis,} {\bf 6} (2024), 253--303.
			
			\bibitem[Carrapatoso \& Mischler (2017)]{CM2017}
			{\sc K. Carrapatoso, S. Mischler,}
			\newblock Landau equation for very soft and Coulomb potentials near Maxwellians,			
			\newblock \textit{Ann. PDE}, {\bf 3} (2017), Paper No. 1, 65 pp.
			
			\bibitem[Cercignani (1988)]{Cercignani}
			\textsc{C. Cercignani},
			\newblock \textbf{The Boltzmann equation and its applications},
			\newblock \textit{Applied Mathematical Sciences,} {\bf 67,} Springer-Verlag, New York, 1988.
			
			\bibitem[Crevat et~al. (2019)]{crevat}
			\textsc{J. Crevat, G. Faye, F. Filbet,}
			\newblock Rigorous derivation of the nonlocal reaction-diffusion Fitzhugh-Nagumo system,
			\newblock \textit{SIAM J. Math. Anal.}, {\bf 51} (2019), 346--373.
			 
			\bibitem[Dechicha \& Puel (2023)]{Puel}
			\textsc{D. Dechicha, M. Puel},	
			\newblock Fractional diffusion for Fokker-Planck equation with heavy tail equilibrium: an \`a la Koch spectral method in any dimension,
			\newblock \textit{preprint} (2023), \texttt{https://arXiv:2303.07162}.

			\bibitem[De Masi et~al. (1989)]{demasi}
			\textsc{A. De Masi, R. Esposito, J. L. Lebowitz,} 
			\newblock Incompressible Navier-Stokes and Euler limits of the Boltzmann equation, 
			\newblock \textit{Comm. Pure Appl. Math.}, {\bf 42} (1989), 1189--1214.
			
			\bibitem[De Lellis, \& Sz\'ekelyhidi (2012)]{delellis}
			\textsc{C. De Lellis, L. Jr Sz\'ekelyhidi, }
			\newblock The h-principle and the equations of fluid dynamics,
			\newblock \textit{Bull. Amer. Math. Soc. (N.S.)} {\bf 49} (2012),  347--375.
			
			%\bibitem[Di Perna \& Lions (1990)]{diperna}
			%\textsc{R. J. DiPerna, P.-L. Lions,} 
			%\newblock On the Cauchy problem for the Boltzmann equation: global existence and weak stability results, 
			%\newblock \textit{Ann. Math.}, {\bf 130}  (1990), 321--366. 
			
						 
			\bibitem[Dolbeault (1994)]{D1994}
			{\sc J. Dolbeault,}
			\newblock  Kinetic models and quantum effects: A modified Boltzmann equation for Fermi-Dirac particles,
			\newblock {\em  Arch. Ration. Mech. Anal. }, {\bf 127} (1994), 101--131.
	
			\bibitem[Duan (2011)]{D2011}
			{\sc R. Duan,}
			\newblock Hypocoercivity of linear degenerately dissipative kinetic equations,
			\newblock {\em Nonlinearity}, {\bf 24} (2011), 2165--2189.
			
			\bibitem[Duan \& Li (2012)]{DL2012}
			{\sc Duan, R., and Li, W.-X.}
			\newblock Hypocoercivity for the linear {Boltzmann} equation with confining
			forces,
			\newblock {\em J. Stat. Phys.},  {\bf 148} (2012), 306--324.
			
			\bibitem[Engel \& Nagel (1999)]{engel}
			\textsc{K. J. Engel, R. Nagel}, 
			\newblock \textbf{One--parameter semigroups for linear evolution equations}, 
			\newblock Springer, 1999.
			
			\bibitem[Enskog (1917)]{E1917}
			\textsc{D. Enskog},
			\newblock Kinetische Theorie des Vorg\"ange in m\"assig verd\"unnten Gasen, Uppsala, Almqvist \& Wiksell, 1917.
			\newblock translated in \textit{Kinetic Theory,} S.G. Brush, Ed., Pergamon Press, Oxford, 1972, 125--225.
			
			\bibitem[Figalli \& Kang (2019)]{figalli}
			\textsc{A. Figalli, M-J. Kang,}
			\newblock A rigorous derivation from the kinetic Cucker-Smale model to the pressureless Euler system with nonlocal alignment, 
			\newblock \textit{Anal. PDE}, {\bf 12} (2019), 843--866.
			
			
			\bibitem[Ellis \& Pinsky (1975)]{EP1975}
			\textsc{R. S. Ellis, M. A. Pinsky,} 
			\newblock The first and second fluid approximations to the linearized Boltzmann equation, 
			\newblock \textit{J. Math. Pures Appl.}, {\bf 54} (1975), 125--156.
			
			\bibitem[Gallagher \& Tristani (2020)]{GT2020}
			\textsc{I. Gallagher, I. Tristani}, 
			\newblock On the convergence of smooth solutions from Boltzmann to Navier-Stokes, 
			\newblock \textit{Ann. H. Lebesgue}, {\bf 3} (2020), 561--614.
			
			\bibitem[Gervais (2021)]{G2021}
			\textsc{P. Gervais,}
			\newblock Spectral study of the linearized Boltzmann operator in $L^2$ spaces with polynomial and gaussian weights, 
			\newblock \textit{Kinet. Relat. Models}, {\bf 14} (2021), 725--747. 
			
			\bibitem[Gervais (2023)]{G2023}
			\textsc{P. Gervais,}
			\newblock On the convergence from Boltzmann to Navier-Stokes-Fourier for general initial data,
			\newblock \textit{SIAM J. Math. Anal.}, {\bf 55}  (2023), 805--848.
			
			\bibitem[Gervais \& Lods (2024)]{GL2023}
			\textsc{P. Gervais, B. Lods,}
			\newblock Strong convergence from Boltzmann-Fermi-Dirac equation to Navier-Stokes-Fourier system,
			\newblock work in preparation.
			
			
			\bibitem[Golse \& Saint-Raymond (2004)]{golseSR}
			\textsc{F. Golse, L. Saint-Raymond,} 
			\newblock The Navier-Stokes limit of the Boltzmann equation for bounded collision kernels, 
			\newblock \textit{Invent. Math.}, {\bf 155} (2004), 81--161.
			
			\bibitem[Golse \& Saint-Raymond (2009)]{golseSR1}
			\textsc{F. Golse, L. Saint-Raymond,} 
			\newblock The incompressible Navier-Stokes limit of the Boltzmann equation for hard cutoff potentials,
			\newblock \textit{J. Math. Pures Appl.},  {\bf 91} (2009), 508--552.
			
			\bibitem[Golse \& Saint-Raymond (2005)]{golseSR2}
			\textsc{F. Golse, L. Saint-Raymond,} 
			\newblock Hydrodynamic limits for the Boltzmann equation, 
			\newblock \textit{Riv. Mat. Univ. Parma}, {\bf 7} (2005), 1--144.
			
			\bibitem[Golse (2014)]{golse}
			\textsc{F. Golse},
			\newblock Fluid dynamic limits of the kinetic theory of gases. 
			\newblock \textit{From particle systems to partial differential equations,} 3--91,
			Springer Proc. Math. Stat., 75, Springer, Heidelberg, 2014.
			
			\bibitem[Goudon et~al. (2004a)]{goudona}
			\textsc{T. Goudon, P.-E. Jabin, A. Vasseur,}
			\newblock Hydrodynamic limit for the Vlasov-Navier-Stokes equations. I. Light particles regime,
			\newblock \textit{Indiana Univ. Math. J.}, {\bf 53} (2004),  1495--1515.
			
			\bibitem[Goudon et~al. (2004b)]{goudonb}
			\textsc{T. Goudon, P.-E. Jabin, A. Vasseur,}
			\newblock Hydrodynamic limit for the Vlasov-Navier-Stokes equations. II. Fine particles regime.
			\newblock \textit{Indiana Univ. Math. J.}, {\bf 53} (2004), 1517--1536.
			
			
			
			%\bibitem[Grafakos(2014)]{grafakos}
			%\textsc{L. Grafakos,} 
			%\newblock \textbf{Classical Fourier analysis,} Third edition. 
			%\newblock Graduate Texts in Mathematics, 249. Springer, New York, 2014.
			
			
			
			
			\bibitem[Grad (1963)]{G1963}
			{\sc H. Grad,}
			\newblock Asymptotic theory of the {Boltzmann} equation.
			\newblock {\em Phys. Fluids}, {\bf 6} (1963), 147--181.
			
			\bibitem[Gressman \& Strain (2011)]{GS2011}
			{\sc P.~T., Gressman,  and R.~M. Strain,}
			\newblock Global classical solutions of the {Boltzmann} equation without
			angular cut-off,
			\newblock {\em J. Am. Math. Soc.}, {\bf 24}  (2011), 771--847.
			
			
			
			\bibitem[Gualdani et~al.(2017)]{GMM2017}
			\textsc{M. P. Gualdani, S. Mischler,  C. Mouhot,} 
			\newblock {Factorization for non-symmetric operators and exponential H-theorem,} 
			\newblock \textit{M\'emoires de la SMF}, {\bf 153}, 2017.
			
			
			\bibitem[Guo (2002)]{G2002}
			{\sc Y. Guo,}
			\newblock The {Landau} equation in a periodic box,
			\newblock {\em Commun. Math. Phys.}, {\bf 231} (2002), 391--434.
			
			\bibitem[Guo (2004)]{G2004}
			{\sc Y. Guo,}
			\newblock The {Boltzmann} equation in the whole space,
			\newblock {\em Indiana Univ. Math. J.}, {\bf 53} (2004), 1081--1094.
			
			\bibitem[Guo (2006)]{G2006}
			{\sc Y. Guo,}
			\newblock Boltzmann diffusive limit beyond the {Navier}-{Stokes} approximation,
			\newblock {\em Commun. Pure Appl. Math.}, {\bf 59}  (2006), 626--687.
			
			
			\bibitem[Guo et~al. (2010)]{guo}
			\textsc{Y. Guo, J. Jang, N. Jiang,} 
			\newblock Acoustic limit for the Boltzmann equation in optimal scaling, 
			\newblock \textit{Comm. Pure Appl. Math.}, {\bf 63} (2010), 337--361.
			
			\bibitem[Guo (2016)]{guo2}
			\textsc{Y. Guo}, 
			\newblock $L^{6}$ bound for Boltzmann diffusive limit, 
			\newblock \textit{Ann. Appl. Math.}, {\bf 32} (2016), 249--265.
			
			
			\bibitem[Guo \& Wu (2017)]{GW}
			\textsc{Y. Guo, L. Wu},
\newblock Geometric correction in diffusive limit of neutron transport equation in 2D convex domains,
\newblock \textit{Arch. Ration. Mech. Anal.} {\bf 226} (2017), 321--403.
			
			\bibitem[Han-Kwan \& Michel (2021)]{daniel}
			\textsc{D. Han-Kwan, D. Michel,}
			\newblock On hydrodynamic limits of the Vlasov-Navier-Stokes system,
			\newblock  \textit{Mem. Amer. Math. Soc.}, (2023), 
			
			
			\bibitem[H\'erau et al. (2020)]{HTT2020}
			{\sc F. H{\'e}rau, D. Tonon, and I. Tristani,}
			\newblock Regularization estimates and {Cauchy} theory for inhomogeneous
			{Boltzmann} equation for hard potentials without cut-off,
			\newblock {\em Commun. Math. Phys.}, {\bf 377} (2020), 697--771.
			
			\bibitem[Hilbert (1912)]{H1912}
			{\sc D. Hilbert,}
			\newblock Begr{\"u}ndung der kinetischen {Gastheorie}.
			\newblock {\em Math. Ann.}, {\bf 72} (1912), 562--577.
			
			
			\bibitem[Jiang \& Masmoudi (2017)]{jiang-masm}
			\textsc{N. Jiang, N. Masmoudi,} 
			\newblock Boundary layers and incompressible Navier-Stokes-Fourier limit of the Boltzmann equation in bounded domain I, 
			\newblock \textit{Comm. Pure Appl. Math.}, {\bf 70} (2017), 90--171.
			
			
			\bibitem[Jiang et~al. (2018)]{zhao}
			\textsc{N. Jiang, C.-J. Xu, H. Zhao,} 
			\newblock Incompressible Navier-Stokes-Fourier limit from the Boltzmann equation: classical solutions,
			\newblock \textit{Indiana Univ. Math. J.}, {\bf 67} (2018), 1817--1855.
			
			
\bibitem[Jiang et~al. (2022)]{jiang}
\textsc{N. Jiang, L. Xiong, K. Zhou,}
\newblock The incompressible {N}avier-{S}tokes-{F}ourier limit from
              {B}oltzmann-{F}ermi-{D}irac equation,
\newblock \textit{J. Differential Equations},  {\bf 308} (2022), {77--129}.
	
			
			\bibitem[Kato(1966)]{K1995}
			\newblock {\sc T. Kato}, 
			\newblock \textbf{Perturbation theory for linear operators}, 
			\newblock Classics in Mathematics, Springer Verlag, 1966.
			
			\bibitem[Karper et~al. (2015)]{karper}
			\textsc{T. K. Karper, A. Mellet, K. Trivisa,}
			\newblock Hydrodynamic limit of the kinetic Cucker-Smale flocking model,
			\newblock \textit{Math. Models Methods Appl. Sci.}, {\bf 25} (2015), 131--163.
			
			%\bibitem[Lachowicz (1987)]{lacho}
			%\textsc{M. Lachowicz,} 
			%\newblock On the initial layer and the existence theorem for the nonlinear Boltzmann equation, 
			%\newblock \textit{Math. Methods Appl. Sci.}, {\bf 9} (1987), 342--366.
			
			
			\bibitem[Lemari\'e--Rieusset (2016)]{R2016}
			\textsc{P.-M. Lemari\'e--Rieusset,}
			\newblock \textbf{The Navier-Stokes problem in the 21st century.}
			\newblock {CRC Press, Boca Raton}, FL, 2016.
			
			
			
			\bibitem[Levermore \& Masmoudi(2010)]{lever}
			\newblock\textsc{C. D. Levermore, N. Masmoudi,} 
			\newblock From the Boltzmann equation to an incompressible Navier-Stokes-Fourier system, 
			\newblock \textit{Arch. Ration. Mech. Anal.}, {\bf 196} (2010), 753--809.
			
			%\bibitem[Lions \& Masmoudi(1999)]{lions-masm}
			%\textsc{P.-L. Lions, N. Masmoudi,} 
			%\newblock Une approche locale de la limite incompressible, 
			%\newblock \textit{C. R. Acad. Sci. Paris S\'er. I Math.} {\bf 329} (1999), 387--392.
			
			\bibitem[Lions \& Masmoudi(2001a)]{lions-masm1}
			\newblock\textsc{P.-L. Lions, N. Masmoudi,} 
			\newblock From Boltzmann equation to the Navier-Stokes and Euler equations I, 
			\newblock \textit{Arch. Ration. Mech. Anal.}, {\bf 158} (2001), 173--193. 
			
			
			\bibitem[Lions \& Masmoudi(2001b)] {lions-masm2}
			\textsc{P.-L. Lions, N. Masmoudi,} 
			\newblock From Boltzmann equation to the Navier-Stokes and Euler equations II, 
			\newblock \textit{Arch. Ration. Mech. Anal.}, {\bf 158} (2001), 195--211.
			
			
			\bibitem[Luo \& Yu (2016)]{LY2016}
			\textsc{L. Luo, H. Yu,}
			\newblock  Spectrum analysis of the linearized relativistic Landau equation, 
			\newblock \textit{J. Stat. Phys.}, {\bf 163} (2016),  914--935.
			
			\bibitem[Luo \& Yu (2017)]{LY2017}
			\textsc{L. Luo, H. Yu,}
			\newblock Spectrum analysis of the linear Fokker-Planck equation,
			\newblock \textit{Anal. Appl. (Singap.)}, {\bf 15} (2017), 313--331.
			
			\bibitem[Mellet et~al. (2011)]{Mellet}
			\textsc{A. Mellet, S. Mischler,  C. Mouhot,}
			\newblock Fractional diffusion limit for collisional kinetic equations,
            \newblock \textit{Arch. Ration. Mech. Anal.}, {\bf 199} (2011),  493--525.
			
			%\bibitem[Mouhot (2006)]{M2006}
			%\textsc{C. Mouhot,} 
			%\newblock Rate of convergence to equilibrium for the spatially homogeneous {Boltzmann} equation with hard potentials, 
			%\newblock \textit{Commun. Math. Phys.}, {\bf 261} (2006), 629--672.
			
			
			%\bibitem[Mouhot \& Neumann(2006)]{neuman}
			%\textsc{C. Mouhot, L. Neumann,} 
			%\newblock Quantitative perturbative study of convergence to equilibrium for collisional kinetic models in the torus, 
			%\newblock \textit{Nonlinearity}, {\bf 19} (2006), 969--998.
			
			
			%\bibitem[Mischler \& Mouhot(2016)]{MiMoFP} 
			%\textsc{S. Mischler,  C. Mouhot}, 
			%\newblock Exponential stability of slowly decaying solutions to the kinetic-Fokker-Planck equation, 
			%\newblock \textit{Arch. Ration. Mech. Anal.}, {\bf 221} (2016), 677--723.
			
			\bibitem[Nishida(1978)]{nishida}
			\textsc{T. Nishida,} 
			\newblock Fluid dynamical limit of the nonlinear Boltzmann equation to the level of the compressible Euler equation, 
			\newblock \textit{Comm. Math. Phys.} {\bf 61} (1978), 119--148.
			
			
			\bibitem[Rachid(2021)]{R2021}
			\textsc{M. Rachid},
			\newblock Incompressible Navier-Stokes-Fourier limit from the Landau equation,
			\newblock \textit{Kinet. Relat. Models}, {\bf 14} (2021), 599--638.
			
			
			
			\bibitem[Saint-Raymond(2009)]{SR}
			\textsc{L. Saint-Raymond,} 
			\newblock \textbf{Hydrodynamic limits of the Boltzmann equation,} 
			\newblock Lecture Notes in Mathematics, 1971. Springer-Verlag, Berlin, 2009.
			
			%\bibitem[Saint-Raymond (2009b)]{sr}
			%\textsc{L. Saint-Raymond,} 
			%\newblock Hydrodynamic limits: some improvements of the relative entropy method, 
			%\newblock \textit{Ann. Inst. H. Poincar\'e Anal. Non Lin\'eaire} {\bf 26} (2009), 705--744.
			
			\bibitem[Sone (2002)]{Sone}
			\textsc{Y. Sone,}
			\newblock \textbf{Kinetic theory and fluid dynamics}, 
			\newblock Modeling and Simulation in Science, Engineering and Technology. Birkh\"auser Boston, Inc., Boston, MA, 2002.
			
			
			\bibitem[Tristani(2016)]{T2016}
			\textsc{I. Tristani}, 
			\newblock Boltzmann equation for granular media with thermal force in a weakly inhomogeneous setting, 
			\newblock {\it J. Funct. Anal.}, {\bf 270} (2016), 1922--1970. 
			
			
			\bibitem[Ukai(1974)]{U1974}
			\textsc{S. Ukai,} 
			\newblock On the existence of global solutions of a mixed problem for nonlinear Boltzman equation, 
			\newblock \textit{Proc. Japan Acad.}, {\bf 50} (1974), 179--184.
			
			\bibitem[Villani (2002)]{V2002}
			\textsc{C. Vilani,} 
			\newblock A review of mathematical topics in collisional kinetic theory, 
			\newblock \textit{Handbook of mathematical fluid dynamics.}, {\bf 1} (2002), 71--305.
						
			%\bibitem[Villani(2009)]{hypo}
			%\textsc{C. Villani}
			%\newblock Hypocoercivity,
			%\newblock \textit{Mem. Amer. Math. Soc.}, {\bf 202}, 2009, iv+141 pp.
			
			\bibitem[Yang \& Yu (2016)]{YY}
			\textsc{T. Yang, H. Yu,}
			\newblock  {Spectrum analysis of some kinetic equations},
			\newblock \textit{Arch. Ration. Mech. Anal.}, {\bf 222} (2016), {731--768}.
			
			\bibitem[Yang \& Yu (2023)]{YY23}
			\textsc{T. Yang, H. Yu,}
			\newblock Spectrum structure and decay rate estimates on the Landau equation with Coulomb potential,
			\newblock \textit{Sci. China Math.}, {\bf 66} (2023), 37--78.
			
			\bibitem[Yang \& Zhou (2022a)]{YZ22}
			\textsc{T. Yang, Y-L. Zhou,}
			\newblock An explicit coercivity estimate of the linearized quantum Boltzmann operator without angular cutoff,
			\newblock \textit{J. Funct. Anal.} {\bf 286} (2024), Paper No. 110197, 33 pp.

			
			\bibitem[Yang \& Zhou (2022b)]{Z22}
			\textsc{T. Yang, Y-L. Zhou,}
			\newblock Global well-posedness of the quantum Boltzmann equation for bosons interacting via inverse power law potentials,
			\newblock \textit{preprint}, (2022) \texttt{https://arxiv.org/abs/2210.08428}.
			
			
			
		\end{thebibliography}

	\end{document}